\documentclass[a4paper,oneside,english]{amsart}
\usepackage[T1]{fontenc}
\usepackage[utf8]{inputenc}
\usepackage{babel}
\usepackage{mathtools}
\usepackage{amstext}
\usepackage{amsthm}
\usepackage{amssymb}
\usepackage{esint}
\usepackage[unicode=true,pdfusetitle,
 bookmarks=true,bookmarksnumbered=false,bookmarksopen=false,
 breaklinks=false,pdfborder={0 0 1},backref=false,colorlinks=false]
 {hyperref}

\makeatletter

\pdfpageheight\paperheight
\pdfpagewidth\paperwidth

\numberwithin{equation}{section}
\theoremstyle{plain}
\newtheorem{thm}{\protect\theoremname}[section]
\theoremstyle{plain}
\newtheorem{cor}[thm]{\protect\corollaryname}
\theoremstyle{plain}
\newtheorem{prop}[thm]{\protect\propositionname}
\theoremstyle{remark}
\newtheorem{rem}[thm]{\protect\remarkname}
\theoremstyle{plain}
\newtheorem{lem}[thm]{\protect\lemmaname}

\usepackage{tikz}

\usepackage{tensor}
\usepackage{enumitem}

\definecolor{green}{rgb}{0,0.8,0} 

\newcommand{\nrm}{\@ifstar{\nrmb}{\nrmi}}
\newcommand{\nrmi}[1]{\Vert{#1}\Vert}
\newcommand{\nrmb}[1]{\left\Vert{#1}\right\Vert}
\newcommand{\abs}{\@ifstar{\absb}{\absi}}
\newcommand{\absi}[1]{\vert{#1}\vert}
\newcommand{\absb}[1]{\left\vert{#1}\right\vert}
\newcommand{\brk}{\@ifstar{\brkb}{\brki}}
\newcommand{\brki}[1]{\langle{#1}\rangle}
\newcommand{\brkb}[1]{\left\langle{#1}\right\rangle}
\newcommand{\set}{\@ifstar{\setb}{\seti}}
\newcommand{\seti}[1]{\{#1\}}
\newcommand{\setb}[1]{\left\{ #1\right\}}

\newcommand{\td}[1]{\widetilde{#1}}
\newcommand{\br}[1]{\overline{#1}}
\newcommand{\ul}[1]{\underline{#1}}
\newcommand{\wh}[1]{\widehat{#1}}

\newcommand{\VERT}[1]{{\left\vert\kern-0.25ex\left\vert\kern-0.25ex\left\vert #1 
    \right\vert\kern-0.25ex\right\vert\kern-0.25ex\right\vert}}

\let\Re\relax
\DeclareMathOperator{\Re}{Re}
\let\Im\relax
\DeclareMathOperator{\Im}{Im}

\newcommand{\aeq}{\sim}
\newcommand{\aleq}{\lesssim}
\newcommand{\ageq}{\gtrsim}

\newcommand{\lap}{\Delta}

\newcommand{\ud}{\mathrm{d}}
\newcommand{\rd}{\partial}

\newcommand{\impmi}{\Leftrightarrow}

\newcommand{\peq}{\relphantom{=}}			

\newcommand{\alp}{\alpha}
\newcommand{\bt}{\beta}
\newcommand{\gmm}{\gamma}
\newcommand{\Gmm}{\Gamma}
\newcommand{\dlt}{\delta}

\newcommand{\eps}{\epsilon}

\newcommand{\lmb}{\lambda}
\newcommand{\Lmb}{\Lambda}

\newcommand{\tht}{\theta}

\newcommand{\bfv}{{\bf v}}

\newcommand{\bfD}{{\bf D}}

\newcommand{\bfQ}{{\bf Q}}

\newcommand{\bfS}{{\bf S}}

\newcommand{\bbC}{\mathbb C}

\newcommand{\bbN}{\mathbb N}

\newcommand{\bbQ}{\mathbb Q}
\newcommand{\bbR}{\mathbb R}

\newcommand{\bbZ}{\mathbb Z}

\newcommand{\calA}{\mathcal A}

\newcommand{\calC}{\mathcal C}

\newcommand{\calE}{\mathcal E}
\newcommand{\calF}{\mathcal F}

\newcommand{\calH}{\mathcal H}
\newcommand{\calI}{\mathcal I}

\newcommand{\calK}{\mathcal K}
\newcommand{\calL}{\mathcal L}
\newcommand{\calM}{\mathcal M}
\newcommand{\calN}{\mathcal N}
\newcommand{\calO}{\mathcal O}
\newcommand{\calP}{\mathcal P}

\newcommand{\calR}{\mathcal R}
\newcommand{\calS}{\mathcal S}
\newcommand{\calT}{\mathcal T}
\newcommand{\calU}{\mathcal U}

\newcommand{\calZ}{\mathcal Z}

\newcommand{\To}{\longrightarrow}

\newcommand{\weakto}{\rightharpoonup}
\newcommand{\embed}{\righthookarrow}

\newcommand{\rst}[1]{\left. #1 \right\vert}

\newcommand{\CR}{\bfD_{+}}
\newcommand{\chf}{\boldsymbol{1}}
\newcommand{\Mod}{\mathbf{Mod}}
\newcommand{\out}{\mathrm{out}}
\newcommand{\init}{\mathrm{init}}
\newcommand{\dec}{\mathrm{dec}}
\newcommand{\err}{\mathrm{err}}
\renewcommand{\eps}{\varepsilon}

\newcommand{\tint}[2]{\textstyle \int_{#1}^{#2}}

\providecommand{\corollaryname}{Corollary}
\providecommand{\lemmaname}{Lemma}
\providecommand{\propositionname}{Proposition}
\providecommand{\remarkname}{Remark}
\providecommand{\theoremname}{Theorem}

\makeatother

\providecommand{\corollaryname}{Corollary}
\providecommand{\lemmaname}{Lemma}
\providecommand{\propositionname}{Proposition}
\providecommand{\remarkname}{Remark}
\providecommand{\theoremname}{Theorem}

\begin{document}
\global\long\def\bbC{\mathbb{C}}%
\global\long\def\bbN{\mathbb{N}}%
\global\long\def\bbQ{\mathbb{Q}}%
\global\long\def\bbR{\mathbb{R}}%
\global\long\def\bbZ{\mathbb{Z}}%

\global\long\def\bfD{{\bf D}}%

\global\long\def\bfv{{\bf v}}%

\global\long\def\calA{\mathcal{A}}%
\global\long\def\calC{\mathcal{C}}%
\global\long\def\calE{\mathcal{E}}%
\global\long\def\calF{\mathcal{F}}%
\global\long\def\calH{\mathcal{H}}%
\global\long\def\calI{\mathcal{I}}%
\global\long\def\calK{\mathcal{K}}%
\global\long\def\calL{\mathcal{L}}%
\global\long\def\calM{\mathcal{M}}%
\global\long\def\calN{\mathcal{N}}%
\global\long\def\calO{\mathcal{O}}%
\global\long\def\calP{\mathcal{P}}%
\global\long\def\calR{\mathcal{R}}%
\global\long\def\calS{\mathcal{S}}%
\global\long\def\calT{\mathcal{T}}%
\global\long\def\calU{\mathcal{U}}%
\global\long\def\calZ{\mathcal{Z}}%

\global\long\def\dlt{\delta}%
\global\long\def\eps{\epsilon}%
\global\long\def\gmm{\gamma}%
\global\long\def\Gmm{\Gamma}%
\global\long\def\tht{\theta}%
\global\long\def\lmb{\lambda}%
\global\long\def\Lmb{\Lambda}%

\global\long\def\rd{\partial}%
\global\long\def\aleq{\lesssim}%
\global\long\def\ageq{\gtrsim}%

\global\long\def\peq{\mathrel{\phantom{=}}}%
\global\long\def\To{\longrightarrow}%
\global\long\def\weakto{\rightharpoonup}%
\global\long\def\embed{\hookrightarrow}%
\global\long\def\Re{\mathrm{Re}}%
\global\long\def\Im{\mathrm{Im}}%
\global\long\def\chf{\mathbf{1}}%
\global\long\def\td#1{\widetilde{#1}}%
\global\long\def\br#1{\overline{#1}}%
\global\long\def\ul#1{\underline{#1}}%
\global\long\def\wh#1{\widehat{#1}}%
\global\long\def\tint#1#2{{\textstyle \int_{#1}^{#2}}}%

\global\long\def\CR{\mathbf{D}_{+}}%
\global\long\def\Mod{\mathbf{Mod}}%
\global\long\def\out{\mathrm{out}}%
\global\long\def\init{\mathrm{init}}%
\global\long\def\dec{\mathrm{dec}}%
\global\long\def\err{\mathrm{err}}%

\global\long\def\eps{\varepsilon}%

\title[Blow-up dynamics for CSS]{Blow-up dynamics for \\
 smooth finite energy radial data solutions to the self-dual Chern--Simons--Schrödinger
equation}
\author{Kihyun Kim}
\email{khyun1215@kaist.ac.kr (current address) khyun@ihes.fr}
\address{Department of Mathematical Sciences, Korea Advanced Institute of Science
and Technology, 291 Daehak-ro, Yuseong-gu, Daejeon 34141, Korea\\
(Current address) IHES, 35 route de Chartres, Bures-sur-Yvette 91440,
France}
\author{Soonsik Kwon}
\email{soonsikk@kaist.edu}
\address{Department of Mathematical Sciences, Korea Advanced Institute of Science
and Technology, 291 Daehak-ro, Yuseong-gu, Daejeon 34141, Korea}
\author{Sung-Jin Oh}
\email{sjoh@math.berkeley.edu}
\address{Department of Mathematics, UC Berkeley, Evans Hall 970, Berkeley,
CA 94720-3840, USA and Korea Institue for Advanced Study, 80 Hoegi-ro,
Dongdaemun-gu, Seoul 02455, Korea}
\begin{abstract}
We consider the finite-time blow-up dynamics of solutions to the self-dual
Chern--Simons--Schrödinger (CSS) equation (also referred to as the
Jackiw--Pi model) near the radial soliton $Q$ with the least $L^{2}$-norm
(ground state). While a formal application of pseudoconformal symmetry
to $Q$ gives rise to an $L^{2}$-continuous curve of initial data
sets whose solutions blow up in finite time, they all have infinite
energy due to the slow spatial decay of $Q$. In this paper, we exhibit
initial data sets that are smooth finite energy radial perturbations
of $Q$, whose solutions blow up in finite time. It turns out that
their blow-up rate differs from the pseudoconformal rate by a power
of logarithm. Applying pseudoconformal symmetry in reverse, this also
yields a first example of an infinite-time blow-up solution, whose
blow-up profile contracts at a logarithmic rate.

Our analysis builds upon the ideas of previous works of the first
two authors on (CSS) as well as celebrated works on energy-critical
geometric equations by Merle, Raphaël, and Rodnianski. A notable feature
of this paper is a systematic use of nonlinear covariant conjugations
by the covariant Cauchy--Riemann operators in all parts of the argument.
This not only overcomes the nonlocality of the problem, which is the
principal challenge for (CSS), but also simplifies the structure of
nonlinearity arising in the proof.
\end{abstract}

\maketitle

\tableofcontents{}

\section{Introduction}

The subject of this paper is the nonrelativistic Chern--Simons gauge
field theory introduced by Jackiw--Pi \cite{JackiwPi1990PRL}, which
is a Lagrangian field theory with the action 
\begin{equation}
\calS[\phi,A]\coloneqq\frac{1}{2}\int_{\bbR^{1+2}}A\wedge F+\int_{\bbR^{1+2}}\frac{1}{2}\Im(\br{\phi}\bfD_{t}\phi)+\frac{1}{2}\abs{\bfD_{x}\phi}^{2}-\frac{g}{4}\abs{\phi}^{4}\,dtdx,\label{eq:CSS-action}
\end{equation}
where $\phi:\bbR^{1+2}\to\bbC$ is a complex-valued scalar field,
$\bfD_{\alp}=\rd_{\alp}+iA_{\alp}$ $(\alp=t,1,2)$ are the covariant
derivatives associated with a real-valued $1$-form $A=A_{t}dt+A_{1}dx^{1}+A_{2}dx^{2}$
(connection $1$-form) and $F=dA$ is the corresponding curvature
$2$-form. Note that \eqref{eq:CSS-action} is simply the sum of the
\emph{Chern--Simons action}, $\frac{1}{2}\int A\wedge F$, and the
action for the (gauge-covariant) cubic nonlinear Schrödinger equation.
Following a widespread usage in the mathematical literature, we will
refer the resulting Euler--Lagrange equation, written below in Section~\ref{subsec:CSS},
as the \emph{Chern--Simons--Schrödinger equation}.

The Chern--Simons action has been employed in high energy physics
and condensed matter physics to describe interesting planar physics,
such as topological massive gauge theories and the quantum Hall effect;
we refer to \cite{JackiwPi1990PRD,JackiwPi1990PRL,JackiwPi1991PRD,JackiwPi1992Progr.Theoret.}
for detailed reviews. The model \eqref{eq:CSS-action} under consideration
is of particular interest as it is the simplest model that is nonrelativistic
(which is the setting of condensed matter physics) and, after a particular
choice of the coupling constant $g$ (namely $g=1$), \emph{self-dual}.
A remarkable consequence of the self-duality, which was observed in
the seminal paper of Jackiw--Pi \cite{JackiwPi1990PRL}, is the existence
of explicit(!)~spatially-localized static solutions to the model
(also referred to as \emph{solitons} or \emph{nontopological vortices})
that are parametrized by the solutions to the (explicitly solvable)
Liouville equation. In what follows, we refer to these solutions as
\emph{Jackiw--Pi vortices}.

Most basic among the Jackiw--Pi vortices is the \emph{ground state}
$(\bfQ,A)$, given in the polar coordinates $(r,\tht)$ by 
\begin{equation}
\bfQ(r,\tht)=\sqrt{8}\frac{1}{1+r^{2}},\quad A_{t}=\frac{1}{2}\abs{\bfQ}^{2},\quad A_{r}=0,\quad A_{\tht}=-2\frac{r^{2}}{1+r^{2}},\label{eq:Q-full}
\end{equation}
which has the minimal charge (i.e., the integral of $\abs{\bfQ}^{2}$)
among all Jackiw--Pi vortices. The charge is a natural measure of
the size of a solution, as it is invariant under the scaling symmetry
of \eqref{eq:CSS-action}. The ground state $\bfQ$ plays a pivotal
role in the dynamics of solutions. Indeed, within radial symmetry,
it is known that the $L^{2}$-norm of $\bfQ(x)$ serves as the threshold
for global regularity and scattering \cite{LiuSmith2016}. An outstanding
problem, then, is \emph{to understand the dynamics of solutions associated
to initial data in the vicinity of $\bfQ(x)$, with the $L^{2}$-norm
greater than or equal to that of $\bfQ(x)$.}

In this regime, an interesting formal dynamics describing finite-time
blow-up follows from the pseudoconformal symmetry of \eqref{eq:CSS-cov}.
Like the well-known cubic NLS on $\bbR^{1+2}$, the Chern--Simons--Schrödinger
equation is invariant under the pseudoconformal transformations 
\begin{align*}
(t,x)=(\tfrac{T}{1-bT},\tfrac{X}{1-bT}),\qquad\Phi_{b}(T,X)=\tfrac{1}{1-bT}e^{-ib\frac{\abs{X}^{2}}{1-bT}}\phi(\tfrac{T}{1-bT},\tfrac{X}{1-bT}),
\end{align*}
where $b\in\bbR$. Applying such transformations with $b>0$ to the
ground state, we obtain a one parameter family of solutions $(\bfS_{b},A_{b})$
blowing up in finite time (namely, at $T=b^{-1}$). Each $\bfS_{b}$
has the same $L^{2}$-norm as $\bfQ$ and $\bfS_{b}(t=0)\to\bfQ$
in $L^{2}$ as $b\to0+$. However, because of the slow spatial decay
of $\bfQ$, each $\bfS_{b}$ $(b>0)$ has \emph{infinite} $\dot{H}^{1}$-norm
(as well as infinite conserved energy, which is defined below). As
a result, if we consider the dynamics of \emph{finite energy} solutions
in the vicinity of $\bfQ$, the relevance of $\bfS_{b}$ and even
the possibility of a finite-time blow-up are dubious\footnote{Another standard method to deduce finite-time blow-up is using the
virial identity à la Glassey, but in the self-dual case, it only leads
to a pseudoconformal transform of a static solution; see \cite{KimKwon2019arXiv}.}.

The main result of this paper is the first construction of finite
time blow-up solutions with smooth finite energy radial initial data,
which are arbitrarily close to $\bfQ$ in the $L^{2}$-topology. A
detailed description of the blow-up dynamics is given; in particular,
we provide a codimension one set of data leading to the blow-up, as
well as a sharp description of the rate. The blow-up rate differs
from the pseudoconformal rate by a factor of logarithm. This is a
sharp contrast to the case of higher equivariance indices $m\geq1$,
in which case the pseudoconformal blow-up rate is obtained \cite{KimKwon2020arXiv}.
Interestingly, our blow-up rate is identical to that obtained in the
$1$-equivariant Schrödinger maps \cite{MerleRaphaelRodnianski2013InventMath}.
Via the pseudoconformal transform, we also construct infinite-time
blow-up solutions with the blow-up profile $\bfQ$, whose scale contracts
at a rate logarithmic in $t$.

Our analysis follows the road map furnished by the seminal works of
Rodnianski--Sterbenz \cite{RodnianskiSterbenz2010Ann.Math.}, Raphaël--Rodnianski
\cite{RaphaelRodnianski2012Publ.Math.}, and Merle--Raphaël--Rodnianski
\cite{MerleRaphaelRodnianski2013InventMath} in the cases of wave
maps, Yang--Mills, and Schrödinger maps. Compared to the previously
considered cases, a key challenge in the Chern--Simons--Schrödinger
case is the nonlocality of the nonlinearity, which results in a stronger
soliton-radiation interaction. Notable features of our proof are a
systematic use of nonlinear covariant conjugations, and the treatment
of the self-dual Chern--Simons--Schrödinger equation as a coupled
system of nonlinearly conjugated variables of varying orders. These
ideas provide a simple and efficient way to overcome the nonlocality
of the problem. This point of view pervades all steps of our arguments,
such as the derivation of modified profiles and sharp modulation laws,
decomposition of solutions, and energy estimates. See Section~\ref{subsec:Strategy}
for more details.

\subsection{The self-dual Chern--Simons--Schrödinger equation}

\label{subsec:CSS} The Euler--Lagrange equation for \eqref{eq:CSS-action}
in the self-dual case $g=1$ takes the form 
\begin{equation}
\left\{ \begin{aligned}\bfD_{t}\phi & =i(\bfD_{1}\bfD_{1}+\bfD_{2}\bfD_{2})\phi+i\abs{\phi}^{2}\phi,\\
F_{t1} & =-\Im(\br{\phi}\bfD_{2}\phi),\\
F_{t2} & =\Im(\br{\phi}\bfD_{1}\phi),\\
F_{12} & =-\tfrac{1}{2}\abs{\phi}^{2}.
\end{aligned}
\right.\label{eq:CSS-cov}
\end{equation}
We remind the reader that $\phi:\bbR^{1+2}\to\bbC$ is a complex-valued
scalar field, $\bfD_{\alp}=\rd_{\alp}+iA_{\alp}$ $(\alp=t,1,2)$
are the covariant derivatives associated with a real-valued $1$-form
$A=A_{t}dt+A_{1}dx^{1}+A_{2}dx^{2}$ (connection $1$-form) and $F=dA$
is the corresponding curvature $2$-form. We will refer to this equation
as the (self-dual) \emph{Chern--Simons--Schrödinger} (CSS) equation.

\subsubsection*{Symmetries and conservation laws}

We describe some gauge-covariant symmetries and their associated conservation
laws of \eqref{eq:CSS-cov} that are of importance in the present
work. Each symmetry described here consists of a pre-composition of
$\phi$ with a coordinate transform $(t',x')\mapsto(t,x)$ and a further
transformation of the resulting $\phi(t',x')$. \emph{Gauge covariance}
refers to the feature that the $1$-form $A$ is simply pulled back
by $(t',x')\mapsto(t,x)$.

Among the most basic symmetries are the \emph{time translation symmetry}
\[
(t,x)=(t'+t_{0},x'),\quad\td{\phi}=\phi,\qquad(t_{0}\in\bbR)
\]
and the \emph{phase rotation symmetry} 
\[
(t,x)=(t',x'),\quad\td{\phi}=e^{i\gmm}\phi.\qquad(\gmm\in\bbR)
\]
Associated to these symmetries are the conservation laws for the \emph{energy}
and the \emph{charge}: 
\begin{align*}
E[\phi,A] & \coloneqq\int_{\bbR^{2}}\frac{1}{2}\abs{\bfD_{x}\phi}^{2}-\frac{1}{4}\abs{\phi}^{4}\,dx\\
M[\phi] & \coloneqq\int_{\bbR^{2}}\abs{\phi}^{2}\,dx.
\end{align*}

Next, of particular importance in this work are the \emph{scaling
symmetry}, 
\[
(t,x)=(\lmb^{-2}t',\lmb^{-1}x'),\quad\phi'=\lmb^{-1}\phi,\qquad(\lmb>0)
\]
under which the $L^{2}$-norm (or $M[\phi]$) is invariant, and the
discrete \emph{pseudoconformal symmetry}, 
\begin{equation}
(t,x)=(-\tfrac{1}{t'},\tfrac{x'}{t'}),\quad\phi'(t',x')=\tfrac{1}{t'}e^{i\frac{\abs{x'}^{2}}{4t'}}\phi.\label{eq:discrete-pseudo-transf}
\end{equation}
The aforementioned continuous family of pseudoconformal transformations
arise by composing the discrete version with the symmetries discussed
so far. Associated to these symmetries are the \emph{virial identities}
\[
\left\{ \begin{aligned}\rd_{t}\left(\int_{\bbR^{2}}\abs{x}^{2}\abs{\phi}^{2}dx\right) & =4\int_{\bbR^{2}}x^{j}\Im(\br{\phi}\bfD_{j}\phi)dx,\\
\rd_{t}\left(\int_{\bbR^{2}}x^{j}\Im(\br{\phi}\bfD_{j}\phi)dx\right) & =4E[\phi,A].
\end{aligned}
\right.
\]

In this aspect, \eqref{eq:CSS-cov} shares many similarities with
the cubic NLS $i\rd_{t}\phi+\lap\phi+\abs{\phi}^{2}\phi=0$ on $\bbR^{1+2}$.

\subsubsection*{Self-duality}

The particular choice of the coefficient $g=1$ in front of $\abs{\phi}^{2}\phi$
in \eqref{eq:CSS-cov} makes this system \emph{self-dual}: the minimizers
of the Hamiltonian $E[\phi]$, which turn out to coincide with static
solutions, are characterized by a first order (as opposed to second
order) elliptic equation (see \eqref{eq:bog} below).

We introduce the \emph{covariant Cauchy--Riemann operator} $\CR$
and its formal $L^{2}$-adjoint: 
\[
\CR\coloneqq\bfD_{1}+i\bfD_{2},\qquad\CR^{\ast}=-\bfD_{1}+i\bfD_{2}.
\]
Observe that 
\[
\CR^{\ast}\CR=-\bfD_{1}^{2}-\bfD_{2}^{2}-\tfrac{1}{2}\abs{\phi}^{2}.
\]
As a consequence, the first equation of \eqref{eq:CSS-cov} can be
written in the form 
\begin{equation}
(i\bfD_{t}+\tfrac{1}{2}\abs{\phi}^{2})\phi-\CR^{\ast}\CR\phi=0.\label{eq:CSS-sd}
\end{equation}
Moreover, observe that 
\begin{align*}
\frac{1}{2}\int_{\bbR^{2}}\abs{\CR\phi}^{2}\,dx & =\frac{1}{2}\int_{\bbR^{2}}\Re(\br{\phi}\CR^{\ast}\CR\phi)\,dx\\
 & =-\frac{1}{2}\int_{\bbR^{2}}\Re(\br{\phi}(\bfD_{1}^{2}+\bfD_{2}^{2})\phi)\,dx-\frac{1}{4}\int_{\bbR^{2}}\abs{\phi}^{4}\,dx.
\end{align*}
After an integration by parts, the last line is exactly the conserved
energy of the self-dual(!) Chern--Simons--Schrödinger equation,
i.e., 
\begin{equation}
E[\phi,A]=\frac{1}{2}\int_{\bbR^{2}}\abs{\CR\phi}^{2}\,dx.\label{eq:energy-sd}
\end{equation}
Therefore, the minimum energy is zero, and the energy minimizers obey
the \emph{Bogomol'nyi equation} 
\begin{equation}
\left\{ \begin{aligned}\CR\phi & =0,\\
F_{12} & =-\frac{1}{2}\abs{\phi}^{2}.
\end{aligned}
\right.\label{eq:bog}
\end{equation}
The last property is the manifestation of \emph{self-duality}. Any
zero-energy solution (or equivalently, a solution to \eqref{eq:bog})
is a static (i.e., $\rd_{t}\phi=0$) solution to \eqref{eq:CSS-cov}
with $A_{t}=-\tfrac{1}{2}\abs{\phi}^{2}$. Conversely, any static
solution with $\phi\in H^{1}$ and mild conditions on $A_{t},A_{j}$
(e.g., boundedness) necessarily has zero energy and $A_{t}=-\tfrac{1}{2}\abs{\phi}^{2}$
\cite{HuhSeok2013JMP}.

It was observed by Jackiw--Pi \cite{JackiwPi1990PRD} that, at points
where $\phi$ is nonzero, \eqref{eq:bog} implies that $\abs{\phi}^{2}$
solves the Liouville equation $\lap(\log\abs{\phi}^{2})=-\abs{\phi}^{2}$.
The ground state $\abs{\bfQ}^{2}$ is the unique (up to obvious symmetries)
positive finite charge solution to the Liouville equation \cite{ChouWan1994PacificJMath}.

\subsubsection*{Cauchy problem formulation and the Coulomb gauge}

The equation \eqref{eq:CSS-cov} has \emph{gauge invariance}, i.e.,
for any real-valued function $\chi$ (gauge transformation), if $(\phi,A)$
is a solution, then so is its \emph{gauge transform} $(e^{i\chi}\phi,A-d\chi)$.
Accordingly, uniqueness of a solution to the Cauchy problem may be
formulated only up to gauge invariance. In order to fix gauge invariance
and obtain a (locally) well-posed Cauchy problem, we need to impose
a condition on $A$.

In this paper, we impose the \emph{Coulomb gauge condition}, 
\begin{equation}
\rd_{1}A_{1}+\rd_{2}A_{2}=0,\label{eq:coulomb}
\end{equation}
along with a suitable decay condition for $A(t,x)$ as $\abs{x}\to\infty$
at every $t$ (that will be implicit in the formulae for the components
of $A$ in \eqref{eq:A-quad} below) to rule out nontrivial gauge
transformations. We mention that \eqref{eq:CSS-cov} in Coulomb gauge,
viewed as an evolution equation solely for $\phi$, admits the following
Hamiltonian formulation \cite{JackiwPi1990PRD}: 
\begin{equation}
\rd_{t}\phi=-i\frac{\dlt E[\phi]}{\dlt\phi},\label{eq:CSS-ham}
\end{equation}
where $\frac{\dlt}{\dlt\phi}$ is the Fréchet derivative with respect
to the real inner product $\int_{\bbR^{2}}\Re(\br{\psi}\phi)dx$,
and $E[\phi]$ is the energy with $A$ determined by $\phi$ and the
Coulomb gauge condition.

\subsubsection*{Equivariance within Coulomb gauge}

We begin with a short general discussion of the general equivariance
ansatz for \eqref{eq:CSS-cov}. A complex-valued function $\psi$
on $\bbR^{2}$ is said to be \emph{$m$-equivariant} if 
\begin{equation}
\psi(r,\tht)=e^{im\tht}v(r)\label{eq:def-equivariance}
\end{equation}
for some radial function $v(r)$, which we refer to as the radial
profile of $\psi$. Note that $0$-equivariance is equivalent to radiality.
By \eqref{eq:CSS-cov}, if $\phi$ is $m$-equivariant at a fixed
$t$, then $F_{tr}$, $F_{t\tht}$ and $F_{r\tht}$ are radial. As
the Coulomb gauge condition is also radially symmetric, it follows
that, as long as local wellposedness holds, \eqref{eq:CSS-cov} in
Coulomb gauge preserves $m$-equivariance of $\phi$ for any $m\in\bbZ$.

Under the $m$-equivariance and Coulomb gauge conditions, $A_{t}$,
$A_{r}$, $A_{\tht}$ are radial and the Coulomb gauge condition reduces
to $A_{r}=0$. The radial profile $u$ of $\phi$, defined by 
\[
\phi(t,r,\tht)=e^{im\tht}u(t,r),
\]
obeys 
\begin{equation}
i(\rd_{t}+iA_{t}[u])u+\rd_{r}^{2}u+\frac{1}{r}\rd_{r}u-\frac{1}{r^{2}}(m+A_{\tht}[u])^{2}u+\abs{u}^{2}u=0,\label{eq:CSS-equiv-u}
\end{equation}
where $A_{t}[u]$, $A_{\tht}[u]$ are given by 
\begin{equation}
A_{t}[u]=-\int_{r}^{\infty}(m+A_{\tht})|u|^{2}\frac{dr'}{r'},\qquad A_{\tht}[u]=-\frac{1}{2}\int_{0}^{r}|u|^{2}r'dr'.\label{eq:A-quad}
\end{equation}
We write $A_{\tht}[u,v]=-\frac{1}{2}\int_{0}^{r}\Re(\br uv)r'dr'$
for the real bilinear form obtained by polarization. Using $\rd_{r}A_{t}=F_{rt}$
and $\rd_{r}A_{\tht}=F_{r\tht}$, as well the decay and smoothness
properties of $A_{t}$ and $A_{\tht}$, it may be easily verified
that the connection $1$-form $A$ agrees with $A_{t}[u]dt+A_{\tht}[u]d\tht$.
Equations \eqref{eq:CSS-equiv-u} and \eqref{eq:A-quad} furnish an
evolutionary equation for the radial profile $u$ of an $m$-equivariant
solution $\phi$ to \eqref{eq:CSS-cov} in Coulomb gauge.

The Cauchy--Riemann operator $\bfD_{+}$ maps $m$-equivariant functions
to $(m+1)$-equivariant functions (the standard Cauchy--Riemann operator
$\rd_{+}=\rd_{1}+i\rd_{2}$ has this property and $A_{1}+iA_{2}$
under Coulomb gauge is a $1$-equivariant function). Given $A_{r}=0$
and $A_{\tht}=A_{\tht}[v]$, where $v$ may be the radial profile
of an arbitrary $m'$-equivariant function, the radial Cauchy--Riemann
operator $^{(m)}\bfD_{v}$ acting on an $m$-equivariant function
is defined by the relation 
\[
\bfD_{+}(e^{im\tht}w(r))=e^{i(m+1)\tht}[{}^{(m)}\bfD_{v}w](r),
\]
and takes the form 
\begin{equation}
^{(m)}\bfD_{v}w=\rd_{r}w-\frac{1}{r}(m+A_{\tht}[v])w.\label{eq:CR-radial}
\end{equation}

As observed in \cite{KimKwon2019arXiv}, the nonlinear equation \eqref{eq:CSS-equiv-u}
can be written in a self-dual form. More precisely, the spatial part
of \eqref{eq:CSS-equiv-u}, which is a second-order nonlinear operator,
can be factorized into first-order (nonlinear) operators. For radial
functions $v,w$, we also introduce the notation $^{(m)}L_{v}w$ for
the linearization of the (radial) Bogomol'nyi operator $v\mapsto{}^{(m)}\bfD_{v}v$
around $v$. It may be expressed as 
\begin{equation}
\begin{aligned}^{(m)}L_{v}w & =^{(m)}\bfD_{v}w-\frac{2}{r}A_{\tht}[v,w]v\\
 & =\rd_{r}w-\frac{1}{r}(m+A_{\tht}[v])w+\frac{v}{r}\int_{0}^{r}\Re(\br vw)r'dr'.
\end{aligned}
\label{eq:Lmv}
\end{equation}
As an immediate application of the self-duality \eqref{eq:energy-sd}
and the Hamiltonian formulation \eqref{eq:CSS-ham}, we see that the
evolution equation \eqref{eq:CSS-equiv-u} for $u$ takes the \emph{self-dual
form}: 
\begin{equation}
\rd_{t}u+i{}^{(m)}L_{u}^{\ast}{}^{(m)}\bfD_{u}u=0,\label{eq:CSS-m-self-dual-form}
\end{equation}
where 
\[
^{(m)}L_{u}^{\ast}w=-\rd_{r}w-\frac{1}{r}(m+1+A_{\tht}[u])w+u\int_{r}^{\infty}\Re(\br uw)dr'
\]
is the formal $L^{2}$-adjoint of $^{(m)}L_{u}$.

Finally, for each $m\geq0$, there is an \emph{explicit} $m$-equivariant
Jackiw--Pi vortex, which is unique up to the symmetries of the equation:
\[
Q^{(m)}(r)e^{im\theta}=\sqrt{8}(m+1)\frac{r^{m}}{1+r^{2m+2}}e^{im\theta}.
\]

\subsection{Known results}

A brief discussion of the known results on the Cauchy problem for
\eqref{eq:CSS-cov} is in order. The well-posedness of \eqref{eq:CSS-cov}
was first studied in Coulomb gauge; after the earlier works \cite{BergeDeBouardSaut1995Nonlinearity,Huh2013Abstr.Appl.Anal},
Lim \cite{Lim2018JDE} proved $H^{1}$-local well-posedness. Under
the heat gauge, small data $H^{0+}$ local well-posedness is proved
by Liu--Smith--Tataru \cite{LiuSmithTataru2014IMRN}. Under equivariance
within Coulomb gauge, the equation becomes semilinear and the $L^{2}$-critical
local well-posedness can be achieved; see \cite[Section 2]{LiuSmith2016}.

There are also works on the long-term dynamics. Bergé--de Bouard--Saut
\cite{BergeDeBouardSaut1995Nonlinearity} used Glassey's convexity
argument \cite{Glassey1977JMP} to derive a sufficient condition for
finite-time blow-up. However, this method essentially applies for
negative energy solutions, which exist only if $g>1$. The same authors
\cite{BergeDeBouardSaut1995PRL} carried out a formal computation
to derive the log-log blow-up for negative energy solutions. Recently,
Oh--Pusateri \cite{OhPusateri2015} showed global existence and scattering
for small data in weighted Sobolev spaces. Under equivariance within
Coulomb gauge, Liu--Smith \cite{LiuSmith2016} proved global well-posedness
and scattering below the charge of the ground state, $M[Q^{(m)}]$,
for each equivariance class.

Within each equivariance class, a natural question is the dynamics
beyond the threshold. At the threshold charge, in addition to the
vortex solution $Q^{(m)}$, there is an explicit finite-time blow-up
solution 
\[
S^{(m)}(t,r)=\frac{1}{|t|}Q^{(m)}\Big(\frac{r}{|t|}\Big)e^{-i\frac{r^{2}}{4|t|}},\qquad t<0,
\]
which is obtained by applying the pseudoconformal transform to $Q^{(m)}$.
Recently, the first and second authors gave a quantitative description
of the dynamics in the vicinity of $S^{(m)}(t)$. When $m\geq1$,
the authors in \cite{KimKwon2019arXiv} constructed pseudoconformal
blow-up solutions with a prescribed asymptotic profile. Here, a pseudoconformal
blow-up solution means a finite-time blow-up solution $u$ that decomposes
as $u(t,r)\approx S^{(m)}(t,r)+z(t,r)$ with some regular $z(t,r)$
near the blow-up time. Moreover, they exhibited the \emph{rotational
instability} (see the discussion following \eqref{eq:RotationalInstability})
of these solutions. This is a backward construction, and an analogue
of the construction of Bourgain--Wang solutions and their instability
in the NLS context \cite{BourgainWang1997,MerleRaphaelSzeftel2013AJM}.

On the other hand, when $m\geq1$, the same authors \cite{KimKwon2020arXiv}
studied conditional stability of pseudoconformal blow-up solutions
in the context of the Cauchy problem. Indeed, they considered the
forward construction problem, and constructed a codimension one set
of initial data leading to pseudoconformal blow-up, i.e., 
\[
u(t,r)-\frac{e^{i\gmm^{\ast}}}{\ell(T-t)}Q^{(m)}\Big(\frac{r}{\ell(T-t)}\Big)\to u^{\ast}\quad\text{in }L^{2}
\]
for some $\gmm^{\ast}\in\bbR$ and $\ell\in(0,\infty)$ as $t\to T$.
The blow-up solutions constructed there are smooth and have finite
energy. Moreover, when $m\geq3$, they constructed a codimension one
\emph{Lipschitz manifold of initial data} yielding pseudoconformal
blow-up. In view of \cite{KimKwon2019arXiv}, the codimension one
condition seems to be optimal.

The aforementioned works \cite{KimKwon2019arXiv,KimKwon2020arXiv}
only deal with the $m\geq1$ case. In the current paper, we consider
the most physically relevant (and also delicate) case: $m=0$.

\subsection{Main results}

Now we specialize to the setting of the present paper. Note, from
\eqref{eq:Q-full}, that the ground state $(\bfQ,A)$ is radial ($m=0$)
and obeys the Coulomb gauge condition, with the radial profile 
\begin{equation}
Q(r)\coloneqq Q^{(0)}(r)=\sqrt{8}\frac{1}{1+r^{2}}.\label{eq:Q-formula}
\end{equation}
In the remainder of the paper, \textbf{unless otherwise stated, we
assume that $(\phi,A)$ is a radial solution to \eqref{eq:CSS-cov}
in Coulomb gauge. }Namely, we let $m=0$ and consider 
\[
\phi(t,x)=u(t,r),\quad A_{t}[u]=-\int_{r}^{\infty}A_{\theta}|u|^{2}\frac{dr'}{r'},\quad A_{\theta}[u]=-\frac{1}{2}\int_{0}^{r}|u|^{2}r'dr'.
\]
The equation for $u$ is given by 
\begin{equation}
i(\rd_{t}+iA_{t}[u])u+\rd_{r}^{2}u+\frac{1}{r}\rd_{r}u-\frac{1}{r^{2}}A_{\tht}^{2}[u]u+\abs{u}^{2}u=0.\tag{CSS}\label{eq:CSS-rad-u}
\end{equation}
To simplify the notation, we introduce the following shorthands for
the first two radial Cauchy--Riemann operators: 
\begin{align}
\bfD_{v}w & \coloneqq{}^{(0)}\bfD_{v}w=\rd_{r}w-\frac{1}{r}A_{\tht}[v]w,\label{eq:Dv}\\
A_{v}w & \coloneqq{}^{(1)}\bfD_{v}w=\rd_{r}w-\frac{1}{r}(1+A_{\tht}[v])w.\label{eq:Av}
\end{align}
We also use the shorthand 
\begin{equation}
L_{v}w\coloneqq{}^{(0)}L_{v}w={}^{(0)}\bfD_{v}w-\frac{2}{r}A_{\tht}[v,w]v.\label{eq:Lv}
\end{equation}
Note that $\bfD_{v}$ and $A_{v}$ are local operators, but $L_{v}$
is a nonlocal operator. The aforementioned self-dual form \eqref{eq:CSS-m-self-dual-form}
reads 
\begin{equation}
\rd_{t}u+iL_{u}^{\ast}\bfD_{u}u=0.\label{eq:CSS-self-dual-form}
\end{equation}

One of the fundamental differences between the $m\geq1$ case and
the present case $m=0$ is that $S^{(0)}(t)$ is no longer a finite
energy solution, due to the slow decay of $Q$. Though $S^{(0)}(t)$
provides an example of finite-time blow-up $L^{2}$-solution (with
the pseudoconformal blow-up rate $|t|$), it was left open until now
whether \eqref{eq:CSS-rad-u} possesses a smooth \emph{finite energy}
blow-up solutions. Our main result answers that such solutions do
exist.

By a forward construction, sharper descriptions of the constructed
blow-up solutions can be provided. In fact, we show that there exists
a codimension one set of initial data yielding finite-time blow-up
solutions, whose blow-up rate differs logarithmically from the pseudoconformal
blow-up rate.

We introduce the relevant initial data set and the codimension one
condition. We denote by $H_{0}^{3}$ the Sobolev space $H^{3}(\bbR^{2})$
restricted to radial (i.e., $m=0$ in \eqref{eq:def-equivariance})
functions. For some small $b^{\ast}>0$ and codimension four linear
subspace $\calZ^{\perp}$ (see \eqref{eq:def-Z-perp}) of the radial
Sobolev space $H_{0}^{3}$, let 
\begin{equation}
\td{\calU}_{\init}\coloneqq\{(\lmb_{0},\gmm_{0},b_{0},\eps_{0})\in\bbR_{+}\times\bbR/2\pi\bbZ\times\bbR\times\calZ^{\perp}:b_{0}\in(0,b^{\ast}),\ \|\eps_{0}\|_{H_{0}^{3}}<b_{0}^{3}\}.\label{eq:def-Utilde-init}
\end{equation}
We define the set $\calU_{\init}$ of coordinates 
\begin{equation}
\calU_{\init}\coloneqq\{(\lmb_{0},\gmm_{0},b_{0},\eta_{0},\eps_{0}):(\lmb_{0},\gmm_{0},b_{0},\eps_{0})\in\td{\calU}_{\init},\ \eta_{0}\in(-\tfrac{b_{0}}{2|\log b_{0}|},\tfrac{b_{0}}{2|\log b_{0}|})\}.\label{eq:def-U-init}
\end{equation}
We define the set $\calO_{\init}$ by the set of images: 
\begin{equation}
\calO_{\init}\coloneqq\{\frac{e^{i\gmm_{0}}}{\lmb_{0}}[P(\cdot;b_{0},\eta_{0})+\eps_{0}]\Big(\frac{r}{\lmb_{0}}\Big):(\lmb_{0},\gmm_{0},b_{0},\eta_{0},\eps_{0})\in\calU_{\init}\}\subseteq H_{0}^{3},\label{eq:def-O-init}
\end{equation}
where $P(\cdot;b_{0},\eta_{0})$ is the modified profile defined in
Section~\ref{sec:Modified-profiles} such that $P(\cdot;0,0)=Q$.
It will be shown that the set $\calO_{\init}$ is open, $Q$ lies
in the boundary of $\calO_{\init}$, and the elements of $\calU_{\init}$
serve as coordinates of the elements of $\calO_{\init}$. See Lemma~\ref{lem:decomp}
for more details. The precise statement of our main result is as follows.
\begin{thm}[Smooth finite energy blow-up solutions]
\label{thm:main-thm}There exists $b^{\ast}>0$ with the following
properties. Let $(\widehat{\lmb}_{0},\widehat{\gmm}_{0},\widehat{b}_{0},\widehat{\eps}_{0})\in\td{\calU}_{\init}$.
Then, there exists $\widehat{\eta}_{0}\in(-\frac{\widehat{b}_{0}}{2|\log\widehat{b}_{0}|},\frac{\widehat{b}_{0}}{2|\log\widehat{b}_{0}|})$
such that the solution $u(t,r)$ to \eqref{eq:CSS-rad-u} starting
from the initial data 
\begin{equation}
u_{0}(r)=\frac{e^{i\widehat{\gmm}_{0}}}{\widehat{\lmb}_{0}}[P(\cdot;\widehat{b}_{0},\widehat{\eta}_{0})+\widehat{\eps}_{0}]\Big(\frac{r}{\widehat{\lmb}_{0}}\Big)\in\calO_{\init}\label{eq:initial-u0}
\end{equation}
satisfies: 
\begin{itemize}
\item (Finite-time blow-up) $u$ blows up in finite time $T=T(u_{0})\in(0,\infty)$. 
\item (Sharp description of the blow-up) There exist $\ell=\ell(u_{0})\in(0,\infty)$,
$\gmm^{\ast}=\gmm^{\ast}(u_{0})\in\bbR$, and $u^{\ast}=u^{\ast}(u_{0})\in L^{2}$
such that 
\[
u(t,r)-e^{i\gmm^{\ast}}\frac{|\log(T-t)|^{2}}{\ell(T-t)}Q\Big(\frac{|\log(T-t)|^{2}}{\ell(T-t)}r\Big)\to u^{\ast}\text{ in }L^{2}
\]
as $t\to T$. 
\item (Regularity of the asymptotic profile) $u^{\ast}$ has the regularity
\[
u^{\ast}\in H_{0}^{1}.
\]
\end{itemize}
\end{thm}

Applying the pseudoconformal transform to the solution constructed
in Theorem~\ref{thm:main-thm}, we can construct an infinite-time
blow-up solution to \eqref{eq:CSS-rad-u}. 
\begin{cor}[Infinite-time blow-up]
\label{cor:InfiniteBlowup}There exists a smooth compactly supported
(radial) initial data $u_{0}$ such that the corresponding forward-in-time
solution $u$ to \eqref{eq:CSS-rad-u} blows up in infinite time with
\[
u(t,r)-(\log t)^{2}Q\big((\log t)^{2}r\big)-e^{it\Delta}u^{\ast}\to0\text{ in }L^{2}
\]
as $t\to\infty$, for some $u^{\ast}\in L^{2}$. 
\end{cor}

\subsubsection*{Comments on Theorem~\ref{thm:main-thm} and Corollary~\ref{cor:InfiniteBlowup}}

\ 

1. \emph{Finite energy solutions}. Not only do the constructed blow-up
solutions have finite energy, we can take their initial data to be
smooth and compactly supported. Indeed, the profile $P$ itself does
not have a compact support due to the $Q$-part of $P$. However,
by carefully choosing $\widehat{\eps}_{0}$ to delete the tail of
$Q$, it is possible to make $u_{0}$ compactly supported.

The deviation by a logarithmic factor from the pseudoconformal blow-up
rate stems from the fact that $S^{(0)}(t)$ has infinite energy. In
the context of (NLS), the well-known log-log blow-up rate \cite{MerleRaphael2003GAFA,MerleRaphael2006JAMS},
which deviates by a log-log factor from the self-similar blow-up rate,
comes from the fact that the exact self-similar solution barely fails
to lie in $L^{2}$. A similar remark applies to the wave maps \cite{RaphaelRodnianski2012Publ.Math.}.

2. \emph{Forward construction}. Our method relies on the forward construction
and modulation analysis. When $m\geq1$, $S^{(m)}(t)$ has finite
energy, and the forward construction in the previous work \cite{KimKwon2020arXiv}
yields exactly the pseudoconformal blow-up, which is different from
here. See Section~\ref{subsec:Strategy} for more details on the
forward construction and novel ideas in the present paper. The arguments
used here supercede the old argument in \cite{KimKwon2020arXiv}.
See Remark \ref{rem:Nonlin-ansatz-P-P2}.

3. \emph{Backward construction}. When $m\geq1$, the first two authors
considered the backward construction of blow-up solutions in \cite{KimKwon2019arXiv}.
There, the interaction between the blow-up profile and the asymptotic
profile is weak (though some nontrivial nonlocal interactions lead
to extra phase rotation of the solution) and the blow-up is given
by the pure pseudoconformal blow-up. However, the current $m=0$ case
can be viewed as a strongly interacting regime, as can be seen in
the logarithmic corrections to the blow-up rates in Theorem~\ref{thm:main-thm}
and Corollary~\ref{cor:InfiniteBlowup}. We expect that continuous
blow-up rates as in \cite{KriegerSchlagTataru2008Invent,KriegerSchlagTataru2009Duke,Perelman2014CMP,JendrejLawrieRodriguez2022ASENS}
might be available in the $m=0$ case by a suitable backward construction.

4. \emph{Comparison with the mass-critical NLS}: \eqref{eq:CSS-rad-u}
shares all the symmetries and conservation laws with (NLS). (NLS)
has a standing wave solution $e^{it}R(x)$ with exponentially decaying
profile $R$, but the static solution $Q$ to \eqref{eq:CSS-rad-u}
only shows a polynomial spatial decay. Thanks to the pseudoconformal
symmetry, there are explicit pseudoconformal blow-up solutions like
$S(t)$ in both cases.

In (NLS), there is a stable blow-up regime, the log-log blow-up for
negative energy solutions. However, in \eqref{eq:CSS-rad-u}, the
energy is always non-negative and we believe that stable blow-up regimes
do not exist for \eqref{eq:CSS-rad-u}. Nevertheless, the \emph{non-self-dual}
case $g>1$ is expected to have stable blow-up dynamics as in the
(NLS) case; see \cite{BergeDeBouardSaut1995PRL} for a formal derivation
of the log-log blow-up for negative energy solutions.

Bourgain-Wang type solutions exist and are unstable in both cases.
However, the instability mechanisms differ drastically; we expect
the rotational instability for \eqref{eq:CSS-rad-u}, but the Bourgain-Wang
solutions arise as the border of log-log blow-up solutions and global
scattering solutions. The difference is due to the different spectral
properties of the linearized operator, see \cite{KimKwon2019arXiv}.

One notable feature only arising in the \eqref{eq:CSS-rad-u} case
is that we have a log-corrected pseudoconformal rate due to the slow
spatial decay of $Q$.

5. \emph{Comparison with Schrödinger and wave maps}. \eqref{eq:CSS-rad-u}
has a remarkable parallel with the Schrödinger and wave maps. First,
as observed in \cite{KimKwon2020arXiv}, after a linear conjugation,
the resulting linearized operator is the same as those of wave maps,
Schrödinger maps, and harmonic map heat flows. Second, the first correction
in the profile construction, which is the source of logarithmic correction
to the blow-up rate, is the same as in the wave maps case \cite{RaphaelRodnianski2012Publ.Math.}.
However, the modulation equations are quite different from the wave
maps case, due to the difference between Schrödinger and wave nature
of the equations. In particular, our modulation equations are of the
form $\frac{\lmb_{s}}{\lmb}+b\approx0$ and $b_{s}+b^{2}+\frac{2b^{2}}{|\log b|}\approx0$,
while in the wave maps case, the $b^{2}$ term is missing. Interestingly,
this equation is the same as that of the Schrödinger maps, which gives
arise the same asymptotics for the blow-up rate \cite{MerleRaphaelRodnianski2013InventMath}.
However, the higher order terms in the $b$ and $\eta$ equations,
which are irrelevant to the blow-up rate, are different. Finally,
we note that the blow-up dynamics in the higher equivariance case
has a completely different story from the Schrödinger maps case; \eqref{eq:CSS-rad-u}
has pseudoconformal blow-up solutions for all $m\geq1$ \cite{KimKwon2019arXiv,KimKwon2020arXiv},
but the asymptotic stability is known for $k\geq3$-equivariant Landau--Lifschitz--Gilbert
flows (including both the Schrödinger maps and harmonic map heat flows)
\cite{GustafsonKangTsai2008Duke,GustafsonNakanishiTsai2010CMP}.

6. \emph{Connection with moduli space dynamics}. As pointed out in
\cite{RodnianskiSterbenz2010Ann.Math.} in the case of wave maps,
the approach of this paper may be thought of as a refinement of the
\emph{adiabatic approximation by a moduli space dynamics}, i.e., approximation
of solutions to \eqref{eq:CSS-rad-u} with data close to $Q$ by a
reduced dynamics on the finite dimensional submanifold formed by the
static solutions $\set{\frac{e^{i\gmm}}{\lmb}Q(\frac{\cdot}{\lmb}):\lmb\in(0,\infty),\,\gmm\in\bbR/2\pi\bbZ}$
(moduli space). Our finite-time blow-up solution is formally connected
with an incomplete trajectory on the moduli space (along which $\lmb\to0$).
This subject has a rich tradition of its own; we refer to \cite{DemouliniStuart}
for the study of a model closely related to ours (Manton's model),
and to the monographs \cite{AtiyahHitchin,MantonSutcliffe} for generalities.

7. \emph{Regularity of the asymptotic profile}. We believe, in parallel
to the Schrödinger maps case, that the regularity of the asymptotic
profile in Theorem~\ref{thm:main-thm} is not better than $H^{1}$.
This would require more precise information of the radiation term
and careful measuring of the flux as in \cite{MerleRaphael2005CMP}.
In this sense, we further expect that different blow-up rates will
be obtained from smooth asymptotic profiles, which is typically assumed
in the backward construction problems.

8. \emph{Rotational instability}. The blow-up solutions constructed
in Theorem~\ref{thm:main-thm} and also in \cite{KimKwon2020arXiv}
(when $m\geq1$) are non-generic and obtained in the regime $|\eta|\ll b$.
A natural question is the dynamics near these blow-up solutions.

When $m\geq1$, in view of the modulation equations $\frac{\lmb_{s}}{\lmb}+b=0$,
$\gmm_{s}=(m+1)\eta$, $b_{s}+b^{2}+\eta^{2}\approx0$, and $\eta_{s}\approx0$,
the regime with $\eta(t)\approx\eta_{0}\neq0$ seems to be generic.
In this regime, solutions concentrate to the spatial scale up to $|\eta_{0}|$,
then stop concentrating but exhibit a quick rotation of the phase
by $\mathrm{sgn}(\eta_{0})(m+1)\pi$ on the time inverval of length
$\sim|\eta_{0}|$, and then spread out. This nonlinear scenario is
presented by constructing an explicit one-parameter family of solutions
\cite{KimKwon2019arXiv}. It is conjectured in \cite{KimKwon2020arXiv}
that the aforementioned rotational instability is universal in the
vicinity of pseudoconformal blow-up solutions.

When $m=0$, even the construction of a continuous family of solutions
exhibiting the instability of blow-up solutions (i.e., the analogue
of \cite{KimKwon2019arXiv}) remains as an interesting open question.
In particular, the analysis of the instability mechanism would require
even more refined understanding of the modulation equations; the modulation
equations found in our proof of Theorem~\ref{thm:main-thm} are only
valid under $|\eta|\leq\frac{b}{|\log b|}$ (rotational instability
is turned off) and $b>0$ (shrinking regime).

Rotational instability is also expected in other relevant equations.
Authors in \cite{BergWilliams2013} present formal computations and
numerical evidences for a quick rotation by the angle $\pi$ for the
$1$-equivariant Landau--Lifschitz--Gilbert equation.

\subsection{\label{subsec:Strategy}Strategy of the proof}

We use the notation collected in Section \ref{subsec:Notation}.

We use the forward construction with modulation analysis. We view
solutions $u$ of the form 
\begin{equation}
u(t,r)=\frac{e^{i\gmm(t)}}{\lmb(t)}[P(\cdot;b(t),\eta(t))+\eps(t,\cdot)]\Big(\frac{r}{\lmb(t)}\Big),\label{eq:Strategy-decomp}
\end{equation}
where $P(\cdot;b,\eta)$ is some modified profile with $P(\cdot;0,0)=Q$
and $\eps$ is the error term. The main steps of the proof are the
construction of the modified profiles $P$ and the control of $\eps$.
We use the method of tail computations to construct the modified profile
$P$ and derive the sharp modulation laws of $\lmb,\gmm,b,\eta$.
In order to control $\eps$ forward-in-time, we use a robust energy
method (with repulsivity) to higher order derivatives of $\eps$.

This argument was used to address the forward construction of blow-up
dynamics in various contexts. To list a few, we refer to Rodnianski--Sterbenz
\cite{RodnianskiSterbenz2010Ann.Math.}, Raphaël--Rodnianski \cite{RaphaelRodnianski2012Publ.Math.},
and Merle--Raphaël--Rodnianski \cite{MerleRaphaelRodnianski2013InventMath}
for energy-critical wave maps and Schrödinger maps. We also refer
to \cite{HillairetRaphael2012AnalPDE,RaphaelSchweyer2013CPAM,RaphaelSchweyer2014AnalPDE}
for other energy-critical equations. The method also extends to the
energy-supercritical equations \cite{MerleRaphaelRodnianski2015CambJMath,Collot2017AnalPDE,Collot2018MemAMS}.
For (CSS) with $m\geq1$, the first two authors \cite{KimKwon2020arXiv}
used this argument for blow-up constructions. This list is not exhaustive.
The most relevant ones to this work are \cite{RaphaelRodnianski2012Publ.Math.,MerleRaphaelRodnianski2013InventMath,KimKwon2020arXiv}.

On top of such an existing road map, our main novelty is a systematic
use of nonlinear \emph{covariant conjugation identities} for the self-dual
Chern--Simons--Schrödinger equation. With this strategy, we are
able to overcome most of the difficulties coming from nonlocal nonlinearities.
We use this strategy in \emph{all steps} of the proof.

1. \emph{Covariant conjugation identities}. The main idea is to view
the dynamics not only in the $u$-variable \eqref{eq:CSS-rad-u},
but also in its covariant higher order derivatives of $u$. The reader
may keep in mind that $u$ has a decomposition of the form \eqref{eq:Strategy-decomp}.

Motivated from $\bfD_{Q}Q=0$, we look at the variable $u_{1}=\bfD_{u}u$.
This nonlinear transform hides (or kills) the modulated $Q$ part,
and thus $u_{1}$ enjoys \emph{degeneracy}, i.e., $u_{1}=0$ if $u$
coincides with a modulated $Q$. Moreover, the conjugation via $\bfD_{u}$
behaves very nicely with the original equation \eqref{eq:CSS-rad-u};
$u_{1}$ solves a surprisingly simple equation 
\begin{equation}
\partial_{t}u_{1}+iA_{u}^{\ast}A_{u}u_{1}+\Big(\int_{r}^{\infty}\Re(\overline{u}u_{1})dr'\Big)iu_{1}=0.\label{eq:Strategy-conj1}
\end{equation}
This is the first covariant conjugation identity. This covariant conjugation
shares a similar spirit with the Hasimoto transform \cite{ChangShatahUhlenbeck2000CPAM}
in the Schrödinger maps context, which makes the equation semi-linear.
In the near-soliton dynamics \cite{GustafsonKangTsai2008Duke,GustafsonNakanishiTsai2010CMP},
the Hasimoto transform hides the harmonic map portion of the solution
and leaves out the degenerate variable (the analogue of $u_{1}$).
Because the transform hides the harmonic map portion, the modulation
laws (of the scale and the spatial rotation) can only be dictated
at the map level. In our case, the original equation \eqref{eq:CSS-rad-u}
is used to detect the modulation laws of $\lmb$ and $\gmm$.

The equation \eqref{eq:Strategy-conj1} was derived in \cite{KimKwon2020arXiv},
but it was used in a linearized form, 
\begin{equation}
L_{Q}iL_{Q}^{\ast}=iA_{Q}^{\ast}A_{Q}.\label{eq:Strategy-linear-conj}
\end{equation}
As opposed to $L_{Q}$, which is only $\bbR$-linear and nonlocal,
the operator $A_{Q}$ is $\bbC$-linear and local. Remarkably, the
second order operator $A_{Q}^{\ast}A_{Q}=H_{Q}$ coincides with the
linearized operator arising in Schrödinger maps, wave maps, and harmonic
map heat flows. Experience from these equations further reveals a
hidden monotonicity structure for the linearized dynamics of (CSS),
see for example the repulsivity \eqref{eq:Def-Vtilde} of the operator
$A_{Q}A_{Q}^{\ast}$, which enabled the analysis in \cite{KimKwon2020arXiv}.

Although the nonlinear transform $u\mapsto u_{1}=\bfD_{u}u$ kills
the modulated $Q$ part (which is the degeneracy of $u_{1}$ mentioned
above), the generalized null modes $i\tfrac{r^{2}}{4}Q$ and $\rho$
(see Lemma~\ref{lem:def-rho}) are still alive, in view of $L_{Q}i\tfrac{r^{2}}{4}Q=\frac{1}{2}irQ$
and $L_{Q}\rho=\frac{1}{2}rQ$. We now notice that these generalized
null modes can also be removed if we take further conjugation by $A_{Q}$,
in view of $A_{Q}(rQ)=0$. Motivated from this observation, we consider
the further conjugated variable $u_{2}=A_{u}\bfD_{u}u$ and naturally
expect \emph{further degeneracy of $u_{2}$} over $u_{1}$ in this
linearized context. The further conjugation also behaves very nicely
with the $u_{1}$-equation \eqref{eq:Strategy-conj1} and yields the
following simple equation for $u_{2}$ (the second covariant conjugation
identity): 
\begin{equation}
\partial_{t}u_{2}+iA_{u}A_{u}^{\ast}u_{2}-i\overline{u}(u_{1})^{2}+\Big(\int_{r}^{\infty}\Re(\overline{u}u_{1})dr'\Big)iu_{2}=0.\label{eq:Strategy-conj2}
\end{equation}

In the following analysis, we will view (CSS) as a system of all the
above equations \eqref{eq:CSS-self-dual-form}, \eqref{eq:Strategy-conj1},
and \eqref{eq:Strategy-conj2}, with compatibility conditions $u_{1}=\bfD_{u}u$
and $u_{2}=A_{u}\bfD_{u}u$. We also take advantage of the degeneracies
of the variables $u_{1}$ and $u_{2}$. We note that the derivation
of these equations becomes apparent if we formulate \eqref{eq:CSS-rad-u}
in terms of Wirtinger derivatives, as is done in Section~\ref{sec:Covariant-conjugation-identities}
below.

\ 

2. \emph{Setup for the modulation analysis}. Fix $(\widehat{\lmb}_{0},\widehat{\gmm}_{0},\widehat{b}_{0},\widehat{\eps}_{0})\in\td{\calU}_{\init}$
and let $\eta_{0}$ vary. Consider the initial data 
\[
u_{0}(r)=\frac{e^{i\widehat{\gmm}_{0}}}{\widehat{\lmb}_{0}}[P(\cdot;\widehat{b}_{0},\widehat{\eta}_{0})+\widehat{\eps}_{0}]\Big(\frac{r}{\widehat{\lmb}_{0}}\Big)\in\calO_{\init},
\]
where $P(\cdot;\widehat{b}_{0},\widehat{\eta}_{0})$ is our modified
profile with $P(\cdot;0,0)=Q$ to be introduced in the next step.
The set of four parameters $\lmb,\gmm,b,\eta$ is motivated from the
four dimensional generalized null space of the linearized operator.

We let $u$ be the forward-in-time evolution of $u_{0}$. We will
decompose $u$ as 
\[
u(t,r)=\frac{e^{i\gmm(t)}}{\lmb(t)}[P(\cdot;b(t),\eta(t))+\eps(t,\cdot)]\Big(\frac{r}{\lmb(t)}\Big).
\]
Several issues such as the construction of $P$, fixing the decomposition
(parameters and $\eps$), and the control of $\eps$ forward-in-time,
will be explained on the way.

To analyze the blow-up dynamics, we renormalize the variables by introducing
\[
\frac{ds}{dt}=\frac{1}{\lmb^{2}},\quad y=\frac{r}{\lmb},\quad w(s,y)=\lmb e^{-i\gmm}u(t,\lmb y)|_{t=t(s)}.
\]
Moreover, we renormalize $u_{1}$ and $u_{2}$ in the previous step
by 
\[
w_{1}\coloneqq\bfD_{w}w\qquad\text{and}\qquad w_{2}\coloneqq A_{w}\bfD_{w}w.
\]
The renormalized variables $w$, $w_{1}$, and $w_{2}$ satisfy the
equations \eqref{eq:w-eqn-sd}, \eqref{eq:w1-eqn-sd}, and \eqref{eq:w2-eqn-sd}.
In these equations, we further introduce the \emph{modified phase
parameter} $\td{\gmm}$ with the relation 
\[
\td{\gmm}_{s}\coloneqq\gmm_{s}+\int_{0}^{\infty}\Re(\overline{w}w_{1})dy.
\]
This takes into account some nonlocal interactions leading to an extra
phase rotation of the solutions. In particular, it changes the $\int_{r}^{\infty}$-integral
to a $\int_{0}^{r}$-integral, which is also important to make sense
the tail computation in the next step.

The proof of Theorem~\ref{thm:main-thm} is a combination of bootstrapping
and a topological (connectivity) argument. Smallness of $\eps$ will
be bootstrapped in the regime $|\eta|\leq\frac{b}{|\log b|}$. As
will be explained later, $\eta$ is an unstable parameter and the
regime $|\eta|\leq\frac{b}{|\log b|}$ cannot be bootstrapped; we
show by a connectivity argument that $|\eta|\leq\frac{b}{|\log b|}$
on the maximal forward lifespan is guaranteed for some special initial
choice $\widehat{\eta}_{0}$. Such special solutions are called \emph{trapped
solutions}, and they will be shown to blow up in finite time as described
in Theorem~\ref{thm:main-thm}.

\ 

3. \emph{Modified profile and sharp modulation equations}. The construction
of modified profiles and the derivation of sharp modulation laws are
among the main challenges of this work. In \cite{KimKwon2019arXiv,KimKwon2020arXiv},
the authors introduced a nonlinear profile ansatz, which was an efficient
way to derive pseudoconformal blow-up when $m\geq1$. However, when
$m=0$ this profile ansatz produces an unacceptable profile error.
As we also see a-posteriori from the resulting logarithmically corrected
blow-up rate, it seems that the profiles in \cite{KimKwon2019arXiv,KimKwon2020arXiv}
do not work. Hence we search for sharper modified profiles and modulation
laws.

Since we view the system of $w$, $w_{1}$, $w_{2}$ equations, we
construct modified profiles $P$, $P_{1}$, $P_{2}$ for $w$, $w_{1}$,
$w_{2}$, respectively, and derive sharp modulation equations using
the tail computation (under the adiabatic ansatz $\frac{\lmb_{s}}{\lmb}+b=0$
and $\td{\gmm}_{s}=-\eta$). This strategy, one of our novelties,
remarkably simplifies the rest of the analysis. Indeed, the degeneracies
of $w_{1}$ and $w_{2}$ (explained in Step 1 for the variables $u_{1}$
and $u_{2}$) imply that $P_{1}$ and $P_{2}$ have \emph{degeneracies
in $b$} as follows: $P_{1}=O(b)$, $P_{2}=O(b^{2})$. As a result,
the following simple profile expansions turn out to be sufficient:
\begin{align*}
P & \coloneqq Q+\chi_{B_{1}}\{-ib\tfrac{y^{2}}{4}Q-\eta\rho\},\\
P_{1} & \coloneqq\chi_{B_{1}}\{-(ib+\eta)\tfrac{y}{2}Q\}+\chi_{B_{0}}\{b^{2}T_{2,0}\},\\
P_{2} & \coloneqq\chi_{B_{0}}\{(b^{2}-2ib\eta-\eta^{2})U_{2}+ib^{3}U_{3,0}\},
\end{align*}
for some profiles $T_{2,0},U_{2},U_{3,0}$ and cutoffs $\chi_{B_{0}}$,
$\chi_{B_{1}}$. When we derive $T_{2,0}$ and $U_{2}$, we will see
that the zero resonance $yQ\notin L^{2}$ to the linearized operator
$H_{Q}=A_{Q}^{\ast}A_{Q}$ leads to a logarithmic correction in the
modulation laws, as in \cite{RaphaelRodnianski2012Publ.Math.,MerleRaphaelRodnianski2013InventMath,RaphaelSchweyer2013CPAM}.
In our setting, this is observed in the \emph{$w_{1}$-equation} and
yields the sharp modulation laws: 
\[
b_{s}+b^{2}+\eta^{2}+c_{b}(b^{2}-\eta^{2})=0,\quad\eta_{s}+2c_{b}b\eta=0,
\]
where $c_{b}\approx\frac{2}{|\log b|}$. We remark that it is necessary
to expand $P_{2}$ up to the $b^{3}$-order. However, again thanks
to the degeneracies of $P_{1}$, $P_{2}$, cruder expansions for $P$
and $P_{1}$ suffice.

In order to guarantee a finite-time blow-up, we need $|\eta|\ll b$.
However, in view of $\eta_{s}+2c_{b}b\eta=0$, the trapped regime
$\abs{\eta}\aleq\frac{b}{\abs{\log b}}$ is non-generic. Thus we view
$\eta$ as an unstable parameter.

\ 

4. \emph{Decomposition and propagation of smallness of $\eps$}. Having
defined the profiles $P$, $P_{1}$, $P_{2}$, we decompose our renormalized
solutions $w$, $w_{1}=\bfD_{w}w$, and $w_{2}=A_{w}\bfD_{w}w$ as
\[
w=P+\eps,\qquad w_{1}=P_{1}+\eps_{1},\qquad w_{2}=P_{2}+\eps_{2},
\]
so that $\eps$ satisfies certain orthogonality conditions. The main
novelty is to study the dynamics of $\eps_{1}$ and $\eps_{2}$ that
are defined via higher order (nonlinearly) conjugated variables. Although
$\eps_{1}\approx L_{Q}\eps$ and $\eps_{2}\approx A_{Q}L_{Q}\eps$
at the leading order, $\eps_{1}$ and $\eps_{2}$ are defined in a
nonlinear fashion. We call them \emph{nonlinear adapted derivatives}.
Linear adapted derivatives such as $L_{Q}\eps$ and $A_{Q}L_{Q}\eps$
were used in \cite{KimKwon2020arXiv}, whose idea goes back to the
works \cite{RaphaelRodnianski2012Publ.Math.,MerleRaphaelRodnianski2013InventMath,MerleRaphaelRodnianski2015CambJMath,Collot2018MemAMS}.
Here, by using nonlinear adapted derivatives, the error terms arising
in $\eps_{1}$ and $\eps_{2}$ equations are significantly simplified
compared to the ones obtained by linear adapted derivatives. As we
will see in Section~\ref{subsec:Energy-estimate}, the equation of
$\eps_{2}$ contains only a few error terms of critical size, which
simplifies the energy estimates as well as the Morawetz corrections.

The roles of $\eps$ and $\eps_{1}$-equations are to detect the modulation
laws. We fix the modulation parameters $\lmb,\gmm,b,\eta$ by imposing
four orthogonality conditions. We make a non-standard choice: we impose
two orthogonality conditions on $\eps$, and two on $\eps_{1}$. The
first two are used to detect the modulation equations of $\lmb$ and
$\gmm$; and the other two are used to detect the modulation equations
of $b$ and $\eta$. For the latter, we can take advantage from the
degeneracy $P_{1}=O(b)$ so that the $\eps_{1}$-equation is essentially
decoupled from the modulation equations of $\lmb$ and $\gmm$.

The $\eps_{2}$-equation will be used to propagate the smallness of
$\eps$ (and $\eps_{1}$ and $\eps_{2}$). The main part is to control
a $\dot{H}^{3}$-level quantity of $\eps$; we apply the energy method
to the $\eps_{2}$-equation whose associated energy functional is
$(\eps_{2},A_{Q}A_{Q}^{\ast}\eps_{2})_{r}=\|A_{Q}\eps_{2}\|_{L^{2}}^{2}\eqqcolon\|\eps_{3}\|_{L^{2}}^{2}$.
Here we can use the repulsivity from the operator $A_{Q}A_{Q}^{\ast}$
\eqref{eq:Def-Vtilde} and also the full degeneracy $P_{2}=O(b^{2})$.
In fact, the sole use of the energy functional $\|\eps_{3}\|_{L^{2}}^{2}$
is not sufficient to close the bootstrap, due to some non-perturbative
terms in the $\eps_{2}$-equation. To overcome this difficulty, we
add a Morawetz-type correction to the energy functional $\|\eps_{3}\|_{L^{2}}^{2}$
and observe that the resulting equation error term (still non-perturbative)
has a \emph{good} sign, thanks to $b>0$ and the repulsivity \eqref{eq:Def-Vtilde}
of $A_{Q}A_{Q}^{\ast}$; see \eqref{eq:MorawetzRepulsivity}. A similar
technique was used in \cite{MerleRaphaelRodnianski2013InventMath}.

In the energy/Morawetz estimates, we benefit from the use of the $\eps_{2}$-variable
in a significant way. If one merely proceeds with linear adapted derivatives,
there appear a lot of errors of critical size $O(b\eps)$ in the equation;
see for example the $R_{\mathrm{L-L}}$ term in \cite{KimKwon2020arXiv}.
Thanks to our approach of covariant conjugations, we significantly
reduced the critical errors. In fact, our variable $\eps_{2}$ is
$A_{Q}L_{Q}\eps$ at the leading order, but a lot of $O(b\eps)$ terms
are hidden in $\eps_{2}$. This enables us to choose a Morawetz correction
in a simple form.

\ 

5. \emph{After bootstrapping}. As mentioned above, $\eta$ is an unstable
parameter. We find a special $\eta_{0}$ ensuring that the solution
remains trapped by a soft connectivity argument. The sharp blow-up
rates are obtained by testing against a better approximation of the
generalized kernel elements. The argument in this step is very similar
to that in \cite{MerleRaphaelRodnianski2013InventMath}.

\subsection{\label{subsec:Notation}Notation}

For $A\in\bbC$ and $B>0$, we use the standard asymptotic notation
$A\aleq B$ or $A=O(B)$ to denote the relation $\abs{A}\leq CB$
for some positive constant $C$. The dependencies of $C$ is specified
by subscripts, e.g., $A\aleq_{E}B\impmi A=O_{E}(B)\impmi\abs{A}\leq C(E)B$.
We also introduce the shorthands 
\[
\brk{\cdot}=(1+(\cdot)^{2})^{\frac{1}{2}},\quad\log_{+}(\cdot)=\max\set{0,\log(\cdot)},\quad\log_{-}(\cdot)=\max\set{0,-\log(\cdot)}.
\]
We let $\chi=\chi(x)$ be a smooth spherically symmetric cutoff function
such that $\chi(x)=1$ for $|x|\leq1$ and $\chi(x)=0$ for $|x|\geq2$.
For $A>0$, we define its rescaled version by $\chi_{A}(x)\coloneqq\chi(x/A)$.

Given a function $f:(0,\infty)\to\bbC$, we introduce the shorthand
\[
\int f=\int_{\bbR^{2}}f(\abs{x})dx=2\pi\int f(r)rdr.
\]
For functions $f,g:(0,\infty)\to\bbC$, their \emph{real} $L^{2}$
inner product is given by 
\[
(f,g)_{r}\coloneqq\int\Re(\br fg).
\]
For $s\in\bbR$, let $\Lmb_{s}$ be the infinitesimal generator of
the $\dot{H}^{s}$-invariant scaling, i.e., 
\[
\Lmb_{s}f\coloneqq\rst{\frac{d}{d\lmb}}_{\lmb=1}\lmb^{1-s}f(\lmb\cdot)=(1-s+r\rd_{r})f.
\]
For a nonnegative integer $k$ and a function $f:(0,\infty)\to\bbC$,
we define 
\begin{align*}
\abs{f}_{k}(r) & \coloneqq\sup_{0\leq\ell\leq k}\abs{r^{\ell}\rd_{r}^{\ell}f(r)},\\
\abs{f}_{-k}(r) & \coloneqq\sup_{0\leq\ell\leq k}\abs{r^{-\ell}\rd_{r}^{k-\ell}f(r)}=r^{-k}\abs{f}_{k}.
\end{align*}
For $f:(0,\infty)\to\bbC$, $B>0$ and a norm $\nrm{\cdot}_{X}$,
we write $f=O_{X}(B)$ to denote $\nrm{f}_{X}\aleq B$.

We will use the Laplacian acting on $m$-equivariant functions: $\Delta_{m}=\partial_{rr}+\tfrac{1}{r}\partial_{r}-\frac{m^{2}}{r^{2}}$.
We will also denote $\partial_{+}=\partial_{1}+i\partial_{2}$. If
$\partial_{+}$ acts on $m$-equivariant functions $f(r)e^{im\theta}$,
then $\partial_{+}[f(r)e^{im\theta}]=[\partial_{+}^{(m)}f]e^{i(m+1)\theta}$,
where $\partial_{+}^{(m)}\coloneqq\partial_{r}-\tfrac{m}{r}$. When
the equivariance index $m$ is clear from the context, we use an abuse
of notation $\partial_{+}f=\partial_{+}^{(m)}f$.

We will use two different localization radii for the modified profiles:
\begin{equation}
B_{0}\coloneqq b^{-\frac{1}{2}},\qquad B_{1}\coloneqq b^{-\frac{1}{2}}|\log b|.\label{eq:LocalizationRadii}
\end{equation}

\subsubsection*{Formulas of frequently used linear operators}

We collect the definitions of various linear operators. Let $A_{\tht}[\psi_{1},\psi_{2}]$
be the polarization of $A_{\tht}[\psi]$: 
\[
A_{\theta}[\psi_{1},\psi_{2}]=-\tfrac{1}{2}\tint 0y\Re(\br{\psi_{1}}\psi_{2})y'dy'.
\]
We will often use the first order operators and their formal $L^{2}$-adjoints:
\begin{align*}
\bfD_{w} & =\rd_{y}-\tfrac{1}{y}(m+A_{\tht}[w]), & \bfD_{w}^{\ast} & =-\rd_{y}-\tfrac{1}{y}(m+1+A_{\tht}[w]),\\
L_{w} & =\bfD_{w}-\tfrac{2}{y}A_{\tht}[w,\cdot], & L_{w}^{\ast} & =\bfD_{w}^{\ast}+w\tint y{\infty}\Re(\br w\cdot)dy',\\
A_{w} & =\bfD_{w}-\tfrac{1}{y}, & A_{w}^{\ast} & =\bfD_{w}^{\ast}-\tfrac{1}{y}.
\end{align*}
The second order operators of particular importance are 
\begin{align*}
\calL_{w} & =\nabla^{2}E[w]\text{, i.e., the Hessian of }E,\\
H_{w} & =-\rd_{yy}-\tfrac{1}{y}\rd_{y}+\tfrac{1}{y^{2}}(1+A_{\tht}[w])^{2}-\tfrac{1}{2}|w|^{2}=A_{w}^{\ast}A_{w},\\
\td H_{w} & =-\rd_{yy}-\tfrac{1}{y}\rd_{y}+\tfrac{1}{y^{2}}(2+A_{\tht}[w])^{2}+\tfrac{1}{2}|w|^{2}=A_{w}A_{w}^{\ast}.
\end{align*}
Most frequently, we will use these operators when $w=Q$, where we
have the following convenient relations
\[
\calL_{Q}=L_{Q}^{\ast}L_{Q},\qquad H_{Q}=A_{Q}^{\ast}A_{Q},\qquad\td H_{Q}=A_{Q}A_{Q}^{\ast}.
\]
See Section~\ref{subsec:lin-Q} for more explanations on these linear
operators.

\subsection*{Organization of the paper}

In Section~\ref{sec:Covariant-conjugation-identities}, we introduce
covariant conjugation identities, which provide the algebraic foundation
of the paper. In Section~\ref{sec:Linearized-operators}, we review
the linearization of \eqref{eq:CSS-rad-u}, study outgoing Green's
functions for linearized operators, and construct adapted function
spaces. In Section~\ref{sec:Modified-profiles}, we construct the
modified profiles. In Section~\ref{sec:Trapped-solutions}, we set
up the bootstrap procedure and prove Theorem~\ref{thm:main-thm}
and Corollary~\ref{cor:InfiniteBlowup}. In Appendix~\ref{sec:Adapted-function-spaces},
we prove various facts regarding the adapted function spaces.

\subsection*{Acknowledgements}

K.~Kim and S.~Kwon are partially supported by Samsung Science \&
Technology Foundation BA1701-01 and NRF-2019R1A5A1028324. S.-J.~Oh
is supported by the Samsung Science and Technology Foundation under
Project Number SSTF-BA1702-02, a Sloan Research Fellowship and a NSF
CAREER Grant DMS-1945615. Part of this work was done while K.~Kim
was visiting Bielefeld University through IRTG 2235. He would like
to appreciate its kind hospitality. The authors are grateful to anonymous
referees for their careful reading of this manuscript.

\section{\label{sec:Covariant-conjugation-identities}Covariant conjugation
identities}

As alluded to in the introduction, we will use higher order variables
$u$, $\bfD_{u}u$, and $A_{u}\bfD_{u}u$ obtained by covariant conjugations.
Our goal in this section is to derive the equations satisfied by $u$,
$\bfD_{u}u$, and $A_{u}\bfD_{u}u$, which provide the starting point
for our analysis. We will also need the renormalized variables $w$,
$w_{1}$, and $w_{2}$ of $u$, $\bfD_{u}u$, and $A_{u}\bfD_{u}u$,
respectively, and the equations satisfied by them. To achieve this
goal, we employ a reformulation of \eqref{eq:CSS-cov} in terms of
the Wirtinger derivatives (see \eqref{eq:CSS-curv-wirt}--\eqref{eq:CSS-wirt}),
which is an elegant way to make the self-dual nature of \eqref{eq:CSS-cov}
manifest.

To make clear the generality of the algebraic manipulations we perform,
we proceed in a gauge-covariant, non-radial fashion in Section~\ref{subsec:CSS-wirt},
and only in Section~\ref{subsec:CSS-conj-eq} do we re-impose the
Coulomb gauge condition and radial symmetry.

\subsection{Self-dual Chern--Simons--Schrödinger in terms of Wirtinger derivatives
and covariant conjugation}

\label{subsec:CSS-wirt} To make the self-dual nature of \eqref{eq:CSS-cov}
manifest, it is expedient to rewrite \eqref{eq:CSS-cov} in terms
of the Wirtinger derivatives 
\[
\rd_{z}=\frac{1}{2}\rd_{1}+\frac{1}{2i}\rd_{2},\qquad\rd_{\br z}=\frac{1}{2}\rd_{1}-\frac{1}{2i}\rd_{2}.
\]
Accordingly, given any connection $1$-form (i.e., a real-valued $1$-form)
$A$, we define\footnote{Geometrically, we are simply complexifying the tangent, co-tangent
and the associated tensor bundle over $\bbR_{x^{1},x^{2}}^{2}$ and
using the basis $(dz,d\br z)=(dx^{1}+idx^{2},dx^{1}-idx^{2})$ for
the complexified co-tangent bundle $T_{\bbC}^{\ast}\bbR^{2}$. The
Wirtinger derivatives arise as the dual basis on the complexified
tangent bundle $T_{\bbC}\bbR^{2}$.} 
\begin{align*}
A_{z} & =A(\rd_{z})=\tfrac{1}{2}A_{1}+\tfrac{1}{2i}A_{2}, & A_{\br z} & =A(\rd_{\br z})=\tfrac{1}{2}A_{1}-\tfrac{1}{2i}A_{2},\\
\bfD_{z} & =\rd_{z}+iA_{z}=\tfrac{1}{2}\bfD_{1}+\tfrac{1}{2i}\bfD_{2}, & \bfD_{\br z} & =\rd_{\br z}+iA_{\br z}=\tfrac{1}{2}\bfD_{1}-\tfrac{1}{2i}\bfD_{2}.
\end{align*}
Since the $1$-form $A$ is real-valued, we have $\br{A_{z}}=A_{\br z}$.
For any complex-valued smooth functions $\phi,\psi$, we have 
\[
\rd_{z}(\br{\psi}\phi)=\br{\psi}\bfD_{z}\phi+\br{\bfD_{\br z}\psi}\phi,\qquad\rd_{\br z}(\br{\psi}\phi)=\br{\psi}\bfD_{\br z}\phi+\br{\bfD_{z}\psi}\phi.
\]
The Cauchy--Riemann operator $\CR$ and its adjoint $\CR^{\ast}$
are expressed as 
\begin{equation}
\CR=2\bfD_{\br z},\qquad\CR^{\ast}=-2\bfD_{z}.\label{eq:CR-wirt}
\end{equation}
We note the following anti-commutator relations: 
\begin{align*}
\rd_{z}\rd_{\br z}+\rd_{\br z}\rd_{z} & =2\rd_{\br z}\rd_{z}=\tfrac{1}{2}(\rd_{1}^{2}+\rd_{2}^{2}),\\
\bfD_{z}\bfD_{\br z}+\bfD_{\br z}\bfD_{z} & =\tfrac{1}{2}(\bfD_{1}^{2}+\bfD_{2}^{2}).
\end{align*}
On the other hand, the commutator of two covariant derivatives is
expressed in terms of the curvature tensor. At the level of the curvature
$2$-form $F$, we introduce 
\begin{align*}
F_{tz} & \coloneqq F(\rd_{t},\rd_{z})=\tfrac{1}{2}F_{t1}+\tfrac{1}{2i}F_{t2}, & F_{t\br z} & =F(\rd_{t},\rd_{\br z})=\tfrac{1}{2}F_{t1}-\tfrac{1}{2i}F_{t2},\\
F_{z\br z} & \coloneqq F(\rd_{z},\rd_{\br z})=-\tfrac{1}{2i}F_{12}.
\end{align*}
Since $F$ is real-valued, we have $\br{F_{tz}}=F_{t\br z}$ and $\br{F_{z\br z}}=F_{\br zz}=-F_{z\br z}$.
Clearly, 
\begin{align*}
\bfD_{\alp}\bfD_{\bt}-\bfD_{\bt}\bfD_{\alp}=iF_{\alp\bt},\qquad\rd_{\alp}A_{\bt}-\rd_{\bt}A_{\alp}=F_{\alp\bt},
\end{align*}
for $\alp,\bt\in\set{t,\br z,z}$.

We now write \eqref{eq:CSS-cov} in terms of the Wirtinger derivatives.
The curvature relations in \eqref{eq:CSS-cov} may be rewritten in
the form 
\begin{equation}
\left\{ \begin{aligned}F_{t\br z} & =\tfrac{1}{2}\br{\phi}\bfD_{\br z}\phi-\tfrac{1}{2}\phi\br{\bfD_{z}\phi}=\br{\phi}\bfD_{\br z}\phi-\tfrac{1}{2}\rd_{\br z}\abs{\phi}^{2},\\
F_{z\br z} & =\tfrac{1}{4i}\abs{\phi}^{2}.
\end{aligned}
\right.\label{eq:CSS-curv-wirt-0}
\end{equation}

At this point, observe that $F_{t\br z}$ cleanly splits into a term
involving $\bfD_{\br z}\phi$ and a total derivative $-\tfrac{1}{2}\rd_{\br z}\abs{\phi}^{2}$.
The latter term can be removed by introducing a \emph{modified connection
$1$-form} $\td A$, 
\begin{equation}
\td A=\td A_{t}dt+\td A_{1}dx^{1}+\td A_{2}dx^{2}=(A_{t}-\tfrac{1}{2}\abs{\phi}^{2})dt+A_{1}dx^{1}+A_{2}dx^{2}.\label{eq:CSS-mA}
\end{equation}
Note that the spatial components of $\td A$ and $A$ are the same.
For the associated curvature $\td F=\ud\td A$, \eqref{eq:CSS-curv-wirt-0}
simplifies to
\begin{equation}
\left\{ \begin{aligned}\td F_{t\br z} & =\br{\phi}\td{\bfD}_{\br z}\phi,\\
\td F_{z\br z} & =\tfrac{1}{4i}\abs{\phi}^{2},
\end{aligned}
\right.\label{eq:CSS-curv-wirt}
\end{equation}
where $\td{\bfD}_{\alp}=\rd_{\alp}+i\td A_{\alp}$ is the covariant
derivative associated with $\td A$. Remarkably, with \eqref{eq:CR-wirt},
\eqref{eq:CSS-mA} and \eqref{eq:CSS-curv-wirt}, \eqref{eq:CSS-sd}
simplifies to 
\begin{equation}
i\td{\bfD}_{t}\phi+4\td{\bfD}_{z}\td{\bfD}_{\br z}\phi=0.\label{eq:CSS-wirt}
\end{equation}
Equations \eqref{eq:CSS-curv-wirt} and \eqref{eq:CSS-wirt} constitute
the self-dual Chern--Simons--Schrödinger equation expressed in terms
of the Wirtinger derivatives. By \eqref{eq:CR-wirt}, \eqref{eq:CSS-curv-wirt},
and the fact that $\td{\bfD}_{\br z}=\bfD_{\br z}$ and $\td{\bfD}_{z}=\bfD_{z}$,
the Bogomol'nyi equation may be written as 
\begin{equation}
\left\{ \begin{aligned}\bfD_{\br z}\phi & =0,\\
F_{z\br z} & =\tfrac{1}{4i}\abs{\phi}^{2}.
\end{aligned}
\right.\label{eq:bog-wirt}
\end{equation}

In this formulation, it is straightforward to derive the following
\emph{covariant conjugation identities}, which will play a key role
in the remainder of this paper:
\begin{prop}[Covariant conjugation identities]
\label{prop:cov-conj} Let $\phi,\td A$ obey \eqref{eq:CSS-curv-wirt}
and \eqref{eq:CSS-wirt}. Then 
\begin{align}
i\td{\bfD}_{t}\td{\bfD}_{\br z}\phi+4\td{\bfD}_{z}\td{\bfD}_{\br z}\td{\bfD}_{\br z}\phi & =0,\label{eq:cov-conj-1}\\
i\td{\bfD}_{t}\td{\bfD}_{\br z}\td{\bfD}_{\br z}\phi+4\td{\bfD}_{\br z}\td{\bfD}_{z}\td{\bfD}_{\br z}\td{\bfD}_{\br z}\phi+\br{\phi}(\td{\bfD}_{\br z}\phi)^{2} & =0.\label{eq:cov-conj-2}
\end{align}
\end{prop}

\begin{rem}
The extensive use of the equations in Proposition~\ref{prop:cov-conj}
is one of the key ideas in this work. An immediate advantage of working
with \eqref{eq:CSS-curv-wirt} and \eqref{eq:cov-conj-1} is that,
thanks to \eqref{eq:bog-wirt}, $\td{\bfD}_{\br z}\phi$ vanishes
when $\phi$ is a modulated soliton. As a consequence, the linearization
of $(i\td{\bfD}_{t}+4\td{\bfD}_{z}\td{\bfD}_{\br z})(\td{\bfD}_{\br z}\phi)$
at a modulated soliton does \emph{not} contain any nonlocal terms
in the corresponding linearization of $\td{\bfD}_{\br z}\phi$, which
is a huge simplification over the case of $(i\td{\bfD}_{t}+4\td{\bfD}_{z}\td{\bfD}_{\br z})\phi$.
Moreover, the simplicity of \eqref{eq:cov-conj-2} already suggests
that $\td{\bfD}_{\br z}\td{\bfD}_{\br z}\phi$ is a very convenient
`nonlinear' high-order variable to prove energy estimate for. See
also Remark~\ref{rem:full-degen} below for a further important cancellation
that occurs for $\td{\bfD}_{\br z}\td{\bfD}_{\br z}\phi$ at the linearized
level.

We remark that \eqref{eq:cov-conj-1} was first proved and used in
\cite{KimKwon2020arXiv} in the context of proving higher order energy
estimates. In this work, the use of \eqref{eq:cov-conj-1} and \eqref{eq:cov-conj-2}
pervades the whole argument, namely in the modified profile construction,
the modulation estimate and the key third-order energy estimate. 
\end{rem}

\begin{proof}
To prove \eqref{eq:cov-conj-1}, we simply apply $\td{\bfD}_{\br z}$
to \eqref{eq:CSS-wirt}, then use \eqref{eq:CSS-curv-wirt} to commute
$\td{\bfD}_{\br z}$ inside. As a result, we obtain 
\begin{align*}
0 & =i\td{\bfD}_{t}\td{\bfD}_{\br z}\phi+4\td{\bfD}_{z}\td{\bfD}_{\br z}\td{\bfD}_{\br z}\phi+i[\td{\bfD}_{\br z},\td{\bfD}_{t}]\phi+4[\td{\bfD}_{\br z},\td{\bfD}_{z}]\td{\bfD}_{\br z}\phi\\
 & =i\td{\bfD}_{t}\td{\bfD}_{\br z}\phi+4\td{\bfD}_{z}\td{\bfD}_{\br z}\td{\bfD}_{\br z}\phi+\br{\phi}\td{\bfD}_{\br z}\phi\phi-\abs{\phi}^{2}\td{\bfD}_{\br z}\phi,
\end{align*}
where the last two terms cancel. To prove \eqref{eq:cov-conj-2},
we apply $\td{\bfD}_{\br z}$ to \eqref{eq:cov-conj-1} and commute
it with $\td{\bfD}_{t}$ using \eqref{eq:CSS-curv-wirt}. \qedhere 
\end{proof}
Finally, for the convenience of the reader, we restate the identities
in Proposition~\ref{prop:cov-conj} in terms of $\CR$ and the original
connection using \eqref{eq:CR-wirt} and \eqref{eq:CSS-mA}: 
\begin{align*}
(i\bfD_{t}+\tfrac{1}{2}|\phi|^{2})\bfD_{+}\phi-\bfD_{+}^{\ast}\bfD_{+}\bfD_{+}\phi & =0,\\
(i\bfD_{t}+\tfrac{1}{2}|\phi|^{2})\bfD_{+}\bfD_{+}\phi-\bfD_{+}\bfD_{+}^{\ast}\bfD_{+}\bfD_{+}\phi+\overline{\phi}(\bfD_{+}\phi)^{2} & =0.
\end{align*}
Note that in the original variables, $\td{\bfD}_{\br z}\phi={\bfD}_{\br z}\phi=2\bfD_{+}\phi$
and $\td{\bfD}_{\br z}\td{\bfD}_{\br z}\phi=4\bfD_{+}\bfD_{+}\phi$.
In our analysis, we use $\bfD_{+}\phi$ and $\bfD_{+}\bfD_{+}\phi$
as our conjugated variables.

\subsection{\label{subsec:CSS-conj-eq}Equations in renormalized variables}

Starting from \eqref{eq:CSS-curv-wirt}--\eqref{eq:CSS-wirt}, we
now impose the Coulomb gauge condition $\rd_{1}\td A_{1}+\rd_{2}\td A_{2}=0$
and the radial symmetry ansatz $\phi(t,r,\tht)=u(t,r)$. Since $\td A_{j}=A_{j}$
$(j=1,2)$, we have, as before, 
\[
\td A_{r}=0,\qquad\td A_{\tht}=A_{\tht}[u]=-\frac{1}{2}\int_{0}^{r}\abs{u}^{2}r'dr'.
\]
By $\td A_{r}=0$, the relation $\rd_{r}=e^{-i\tht}\rd_{\br z}+e^{i\tht}\rd_{z}$,
and \eqref{eq:CSS-curv-wirt}, we have 
\begin{align*}
\rd_{r}\td A_{t}=\td F_{rt}=e^{-i\tht}\td F_{\br zt}+e^{i\tht}\td F_{zt}=-2\Re\left(\br{\phi}(e^{-i\tht}\td{\bfD}_{\br z}\phi)\right)=-\Re(\br u\bfD_{u}u).
\end{align*}
Since $\td A_{t}\to0$ as $r\to\infty$, we may integrate from $\infty$
to obtain 
\[
\td A_{t}=\int_{r}^{\infty}\Re(\br u\bfD_{u}u)dr'.
\]
In this setting, \eqref{eq:cov-conj-1} and \eqref{eq:cov-conj-2}
take the form 
\begin{align}
\left(i\rd_{t}-\int_{r}^{\infty}\Re(\br u\bfD_{u}u)dr'\right)\bfD_{u}u-A_{u}^{\ast}A_{u}\bfD_{u}u & =0,\label{eq:noncov-conj-1}\\
\left(i\rd_{t}-\int_{r}^{\infty}\Re(\br u\bfD_{u}u)dr'\right)A_{u}\bfD_{u}u-A_{u}A_{u}^{\ast}A_{u}\bfD_{u}u+\br u(\bfD_{u}u)^{2} & =0,\label{eq:noncov-conj-2}
\end{align}

Next, given modulation parameters $\lmb:I\to(0,\infty)$ and $\gmm:I\to\bbR$,
which we assume to be $C^{1}$, consider the renormalized independent
variables $(s,y)$ and dependent variable $w$ defined by 
\begin{equation}
\frac{ds}{dt}=\frac{1}{\lmb^{2}},\quad y=\frac{r}{\lmb},\quad w(s,y)=\left.\lmb e^{-i\gmm}u(t,\lmb y)\right|_{t=t(s)}.\label{eq:renrm}
\end{equation}
To simplify the notation, in what follows we write $\lmb(s)=\lmb(t(s))$,
$\gmm(s)=\gmm(t(s))$ and so on. The associated nonlinear higher order
variables are defined by (recall \eqref{eq:Dv}--\eqref{eq:Av})
\begin{align}
w_{1} & =^{(0)}\bfD_{w}w=\bfD_{w}w=\lmb^{2}e^{-i\gmm}(\bfD_{u}u)(t,\lmb y)\big|_{t=t(s)},\label{eq:w1-def}\\
w_{2} & =^{(1)}\bfD_{w}w_{1}=A_{w}w_{1}=\lmb^{3}e^{-i\gmm}(A_{u}\bfD_{u}u)(t,\lmb y)\big|_{t=t(s)}.\label{eq:w2-def}
\end{align}
By applying a simple change of variables to \eqref{eq:CSS-self-dual-form},
\eqref{eq:noncov-conj-1} and \eqref{eq:noncov-conj-2}, and rewriting
\[
-\int_{y}^{\infty}\Re(\br w\bfD_{w}w)\,dy'=-\int_{0}^{\infty}\Re(\br w\bfD_{w}w)\,dy'+\int_{0}^{y}\Re(\br w\bfD_{w}w)\,dy',
\]
we obtain the equations of the renormalized variables $w$, $w_{1}$,
and $w_{2}$:
\begin{prop}[Equations in renormalized variables]
\label{prop:CSS-conj-eq} Let $(\phi,A)$ be a solution on $I\times\bbR^{2}$
obeying the Coulomb gauge condition and radial symmetry (see Section
\ref{subsec:CSS}). Given $C^{1}(I)$ modulation parameters $\lmb(t)>0$
and $\gmm(t)\in\bbR$ for $t\in I$, consider the renormalized variables
$(s,y,w)$ and $w_{1},w_{2}$ defined by \eqref{eq:renrm}, \eqref{eq:w1-def},
and \eqref{eq:w2-def}.

Then the renormalized variables $w$, $w_{1}$, and $w_{2}$ obey
the following equations: 
\begin{gather}
(\rd_{s}-\frac{\lmb_{s}}{\lmb}\Lmb+\gmm_{s}i)w+iL_{w}^{\ast}\bfD_{w}w=0,\label{eq:w-eqn-sd}\\
(\rd_{s}-\frac{\lmb_{s}}{\lmb}\Lmb_{-1}+\td{\gmm}_{s}i)w_{1}+iA_{w}^{\ast}A_{w}w_{1}-\left(\int_{0}^{y}\Re(\br ww_{1})dy'\right)iw_{1}=0,\label{eq:w1-eqn-sd}\\
(\rd_{s}-\frac{\lmb_{s}}{\lmb}\Lmb_{-2}+\td{\gmm}_{s}i)w_{2}+iA_{w}A_{w}^{\ast}w_{2}-\left(\int_{0}^{y}\Re(\br ww_{1})dy'\right)iw_{2}-i\br ww_{1}^{2}=0,\label{eq:w2-eqn-sd}
\end{gather}
where 
\[
\td{\gmm}_{s}\coloneqq\gmm_{s}+\int_{0}^{\infty}\Re(\br ww_{1})dy.
\]
\end{prop}

\begin{rem}
\label{rem:gamma-to-gamma-tilde}At the technical level, the reason
for the introduction of the correction $\td{\gmm}_{s}$ is to switch
the domain of the integration in the nonlocal term $\int\Re(\br ww_{1})\,dy'$
from $[y,\infty)$ to $[0,y]$, which is crucial in the ensuing analysis.
Conceptually, the correction $\td{\gmm}_{s}$ contains the dominant
nonlocal effect of the radiation on the soliton, which results in
extra phase rotation of the soliton in the similar spirit of \cite{KimKwon2019arXiv,KimKwon2020arXiv}. 
\end{rem}

\section{\label{sec:Linearized-operators}Linearized operators at $Q$ and
adapted function spaces}

Our goal is to construct a blow-up solution staying close to the modulated
soliton profiles $Q$. After rescaling our solutions, it is necessary
to study the linearized dynamics of \eqref{eq:CSS-rad-u} around $Q$.
In Section~\ref{subsec:lin-Q}, we first review the linearization
of \eqref{eq:CSS-rad-u}. In Section~\ref{subsec:Outgoing-Green's-function},
we construct right inverses of some linear operators that will be
used for the construction of modified profiles. In Section~\ref{subsec:Adapted-function-spaces},
we introduce adapted function spaces and associated coercivity estimates
to be used in the modulation and higher order energy estimates.

\subsection{\label{subsec:lin-Q}Linearization of the Bogomol'nyi equation and
\eqref{eq:CSS-rad-u} at $Q$}

In this subsection, we briefly collect facts about the linearization
of \eqref{eq:CSS-rad-u} around $Q$, which already appeared in \cite[Section~3]{KimKwon2019arXiv}
and \cite[Section~2.1]{KimKwon2020arXiv} (for the case of higher
equivariance case $m\ge1$). Note that we also recorded frequently
used formulas in the notation section for convenience of the readers.

As we have seen in \eqref{eq:CSS-self-dual-form}, we first linearize
the Bogomol'nyi operator $w\mapsto\bfD_{w}w$ and then linearize \eqref{eq:CSS-rad-u}.

Consider the (radial Coulomb-gauge) Bogomol'nyi operator $w\mapsto\bfD_{w}w$.
We write 
\begin{equation}
\bfD_{w+\eps}(w+\eps)=\bfD_{w}w+L_{w}\eps+N_{w}(\eps),\label{eq:LinearizationBogomolnyi}
\end{equation}
where (cf.~\eqref{eq:Lv}) 
\begin{align*}
L_{w}\eps & \coloneqq\bfD_{w}\eps+wB_{w}\eps,\\
N_{w}(\eps) & \coloneqq\eps B_{w}\eps+\tfrac{1}{2}wB_{\eps}\eps+\tfrac{1}{2}\eps B_{\eps}\eps,\\
B_{w}\eps & \coloneqq-\tfrac{2}{y}A_{\tht}[w,\eps]=\tfrac{1}{y}{\textstyle \int_{0}^{y}}\Re(\br w\eps)y'\,dy'.
\end{align*}
The $L^{2}$-adjoint $L_{w}^{\ast}$ of $L_{w}$ takes the form 
\begin{align*}
L_{w}^{\ast}v & =\bfD_{w}^{\ast}v+B_{w}^{\ast}(\br wv),\\
B_{w}^{\ast}v & =w{\textstyle \int_{y}^{\infty}}\Re v\,dy'.
\end{align*}

Next, we linearize \eqref{eq:CSS-rad-u}, which we write in the self-dual
form \eqref{eq:CSS-self-dual-form}: $\rd_{t}u+iL_{u}^{\ast}\bfD_{u}u=0$.
We write 
\[
iL_{w+\eps}^{\ast}\bfD_{w+\eps}(w+\eps)=iL_{w}^{\ast}\bfD_{w}w+\calL_{w}\eps+(\text{h.o.t. in \ensuremath{\eps}}),
\]
where 
\begin{align*}
\calL_{w}\eps & \coloneqq L_{w}^{\ast}L_{w}\eps+\bfD_{w}w(B_{w}\eps)+B_{w}^{\ast}(\br{\eps}\bfD_{w}w)+B_{\eps}^{\ast}(\br w\bfD_{w}w).
\end{align*}
In particular, from $\bfD_{Q}Q=0$, we observe the self-dual factorization
of $i\calL_{Q}$: 
\begin{equation}
i\calL_{Q}=iL_{Q}^{\ast}L_{Q}.\label{eq:calLQ}
\end{equation}
This identity was first observed by Lawrie, Oh, and Shahshahani in
their unpublished note and its derivation can be found in \cite{KimKwon2019arXiv}.
Thus, the linearization of \eqref{eq:CSS-self-dual-form} at $Q$
is 
\begin{equation}
\rd_{t}\eps+i\calL_{Q}\eps=0,\quad\hbox{or}\quad\rd_{t}\eps+iL_{Q}^{\ast}L_{Q}\eps=0.\label{eq:lin-CSS-Q}
\end{equation}

Next, observe that if we linearize \eqref{eq:noncov-conj-1} at $Q$,
then we obtain 
\begin{equation}
\rd_{t}L_{Q}\eps+iA_{Q}^{\ast}A_{Q}L_{Q}\eps=0.\label{eq:LinearizedFlow-1}
\end{equation}
Comparing this equation with the application of $L_{Q}$ to \eqref{eq:lin-CSS-Q}
(as well as the right-invertibility of $L_{Q}$ from Proposition~\ref{prop:LQ-green}
below), we arrive at the remarkable \emph{linearized conjugation identity}
\begin{equation}
iA_{Q}^{\ast}A_{Q}=L_{Q}iL_{Q}^{\ast}.\label{eq:conj-lin}
\end{equation}
This identity was first observed in \cite{KimKwon2020arXiv}. Note
that while $L_{Q}$ and $L_{Q}^{\ast}$ are separately nonlocal operators,
the left-hand side is manifestly local. We introduce the notation
\begin{equation}
H_{Q}:=A_{Q}^{\ast}A_{Q}.\label{eq:HQ}
\end{equation}
Note that while $L_{Q}$, $L_{Q}^{\ast}$ and $\calL_{Q}$ are only
$\bbR$-linear, $\bfD_{Q}$, $A_{Q}$, $H_{Q}$ and their adjoints
are $\bbC$-linear. We further remark that $A_{Q}$ and $H_{Q}$ are
exactly same as those in the wave maps and Schrödinger maps problems,
see \cite[(2.4) and (2.5)]{RaphaelRodnianski2012Publ.Math.} and \cite[(2.11)]{MerleRaphaelRodnianski2013InventMath}.
See also \cite{RaphaelSchweyer2013CPAM,RaphaelSchweyer2014AnalPDE}
for the related harmonic map heat flows.

Finally, we linearize \eqref{eq:noncov-conj-2} at $Q$ to arrive
at 
\begin{equation}
\partial_{t}A_{Q}L_{Q}\eps+iA_{Q}A_{Q}^{\ast}A_{Q}L_{Q}\eps=0.\label{eq:LinearizedFlow-2}
\end{equation}
A crucial fact is that $A_{Q}A_{Q}^{\ast}$ has a \emph{positive repulsive}
potential: 
\[
A_{Q}A_{Q}^{\ast}=-\partial_{yy}-\frac{1}{y}\partial_{y}+\frac{\td V}{y^{2}},
\]
where 
\begin{equation}
\td V=(2+A_{\theta}[Q])^{2}+\tfrac{1}{2}y^{2}Q^{2}\geq0\quad\text{and}\quad-y\partial_{y}\td V\geq0.\label{eq:Def-Vtilde}
\end{equation}
The repulsivity of $A_{Q}A_{Q}^{\ast}$ was first used in \cite{RodnianskiSterbenz2010Ann.Math.}.
This is also used in the Chern--Simons--Schrödinger setting \cite{KimKwon2020arXiv}.
Note that the positivity of $\td V$ is much weaker than that of the
higher equivariance case. Indeed, we have $\td V\sim\langle y\rangle^{-2}$
when $m=0$ but $\td V\sim1$ when $m\geq1$. See Remark~\ref{rem:energy-estimate-weak-repul}
for related discussions.

Next, we consider the kernels of the above linearized operators at
$Q$. Via the phase rotation and scaling symmetries of the Bogomol'nyi
operator, we have 
\begin{equation}
L_{Q}(\Lmb Q)=0,\qquad L_{Q}(iQ)=0.\label{eq:LQ-ker}
\end{equation}
Despite the presence of a nonlocal term, it can be shown that the
$L^{2}$-kernel of $L_{Q}$ is indeed $\mathrm{span}_{\bbR}\set{\Lmb Q,iQ}$;
see \cite[Section~3]{KimKwon2019arXiv}.

For $\bfD_{Q}$, we have 
\begin{equation}
\bfD_{Q}Q=0.\label{eq:DQ-ker}
\end{equation}
Since $\bfD_{Q}$ is first-order, local, and $\bbC$-linear, its $L^{2}$-kernel
is given by $\mathrm{span}_{\bbC}\set{Q}$.

For $A_{w}$, due to $A_{w}(yv)=y\bfD_{w}v$, it follows that 
\begin{equation}
A_{Q}(yQ)=0.\label{eq:AQ-ker}
\end{equation}
As $A_{Q}$ is also first-order, local, and $\bbC$-linear, its formal
(smooth) kernel is $\mathrm{span}_{\bbC}\set{yQ}$. Moreover, by \eqref{eq:HQ},
it follows that 
\begin{equation}
H_{Q}(yQ)=0.\label{eq:HQ-ker}
\end{equation}
However, $yQ\notin L^{2}$; in fact, it is a \emph{resonance} at zero
for the operator $H_{Q}$. Note that there is another formal kernel
element $\Gamma$ of $H_{Q}$ (see Proposition \ref{prop:HQ-green}
below), but it is singular at the origin ($\Gamma\sim y^{-1}$).

We turn to the formal generalized kernel of $i\calL_{Q}$. By \eqref{eq:calLQ},
it follows that 
\begin{equation}
i\calL_{Q}(\Lmb Q)=0,\qquad i\calL_{Q}(iQ)=0,\label{eq:calLQ-ker}
\end{equation}
and that the $L^{2}$-kernel of $i\calL_{Q}$ is $\mathrm{span}_{\bbR}\set{\Lmb Q,iQ}$.
Concerning the formal kernel of $(i\calL_{Q})^{2}$, which is a part
of the formal generalized kernel of $i\calL_{Q}$, we have 
\begin{equation}
i\calL_{Q}(i\tfrac{y^{2}}{4}Q)=\Lmb Q,\qquad i\calL_{Q}\rho=iQ,\label{eq:calLQ-gen-ker}
\end{equation}
where the first identity is easy to verify and $\rho$ is given in
Lemma~\ref{lem:def-rho} below. Note that $i\frac{y^{2}}{4}Q,\rho\notin L^{2}$.
\begin{lem}[Generalized nullspace element $\rho$]
\label{lem:def-rho}\label{lem:calLQ-genker}There exists a unique
smooth function $\rho:(0,\infty)\to\bbR$ satisfying the following
properties: 
\begin{enumerate}
\item (Smoothness on the ambient space) The $m$-equivariant extension $\rho(x)\coloneqq\rho(y)e^{im\theta}$,
$x=ye^{i\theta}$, is smooth on $\bbR^{2}$. 
\item (Equation) $\rho(r)$ satisfies 
\[
L_{Q}\rho=\tfrac{1}{2}yQ\quad\text{and}\quad\calL_{Q}\rho=Q.
\]
\item (Pointwise bounds) We have 
\begin{equation}
|\rho|_{k}\aleq_{k}y^{2}Q,\qquad\forall k\in\bbN.\label{eq:rho-estimate}
\end{equation}
\end{enumerate}
\end{lem}

For the construction of $\rho$ including the $m=0$ case, see \cite[Lemma 3.6]{KimKwon2019arXiv}.
Further properties of $\rho$ can be proved by following the proof
of \cite[Lemma~2.1]{KimKwon2020arXiv} ($m\geq1$ case) with a suitable
modification. Alternatively, we may construct $\rho$ and prove the
preceding lemma by taking $\rho={}^{(\out)}L_{Q}^{-1}(\tfrac{1}{2}yQ)$,
where $^{(\out)}L_{Q}^{-1}$ is defined by \eqref{eq:out-green} and
Proposition~\ref{prop:LQ-green}. We omit the proof.

When $m\geq2$, the following spaces 
\begin{align*}
N_{g}(i\calL_{Q}) & \coloneqq\mathrm{span}_{\bbR}\{\Lambda Q,iQ,iy^{2}Q,\rho\}\subseteq L^{2},\\
N_{g}(\calL_{Q}i)^{\perp} & \coloneqq\{i\rho,y^{2}Q,Q,i\Lambda Q\}^{\perp}\subseteq L^{2}
\end{align*}
are formally invariant under the flow $\partial_{t}+i\calL_{Q}$.
Moreover, we have a clean splitting of $L^{2}$ by 
\[
L^{2}=N_{g}(i\calL_{Q})\oplus N_{g}(\calL_{Q}i)^{\perp}.
\]
In the previous work \cite{KimKwon2020arXiv}, one was motivated by
this splitting to make a decomposition of the form 
\[
u(r)=\frac{e^{i\gmm}}{\lmb}[P(\cdot;b,\eta)+\eps]\Big(\frac{r}{\lmb}\Big),
\]
where the four modulation parameters $\lmb,\gmm,b,\eta$ take into
account the generalized null space elements ($P(\cdot;0,0)=Q$, $\rd_{\lmb=1}\frac{1}{\lmb}P(\frac{\cdot}{\lmb})=-\Lmb P\approx-\Lmb Q$,
$\rd_{\gmm=0}e^{i\gmm}P=iP\approx iQ$, $\partial_{b}P\approx-i\frac{y^{2}}{4}Q$,
and $\partial_{\eta}P\approx-(m+1)\rho$) and $\eps$ belongs to (a
truncated version of) $N_{g}(\calL_{Q}i)^{\perp}$. When $m\in\{0,1\}$,
the above decomposition does not make sense rigorously, but still
suggests a similar decomposition. It also provides a starting point
of the construction of modified profiles $P$.

The following relation among the formal generalized kernel elements
of $i\calL_{Q}$ and the kernel of $A_{Q}$, which may be read off
of \eqref{eq:conj-lin}, is useful: 
\begin{equation}
L_{Q}\rho=\tfrac{1}{2}yQ,\qquad L_{Q}(i\tfrac{y^{2}}{4}Q)=\tfrac{1}{2}iyQ.\label{eq:LQ-gen-ker}
\end{equation}

\subsection{\label{subsec:Outgoing-Green's-function}Outgoing Green's functions}

In this subsection, we construct right inverses of the (radial) linear
operators $L_{Q}$, $A_{Q}$, and $H_{Q}=A_{Q}^{\ast}A_{Q}$. These
can be used in the construction of modified profiles $P$ (more precisely,
the Taylor expansions).\footnote{In fact, it turns out that outgoing Green's functions for $L_{Q}$
are not used in this work. However, we include their construction
as it may be of independent interest. It should be used when one expands
the modified profile $P$ in higher order.}

Since $L_{Q}$, $A_{Q}$, and $H_{Q}$ have nontrivial kernels, their
right inverses are not unique. To fix them, we simply impose a \emph{good}
behavior (degeneracy) near $y=0$. Concretely, for $T\in\set{L_{Q},i^{-1}L_{Q}i,A_{Q},H_{Q}}$
we construct the \emph{outgoing Green's function} $^{(T)}G(y,y')$,
which is characterized by the properties 
\begin{equation}
\begin{aligned}T\left(^{(T)}G(y,y')\right) & =\delta_{y'}(y),\\
^{(T)}G(y,y') & =0\qquad\text{for }0<y<y',
\end{aligned}
\label{eq:out-green}
\end{equation}
for a linear operator $T$ acting on \emph{real-valued} functions
of the variable $y$. The second property of \eqref{eq:out-green}
concerning the support is the \emph{outgoing} property that uniquely
determines the Green's function $^{(T)}G(y,y')$ (see also the propositions
below). Recall that $L_{Q}$ is only $\bbR$-linear. When $L_{Q}$
acts on complex-valued functions, we need to separate the real and
imaginary parts. For the $\bbC$-linear operators $A_{Q}$ and $H_{Q}$,
the same Green's functions still work for complex-valued functions.
The desired right inverse may then be defined as 
\begin{equation}
\big[{}^{(\out)}T^{-1}f\big](y)=\int_{0}^{\infty}{}^{(T)}G(y,y')f(y')\,dy'.\label{eq:out-inverse}
\end{equation}
By the outgoing property, the domain of the integral on the RHS would
be restricted to $\int_{0}^{y}$, which is the \emph{good} behavior
we need.

\subsubsection*{Outgoing Green's function for $A_{Q}$}

We start with $A_{Q}$, which is the simplest. 
\begin{prop}
\label{prop:AQ-green} The outgoing Green's function for $A_{Q}$
takes the form 
\[
^{(A_{Q})}G(y,y')=\chf_{(0,\infty)}(y-y')\frac{yQ(y)}{y'Q(y')}.
\]
\end{prop}

\begin{proof}
We use the variation of constants. Recall, from \eqref{eq:AQ-ker},
that $A_{Q}(yQ)=0$. Making the substitution $^{(A_{Q})}G(y,y')=g_{y'}(y)\frac{yQ(y)}{y'Q(y')}$
in \eqref{eq:out-green}, we obtain 
\begin{align*}
\rd_{y}g_{y'}(y) & =\dlt_{y'}(y),\\
g_{y'}(y) & =0\qquad\hbox{for}\quad0<y<y',
\end{align*}
thus $g_{y'}(y)=\chf_{(0,\infty)}(y-y')$. The desired formula follows.
\qedhere 
\end{proof}

\subsubsection*{Outgoing Green's function for $H_{Q}$}

Next, we consider the second-order operator $H_{Q}=A_{Q}^{\ast}A_{Q}$. 
\begin{prop}
\label{prop:HQ-green} The outgoing Green's function for $H_{Q}$
takes the form 
\[
^{(A_{Q})}G(y,y')=\chf_{(0,\infty)}(y-y')y'\left(J(y)\Gmm(y')-\Gmm(y)J(y')\right),
\]
where 
\[
J(y)=yQ,\quad\Gmm(y)=J\int_{1}^{y}J^{-2}(y')\frac{dy'}{y'}.
\]
For any nonnegative integer $k$, we have 
\[
\abs{J(y)}_{k}\aleq_{k}\begin{cases}
y & \text{if }y\leq1\\
y^{-1} & \text{if }y\geq1
\end{cases},\qquad\abs{\Gmm(y)}_{k}\aleq_{k}\begin{cases}
y^{-1} & \text{if }y\leq1\\
y & \text{if }y\geq1
\end{cases}.
\]
\end{prop}

This is simply the standard construction of Green's function for the
second-order differential operator $A_{Q}^{\ast}A_{Q}=-\rd_{y}^{2}-\frac{1}{y}\rd_{y}+\frac{V}{y^{2}}$
using the fundamental basis consisting of $J$ (recall that $A_{Q}^{\ast}A_{Q}J=0$)
and $\Gmm$, where the latter is obtained by integrating the Wronskian
relation $\Gmm'J-J'\Gmm=\tfrac{1}{y}$ (or, $\partial_{y}(y\Gamma'J-yJ'\Gamma)=0$).
For details, we refer to \cite[Appendix A]{RaphaelRodnianski2012Publ.Math.}
(see also \cite{MerleRaphaelRodnianski2013InventMath,RaphaelSchweyer2013CPAM}),
where exactly the same operator (in the case $k=1$) is considered.

\subsubsection*{Outgoing Green's function for $L_{Q}$}

Finally, we turn to the first-order operator $L_{Q}$, which is most
involved due to its nonlocality. Unlike $A_{Q}$ and $H_{Q}$, the
operator $L_{Q}$ is \emph{not} $\mathbb{C}$-linear; nevertheless,
it is $\bbR$-linear and preserves the real and imaginary parts. Hence,
in order to invert $L_{Q}u=f$ for a complex-valued function $f$,
we need Green's functions for $L_{Q}$ and $i^{-1}L_{Q}i$. 
\begin{prop}
\label{prop:LQ-green} The outgoing Green's functions for $L_{Q}$
and $i^{-1}L_{Q}i$ are 
\begin{align*}
^{(L_{Q})}G(y,y') & =\chf_{(0,\infty)}(y-y')\frac{Q(y)}{Q(y')}I(y,y'),\\
^{(i^{-1}L_{Q}i)}G(y,y') & =\chf_{(0,\infty)}(y-y')\frac{Q(y)}{Q(y')},
\end{align*}
where $I(y,y')$ is smooth on $\set{(y,y'):0<y'<y}$ and obeys the
following properties: 
\begin{enumerate}
\item (Upper bounds) For any nonnegative integer $k$, we have 
\[
\abs{(y\rd_{y})^{k}I(y,y')}\aleq_{k}\begin{cases}
1+\brk{y'}^{-2}\log\left(2+\tfrac{\brk{y}}{\brk{y'}}\right) & \text{if }k=0,\\
\frac{y-y'}{y}\min\set{\frac{y^{2}}{\brk{y}^{2}},\brk{y'}^{-2}} & \text{if }k=1,\\
\frac{y^{2}}{1+y^{4}}\left(1+\brk{y'}^{-2}\log\left(2+\tfrac{\brk{y}}{\brk{y'}}\right)\right) & \text{if }k\geq2.
\end{cases}
\]
\item (Behavior near the diagonal) We have 
\[
\lim_{y-y'\to0+}I(y,y')=1,\qquad\lim_{y-y'\to0+}y\rd_{y}I(y,y')=0.
\]
Moreover, for any nonnegative integer $k$, define $I^{(k)}(y)\coloneqq\lim_{y'\to y-}(y\rd_{y})^{k}I(y,y')$.
For $k\geq2$ and any nonnegative integer $\ell$, we have 
\[
\abs{I^{(k)}(y)}_{\ell}\aleq_{k,\ell}\frac{y^{2}}{1+y^{4}}.
\]
\end{enumerate}
\end{prop}

\begin{proof}
We begin with the simpler case of $i^{-1}L_{Q}i$. For a real-valued
function $u$, 
\[
i^{-1}L_{Q}iu=\bfD_{Q}u=\rd_{y}u-\frac{1}{y}A_{\theta}[Q]u.
\]
In particular, $i^{-1}L_{Q}i$ is a local operator (acted on real-valued
functions). Moreover, recall from \eqref{eq:DQ-ker} that $\bfD_{Q}Q=0$.
Substituting $^{(i^{-1}L_{Q}i)}G(y,y')=g_{y'}(y)\frac{Q(y)}{Q(y')}$,
\eqref{eq:out-green} becomes 
\begin{align*}
\rd_{y}g_{y'} & =\dlt_{y'}(y),\\
g_{y'}(y) & =0\qquad\hbox{for}\quad0<y<y',
\end{align*}
from which the desired expression for $^{(i^{-1}L_{Q}i)}G(y,y')$
follows.

Next, we turn to $L_{Q}$. While $\ker_{\bbR}L_{Q}=\set{\Lmb Q}$,
variation of constants does not work due to the nonlocal integral
term. Instead, we simply make the same substitution $^{(L_{Q})}G(y,y')=I(y,y')\frac{Q(y)}{Q(y')}$
as before, after which \eqref{eq:out-green} becomes 
\begin{equation}
\begin{aligned}\rd_{y}I(y,y')+\frac{1}{y} & \int_{0}^{y}zQ^{2}(z)I(z,y')dz=\dlt_{y'}(y),\\
I(y,y') & =0\quad\hbox{for}\quad0<y<y'.
\end{aligned}
\label{eq:LQ-out-green}
\end{equation}
Integrating from $y=0$ and applying Fubini's theorem, we arrive at
the Volterra-type equation 
\begin{equation}
I(y,y')=\chf_{(0,\infty)}(y-y')-\int_{0}^{y}zQ^{2}\log\frac{y}{z}I(z,y')\,dz.\label{eq:LQ-green-volterra}
\end{equation}
By a standard Picard iteration argument applied to \eqref{eq:LQ-green-volterra},
the existence of a unique solution $I(y,y')$ for $y\in(y',y_{+})$
for some $y_{+}>y'$ follows. Moreover, it is clear that $\lim_{y-y'\to0+}I(y,y')=1$.
Finally, observe that $I(y,y')$ may be extended past $y_{+}$ as
long as $\limsup_{y\to y_{+}-}\abs{I(y,y')}<\infty$.

In order to construct and estimate $I(y,y')$ on the whole interval
$(y',\infty)$, we introduce a parameter $C_{0}>1$ to be fixed below,
and split the argument into the following two cases:

\textbf{Case 1:} $y<2C_{0}$. We may assume that $y'\leq y<\min\set{2C_{0},y_{+}}$,
since $I(y,y')$ is zero otherwise. Then by \eqref{eq:LQ-green-volterra},
\begin{align*}
\abs{I(y,y')}\leq1+\int_{y'}^{y}zQ^{2}\log\frac{y}{z}\abs{I(z,y')}\,dz,
\end{align*}
so by Gronwall's inequality, 
\[
\abs{I(y,y')}\leq\exp\left(\int_{y'}^{y}zQ^{2}\log\frac{y}{z}\,dz\right)\aleq_{C_{0}}1.
\]
In particular, if we take $C_{0}\to\infty$, it already follows that
$I(y,y')$ exists for all $y\in(y',\infty)$. However, the resulting
bound for large $y$'s is bad, so we give a separate argument in that
case as follows.

\textbf{Case 2:} $y>C_{0}$, where $C_{0}$ is a parameter to be fixed
below. By the preceding remark, we may assume that $I(y,y')$ exists
on $y\in(y',\infty)$. In this case, for $\max\set{y',C_{0}}<y$,
we rewrite \eqref{eq:LQ-green-volterra} as 
\begin{align*}
\ I(y,y') & =\underbrace{1-\int_{0}^{C_{0}}zQ^{2}\log\frac{y}{z}I(z,y')\,dz}_{=:g_{0}(\cdot,y')}-\underbrace{\int_{\max\set{y',C_{0}}}^{y}zQ^{2}\log\frac{y}{z}I(z,y')\,dz}_{=:TI(\cdot,y')}.
\end{align*}
Consider the norm 
\[
\nrm{g}\coloneqq\sup_{y>\max\set{y',C_{0}}}\left(1+\brk{C_{0}^{-1}y'}^{-2}\log\left(2+\tfrac{\brk{C_{0}^{-1}y}}{\brk{C_{0}^{-1}y'}}\right)\right)^{-1}\abs{g(y)}.
\]
Observe that $g_{0}=1$ if $y'>C_{0}$ and $\abs{g_{0}}\aleq_{C_{0}}1$
by Case~1 otherwise; hence $\nrm{g}\aleq_{C_{0}}1$. On the other
hand, we claim that 
\begin{equation}
\nrm{Tg}\aleq C_{0}^{-2}\nrm{g}.\label{eq:LQ-green-contract}
\end{equation}
To verify \eqref{eq:LQ-green-contract}, we may normalize $\nrm{g}=1$.
For simplicity, we only consider the case $y'>C_{0}$; the alternative
case may be handled similarly. Since $zQ^{2}\aleq z^{-3}$ on the
domain of integration, we have 
\begin{align*}
\abs*{\int_{y'}^{y}zQ^{2}\log\frac{y}{z}g(z)\,dz} & \aleq\int_{y'}^{y}\frac{1}{z^{3}}\log\frac{y}{z}\left(1+\brk{C_{0}^{-1}y'}^{-2}\log\left(2+\tfrac{\brk{C_{0}^{-1}z}}{\brk{C_{0}^{-1}y'}}\right)\right)\,dz\\
 & \aleq(y')^{-2}\log\frac{y}{y'}\left(1+\brk{C_{0}^{-1}y'}^{-2}\right)\\
 & \aleq C_{0}^{-2}\brk{C_{0}^{-1}y'}^{-2}\log\left(2+\tfrac{\brk{C_{0}^{-1}y}}{\brk{C_{0}^{-1}y'}}\right),
\end{align*}
which proves \eqref{eq:LQ-green-contract}.

By \eqref{eq:LQ-green-contract}, $T$ is a contraction with respect
to $\nrm{\cdot}$ once we fix a large enough $C_{0}>1$. By the contraction
mapping principle, it follows that, 
\[
\abs{I(y,y')}\aleq_{C_{0}}1+\brk{C_{0}^{-1}y'}^{-2}\log\left(2+\tfrac{\brk{C_{0}^{-1}y}}{\brk{C_{0}^{-1}y'}}\right)\aleq_{C_{0}}1+\brk{y'}^{-2}\log\left(2+\tfrac{\brk{y}}{\brk{y'}}\right),
\]
which is the desired bound for $I(y,y')$.

For $y\rd_{y}I(y,y')$, we use the equation 
\begin{equation}
y\rd_{y}I(y,y')=-\int_{0}^{y}zQ^{2}(z)I(z,y')dz\qquad\hbox{for}\quad y'<y,\label{eq:LQ-green-ydy}
\end{equation}
which immediately follows from \eqref{eq:LQ-out-green}. From \eqref{eq:LQ-green-ydy},
$\lim_{y-y'\to0+}y\rd_{y}I(y,y')=0$ is immediate. To verify the asserted
bound for $\abs{y\rd_{y}I(y,y')}$, it suffices to establish 
\[
\abs{y\rd_{y}I(y,y')}\aleq\begin{cases}
(y-y')y' & y'<y\leq2y',\ y'\leq2\\
\frac{y^{2}}{\brk{y}^{2}} & 2y'<y,\ y'\leq2\\
\frac{y-y'}{(y')^{3}} & y'<y\leq2y',\ y'>2\\
(y')^{-2} & 2y'<y,\ y'>2,
\end{cases}
\]
each of which is a straightforward consequence of \eqref{eq:LQ-green-ydy},
$\abs{zQ^{2}}\aleq\frac{z}{1+z^{4}}$ and the preceding bound for
$I(y,y')$. Finally, the assertions concerning $(y\rd_{y})^{k}I(y,y')$
follow in an inductive manner from 
\[
(y\rd_{y})^{2}I(y,y')=-y^{2}Q^{2}(y)I(y,y')\qquad\hbox{for}\quad y'<y,
\]
which is obtained by differentiating \eqref{eq:LQ-green-ydy}; we
omit the details. \qedhere 
\end{proof}

\subsection{\label{subsec:Adapted-function-spaces}Adapted function spaces}

In this subsection, we briefly review the definitions of equivariant
Sobolev spaces $H_{m}^{k}$ and construct adapted function spaces
$\dot{\calH}_{0}^{1}$, $\dot{\calH}_{2}^{1}$, $\dot{\calH}_{1}^{2}$,
and $\dot{\calH}_{0}^{3}$. These function spaces are designed to
have (sub-)coercivity estimates of the linear operators $L_{Q}$,
$A_{Q}$, and $A_{Q}^{\ast}$ at various levels of regularity. Moreover,
since $L_{Q}$ and $A_{Q}$ shift the equivariance index by $1$,
and $A_{Q}^{\ast}$ shifts the equivariance index by $-1$ when viewed
as acting on functions on the ambient space $\bbR^{2}$, we need to
handle various equivariance indices too.

\subsubsection*{Equivariant Sobolev spaces}

Perhaps a natural starting point is to consider equivariant Sobolev
spaces. Let $m\geq0$. Given an $m$-equivariant function $f$ (see
\eqref{eq:def-equivariance} for the definition), we will often identify
it with its \emph{radial part} $g:\bbR_{+}\to\bbC$, i.e. $f(x)=g(r)e^{im\theta}$,
under the usual polar coordinates relation $x_{1}+ix_{2}=re^{i\theta}$.
We often consider $g$ as an $m$-equivariant function, i.e. we say
that $g$ belongs to some $m$-equivariant function space if its $m$-equivariant
extension belongs to that.

For $s\geq0$, we denote by $H_{m}^{s}$ the set of $m$-equivariant
$H^{s}(\bbR^{2})$ functions. The set of $m$-equivariant Schwartz
functions is denoted by $\calS_{m}$. The $H_{m}^{s}$-norm and $\dot{H}_{m}^{s}$-norm
mean the usual $H^{s}(\bbR^{2})$-norm and $\dot{H}^{s}(\bbR^{2})$-norm,
but the subscript $m$ indicates the equivariance index. When $0\leq k\leq m$,
we have \emph{generalized Hardy's inequality} \cite[Lemma A.7]{KimKwon2019arXiv}:
\begin{equation}
\|\sup_{0\leq\ell\leq k}|r^{-\ell}\rd_{r}^{k-\ell}f|\|_{L^{2}}=\||f|_{-k}\|_{L^{2}}\sim_{k,m}\|f\|_{\dot{H}_{m}^{k}},\qquad\forall f\in\calS_{m}.\label{eq:GenHardyAppendix-1}
\end{equation}
In addition, when $m\geq1$ and $k=1$, we have the \emph{Hardy-Sobolev
inequality} \cite[Lemma A.6]{KimKwon2019arXiv}: 
\begin{equation}
\|r^{-1}f\|_{L^{2}}+\|f\|_{L^{\infty}}\aleq\|f\|_{\dot{H}_{m}^{1}}.\label{eq:HardySobolevAppendix-1}
\end{equation}
As is well known, \eqref{eq:HardySobolevAppendix-1} \emph{fails}
when $m=0$, but we can have a logarithmically weakened version of
it; see \eqref{eq:LogHardy}. The generalized Hardy's inequality \eqref{eq:GenHardyAppendix-1}
allows us define the space $\dot{H}_{m}^{k}$ when $0\leq k\leq m$
by taking the completion of $\calS_{m}$ under the $\dot{H}_{m}^{k}$-norm,
with the embedding properties 
\[
\calS_{m}\hookrightarrow H_{m}^{k}\hookrightarrow\dot{H}_{m}^{k}\hookrightarrow L_{\mathrm{loc}}^{2}.
\]

\subsubsection*{Adapted function spaces}

As alluded to above, we will track the dynamics of $w$, $w_{1}=\bfD_{w}w$,
and $w_{2}=A_{w}w_{1}=A_{w}\bfD_{w}w$: see the equations \eqref{eq:w-eqn-sd},
\eqref{eq:w1-eqn-sd}, and \eqref{eq:w2-eqn-sd}. The related linearized
equations are \eqref{eq:lin-CSS-Q}, \eqref{eq:LinearizedFlow-1},
and \eqref{eq:LinearizedFlow-2}, respectively. Thus we need to handle
adapted derivatives $L_{Q}\eps$, $A_{Q}L_{Q}\eps$, and so on. Here
we investigate how these derivatives control the original $\eps$.
The preceding equivariant Sobolev spaces do not work very well with
those adapted derivatives. We need to introduce new \emph{adapted
function spaces} $\dot{\calH}_{m}^{k}$, which are slightly modified
from the original equivariant Sobolev spaces $\dot{H}_{m}^{k}$. More
precisely, we will obtain (sub-)coercivity properties of $L_{Q}$,
$A_{Q}$, and $A_{Q}^{\ast}$ in terms of $\dot{\calH}_{m}^{k}$-norms.

We define the $\dot{\calH}_{m}^{k}$-norms for $(k,m)\in\{(1,0),(1,2),(2,1),(3,0)\}$
by (recall $\log_{\pm}r=\max\{0,\pm\log r\}$)
\begin{gather*}
\|v\|_{\dot{\calH}_{0}^{1}}\coloneqq\|\partial_{r}v\|_{L^{2}}+\|r^{-1}\langle\log_{-}r\rangle^{-1}v\|_{L^{2}},\\
\|v\|_{\dot{\calH}_{2}^{1}}\coloneqq\|\partial_{r}v\|_{L^{2}}+\|r^{-1}\langle\log_{+}r\rangle^{-1}v\|_{L^{2}},\\
\|v\|_{\dot{\calH}_{1}^{2}}\coloneqq\|\partial_{rr}v\|_{L^{2}}+\|r^{-1}\langle\log r\rangle^{-1}|v|_{-1}\|_{L^{2}},\\
\|v\|_{\dot{\calH}_{0}^{3}}\coloneqq\|\partial_{rrr}v\|_{L^{2}}+\|r^{-1}\langle\log r\rangle^{-1}|\partial_{r}v|_{-1}\|_{L^{2}}+\|r^{-1}\langle r\rangle^{-2}\langle\log r\rangle^{-1}v\|_{L^{2}}.
\end{gather*}
The space $\dot{\calH}_{m}^{k}$ is defined by the completion of the
space $\calS_{m}$ of $m$-equivariant Schwartz functions under the
$\dot{\calH}_{m}^{k}$-norms. It turns out that $\dot{\calH}_{0}^{1}\hookrightarrow\dot{H}_{0}^{1}$,
$\dot{\calH}_{1}^{2}\hookrightarrow\dot{H}_{1}^{2}$, and $\dot{\calH}_{0}^{3}\hookrightarrow\dot{H}_{0}^{3}$.
But we have a reverse embedding for $\dot{\calH}_{2}^{1}$: $\dot{H}_{2}^{1}\hookrightarrow\dot{\calH}_{2}^{1}$.
Note that the norms $\dot{\calH}_{m}^{k}$ are same as $\dot{H}_{m}^{k}$
norms for high frequency pieces. In particular, one has $\dot{\calH}_{m}^{k}\cap L^{2}=H_{m}^{k}$.
See Appendix~\ref{sec:Adapted-function-spaces} for more details.

The spaces $\dot{\calH}_{m}^{k}$ are constructed in order to have
boundedness and subcoercivity estimates of $L_{Q}$, $A_{Q}$, and
$A_{Q}^{\ast}$. Actually this is how we chose the weights in the
definitions of the $\dot{\calH}_{m}^{k}$-norms. For more details,
we refer to \cite[Section 2.3]{KimKwon2020arXiv}. Since $L_{Q}$
and $A_{Q}$ have nontrivial kernels, we cannot have a coercivity
estimate like $\|L_{Q}v\|_{L^{2}}\sim\|v\|_{\dot{\calH}_{0}^{1}}$.
Instead, we can have a subcoercivity estimate as 
\[
\|L_{Q}v\|_{L^{2}}+\|\chf_{r\sim1}v\|_{L^{2}}\sim\|v\|_{\dot{\calH}_{0}^{1}}.
\]
The associated coercivity can be obtained by ruling out the kernel
elements of $L_{Q}$. The same remark applies to $A_{Q}$. For $A_{Q}^{\ast}$,
due to the positivity \eqref{eq:Def-Vtilde} of $A_{Q}A_{Q}^{\ast}$,
the unconditional coercivity estimate for $A_{Q}^{\ast}$ holds. As
a result, we have the following coercivity estimates (see Appendix
\ref{sec:Adapted-function-spaces} for the proof). 
\begin{prop}[Linear coercivity estimates]
\label{prop:LinearCoercivity}\ 
\begin{enumerate}
\item (Coercivity of $L_{Q}$ at $\dot{H}^{1}$-level) Let $\psi_{1},\psi_{2}\in(\dot{\calH}_{0}^{1})^{\ast}$
be such that the $2\times2$ matrix $(a_{ij})$ defined by $a_{i1}=(\psi_{i},\Lambda Q)_{r}$
and $a_{i2}=(\psi_{i},iQ)_{r}$ has nonzero determinant. Then, we
have a coercivity estimate 
\begin{equation}
\|v\|_{\dot{\calH}_{0}^{1}}\aleq_{\psi_{1},\psi_{2}}\|L_{Q}v\|_{L^{2}}\aleq\|v\|_{\dot{\calH}_{0}^{1}},\qquad\forall v\in\dot{\calH}_{m}^{1}\cap\{\psi_{1},\psi_{2}\}^{\perp},\label{eq:LQ-coer-H1-sec2}
\end{equation}
where $\perp$ is defined with respect to the real inner product $(\cdot,\cdot)_{r}$.
\item (Coercivity of $L_{Q}$ at $\dot{H}^{3}$-level) Let $\psi_{1},\psi_{2}\in(\dot{\calH}_{0}^{3})^{\ast}$
be such that the $2\times2$ matrix $(a_{ij})$ defined by $a_{i1}=(\psi_{i},\Lambda Q)_{r}$
and $a_{i2}=(\psi_{i},iQ)_{r}$ has nonzero determinant. Then, we
have a coercivity estimate 
\begin{equation}
\|v\|_{\dot{\calH}_{0}^{3}}\aleq_{\psi_{1},\psi_{2}}\|L_{Q}v\|_{\dot{\calH}_{1}^{2}}\aleq\|v\|_{\dot{\calH}_{0}^{3}},\qquad\forall v\in\dot{\calH}_{0}^{3}\cap\{\psi_{1},\psi_{2}\}^{\perp}.\label{eq:LQ-coer-H3-sec2}
\end{equation}
\item (Coercivity of $A_{Q}$ at $\dot{H}^{2}$-level) Let $\psi_{1},\psi_{2}\in(\dot{\calH}_{1}^{2})^{\ast}$
be such that the $2\times2$ matrix $(a_{ij})$ defined by $a_{i1}=(\psi_{i},rQ)_{r}$
and $a_{i2}=(\psi_{i},irQ)_{r}$ has nonzero determinant. Then, we
have a coercivity estimate 
\begin{equation}
\|v\|_{\dot{\calH}_{1}^{2}}\aleq_{\psi_{1},\psi_{2}}\|A_{Q}v\|_{\dot{\calH}_{2}^{1}}\aleq\|v\|_{\dot{\calH}_{1}^{2}},\qquad\forall v\in\dot{\calH}_{1}^{2}\cap\{\psi_{1},\psi_{2}\}^{\perp}.\label{eq:AQ-coer-H2-sec2}
\end{equation}
\item (Unconditional coercivity of $A_{Q}^{\ast}$ at $\dot{H}^{1}$-level)
We have 
\begin{equation}
\|A_{Q}^{\ast}v\|_{L^{2}}\sim\|v\|_{\dot{\calH}_{2}^{1}},\qquad\forall v\in\dot{\calH}_{2}^{1}.\label{eq:positivity-AQAQstar}
\end{equation}
\end{enumerate}
\end{prop}

In later applications, we will use orthogonality conditions depending
on a large truncation parameter $M$. Thus in the above coercivity
estimates $\aleq_{\psi_{1},\psi_{2}}$ becomes $\aleq_{M}$.

We will later decompose $w$, $w_{1}$, $w_{2}$ as 
\[
w=P+\eps,\quad w_{1}=P_{1}+\eps_{1},\quad w_{2}=P_{2}+\eps_{2},
\]
where $P$, $P_{1}$, $P_{2}$ are some modified profiles, and $\eps$,
$\eps_{1}$, $\eps_{2}$ are the errors. Thus $\eps$, $\eps_{1}$,
$\eps_{2}$ are $0$, $1$, $2$-equivariant functions, respectively.
Although $\eps_{1}$ and $\eps_{2}$ are constructed in a nonlinear
fashion (later called \emph{nonlinear adapted derivatives}), we approximately
have $\eps_{1}\approx L_{Q}\eps$ and $\eps_{2}\approx A_{Q}\eps_{1}$.
We will also use $\eps_{3}=A_{Q}^{\ast}\eps_{2}$. In bootstrap analysis,
we want to control $\|\eps\|_{L^{2}}$, $\|\eps_{1}\|_{L^{2}}$, and
$\|\eps_{3}\|_{L^{2}}$. In view of the above coercivity estimates,
$\|\eps_{1}\|_{L^{2}}$ will control $\|\eps\|_{\dot{\calH}_{0}^{1}}$,
and $\|\eps_{3}\|_{L^{2}}$ will control $\|\eps_{2}\|_{\dot{\calH}_{2}^{1}},\|\eps_{1}\|_{\dot{\calH}_{1}^{2}},\|\eps\|_{\dot{\calH}_{0}^{3}}$.
See Lemma~\ref{lem:NonlinearCoercivity}.

Finally, for technical reasons, we will need an auxiliary norm $\nrm{\cdot}_{X}$
\begin{equation}
\|f\|_{X}\coloneqq\|\langle y\rangle^{-2}\langle\log_{+}y\rangle f\|_{L^{2}}.\label{eq:Def-Xnorm}
\end{equation}
This will be used in the Morawetz correction (Section~\ref{subsec:Energy-estimate}),
e.g.~in the estimate 
\begin{align*}
(i\eps_{2},yQ^{2}\eps_{1})_{r} & \aleq\|\eps_{2}\|_{\dot{\calH}_{2}^{1}}\|\langle y\rangle^{-2}\langle\log_{+}y\rangle\eps_{1}\|_{L^{2}}\sim\|\eps_{3}\|_{L^{2}}\|\eps_{1}\|_{X}.
\end{align*}

\section{\label{sec:Modified-profiles}Modified profiles}

This section is devoted to the construction of modified profiles and
the derivation of a sharp logarithmic correction to the pseudoconformal
blow-up rate, which are one of the novelties of this work.

As we have seen in Section~\ref{subsec:lin-Q}, the information on
the generalized nullspace of $i\calL_{Q}$ suggests a decomposition
of the form 
\[
u(t,r)=\frac{e^{i\gmm(t)}}{\lmb(t)}[P(\cdot;b(t),\eta(t))+\eps(t,\cdot)]\Big(\frac{r}{\lmb(t)}\Big),
\]
where $P(\cdot;0,0)=Q$, $\partial_{b}P\approx-i\frac{y^{2}}{4}Q$,
$\partial_{\eta}P\approx-(m+1)\rho$. Here we focus on the modulated
blow-up profile $P$.

The case considered here ($m=0$) is significantly different from
the case $m\geq1$. When $m\geq1$, the authors in \cite{KimKwon2020arXiv}
constructed pseudoconformal blow-up solutions using the modified profiles
\[
Q_{b}^{(\eta)}(y)=\chi_{B_{0}}(y)Q^{(\eta)}(y)e^{-ib\frac{y^{2}}{4}},
\]
where $Q^{(\eta)}$ is some profile satisfying $Q^{(0)}=Q$ and $\partial_{\eta}Q^{(\eta)}\approx-(m+1)\rho$.
Moreover, $Q^{(\eta)}$ is obtained by solving the \emph{modified
Bogomol'nyi equation} \cite{KimKwon2019arXiv} 
\[
\bfD_{Q^{(\eta)}}Q^{(\eta)}=-\eta\tfrac{y}{2}Q^{(\eta)}
\]
in the region $y\ll|\eta|^{-\frac{1}{2}}$. This profile suggests
the modulation equation of the form 
\begin{equation}
\frac{\lmb_{s}}{\lmb}+b=0,\quad\gmm_{s}\approx(m+1)\eta,\quad b_{s}+b^{2}+\eta^{2}=0,\quad\eta_{s}=0.\label{eq:high-equiv-mod-eqn}
\end{equation}
This nonlinear profile ansatz was a quick and efficient way to derive
the above modulation equation. Moreover, when $m\geq1$, the profile
error $\Psi$ (generated by the truncation $\chi_{B_{0}}$) is sufficiently
small to guarantee pseudoconformal blow-up.

Moreover, the rotational instability for $m\geq1$ can be read off
from \eqref{eq:high-equiv-mod-eqn}. Setting $\eta$ as a fixed small
constant $\eta_{0}$, \eqref{eq:high-equiv-mod-eqn} has solutions
\begin{equation}
\begin{gathered}b(t)=|t|,\quad\lmb(t)=(t^{2}+\eta^{2})^{\frac{1}{2}},\quad\eta(t)=\eta_{0},\\
\gmm(t)=\begin{cases}
0 & \text{if }\eta_{0}=0,\\
\mathrm{sgn}(\eta)(m+1)\tan^{-1}(\tfrac{t}{|\eta|}) & \text{if }\eta_{0}\neq0.
\end{cases}
\end{gathered}
\label{eq:RotationalInstability}
\end{equation}
When $\eta_{0}=0$, the solution blows up in the pseudoconformal regime
and shows no phase rotation. However, when $\eta_{0}\neq0$, regardless
how much small $|\eta_{0}|$ is, the solution is global and shows
an abrupt phase rotation on the short time interval $|t|\aleq|\eta_{0}|$,
by the fixed amount of angle $(m+1)\pi$. In \cite{KimKwon2019arXiv},
an explicit family of solutions for $\eta_{0}\geq0$ was constructed
to establish the (one-sided) rotational instability for $m\geq1$.

Unfortunately when $m=0$, the above nonlinear profile ansatz does
not work; it generates a profile error $\Psi$ of critical size. Hence
we search for a more refined profile. Experiences from other critical
equations such as wave maps, Schrödinger maps, and harmonic map heat
flows \cite{RaphaelRodnianski2012Publ.Math.,MerleRaphaelRodnianski2013InventMath,RaphaelSchweyer2013CPAM}
tell us that there might be a logarithmic correction to the blow-up
rate, driven by the zero \emph{resonance} for the linearized operator
$H_{Q}$.

The authors in \cite{KimKwon2020arXiv} found a remarkable conjugation
identity \eqref{eq:conj-lin}, which bridges \eqref{eq:CSS-cov} to
the above critical equations. This connection is observed when we
proceed to the variable $L_{Q}\eps$ in the linearized equation 
\[
\partial_{t}L_{Q}\eps+iH_{Q}L_{Q}\eps=0,\qquad H_{Q}=A_{Q}^{\ast}A_{Q}.
\]
As mentioned earlier, this $H_{Q}$ is the same as the one appearing
in the above critical equations and has the \emph{zero resonance}
$yQ\notin L^{2}$. This connection motivates us to look at the $w_{1}$-equation,
instead of the original equation for $w$. Moreover, we are able to
extract, \emph{from the $w_{1}$-equation}, logarithmic corrections
to $b_{s}+b^{2}+\eta^{2}=0$, which results in a logarithmic correction
to the pseudoconformal blow-up rate.

Motivated from the previous discussion, we not only track the dynamics
of $w$, but also its covariant higher order variables $w_{1}=\bfD_{w}w$
and $w_{2}=A_{w}w_{1}$. Using the conjugation identities, we derived
evolution equations of $w_{1}$ and $w_{2}$. We view \eqref{eq:CSS-rad-u}
as a system of evolution equations of $w$, $w_{1}$, $w_{2}$ under
the compatibility conditions $w_{1}=\bfD_{w}w$ and $w_{2}=A_{w}w_{1}$.
We are about to construct modified profiles $P$, $P_{1}$, $P_{2}$
for $w$, $w_{1}$, $w_{2}$, respectively.

Of course one can try to set $P_{1}=\bfD_{P}P$ and $P_{2}=A_{P}P_{1}$,
but this choice is nothing but looking at only the $w$-equation.
One of the main novelties here is to construct $P$, $P_{1}$, $P_{2}$
that approximately solve the evolution equations as well as the compatibility
conditions. Here, the point is that we also relax the compatibility
conditions: $\bfD_{P}P\approx P_{1}$ and $A_{P}P_{1}\approx P_{2}$.

In this setting, we have another advantage. It turns out that we do
not need to expand $P$ and $P_{1}$ to very higher orders. In fact,
it suffices to expand $P$ \emph{only} up to linear order and $P_{1}$
up to quadratic order. This is because the degeneracies of the profiles
$P_{1}=O(b)$ and $P_{2}=O(b^{2})$, which ultimately relies on the
facts that $\bfD_{Q}Q=0$ and $A_{Q}L_{Q}$ kills all the elements
of $\{\Lambda Q,iQ,i\tfrac{y^{2}}{4}Q,\rho\}$, as explained in Section
\ref{subsec:Strategy}. As we will apply the energy estimate for the
variable $w_{2}$, $P_{2}$ should be constructed to the highest order
compared to $P$ and $P_{1}$. However, thanks to the degeneracy of
$P_{2}$, it contains only the quadratic and cubic order terms, which
are still quite simple.

Finally, we remark that we are able to observe logarithmic corrections
in the modulation laws from the $w_{1}$-equation. As explained above,
at the linear level, the $w_{1}$-equations solves a similar equation
to the Schrödinger map case. The effects of the logarithmic corrections
can be seen in the quadratic terms of $P_{1}$ and $P_{2}$ expansions.

\subsection{Formal derivation of the profiles}

Our starting points are the evolution equations \eqref{eq:w-eqn-sd},
\eqref{eq:w1-eqn-sd}, and \eqref{eq:w2-eqn-sd} for $w$, $w_{1}=\bfD_{w}w$,
and $w_{2}=A_{w}w_{1}$, derived in Proposition~\ref{prop:CSS-conj-eq}.
After substitutions $w_{1}=\bfD_{w}w$ and $w_{2}=A_{w}w_{1}$, they
are written as 
\begin{gather}
(\partial_{s}-\frac{\lmb_{s}}{\lmb}\Lambda+\gmm_{s}i)w+iL_{w}^{\ast}w_{1}=0,\label{eq:w-eqn-sd-profile}\\
(\partial_{s}-\frac{\lmb_{s}}{\lmb}\Lambda_{-1}+\td{\gmm}_{s}i)w_{1}+iA_{w}^{\ast}w_{2}-\Big(\int_{0}^{y}\Re(\overline{w}w_{1})dy'\Big)iw_{1}=0,\label{eq:w1-eqn-sd-profile}\\
(\rd_{s}-\frac{\lmb_{s}}{\lmb}\Lmb_{-2}+\td{\gmm}_{s}i)w_{2}+iA_{w}A_{w}^{\ast}w_{2}-\Big(\int_{0}^{y}\Re(\br ww_{1})dy'\Big)iw_{2}-i\br ww_{1}^{2}=0,\label{eq:w2-eqn-sd-profile}
\end{gather}
where 
\[
\td{\gmm}_{s}=\gmm_{s}+{\textstyle \int_{0}^{\infty}}\Re(\overline{w}w_{1})dy.
\]
Recall that the role of the phase correction $\gmm_{s}\mapsto\td{\gmm}_{s}$
is to replace the above $\int_{y}^{\infty}$-integral by $\int_{0}^{y}$.
Note that $\int_{y}^{\infty}$ has the technical problem that it cannot
be defined for functions with growing tails, which typically arise
in the Taylor expansion of the profiles. See also Remark~\ref{rem:gamma-to-gamma-tilde}.
Assume the \emph{adiabatic ansatz} 
\[
\frac{\lmb_{s}}{\lmb}+b=0\quad\text{and}\quad\td{\gmm}_{s}=-\eta.
\]

We will construct an approximate solution of the form 
\[
(w,w_{1},w_{2})=(P,P_{1},P_{2}),
\]
to \eqref{eq:w-eqn-sd-profile}--\eqref{eq:w2-eqn-sd-profile} and
the compatibility conditions $w_{1}=\bfD_{w}w$ and $w_{2}=A_{w}w_{1}$.
Here, $P$, $P_{1}$, and $P_{2}$ will be suitable localizations
of 
\begin{equation}
\begin{aligned}\widehat{P} & \coloneqq Q-ib\tfrac{y^{2}}{4}Q-\eta\rho,\\
\widehat{P}_{1} & \coloneqq-(ib+\eta)\tfrac{y}{2}Q+b^{2}T_{2,0},\\
\widehat{P}_{2} & \coloneqq(b^{2}-2ib\eta-\eta^{2})U_{2}+ib^{3}U_{3,0},
\end{aligned}
\label{eq:P-hat-ansatz}
\end{equation}
where $T_{2,0}$, $U_{2}$, $U_{3,0}$ are real-valued. The profiles
$T_{2,0}$, $U_{2}$, $U_{3,0}$, as well as the laws for $b_{s}$
and $\eta_{s}$, are unknowns and will be chosen subsequently to minimize
the profile error.

The profiles up to the first order in $\widehat{P}$, $\widehat{P}_{1}$,
and $\widehat{P}_{2}$ are easily derived from the generalized nullspace
relations and the adiabatic ansatz. Indeed, if we start from $w=Q$,
then $\bfD_{Q}Q=0$ and the compatibility conditions suggest that
zeroth order terms of $w_{1}$ and $w_{2}$ should vanish. Next, from
\eqref{eq:w-eqn-sd-profile} and the adiabatic ansatz $\frac{\lmb_{s}}{\lmb}+b=0$
and $\gmm_{s}\approx\eta$, we are led to 
\[
L_{Q}^{\ast}w_{1}\approx_{1}ib\Lambda Q-\eta Q,
\]
in the sense that both hand sides are equal up to the first order.
This suggests us the choice $w_{1}\approx_{1}-(ib+\eta)\tfrac{y}{2}Q$.
By linearizing the compatibility relation $w_{1}=\bfD_{w}w$, we have
\[
L_{Q}(w-Q)\approx_{1}-(ib+\eta)\tfrac{y}{2}Q,
\]
which motivates the choice $w\approx_{1}Q-ib\tfrac{y^{2}}{4}Q-\eta\rho$.
Finally, $A_{Q}(yQ)=0$ and the compatibility relation $w_{2}=A_{w}w_{1}$
suggest $w_{2}\approx_{1}0$. In summary, we are led to 
\begin{align*}
\widehat{P} & \approx_{1}Q-ib\tfrac{y^{2}}{4}Q-\eta\rho,\\
\widehat{P}_{1} & \approx_{1}-(ib+\eta)\tfrac{y}{2}Q,\\
\widehat{P}_{2} & \approx_{1}0.
\end{align*}

We now search for higher order expansions for $\widehat{P}$, $\widehat{P}_{1}$,
and $\widehat{P}_{2}$. In the following, we will also assume 
\[
|\eta|\leq\frac{b}{|\log b|}\qquad\text{and}\qquad0<b\ll1.
\]
Although our sharp modulation equation will be slightly different
from \eqref{eq:high-equiv-mod-eqn} of the $m\geq1$ case, \eqref{eq:high-equiv-mod-eqn}
still motivates us to assume $|\eta|\ll b$ to guarantee the blow-up.
\begin{rem}
In order to obtain the sharp energy estimate \eqref{eq:Psi2-SharpEnergy}
under $|\eta|\leq\frac{b}{|\log b|}$, it is necessary to expand $\widehat{P}_{2}$
up to $b^{3}$-order terms. Thus one may start from considering a
general expansion 
\begin{align*}
\widehat{P} & =Q-ib\tfrac{y^{2}}{4}Q-\eta\rho+b^{2}\widehat{S}_{2,0}+b\eta\widehat{S}_{1,1}+\eta^{2}\widehat{S}_{0,2}+\cdots,\\
\widehat{P}_{1} & \coloneqq-(ib+\eta)\tfrac{y}{2}Q+b^{2}\widehat{T}_{2,0}+b\eta\widehat{T}_{1,1}+\eta^{2}\widehat{T}_{0,2}+\cdots,\\
\widehat{P}_{2} & \coloneqq b^{2}\widehat{U}_{2,0}+b\eta\widehat{U}_{1,1}+\eta^{2}\widehat{U}_{0,2}+b^{3}\widehat{U}_{3,0},
\end{align*}
for some complex-valued profiles $\widehat{S}_{i,j}$, $\widehat{T}_{i,j}$,
and $\widehat{U}_{i,j}$. Due to \eqref{eq:Psi2-SharpEnergy} and
$|\eta|\leq\frac{b}{|\log b|}$, it is enough to stop at $b^{3}\widehat{U}_{3,0}$;
our main goal is to construct $\widehat{U}_{3,0}$.

In the following, we will use the ansatz \eqref{eq:P-hat-ansatz}
for the simplicity of presentation. On the way, the reader may see
that the linear expansion is enough for $\widehat{P}$, and the expansion
up to the $b^{2}$-term is enough for $\widehat{P}_{1}$. The other
quadratic terms $b\eta\widehat{T}_{1,1}$ and $\eta^{2}\widehat{T}_{0,2}$
are not necessary, due to $|\eta|\leq\frac{b}{|\log b|}$. Moreover,
the coefficients in the ansatz \eqref{eq:P-hat-ansatz} naturally
appear in the derivation. 
\end{rem}

\subsubsection*{Derivation of $U_{2}$ and $T_{2,0}$}

\ 

Here we search for the quadratic terms of the expansions. We look
at the $w_{1}$-equation \eqref{eq:w1-eqn-sd-profile}. At this point,
we assume that $b_{s}$ and $\eta_{s}$ have unknown quadratic terms
in $b$ and $\eta$, though we expect that $b_{s}\approx-b^{2}-\eta^{2}$
and $\eta_{s}\approx0$ from \eqref{eq:high-equiv-mod-eqn}. We collect
the $O(b^{2},b\eta,\eta^{2})$-terms (not including $O(1,b,\eta)$
terms) in the equation \eqref{eq:w1-eqn-sd-profile}: 
\begin{align*}
\partial_{s}w_{1}\quad & \to\quad(-ib_{s}-\eta_{s})(\tfrac{y}{2}Q),\\
b\Lambda_{-1}w_{1}\quad & \to\quad(-ib^{2}-b\eta)\Lambda_{-1}(\tfrac{y}{2}Q),\\
-\eta iw_{1}\quad & \to\quad(-b\eta+i\eta^{2})(\tfrac{y}{2}Q),\\
iA_{w}^{\ast}w_{2}\quad & \to\quad(ib^{2}+2b\eta-i\eta^{2})A_{Q}^{\ast}U_{2},\\
-({\textstyle \int_{0}^{y}}\Re(\overline{w}w_{1})dy')iw_{1}\quad & \to\quad(b\eta-i\eta^{2})(2-\Lambda)(\tfrac{y}{2}Q),
\end{align*}
where in the last one we used 
\begin{equation}
({\textstyle \int_{0}^{y}}\tfrac{y'}{2}Q^{2}dy')\tfrac{y}{2}Q=-A_{\theta}[Q]\tfrac{y}{2}Q=\tfrac{y}{2}Q+(yA_{Q}-y\partial_{y})\tfrac{y}{2}Q=(2-\Lambda)(\tfrac{y}{2}Q).\label{eq:some-computation}
\end{equation}
Summing up, we arrive at 
\[
(-i(b_{s}+b^{2}+\eta^{2})-\eta_{s})(\tfrac{y}{2}Q)+(ib^{2}+2b\eta-i\eta^{2})(A_{Q}^{\ast}U_{2}-\Lambda(\tfrac{y}{2}Q))=0.
\]
Here, the key point is that $\Lambda(\tfrac{y}{2}Q)$ exhibits better
spatial decay (by order $2$) compared to the main term $yQ$, which
is grouped together with the modulation differentials $b_{s}$, $\eta_{s}$.
Roughly speaking, the term with the worst growth $yQ$ is cancelled
by choosing $b_{s}$, $\eta_{s}$ appropriately, whereas we attempt
to introduce profile $U_{2}$ (and also $T_{2,0}$ below) to solve
away the remaining better decaying terms. This is the \emph{tail computation}
due to \cite{RaphaelRodnianski2012Publ.Math.,MerleRaphaelRodnianski2013InventMath,MerleRaphaelRodnianski2015CambJMath}.

This motivates us to formally set 
\[
b_{s}+b^{2}+\eta^{2}=0\quad\text{and}\quad\eta_{s}=0
\]
up to quadratic terms. For the profile $U_{2}$, a naive choice would
be to solve $A_{Q}^{\ast}U_{2}-\Lambda(\tfrac{y}{2}Q)=0$. However,
with this choice we cannot avoid the profile error $\Psi_{2}$ of
critical size. Indeed, solving $A_{Q}^{\ast}U_{2}-\Lambda(\tfrac{y}{2}Q)=0$,
we have $U_{2}\sim1$ near infinity. This lack of decay is due to
the violation of the $L^{2}$-solvability condition $(\Lambda(\tfrac{y}{2}Q),\tfrac{y}{2}Q)_{r}=2\pi\neq0$,
which in turn is due to $yQ\not\in L^{2}$. Continuing the expansion
with this $U_{2}$, we would arrive at $U_{3,0}\sim y^{2}$ near infinity.
In the computation of the profile error $\Psi_{2}$, with any cutoff
at some $y=B$, $\|\Psi_{2}\|_{\dot{\calH}_{2}^{1}}$ would see the
cutoff error of $U_{3,0}$ at $y=B$, which is 
\[
b^{3}\|\chf_{y\sim B}|U_{3,0}|_{-3}\|_{L^{2}}\sim b^{3}\|\chf_{y\sim B}\tfrac{1}{y}\|_{L^{2}}\sim b^{3}.
\]
This error is \emph{of critical size,} in the sense that we would
not be able to make $\|\eps\|_{\dot{\calH}_{0}^{3}}\ll b^{2}$ in
the energy argument because of it. This also explains why we cannot
use the profile ansatz used in the case $m\geq1$.

To overcome this issue, we follow \cite{RaphaelRodnianski2012Publ.Math.}
and use the fact that $\tfrac{y}{2}Q$ is a \emph{resonance} to the
operator $A_{Q}^{\ast}A_{Q}$. From the compatibility condition $A_{w}w_{1}=w_{2}$
(compare $b^{2}$-order terms), we choose $T_{2,0}$ such that 
\[
A_{Q}T_{2,0}=U_{2}.
\]
Thus if $A_{Q}^{\ast}U_{2}=\Lambda(\tfrac{y}{2}Q)$, then $T_{2,0}$
should satisfy $A_{Q}^{\ast}A_{Q}T_{2,0}=\Lambda(\frac{y}{2}Q)$.
Note again that the $L^{2}$-solvability condition \emph{does not}
hold because $\tfrac{y}{2}Q\notin L^{2}$: 
\[
(\Lambda(\tfrac{y}{2}Q),\tfrac{y}{2}Q)_{r}=2\pi\neq0.
\]
As in \cite[p.31 Step 6]{RaphaelRodnianski2012Publ.Math.}, we introduce
\[
c_{b}\coloneqq\frac{(\Lambda(\tfrac{y}{2}Q),\tfrac{y}{2}Q)_{r}}{(\tfrac{y}{2}Q\chi_{B_{0}},\tfrac{y}{2}Q)_{r}}=\frac{2}{|\log b|}+O\Big(\frac{1}{|\log b|^{2}}\Big)
\]
and solve instead\footnote{For interested readers to the case $m\geq1$, we note that the solvability
condition $(\Lambda(\tfrac{y}{2}Q),\tfrac{y}{2}Q)_{r}=0$ holds because
$\tfrac{y}{2}Q\in L^{2}$. Thus one may define $U_{2}$ and $T_{2,0}$
by solving $A_{Q}^{\ast}U_{2}=\Lambda(\tfrac{y}{2}Q)$ and $A_{Q}T_{2,0}=U_{2}$
instead. Note that one can find explicit formulae $U_{2}=-\tfrac{y^{2}}{4}Q$
and $T_{2,0}=-\tfrac{y^{3}}{8}Q$, as motivated from the Taylor expansion
of the pseudoconformal phase $e^{-ib\frac{y^{2}}{4}}$. This leads
to the pseudoconformal blow-up rate.} 
\[
A_{Q}^{\ast}A_{Q}T_{2,0}=\Lambda(\tfrac{y}{2}Q)-c_{b}\tfrac{y}{2}Q\chi_{B_{0}}\eqqcolon g_{2}.
\]
Because $g_{2}$ is now orthogonal to $\frac{y}{2}Q$, it can be shown
(see Lemma \ref{lem:profiles} below) that $T_{2,0}$ has a logarithmically
improved decay at $y\sim B_{0}$ compared to the formal diverging
kernel $\Gamma\sim y$ of $H_{Q}=A_{Q}^{\ast}A_{Q}$. For the choice
of the radius $B_{0}$, see Remark~\ref{rem:profile-loc}. We remark
that the power $-\frac{1}{2}$ of $B_{0}=b^{-\frac{1}{2}}$ is tied
to the sharp blow-up rate.

Therefore, we will choose $U_{2}$ and $T_{2,0}$ such that 
\begin{align}
A_{Q}^{\ast}U_{2} & =\Lambda(\tfrac{y}{2}Q)-c_{b}\tfrac{y}{2}Q\chi_{B_{0}}=g_{2},\label{eq:U2-property}\\
A_{Q}T_{2,0} & =U_{2}.\label{eq:T20-property}
\end{align}
With this $U_{2}$, it turns out that one has a logarithmic gain $\frac{1}{|\log b|}$
in the region $y\ageq B_{0}$, so the previous issue is overcome.
On the other hand, the equation \eqref{eq:w1-eqn-sd-profile} is solved
up to quadratic terms with the additional error 
\[
(ib^{2}+2b\eta-i\eta^{2})c_{b}\tfrac{y}{2}Q\chi_{B_{0}}.
\]
This will give rise to additional terms of order $O(\frac{b^{2}}{|\log b|},\frac{b\eta}{|\log b|},\frac{\eta^{2}}{|\log b|})$
in the equations for $b_{s}$ and $\eta_{s}$, which in turn cause
the logarithmic correction to the blow-up rate. As a result, we get
the formal parameter law: 
\begin{equation}
\frac{\lmb_{s}}{\lmb}+b=0,\quad\td{\gmm}_{s}=-\eta,\quad b_{s}+b^{2}+\eta^{2}+c_{b}(b^{2}-\eta^{2})=0,\quad\eta_{s}+2c_{b}b\eta=0,\label{eq:FormalParameterLaw}
\end{equation}
with $c_{b}\approx\frac{2}{|\log b|}$ defined above. 
\begin{rem}[Full quadratic expansion for $\widehat{P}_{1}$]
By the same way, but using $A_{w}^{\ast}A_{w}w_{1}$ instead of $A_{w}^{\ast}w_{2}$
in \eqref{eq:w1-eqn-sd-profile} and collecting the quadratic terms
$O(b^{2},b\eta,\eta^{2})$, one can derive the full quadratic expansion
of $\widehat{P}_{1}$: 
\[
\widehat{P}_{1}=-(ib+\eta)\tfrac{y}{2}Q+(b^{2}-2ib\eta-\eta^{2})T_{2,0}+(ib\eta+\eta^{2})\td T_{2},
\]
where $T_{2,0}$ is as above and $\td T_{2}$ solves $A_{Q}\td T_{2}=A_{\theta}[Q,\rho]Q$.
As mentioned in the previous remark, $O(b\eta)$ and $O(\eta^{2})$
terms are not necessary in the derivation of $U_{3,0}$ and later
analysis. 
\end{rem}

\subsubsection*{Derivation of $U_{3,0}$}

\ 

We finally search for the $b^{3}$ term of the $\widehat{P}_{2}$-expansion.
We again look at the $w_{1}$-equation \eqref{eq:w1-eqn-sd-profile}.
We collect $b^{3}$-terms of the error. 
\begin{align*}
\partial_{s}w_{1}\quad & \to\quad-2b^{3}T_{2,0},\\
b\Lambda_{-1}w_{1}\quad & \to\quad b^{3}\Lambda_{-1}T_{2,0},\\
-\eta iw_{1}\quad & \to\quad0,\\
iA_{w}^{\ast}w_{2}\quad & \to\quad-b^{3}A_{Q}^{\ast}U_{3,0},\\
-({\textstyle \int_{0}^{y}}\Re(\overline{w}w_{1})dy')iw_{1}\quad & \to\quad-b^{3}({\textstyle \int_{0}^{y}}(QT_{2,0}+\tfrac{(y')^{3}}{8}Q^{2})y'dy')\tfrac{y}{2}Q.
\end{align*}
Summing these up, we are motivated to choose $U_{3,0}$ by solving
\begin{equation}
A_{Q}^{\ast}U_{3,0}=\Lambda_{1}T_{2,0}-({\textstyle \int_{0}^{y}}(QT_{2,0}+\tfrac{(y')^{3}}{8}Q^{2})y'dy')\tfrac{y}{2}Q\eqqcolon g_{3,0}.\label{eq:U30-property}
\end{equation}
Taking $A_{Q}$, we obtain the identity for later use: 
\begin{equation}
A_{Q}g_{3,0}=A_{Q}\Lambda_{1}T_{2,0}-(QT_{2,0}+\tfrac{y^{3}}{8}Q^{2})(\tfrac{y}{2}Q).\label{eq:AQg30}
\end{equation}

\subsection{Estimates of profiles in Taylor expansions}

In the previous subsection, we discussed how we choose the higher
order profiles $T_{2,0},U_{2},U_{3,0}$ used in the definitions $\widehat{P}_{1}$
and $\widehat{P}_{2}$. Here we construct these profiles satisfying
\eqref{eq:U2-property}, \eqref{eq:T20-property}, and \eqref{eq:U30-property},
using the outgoing Green's function discussed in Section~\ref{subsec:Outgoing-Green's-function}. 
\begin{lem}[Profiles $T_{2,0},U_{2},U_{3,0}$]
\label{lem:profiles} For any sufficiently small $b>0$, define smooth
functions on $(0,\infty)$ by 
\begin{align*}
T_{2,0}(y;b) & \coloneqq{}^{(\out)}H_{Q}^{-1}g_{2},\\
U_{2}(y;b) & \coloneqq A_{Q}T_{2,0}=-(A_{Q}\Gamma){\textstyle \int_{0}^{y}}g_{2}Jy'dy'=(A_{Q}\Gamma){\textstyle \int_{y}^{\infty}}g_{2}Jy'dy',\\
U_{3,0}(y;b) & \coloneqq A_{Q}{}^{(\out)}H_{Q}^{-1}g_{3,0}=-(A_{Q}\Gamma){\textstyle \int_{0}^{y}}g_{3,0}Jy'dy'.
\end{align*}
where 
\begin{align*}
g_{2}(y;b) & =\Lambda(\tfrac{y}{2}Q)-c_{b}\tfrac{y}{2}Q\chi_{B_{0}},\\
g_{3,0}(y;b) & =\Lambda_{1}T_{2,0}-({\textstyle \int_{0}^{y}}(QT_{2,0}+\tfrac{(y')^{3}}{8}Q^{2})y'dy')(\tfrac{y}{2}Q),\\
c_{b} & =\tfrac{(\Lambda(yQ),yQ)_{r}}{(yQ\chi_{B_{0}},yQ)_{r}}=\tfrac{2}{|\log b|}+O(\tfrac{1}{|\log b|^{2}}).
\end{align*}
Then, for any nonnegative integer $k$, the following properties hold:
\begin{enumerate}
\item (Rough pointwise estimates, only sharp in the compact regions $y\sim1$)
We have 
\begin{equation}
|U_{2}|_{k}+\tfrac{1}{y}|T_{2,0}|_{k}+\tfrac{1}{y^{2}}|U_{3,0}|_{k}\aleq_{k}1.\label{eq:RoughPointwise}
\end{equation}
\item (Sharp pointwise estimates) Recall $B_{0}=b^{-1/2}$.
\begin{enumerate}
\item In the region $1\leq y\leq B_{0}$, we have 
\[
\begin{aligned}|U_{2}|_{k}+\tfrac{1}{y}|T_{2,0}|_{k}+\tfrac{1}{y^{2}}|U_{3,0}|_{k} & \aleq_{k}\tfrac{1}{|\log b|}|\log(b^{\frac{1}{2}}y)|,\\
|b\partial_{b}U_{2}|_{k}+\tfrac{1}{y}|b\partial_{b}T_{2,0}|_{k}+\tfrac{1}{y^{2}}|b\partial_{b}U_{3,0}|_{k} & \aleq_{k}\tfrac{1}{|\log b|^{2}}|\log(b^{\frac{1}{2}}y)|.
\end{aligned}
\]
\item In the region $B_{0}\leq y\leq2B_{0}$, we have 
\[
\begin{aligned}|U_{2}|_{k}+\tfrac{1}{y}|T_{2,0}|_{k}+\tfrac{1}{y^{2}}|U_{3,0}|_{k} & \aleq_{k}\tfrac{1}{|\log b|},\\
|b\partial_{b}U_{2}|_{k}+\tfrac{1}{y}|b\partial_{b}T_{2,0}|_{k}+\tfrac{1}{y^{2}}|b\partial_{b}U_{3,0}|_{k} & \aleq_{k}\tfrac{1}{|\log b|}.
\end{aligned}
\]
\item In the region $y\leq1$, we have 
\[
\tfrac{1}{y^{2}}|U_{2}|_{k}+\tfrac{1}{y^{3}}|T_{2,0}|_{k}+\tfrac{1}{y^{4}}|U_{3,0}|_{k}\aleq_{k}1.
\]
Moreover, the profile $T_{2,0}$ has smooth $1$-equivariant extension
on $\bbR^{2}$; and the profiles $U_{2},U_{3,0}$ have smooth $2$-equivariant
extension on $\bbR^{2}$. 
\end{enumerate}
\end{enumerate}
\end{lem}

\begin{rem}
An important point is that one has \emph{logarithmic gain} in the
region $y\sim B_{0}$. In the region $y\aleq1$, we do not have any
logarithmic gain. 
\end{rem}

\begin{rem}
The rough pointwise estimates are sharp only in the region $y\sim1$,
and not sharp in far regions $y\ageq B_{0}$. Thus rough pointwise
estimates will be effective when the main contributions to errors
come from the compact region $y\sim1$. Of course, the rough pointwise
estimates are easy to implement. 
\end{rem}

\begin{proof}
Bounds of $U_{2}$ are immediate from the bounds of $T_{2,0}$. Henceforth,
we focus on $T_{2,0}$ and $U_{3,0}$.

For $T_{2,0}$, thanks to the \emph{cancellation property} near the
infinity 
\[
\chf_{[1,\infty)}|\Lambda(yQ)|\aleq\chf_{[1,\infty)}y^{-3},
\]
$g_{2}$ satisfies (use $\partial_{b}c_{b}\aleq\frac{1}{b|\log b|^{2}}$
for $\partial_{b}g_{2}$) 
\begin{align}
|g_{2}|_{k} & \aleq_{k}\chf_{(0,1]}y+\chf_{[1,2B_{0}]}\tfrac{1}{|\log b|y}+\chf_{[2B_{0},\infty)}\tfrac{1}{y^{3}},\label{eq:g2-sharp}\\
|\partial_{b}g_{2}|_{k} & \aleq_{k}\tfrac{1}{b|\log b|}(\chf_{(0,1]}\tfrac{y}{|\log b|}+\chf_{[1,B_{0}]}\tfrac{1}{|\log b|y}+\chf_{[B_{0},2B_{0}]}\tfrac{1}{y}).\label{eq:db-g2-sharp}
\end{align}
In particular, by Proposition~\ref{prop:HQ-green}, it easily follows
that 
\[
\chf_{(0,1]}\abs{T_{2,0}}_{k}\aleq_{k}y^{3},\qquad\chf_{(0,1]}\abs{\rd_{b}T_{2,0}}_{k}\aleq_{k}\tfrac{1}{b\abs{\log b}^{2}}y^{3}.
\]
Because $g_{2}$ satisfies the solvability condition $(g_{2},yQ)_{r}=0$
(thus by differentiating it $(\partial_{b}g_{2},yQ)_{r}=0$), we can
rewrite (see Proposition~\ref{prop:HQ-green}) 
\begin{align*}
T_{2,0} & =J{\textstyle \int_{0}^{y}}g_{2}\Gamma y'dy'+\Gamma{\textstyle \int_{y}^{\infty}}g_{2}Jy'dy',\\
\partial_{b}T_{2,0} & =J{\textstyle \int_{0}^{y}}\partial_{b}g_{2}\Gamma y'dy'+\Gamma{\textstyle \int_{y}^{\infty}}\partial_{b}g_{2}Jy'dy'.
\end{align*}
Substituting the pointwise estimates of $g_{2}$ shows the bounds
of $T_{2,0}$: 
\begin{align}
\chf_{[1,\infty)}|T_{2,0}|_{k} & \aleq_{k}\chf_{[1,2B_{0}]}\tfrac{1}{|\log b|}y\langle\log(b^{\frac{1}{2}}y)\rangle+\chf_{[2B_{0},\infty)}\tfrac{1}{y}(\tfrac{1}{b|\log b|}+\log y),\label{eq:T20-lemma-temp}\\
\chf_{[1,\infty)}|\partial_{b}T_{2,0}|_{k} & \aleq_{k}\chf_{[1,2B_{0}]}\tfrac{1}{b|\log b|^{2}}y\langle\log(b^{\frac{1}{2}}y)\rangle+\chf_{[B_{0},\infty)}\tfrac{1}{b^{2}|\log b|}\tfrac{1}{y}.\nonumber 
\end{align}
From these estimates, the sharp pointwise estimates for $T_{2,0}$
follow.

Finally, we bound $U_{3,0}$. We start from estimating $g_{3,0}$.
By the nonsharp bounds 
\[
\abs{QT_{2,0}}_{k}\aleq_{k}\tfrac{y^{3}}{1+y^{4}},\qquad\abs{Q\rd_{b}T_{2,0}}_{k}\aleq_{k}\chf_{(0,1]}\tfrac{1}{b\abs{\log b}^{2}}y^{3}+\chf_{[1,\infty)}\tfrac{1}{b\abs{\log b}}\tfrac{1}{y},
\]
we obtain 
\begin{align*}
\abs{g_{3,0}-\Lambda_{1}T_{2,0}}_{k} & \aleq_{k}\chf_{(0,1]}y^{6}+\chf_{[1,\infty)}\\
\abs{\rd_{b}(g_{3,0}-\Lambda_{1}T_{2,0})}_{k} & \aleq_{k}\chf_{(0,1]}\tfrac{1}{b\abs{\log b}^{2}}y^{6}+\chf_{[1,\infty)}\tfrac{1}{b\abs{\log b}}.
\end{align*}
Hence, using the sharp $T_{2,0}$-estimates for $\Lambda_{1}T_{2,0}$,
it follows that 
\begin{align}
\abs{g_{3,0}}_{k} & \aleq_{k}\chf_{(0,1]}y^{3}+\chf_{[1,2B_{0}]}\tfrac{1}{|\log b|}y\brk{\log(b^{\frac{1}{2}}y)}+\chf_{[2B_{0},\infty)}\tfrac{1}{y}(\tfrac{1}{b|\log b|}+y),\label{eq:g30-sharp}\\
\abs{\partial_{b}g_{3,0}}_{k} & \aleq_{k}\chf_{(0,1]}\tfrac{1}{b\abs{\log b}^{2}}y^{3}+\chf_{[1,2B_{0}]}\tfrac{1}{b|\log b|^{2}}y\langle\log(b^{\frac{1}{2}}y)\rangle\label{eq:db-g30-sharp}\\
 & \peq+\chf_{[B_{0},\infty)}\tfrac{1}{b\abs{\log b}}\tfrac{1}{y}\left(\tfrac{1}{b}+y\right).\nonumber 
\end{align}
Substituting these bounds to 
\[
U_{3,0}=-(A_{Q}\Gamma){\textstyle \int_{0}^{y}}g_{3,0}Jy'dy',\qquad\partial_{b}U_{3,0}=-(A_{Q}\Gamma){\textstyle \int_{0}^{y}}\partial_{b}g_{3,0}Jy'dy',
\]
and using $|A_{Q}\Gamma|_{k}\aleq_{k}\chf_{(0,1]}\tfrac{1}{y^{2}}+\chf_{[1,\infty)}$,
we have 
\begin{align*}
\abs{U_{3,0}}_{k} & \aleq_{k}\chf_{(0,1]}y^{4}+\chf_{[1,2B_{0}]}\tfrac{1}{\abs{\log b}}y^{2}\brk{\log(b^{\frac{1}{2}}y)}\\
 & \peq+\chf_{[2B_{0},\infty)}(\tfrac{1}{b\abs{\log b}}\brk{\log(b^{\frac{1}{2}}y)}+y)\\
\abs{\rd_{b}U_{3,0}}_{k} & \aleq_{k}\chf_{(0,1]}\tfrac{1}{b\abs{\log b}}y^{4}+\chf_{[1,2B_{0}]}\tfrac{1}{b\abs{\log b}^{2}}y^{2}\brk{\log(b^{\frac{1}{2}}y)}\\
 & \peq+\chf_{[B_{0},\infty)}\tfrac{1}{b\abs{\log b}}(\tfrac{1}{b}\brk{\log(b^{-\frac{1}{2}}y)}+y).
\end{align*}
Thus the $U_{3,0}$ estimate follows.

We finally note that the smoothness (analyticity) of the profiles
at the origin follow from the explicit formulae of the involved functions.
This completes the proof. \qedhere 
\end{proof}

\subsection{Modified profiles}

We are now ready to define the modified profiles $P$, $P_{1}$, and
$P_{2}$ by adding suitable truncations. Then we will show that $P$,
$P_{1}$, and $P_{2}$ solve the evolution equations \eqref{eq:w-eqn-sd-profile},
\eqref{eq:w1-eqn-sd-profile}, \eqref{eq:w2-eqn-sd-profile} under
the formal parameter evolution laws \eqref{eq:FormalParameterLaw},
and the compatibility conditions $P_{1}\approx\bfD_{P}P$ and $P_{2}\approx A_{P}P_{1}$
up to admissible errors.

Recall the unlocalized modified profiles 
\begin{align*}
\widehat{P} & =Q-ib\tfrac{y^{2}}{4}Q-\eta\rho,\\
\widehat{P}_{1} & =-(ib+\eta)\tfrac{y}{2}Q+b^{2}T_{2,0},\\
\widehat{P}_{2} & =(b^{2}-2ib\eta-\eta^{2})U_{2}+ib^{3}U_{3,0}.
\end{align*}
We define the \emph{localized modified profiles} with $B_{0}=b^{-\frac{1}{2}}$
and $B_{1}=b^{-\frac{1}{2}}|\log b|$ by 
\begin{align*}
P & \coloneqq Q+\chi_{B_{1}}\{-ib\tfrac{y^{2}}{4}Q-\eta\rho\}\\
P_{1} & \coloneqq\chi_{B_{1}}\{-(ib+\eta)\tfrac{y}{2}Q\}+\chi_{B_{0}}\{b^{2}T_{2,0}\},\\
P_{2} & \coloneqq\chi_{B_{0}}\{(b^{2}-2ib\eta-\eta^{2})U_{2}+ib^{3}U_{3,0}\}.
\end{align*}
We truncated linear terms at $B_{1}$, but higher order terms at $B_{0}$.
It is crucial to take $B_{1}\gg B_{0}$; see Remark~\ref{rem:profile-loc}
below for the motivation. To incorporate the logarithmic corrections
to the modulation equations, we introduce 
\begin{align*}
\Mod & \coloneqq(\frac{\lmb_{s}}{\lmb}+b,\gmm_{s}-\eta,b_{s}+b^{2}+\eta^{2},\eta_{s})^{t},\\
\td{\Mod} & \coloneqq(\frac{\lmb_{s}}{\lmb}+b,\td{\gmm}_{s}+\eta,b_{s}+b^{2}+\eta^{2}+c_{b}(b^{2}-\eta^{2}),\eta_{s}+2c_{b}b\eta)^{t},\\
\mathbf{v}_{k} & \coloneqq(\Lambda_{-k}P_{k},-iP_{k},-\partial_{b}P_{k},-\partial_{\eta}P_{k})^{t},\qquad\forall k\in\{0,1,2\}.
\end{align*}
We will write $\bfv=\bfv_{0}$ and $P=P_{0}$ in short.
\begin{prop}[Modified profile]
\label{prop:ModifiedProfile}Assume the following range of $b$ and
$\eta$: 
\[
|\eta|\leq\frac{b}{|\log b|}\quad\text{and}\quad0<b<b^{\ast}.
\]
If $b^{\ast}>0$ is sufficiently small, then we have the following.
\begin{enumerate}
\item (Estimates for modulation vectors) For $\mathbf{v}=\mathbf{v}_{0}$,
we have
\begin{equation}
\begin{aligned}\chf_{(0,B_{0}/2]}(|\Lambda P-\Lambda Q|+|iP-iQ|) & \aleq b,\\
\chf_{(0,B_{0}/2]}(|\partial_{b}P+i\tfrac{y^{2}}{4}Q|+|\partial_{\eta}P+\rho|) & =0.
\end{aligned}
\label{eq:v-estimate}
\end{equation}
For $\mathbf{v}_{1}$, we have $b$-degeneracy for scalings/phase;
for some constant $C>0$, (recall the $X$-norm \eqref{eq:Def-Xnorm})
we have 
\begin{equation}
\begin{aligned}\|\Lambda_{-1}P_{1}\|_{X}+\|iP_{1}\|_{X} & \aleq b,\\
\|\partial_{b}P_{1}+\chi_{B_{1}}i\tfrac{y}{2}Q\|_{X}+\|\partial_{\eta}P_{1}+\chi_{B_{1}}\tfrac{y}{2}Q\|_{X} & \aleq b|\log b|^{C}.
\end{aligned}
\label{eq:v1-estimate}
\end{equation}
For $\mathbf{v}_{2}$, we have full degeneracy 
\begin{equation}
\begin{aligned}\|\Lambda_{-2}P_{2}\|_{\dot{\calH}_{2}^{1}}+\|iP_{2}\|_{\dot{\calH}_{2}^{1}} & \aleq b^{2},\\
\|\partial_{b}P_{2}\|_{\dot{\calH}_{2}^{1}}+\|\partial_{\eta}P_{2}\|_{\dot{\calH}_{2}^{1}} & \aleq b.
\end{aligned}
\label{eq:v2-estimate}
\end{equation}
\item (Compatibiliity relations of $P,P_{1},P_{2}$) We have 
\begin{align}
\|\bfD_{P}P-P_{1}\|_{L^{2}} & \aleq b,\label{eq:Comparison-P-P1-H1}\\
\|\bfD_{P}P-P_{1}\|_{\dot{\calH}_{1}^{2}} & \aleq b^{2},\label{eq:Comparison-P-P1-H3}\\
\|A_{P}P_{1}-P_{2}\|_{\dot{\calH}_{2}^{1}} & \aleq\tfrac{b^{2}}{|\log b|}.\label{eq:Comparison-P1-P2-H3}
\end{align}
\item (Equation for $P$) We can write 
\begin{equation}
(\partial_{s}-\frac{\lmb_{s}}{\lmb}\Lambda+\gmm_{s}i)P+iL_{P}^{\ast}P_{1}=-\Mod\cdot\mathbf{v}+i\Psi\label{eq:P-equation}
\end{equation}
such that
\begin{align}
\chf_{(0,B_{0}/2]}|\Psi| & \aleq b^{2}\abs{\log b}.\label{eq:Psi-RoughBound}
\end{align}
\item (Equation for $P_{1}$) We can write 
\begin{equation}
(\partial_{s}-\frac{\lmb_{s}}{\lmb}\Lambda_{-1}+\td{\gmm}_{s}i)P_{1}+iA_{P}^{\ast}P_{2}-\Big(\int_{0}^{y}\Re(\overline{P}P_{1})dy'\Big)iP_{1}=-\td{\Mod}\cdot\mathbf{v}_{1}+i\Psi_{1}\label{eq:P1-equation}
\end{equation}
such that we have 
\begin{align}
\|\Psi_{1}\|_{X} & \aleq b^{3}|\log b|^{C}\label{eq:Psi1-LocalEnergy}
\end{align}
for some constant $C>0$.
\item (Equation for $P_{2}$) We can write 
\begin{equation}
\begin{aligned} & (\partial_{s}-\frac{\lmb_{s}}{\lmb}\Lambda_{-2}+\td{\gmm}_{s}i)P_{2}+iA_{P}A_{P}^{\ast}P_{2}-\Big(\int_{0}^{y}\Re(\overline{P}P_{1})dy'\Big)iP_{2}-i\overline{P}(P_{1})^{2}\\
 & =-\td{\Mod}\cdot\mathbf{v}_{2}+i\Psi_{2}
\end{aligned}
\label{eq:P2-equation}
\end{equation}
such that we have a sharp $\dot{\calH}_{2}^{1}$-estimate 
\begin{equation}
\|\Psi_{2}\|_{\dot{\calH}_{2}^{1}}\aleq\frac{b^{3}}{|\log b|}.\label{eq:Psi2-SharpEnergy}
\end{equation}
\end{enumerate}
\end{prop}

\begin{rem}
\label{rem:sharp-error}We make the general remark that, in order
to close the energy estimate in the main bootstrap argument in the
following section, the second line of \eqref{eq:v2-estimate} needs
to be sharp even up to the power of $\abs{\log b}$.

\eqref{eq:Psi2-SharpEnergy} seems to have a very little room. This
can be explained by following the blow-up analysis in the next section.
In the energy estimate, the size of \eqref{eq:Psi2-SharpEnergy} limits
the size of bootstrap assumption on $\|\eps_{3}\|_{L^{2}}$, which
is a $\dot{H}_{0}^{3}$-like quantity of $\eps$, and the size of
$\|\eps_{3}\|_{L^{2}}$ should be sufficiently small to justify the
sharp modulation equations for $b$ and $\eta$ (Lemma~\ref{lem:RefinedModulationEstimates}).
It seems that we have a room of only a small power of $|\log b|$
for \eqref{eq:Psi2-SharpEnergy}.

For the remaining error estimates at the same level, we have more
room; for instance \eqref{eq:Comparison-P1-P2-H3} only needs to be
of size $o(b^{2})$ as $b\to0$.
\end{rem}

\begin{rem}
\label{rem:profile-loc} We note that the larger localization scale
$B_{1}=b^{-\frac{1}{2}}\abs{\log b}$ for the first-order profiles
is needed for the localization errors in \eqref{eq:Comparison-P-P1-H3}
and \eqref{eq:Comparison-P1-P2-H3}; actually, in view of Remark~\ref{rem:sharp-error},
truncating at $y\aeq b^{-\frac{1}{2}}\abs{\log b}^{\alp}$ for any
$\alp>0$ is enough.

All the localization scales in the definition of $P,P_{1},P_{2}$
(i.e., $B_{0}$ and $B_{1}$) should be $b^{-\frac{1}{2}}$ up to
some logarithmic powers. For example, if one uses a smaller radius
$B'=b^{-\alpha}$ for some $0<\alpha<\frac{1}{2}$, then the profile
error $\Psi_{2}$ arsing from applying the cutoff $\chi_{B'}$ to
$U_{3,0}$ cannot satisfy \eqref{eq:Psi2-SharpEnergy}. On the other
hand, if one uses a larger scale $B'=b^{-\alpha}$ for some $\alpha>\frac{1}{2}$,
then the cutoff error measured in lower Sobolev norms might be harmful;
e.g., the second line of \eqref{eq:v2-estimate} would be violated
due to the growing tail of $U_{3,0}$. This explains why the localization
scale for $U_{3,0}$ should be the parabolic scale $b^{-\frac{1}{2}}$.

Moreover, in the definition of $g_{2}$, the radius $B_{0}=b^{-\frac{1}{2}}$
is also sharp in the sense that any other radii $b^{-\alpha}$, $\alpha\neq\frac{1}{2}$
are not allowed. Indeed, if we use some other radius $B'=b^{-\alpha}$
in the definition of $g_{2}$, the logarithmic gain $\frac{1}{|\log b|}$
for the profiles $U_{2}$, $T_{2,0}$, or $U_{3,0}$ would appear
at $y\ageq B'$ (see for example \eqref{eq:T20-lemma-temp}). In order
to obtain \eqref{eq:Psi2-SharpEnergy}, we need to take advantage
of this logarithmic gain, so the cutoff radius used in the definition
of $P_{2}$ (i.e., $B_{0}$) should detect this. In other words, $B'\leq B_{0}$,
i.e., $\alpha\leq\frac{1}{2}$. On the other hand, if $B'$ is too
small compared to $B_{0}$, then $\Psi_{2}$ would collect an error
of the form $(ib^{2}+2b\eta-i\eta^{2})c_{b}\tfrac{y}{2}Q(\chi_{B'}-\chi_{B_{0}}),$
whose $\dot{\calH}_{2}^{1}$- norm cannot satisfy \eqref{eq:Psi2-SharpEnergy}
if $\alpha<\frac{1}{2}$. Thus $\alpha=\frac{1}{2}$ is a tight choice.
\end{rem}

\begin{rem}
As we will see in Section~\ref{sec:Trapped-solutions}, the $P$-equation
\eqref{eq:P-equation} will be used in the modulation estimates of
$\lmb$ and $\gmm$; the $P_{1}$-equation \eqref{eq:P1-equation}
will be used in the modulation estimates of $b$ and $\eta$, and
also in the Morawetz corrections; the $P_{2}$-equation \eqref{eq:P2-equation}
will be used in the sharp third energy estimate. These tell us how
much error is acceptable for the profile errors $\Psi$, $\Psi_{1}$,
and $\Psi_{2}$. It is necessary for $\Psi$ and $\Psi_{1}$ to be
small in order not to disturb the modulation laws \eqref{eq:FormalParameterLaw}.
This says that it is only necessary to have $\Psi=o(b)$ and $\Psi_{1}=o(b^{2})$.
This also explains why it suffices to expand $P$ and $P_{1}$ in
lower order than $P_{2}$. 
\end{rem}

\begin{rem}
\label{rem:full-degen} The full degeneracy estimate \eqref{eq:v2-estimate}
for $\bfv_{2}$ holds thanks to the fact that $P_{2}\approx A_{Q}L_{Q}P$
at the linear level, while $A_{Q}L_{Q}(i\tfrac{y^{2}}{4}Q)=A_{Q}L_{Q}(\rho)=A_{Q}(\tfrac{y}{2}Q)=0$.
This cancellation allows for an easier treatment of the term $\td{\Mod}\cdot\mathbf{v}_{2}$
in the energy estimate compared to the general case without self-duality,
in which a higher derivative of $P$ is not expected to possess such
a degeneracy \cite{HillairetRaphael2012AnalPDE}.
\end{rem}

\begin{rem}
\label{rem:Nonlin-ansatz-P-P2}As mentioned in the introduction, when
$m\geq1$, the pseudoconformal blow-up construction in \cite{KimKwon2020arXiv}
can be further simplified by the current method. In the modified profile
construction, one can further take advantage of the nonlinear profile
ansatz $Q_{b}^{(\eta)}$ of \cite{KimKwon2019arXiv,KimKwon2020arXiv}
(see also the discussions at the beginning of this section) to define
the modified profiles for $w$, $w_{1}$, $w_{2}$ as 
\[
P=Q_{b}^{(\eta)}\chi_{B_{0}},\quad P_{1}=-(ib+\eta)\tfrac{y}{2}Q_{b}^{(\eta)}\chi_{B_{0}},\quad P_{2}=(ib+\eta)^{2}\tfrac{y^{2}}{4}Q_{b}^{(\eta)}\chi_{B_{0}}.
\]
\end{rem}

\begin{proof}
\textbf{Step 1:} \emph{Estimates for the modulation vectors.}

We first show \eqref{eq:v-estimate}. Due to the cutoff $\chf_{(0,B_{0}/2]}$,
we do not need to take care of the cutoff errors from the localizations
$\chi_{B_{0}},\chi_{B_{1}}$ in the definition of $P$. Thus 
\begin{align*}
\chf_{(0,B_{0}/2]}(\Lambda P-\Lambda Q) & =\chf_{(0,B_{0}/2]}(-ib\Lambda(\tfrac{y^{2}}{4}Q)-\eta\Lambda\rho),\\
\chf_{(0,B_{0}/2]}(iP-iQ) & =\chf_{(0,B_{0}/2]}(b\tfrac{y^{2}}{4}Q-i\eta\rho),\\
\chf_{(0,B_{0}/2]}(\partial_{b}P+i\tfrac{y^{2}}{4}Q) & =0,\\
\chf_{(0,B_{0}/2]}(\partial_{\eta}P+\rho) & =0.
\end{align*}
We view the RHS as errors and substitute the pointwise bounds from
the $\rho$-estimates \eqref{eq:rho-estimate}. This shows \eqref{eq:v-estimate}.

We turn to show \eqref{eq:v1-estimate}. We will use the rough estimates
\eqref{eq:RoughPointwise}: $|P_{1}|_{1}\aleq\chf_{(0,2B_{1}]}(b\tfrac{1}{\brk{y}}+b^{2}y)$.
In view of the $X$-norm \eqref{eq:Def-Xnorm}, we multiply by $\langle y\rangle^{-2}\langle\log_{+}y\rangle$
and take the $L^{2}$-norm to get the claims for $\Lambda_{-1}P_{1}$
and $iP_{1}$. For $\partial_{b}P_{1}$, we compute 
\begin{align*}
 & \partial_{b}P_{1}+i\tfrac{y}{2}Q\chi_{B_{1}}\\
 & =\chi_{B_{0}}\{2bT_{2,0}+b^{2}\partial_{b}T_{2,0}\}+(\partial_{b}\chi_{B_{0}})(b^{2}T_{2,0})+(\partial_{b}\chi_{B_{1}})(-(ib+\eta)\tfrac{y}{2}Q).
\end{align*}
Multiplying $\langle y\rangle^{-2}\langle\log_{+}y\rangle$ to the
RHS and taking $L^{2}$ yield the claim for $\partial_{b}P_{1}$.
For $\partial_{\eta}P_{1}$, we in fact have 
\[
\partial_{\eta}P_{1}+\tfrac{y}{2}Q\chi_{B_{1}}=0,
\]
thus the claim for $\partial_{\eta}P_{1}$ follows trivially.

We turn to show \eqref{eq:v2-estimate}. Due to the coercivity \eqref{eq:positivity-AQAQstar}
of $A_{Q}A_{Q}^{\ast}$, it suffices to estimate $\|A_{Q}^{\ast}\mathbf{v}_{2}\|_{L^{2}}$.
We will need to use the logarithmic gain induced by taking $A_{Q}^{\ast}$.
From the definitions of $U_{2}$ and $U_{3,0}$, we have 
\[
A_{Q}^{\ast}U_{2}=g_{2},\qquad A_{Q}^{\ast}U_{3,0}=g_{3,0}.
\]
We also have the scaling identity 
\[
A_{Q}^{\ast}\Lambda_{-2}P_{2}=\Lambda_{-3}A_{Q}^{\ast}P_{2}+\tfrac{1}{2}(yQ^{2})P_{2}.
\]
Thus the desired claim 
\[
\|A_{Q}^{\ast}\Lambda_{-2}P_{2}\|_{L^{2}}+\|A_{Q}^{\ast}iP_{2}\|_{L^{2}}\aleq\||A_{Q}^{\ast}P_{2}|_{1}\|_{L^{2}}+\|\langle y\rangle^{-3}P_{2}\|_{L^{2}}\aleq b^{2}
\]
follows from 
\begin{align*}
|A_{Q}^{\ast}P_{2}|_{1} & \aleq\chf_{(0,2B_{0}]}(b^{2}|g_{2}|_{1}+b^{3}|g_{3,0}|_{1})+\chf_{[B_{0},2B_{0}]}\tfrac{1}{y}|\widehat{P}_{2}|_{1},\\
|\langle y\rangle^{-3}P_{2}| & \aleq\chf_{(0,2B_{0}]}(b^{2}\langle y\rangle^{-3}+b^{3}\langle y\rangle^{-1}),
\end{align*}
and \eqref{eq:g2-sharp}, \eqref{eq:g30-sharp}, Lemma~\ref{lem:profiles}.
For $\partial_{b}P_{2}$, note that 
\begin{align*}
\partial_{b}P_{2} & =\chi_{B_{0}}\{(2b-2i\eta)U_{2}+3ib^{2}U_{3,0}+(b^{2}-2ib\eta-\eta^{2})\partial_{b}U_{2}+ib^{3}\partial_{b}U_{3,0}\}\\
 & \quad+(\partial_{b}\chi_{B_{0}})\widehat{P}_{2}.
\end{align*}
For the first line, we take $A_{Q}^{\ast}$, measure the $L^{2}$-norm
and proceed as before, where we also use \eqref{eq:db-g2-sharp} and
\eqref{eq:db-g30-sharp} for $A_{Q}^{\ast}\rd_{b}U_{2}=\rd_{b}g_{2}$
and $A_{Q}^{\ast}\rd_{b}U_{3,0}=\rd_{b}g_{3,0}$, respectively. For
the second line, we have $\||(\partial_{b}\chi_{B_{0}})\widehat{P}_{2}|_{-1}\|_{L^{2}}\aleq\frac{b}{|\log b|}$
by Lemma~\ref{lem:profiles}. For $\partial_{\eta}P_{2}$, note that
\[
\partial_{\eta}P_{2}=\chi_{B_{0}}(-2ib-2\eta)U_{2}.
\]
Again, we take $A_{Q}^{\ast}$, measure the $L^{2}$-norm and proceed
as before.

\textbf{Step 2:} \emph{The relations between $P,P_{1},P_{2}$.}

We first show \eqref{eq:Comparison-P-P1-H1} and \eqref{eq:Comparison-P-P1-H3}.
From the linearization of the Bogomol'nyi operator, we have 
\begin{align*}
\bfD_{P}P & =L_{Q}(P-Q)-\tfrac{1}{y}A_{\theta}[P-Q]Q-\tfrac{1}{y}(A_{\theta}[P]-A_{\theta}[Q])(P-Q)\\
 & =\chi_{B_{1}}\{-(ib+\eta)\tfrac{y}{2}Q\}+[L_{Q},\chi_{B_{1}}](-ib\tfrac{y^{2}}{4}Q-\eta\rho)\\
 & \quad-\tfrac{1}{y}A_{\theta}[\chi_{B_{1}}(-ib\tfrac{y^{2}}{4}Q-\eta\rho)]Q-\tfrac{1}{y}(A_{\theta}[P]-A_{\theta}[Q])(P-Q).
\end{align*}
By the definition of $P_{1}$, we see that the first term $\chi_{B_{1}}\{-(ib+\eta)\tfrac{y}{2}Q\}$
cancels: 
\begin{equation}
\begin{aligned}\bfD_{P}P-P_{1} & =[L_{Q},\chi_{B_{1}}](-ib\tfrac{y^{2}}{4}Q-\eta\rho)\\
 & \quad-\tfrac{1}{y}A_{\theta}[\chi_{B_{1}}(-ib\tfrac{y^{2}}{4}Q-\eta\rho)]Q\\
 & \quad-\tfrac{1}{y}(A_{\theta}[P]-A_{\theta}[Q])(P-Q)\\
 & \quad-\chi_{B_{0}}\{b^{2}T_{2,0}\}.
\end{aligned}
\label{eq:DP-P1-exp}
\end{equation}
It suffices to measure the $L^{2}$-difference and $\dot{\calH}_{1}^{2}$-difference
of the RHS.

We now estimate each line on the RHS of \eqref{eq:DP-P1-exp}. For
the first line, notice that 
\[
[L_{Q},\chi_{B_{1}}]f=(\rd_{y}\chi_{B_{1}})f+\tfrac{Q}{y}\left(\tint 0y\chi_{B_{1}}\Re fQy'dy'-\chi_{B_{1}}\tint 0y\Re fQy'dy'\right).
\]
Note that the second term is supported on $[B_{1},\infty)$ and only
uses the information of $f$ on $(0,2B_{1}]$. Thus $[L_{Q},\chi_{B_{1}}]f$
satisfies the pointwise estimates 
\begin{equation}
\abs{[L_{Q},\chi_{B_{1}}]f}_{2}\aleq\chf_{[B_{1},2B_{1}]}\tfrac{1}{y}\abs{f}_{2}+\chf_{[B_{1},\infty)}\tfrac{1}{y^{3}}{\textstyle \int_{0}^{2B_{1}}}\abs{f}\tfrac{1}{y'}dy'.\label{eq:LQ-chi-comm}
\end{equation}
Substituting $f=-ib\tfrac{y^{2}}{4}Q-\eta\rho$, Lemma~\ref{lem:profiles}
implies that 
\begin{align*}
\|[L_{Q},\chi_{B_{1}}](-ib\tfrac{y^{2}}{4}Q-\eta\rho)\|_{L^{2}} & \aleq b,\\
\||[L_{Q},\chi_{B_{1}}](-ib\tfrac{y^{2}}{4}Q-\eta\rho)|_{-2}\|_{L^{2}} & \aleq\tfrac{b^{2}}{|\log b|^{2}}.
\end{align*}
We remark that while the contribution of the second term in \eqref{eq:LQ-chi-comm}
is nonlocal, thanks to the fast decay $\tfrac{1}{y^{3}}$, its contribution
is better by $b\abs{\log b}^{C}$ compared to the first term.

For the second line of \eqref{eq:DP-P1-exp}, using the bound 
\[
|A_{\theta}[\chi_{B_{1}}(-ib\tfrac{y^{2}}{4}Q-\eta\rho)]|_{2}\aleq b^{2}\min\{y^{2},B_{1}^{2}\}
\]
we have 
\begin{align*}
\|\tfrac{1}{y}A_{\theta}[\chi_{B_{1}}(-ib\tfrac{y^{2}}{4}Q-\eta\rho)]Q\|_{L^{2}} & \aleq b^{2-},\\
\||\tfrac{1}{y}A_{\theta}[\chi_{B_{1}}(-ib\tfrac{y^{2}}{4}Q-\eta\rho)]Q|_{-2}\|_{L^{2}} & \aleq b^{2}.
\end{align*}

For the third line of \eqref{eq:DP-P1-exp}, we note the bound 
\begin{equation}
\chf_{(0,2B_{1}]}|\tfrac{1}{y}(A_{\theta}[P]-A_{\theta}[Q])|_{2}\aleq\tfrac{b}{|\log b|}\tfrac{\langle\log y\rangle}{\brk{y}}+b^{2}y,\label{eq:AthtP-AthtQ}
\end{equation}
which follows from 
\[
A_{\theta}[P]-A_{\theta}[Q]=-\tint 0y\Re(P-Q)Qy'dy'-\tfrac{1}{2}\tint 0y\abs{P-Q}^{2}y'dy'
\]
and the easy bounds 
\begin{equation}
\begin{aligned}|\Re(P-Q)|_{2} & \aleq\chf_{(0,2B_{1}]}\eta\leq\chf_{(0,2B_{1}]}\tfrac{b}{\abs{\log b}},\\
|\Im(P-Q)|_{2} & \aleq\chf_{(0,2B_{1}]}b.
\end{aligned}
\label{eq:P-Q-est}
\end{equation}
Thus 
\begin{align*}
\|\tfrac{1}{y}(A_{\theta}[P]-A_{\theta}[Q])(P-Q)\|_{L^{2}} & \aleq b^{2-},\\
\||\tfrac{1}{y}(A_{\theta}[P]-A_{\theta}[Q])(P-Q)|_{-2}\|_{L^{2}} & \aleq\tfrac{b^{2}}{|\log b|}.
\end{align*}

For the last line of \eqref{eq:DP-P1-exp}, the sharp estimates show
\[
\|\chi_{B_{0}}b^{2}T_{2,0}\|_{L^{2}}\aleq\tfrac{b}{|\log b|}.
\]
For the $\dot{\calH}_{1}^{2}$ estimate, crudely estimating the $\||\cdot|_{-2}\|_{L^{2}}$-norm
will give only $b^{2}|\log b|^{\frac{1}{2}}$, so we will elaborate
a little bit more. In view of the subcoercivity estimates \eqref{eq:subcoercivity-AQ}
and \eqref{eq:Coercivity-AQstar-appendix}, we have 
\begin{align*}
\|\chi_{B_{0}}b^{2}T_{2,0}\|_{\dot{\calH}_{1}^{2}} & \aleq\|A_{Q}(\chi_{B_{0}}b^{2}T_{2,0})\|_{\dot{\calH}_{2}^{1}}+\|\chf_{y\sim1}\chi_{B_{0}}b^{2}T_{2,0}\|_{L^{2}}\\
 & \aleq\|A_{Q}^{\ast}A_{Q}(\chi_{B_{0}}b^{2}T_{2,0})\|_{\dot{\calH}_{2}^{1}}+\|\chf_{y\sim1}b^{2}T_{2,0}\|_{L^{2}}.
\end{align*}
The second term is obviously bounded by $b^{2}$. Since $A_{Q}^{\ast}A_{Q}T_{2,0}=g_{2}$,
after commuting $A_{Q}^{\ast}A_{Q}$ with $\chi_{B_{0}}$, the first
term can be estimated by (using Lemma~\ref{lem:profiles} and \eqref{eq:g2-sharp})
\[
\|A_{Q}^{\ast}A_{Q}(\chi_{B_{0}}b^{2}T_{2,0})\|_{L^{2}}\aleq\|\chf_{(0,2B_{0}]}b^{2}g_{2}\|_{L^{2}}+\|\chf_{[B_{0},2B_{0}]}\tfrac{1}{y}b^{2}|T_{2,0}|_{-1}\|_{L^{2}}\aleq b^{2}.
\]
Thus \eqref{eq:Comparison-P-P1-H1} and \eqref{eq:Comparison-P-P1-H3}
are proved.

We turn to show \eqref{eq:Comparison-P1-P2-H3}. Using $A_{Q}(yQ)=0$
and $A_{Q}T_{2,0}=U_{2}$, we obtain 
\begin{align*}
A_{P}P_{1} & =A_{Q}P_{1}+(A_{P}-A_{Q})P_{1}\\
 & =\chi_{B_{0}}b^{2}U_{2}+(\partial_{y}\chi_{B_{0}})\{b^{2}T_{2,0}\}+(\partial_{y}\chi_{B_{1}})\{(-ib-\eta)\tfrac{y}{2}Q\}+(A_{P}-A_{Q})P_{1}.
\end{align*}
Therefore, we have 
\begin{align}
A_{P}P_{1}-P_{2} & =\chi_{B_{0}}\big\{(2ib\eta+\eta^{2})U_{2}-ib^{3}U_{3,0}\big\}\label{eq:APP1-P2-temp}\\
 & \quad+(\partial_{y}\chi_{B_{0}})\{b^{2}T_{2,0}\}+(\partial_{y}\chi_{B_{1}})\big\{(-ib-\eta)\tfrac{y}{2}Q\big\}+(A_{P}-A_{Q})P_{1}.\nonumber 
\end{align}
It remains to estimate the RHS of \eqref{eq:APP1-P2-temp} in the
$\dot{\calH}_{2}^{1}$-norm. For the first term, we use \eqref{eq:positivity-AQAQstar}
and $|\eta|\leq\frac{b}{|\log b|}$ to estimate 
\begin{align*}
 & \|A_{Q}^{\ast}(\chi_{B_{0}}\{(2ib\eta+\eta^{2})U_{2}-ib^{3}U_{3,0}\})\|_{L^{2}}\\
 & \aleq\|\chi_{B_{0}}(\tfrac{b^{2}}{|\log b|}|g_{2}|+b^{3}|g_{3,0}|)\|_{L^{2}}+\|\tfrac{1}{y}(\tfrac{b^{2}}{|\log b|}|U_{2}|+b^{3}|U_{3,0}|)\|_{L^{2}}\aleq\tfrac{b^{2}}{|\log b|},
\end{align*}
where in the last inequality we used \eqref{eq:g2-sharp} and \eqref{eq:g30-sharp}.
Using the sharp estimates in Lemma~\ref{lem:profiles} and $|\eta|\leq\frac{b}{|\log b|}$,
we have 
\begin{align*}
\||(\partial_{y}\chi_{B_{0}})b^{2}T_{2,0}|_{-1}\|_{L^{2}} & \aleq\tfrac{b^{2}}{|\log b|},\\
\||(\partial_{y}\chi_{B_{1}})\big\{(-ib-\eta)\tfrac{y}{2}Q\big\}|_{-1}\|_{L^{2}} & \aleq\tfrac{b^{2}}{|\log b|^{2}},
\end{align*}
where we used the logarithmic improvement $\frac{1}{|\log b|}$ of
$T_{2,0}$ in the region $[B_{0},2B_{0}]$ and $B_{1}=B_{0}|\log b|$.
Finally, we use \eqref{eq:AthtP-AthtQ} and $|P_{1}|_{-1}\aleq\chf_{(0,2B_{1}]}b\tfrac{1}{\langle y\rangle^{2}}$
to estimate 
\[
\||(A_{P}-A_{Q})P_{1}|_{-1}\|_{L^{2}}\aleq\|\chf_{(0,2B_{1}]}(\tfrac{b}{|\log b|}\tfrac{\langle\log y\rangle}{\langle y\rangle}+b^{2}y)\cdot b\tfrac{1}{\langle y\rangle^{2}}\|_{L^{2}}\aleq\tfrac{b^{2}}{|\log b|}.
\]
This completes the proof of \eqref{eq:Comparison-P1-P2-H3}.

\textbf{Step 3:}\emph{ Equation for $P$.}

Here, as our aim is to measure $\Psi$ in the region $(0,B_{0}/2]$,
in many cases (only except the $L_{P}^{\ast}$-part) the error computations
are simple and profile localization has no effect.

First, we note the computations 
\begin{equation}
\begin{aligned}\chf_{(0,B_{0}/2]}\partial_{s}P & =\chf_{(0,B_{0}/2]}\{(b_{s}+b^{2}+\eta^{2})\partial_{b}P+\eta_{s}\partial_{\eta}P+O(b^{2})\},\\
\chf_{(0,B_{0}/2]}\{-\frac{\lmb_{s}}{\lmb}\Lambda P\} & =\chf_{(0,B_{0}/2]}\{b\Lambda Q-\Big(\frac{\lmb_{s}}{\lmb}+b\Big)\Lambda P+O(b^{2})\},\\
\chf_{(0,B_{0}/2]}\gmm_{s}iP & =\chf_{(0,B_{0}/2]}\{\eta iQ+(\gmm_{s}-\eta)iP+O(\tfrac{b^{2}}{|\log b|})\},
\end{aligned}
\label{eq:P-equation-claim1}
\end{equation}
which easily follow from 
\begin{align*}
\chf_{(0,B_{0}/2]}(b^{2}+\eta^{2})\partial_{b}P & =O(b^{2}),\\
\chf_{(0,B_{0}/2]}b(\Lambda P-\Lambda Q) & =O(b^{2}),\\
\chf_{(0,B_{0}/2]}\eta(iP-iQ) & =O(\tfrac{b^{2}}{|\log b|}).
\end{align*}

Next, we claim that 
\begin{equation}
\chf_{(0,B_{0}/2]}iL_{P}^{\ast}P_{1}=\chf_{(0,B_{0}/2]}\{-b\Lambda Q-\eta iQ+O(b^{2}|\log b|)\}.\label{eq:P-equation-claim2}
\end{equation}
To see this, let us write 
\begin{align*}
iL_{P}^{\ast}P_{1} & =iL_{Q}^{\ast}(-(ib+\eta)\tfrac{y}{2}Q)\\
 & \quad+iL_{Q}^{\ast}((ib+\eta)(1-\chi_{B_{1}})\tfrac{y}{2}Q)+i(L_{P}^{\ast}-L_{Q}^{\ast})(-(ib+\eta)\tfrac{y}{2}Q\chi_{B_{1}})\\
 & \quad+iL_{P}^{\ast}(\chi_{B_{0}}b^{2}T_{2,0}).
\end{align*}
For the first term, we use $L_{Q}^{\ast}(i\tfrac{y}{2}Q)=-i\Lambda Q$
and $L_{Q}^{\ast}(\tfrac{y}{2}Q)=Q$ to get 
\[
iL_{Q}^{\ast}(-(ib+\eta)\tfrac{y}{2}Q)=-b\Lambda Q-\eta iQ.
\]
For the second term, we have 
\[
\chf_{(0,B_{0}/2]}iL_{Q}^{\ast}((ib+\eta)(1-\chi_{B_{1}})\tfrac{y}{2}Q)\aleq\chf_{(0,B_{0}/2]}bQ{\textstyle \int_{B_{1}}^{\infty}}Q^{2}y'dy'\aleq\chf_{(0,B_{0}/2]}\tfrac{b^{2}}{|\log b|^{2}}.
\]
For the third term, we note that 
\begin{align*}
 & \chf_{(0,B_{0}/2]}i(L_{P}^{\ast}-L_{Q}^{\ast})(-(ib+\eta)\tfrac{y}{2}Q\chi_{B_{1}})\\
 & \aleq\chf_{(0,B_{0}/2]}b\{|A_{\theta}[P]-A_{\theta}[Q]|Q+|P-Q|{\textstyle \int_{0}^{2B_{1}}}Q^{2}y'dy'+|P|{\textstyle \int_{0}^{2B_{1}}}|P-Q|Qy'dy'\}.
\end{align*}
Using $\chf_{(0,B_{0}/2]}|A_{\theta}[P]-A_{\theta}[Q]|+|P-Q|\aleq b$,
which follow from \eqref{eq:P-Q-est} and \eqref{eq:AthtP-AthtQ},
we see that 
\[
\chf_{(0,B_{0}/2]}i(L_{P}^{\ast}-L_{Q}^{\ast})(-(ib+\eta)\tfrac{y}{2}Q\chi_{B_{1}})\aleq b^{2}\abs{\log b}.
\]
For the fourth term, note that 
\[
L_{P}^{\ast}(\chi_{B_{0}}f)\aleq\tfrac{1}{y}\abs{f}_{1}+\tfrac{1}{y}\abs{A_{\tht}[P]}\abs{f}+\abs{P}\tint 0{2B_{0}}\abs{P-Q}\abs{f}dy'+\abs{P}\tint 0{2B_{0}}Q\abs{f}dy'
\]
By the rough pointwise bounds \eqref{eq:RoughPointwise}, we have
$\chf_{(0,B_{0}/2]}\tfrac{1}{y}\abs{A_{\tht}[P]}+\abs{P}\aleq1$.
Then using \eqref{eq:RoughPointwise} again for $T_{2,0}$, we see
that 
\[
\chf_{(0,B_{0}/2]}L_{P}^{\ast}(\chi_{B_{0}}b^{2}T_{2,0})\aleq b^{2}\abs{\log b}.
\]
Thus the claim \eqref{eq:P-equation-claim2} is proved.

Summing up the claims \eqref{eq:P-equation-claim1} and \eqref{eq:P-equation-claim2},
we have 
\begin{align*}
(\partial_{s}-\frac{\lmb_{s}}{\lmb}\Lambda+\gmm_{s}i)P+iL_{P}^{\ast}P_{1} & =-\Mod\cdot\mathbf{v}+i\Psi
\end{align*}
with 
\[
\chf_{(0,B_{0}/2]}|\Psi|\aleq b^{2}\abs{\log b}.
\]

\textbf{Step 4: }\emph{Equation for $P_{1}$ and refined modulation
equations.}

Although we motivated the profile $U_{3,0}$ using the $w_{1}$-equation
by solving up to $O(b^{3})$ correctly, here it is not necessary to
keep track of $O(b^{3})$-terms because the asserted claim \eqref{eq:Psi1-LocalEnergy}
only requires $O_{X}(b^{3}|\log b|^{C})$. Thus we will only keep
up to quadratic terms. However, in Step 5, we need to keep track of
the $O(b^{3})$-terms in order to get the sharp estimate \eqref{eq:Psi2-SharpEnergy}.

First, we claim that 
\begin{equation}
\begin{aligned}\partial_{s}P_{1} & =\chi_{B_{0}}\{(ib^{2}+i\eta^{2}+ic_{b}(b^{2}-\eta^{2})+2c_{b}b\eta)\tfrac{y}{2}Q\}\\
 & \quad+(b_{s}+b^{2}+\eta^{2}+c_{b}(b^{2}-\eta^{2}))\partial_{b}P_{1}+(\eta_{s}+2c_{b}b\eta)\partial_{\eta}P_{1}\\
 & \quad+O_{X}(b^{3}|\log b|^{C}).
\end{aligned}
\label{eq:P1-equation-claim1}
\end{equation}
This would follow from 
\begin{align*}
\partial_{b}P_{1}+\chi_{B_{0}}(i\tfrac{y}{2}Q) & =O_{X}(b|\log b|^{C}),\\
\partial_{\eta}P_{1}+\chi_{B_{0}}\tfrac{y}{2}Q & =O_{X}(b|\log b|^{C}).
\end{align*}
These follow from \eqref{eq:v1-estimate} and 
\[
\nrm{(\chi_{B_{1}}-\chi_{B_{0}})yQ}_{X}\aleq\nrm{\chf_{[B_{0},2B_{1}]}\tfrac{1}{y^{3}}\brk{\log_{+}y}}_{L^{2}}\aleq b\abs{\log b}^{C}.
\]

Next, we claim that 
\begin{equation}
\begin{aligned}-\frac{\lmb_{s}}{\lmb}\Lambda_{-1}P_{1} & =\chi_{B_{0}}\{(-ib^{2}-b\eta)\Lambda_{-1}(\tfrac{y}{2}Q)\}-\Big(\frac{\lmb_{s}}{\lmb}+b\Big)\Lambda_{-1}P_{1}\\
 & \quad+O_{X}(b^{3}|\log b|^{C}).
\end{aligned}
\label{eq:P1-equation-claim2}
\end{equation}
This would follow from 
\[
\Lambda_{-1}P_{1}+\chi_{B_{0}}\{(ib+\eta)\Lambda_{-1}\tfrac{y}{2}Q\}=O_{X}(b^{2}|\log b|^{C}),
\]
which in turn follows from applying the rough estimates \eqref{eq:RoughPointwise}
to 
\begin{align*}
 & \Lambda_{-1}P_{1}+\chi_{B_{0}}\{(ib+\eta)\Lambda_{-1}(\tfrac{y}{2}Q)\}\\
 & =\Lambda_{-1}\{\chi_{B_{0}}b^{2}T_{2,0}\}-(\chi_{B_{1}}-\chi_{B_{0}})(ib+\eta)\Lambda_{-1}(\tfrac{y}{2}Q)-(y\partial_{y}\chi_{B_{1}})((ib+\eta)\tfrac{y}{2}Q).
\end{align*}

Next, we claim that 
\begin{align}
\td{\gmm}_{s}iP_{1} & =\chi_{B_{0}}\{(-b\eta+i\eta^{2})\tfrac{y}{2}Q\}+(\td{\gmm}_{s}+\eta)iP_{1}+O_{X}(b^{3}|\log b|^{C}).\label{eq:P1-equation-claim3}
\end{align}
This would follow from 
\[
P_{1}+\chi_{B_{0}}(ib+\eta)\tfrac{y}{2}Q=O_{X}(b^{2}|\log b|^{C}),
\]
which follows from applying the rough estimates to: 
\[
P_{1}+\chi_{B_{0}}(ib+\eta)\tfrac{y}{2}Q=\chi_{B_{0}}\{b^{2}T_{2,0}\}-(\chi_{B_{1}}-\chi_{B_{0}})(ib+\eta)\tfrac{y}{2}Q.
\]

Next, we claim that 
\begin{equation}
-({\textstyle \int_{0}^{y}}\Re(\overline{P}P_{1})dy')iP_{1}=\chi_{B_{0}}(b\eta-i\eta^{2})(2-\Lambda)\tfrac{y}{2}Q+O_{X}(b^{3}|\log b|^{C}).\label{eq:P1-equation-claim4}
\end{equation}
To show this, we begin with the bounds 
\begin{align*}
\overline{P} & =Q+\chi_{B_{1}}\{ib\tfrac{y^{2}}{4}Q-\eta\rho\},\\
P_{1} & =\chi_{B_{1}}\{-(ib+\eta)\tfrac{y}{2}Q\}+\chf_{(0,2B_{0}]}O(b^{2}y),
\end{align*}
which follow from \eqref{eq:RoughPointwise}. It then follows that
\begin{equation}
\begin{aligned}\Re(\overline{P}P_{1}) & =\chf_{(0,2B_{0}]}(-\eta\tfrac{y}{2}Q^{2})+\chf_{(0,2B_{1}]}O(b^{2}\tfrac{1}{\langle y\rangle}).\end{aligned}
\label{eq:Re-P-P1}
\end{equation}
Hence, 
\begin{align*}
 & \chf_{(0,2B_{1}]}{\textstyle \int_{0}^{y}}\Re(\overline{P}P_{1})dy'\\
 & =\chi_{B_{0}}{\textstyle \int_{0}^{y}}(-\eta\tfrac{y'}{2}Q^{2})dy'+O(\chf_{(0,B_{0}]}b^{2}|\log b|+\chf_{[B_{0},2B_{1}]}\tfrac{b}{|\log b|}).
\end{align*}
Thus 
\begin{align*}
 & -({\textstyle \int_{0}^{y}}\Re(\overline{P}P_{1})dy')iP_{1}\\
 & =\chi_{B_{0}}\{(b\eta-i\eta^{2})({\textstyle \int_{0}^{y}}\tfrac{y'}{2}Q^{2}dy')\tfrac{y}{2}Q\}+O(\chf_{(0,B_{0}]}\tfrac{b^{3}}{|\log b|}y+\chf_{[B_{0},2B_{1}]}\tfrac{b^{2}}{|\log b|}\tfrac{1}{y}).
\end{align*}
The last term contributes to the error $O_{X}(b^{3}|\log b|^{C})$
as desired. The proof of the claim \eqref{eq:P1-equation-claim4}
follows from the computation 
\[
({\textstyle \int_{0}^{y}}\tfrac{y'}{2}Q^{2}dy')\tfrac{y}{2}Q=(2-\Lambda)(\tfrac{y}{2}Q),
\]
where we have used $A_{Q}(\tfrac{y}{2}Q)=0$.

Finally, we claim that 
\begin{equation}
iA_{P}^{\ast}P_{2}=\chi_{B_{0}}\{(ib^{2}+2b\eta-i\eta^{2})g_{2}\}+O_{X}(b^{3}|\log b|^{C}).\label{eq:P1-equation-claim5}
\end{equation}
In fact, we will prove a stronger estimate for later use in Step 5:
\begin{equation}
iA_{P}^{\ast}P_{2}=\chi_{B_{0}}\{(ib^{2}+2b\eta-i\eta^{2})g_{2}-b^{3}g_{3,0}\}+\chf_{(0,2B_{0}]}O_{\||\cdot|_{-2}\|_{L^{2}}}(\tfrac{b^{3}}{|\log b|}).\label{eq:A_P-P_2}
\end{equation}
To see this, we start from 
\[
iA_{P}^{\ast}P_{2}=\chi_{B_{0}}i(A_{Q}^{\ast}\widehat{P}_{2})+(\partial_{y}\chi_{B_{0}})i\widehat{P}_{2}+i(A_{P}^{\ast}-A_{Q}^{\ast})P_{2}.
\]
We keep the first term in the form 
\begin{align*}
\chi_{B_{0}}i(A_{Q}^{\ast}\widehat{P}_{2}) & =\chi_{B_{0}}\{(ib^{2}+2b\eta-i\eta^{2})A_{Q}^{\ast}U_{2}-b^{3}A_{Q}^{\ast}U_{3,0}\}\\
 & =\chi_{B_{0}}\{(ib^{2}+2b\eta-i\eta^{2})g_{2}-b^{3}g_{3,0}\}.
\end{align*}
For the second term, we use $|\partial_{y}\chi_{B_{0}}|_{2}\aleq\chf_{[B_{0},2B_{0}]}\frac{1}{y}$
and $\chf_{[B_{0},2B_{0}]}|\widehat{P}_{2}|_{2}\aleq\frac{b^{2}}{|\log b|}$
to get 
\[
\||(\partial_{y}\chi_{B_{0}})i\widehat{P}_{2}|_{-2}\|\aleq\tfrac{b^{3}}{|\log b|}.
\]
For the last term, we note that, by \eqref{eq:AthtP-AthtQ}, 
\[
\abs{(A_{P}^{\ast}-A_{Q}^{\ast})f}_{2}=\abs{-\tfrac{1}{y}(A_{\tht}[P]-A_{\tht}[Q])f}_{2}\aleq(\tfrac{b}{|\log b|}\tfrac{\langle\log y\rangle}{\brk{y}}+b^{2}y)\abs{f}_{2}.
\]
Using also $|P_{2}|_{2}\aleq\chf_{(0,2B_{0}]}b^{2}$, which follows
from \eqref{eq:RoughPointwise}, we have 
\[
\||(A_{P}^{\ast}-A_{Q}^{\ast})P_{2}|_{-2}\|_{L^{2}}\aleq\tfrac{b^{3}}{|\log b|}.
\]

Summing up the above claims \eqref{eq:P1-equation-claim1}--\eqref{eq:P1-equation-claim5},
we have 
\begin{align*}
 & (\partial_{s}-\frac{\lmb_{s}}{\lmb}\Lambda_{-1}+\td{\gmm}_{s}i)P_{1}+iA_{P}P_{2}-\Big(\int_{0}^{y}\Re(\overline{P}P_{1})dy'\Big)iP_{1}\\
 & =-\td{\Mod}\cdot\mathbf{v}_{1}+\chi_{B_{0}}\{(ib^{2}+2b\eta-i\eta^{2})(g_{2}-\Lambda(\tfrac{y}{2}Q)+c_{b}\tfrac{y}{2}Q)\}+O_{X}(b^{3}|\log b|^{C}).
\end{align*}
By the definition of $g_{2}$, the quadratic order terms \emph{almost
vanish}: 
\begin{align*}
 & \chi_{B_{0}}\{(ib^{2}+2b\eta-i\eta^{2})(g_{2}-\Lambda(\tfrac{y}{2}Q)+c_{b}\tfrac{y}{2}Q)\}\\
 & =\chi_{B_{0}}(1-\chi_{B_{0}})(ib^{2}+2b\eta-i\eta^{2})c_{b}\tfrac{y}{2}Q=O_{X}(b^{3}|\log b|^{C}).
\end{align*}
Therefore, we can rearrange the above display as 
\[
=-\td{\Mod}\cdot\mathbf{v}_{1}+O_{X}(b^{3}|\log b|^{C})\eqqcolon-\td{\Mod}\cdot\mathbf{v}_{1}+i\Psi_{1}.
\]
This coompletes the proof of \eqref{eq:Psi1-LocalEnergy}.

\textbf{Step 5:} \emph{Equation for $P_{2}$ and sharp energy estimates.}

First, we claim that 
\begin{equation}
\begin{aligned}\partial_{s}P_{2} & =-2b^{3}U_{2}\chi_{B_{0}}+(b_{s}+b^{2}+\eta^{2}+c_{b}(b^{2}-\eta^{2}))\partial_{b}P_{2}\\
 & \quad+i(\eta_{s}+2c_{b}b\eta)\partial_{\eta}P_{2}+O_{\dot{\calH}_{2}^{1}}(\tfrac{b^{3}}{|\log b|}).
\end{aligned}
\label{eq:P2-equation-claim1}
\end{equation}
We note that the terms including $c_{b}$ can be considered as an
error, but we include them to match the formula for $\td{\Mod}$.
By \eqref{eq:v2-estimate} and \eqref{eq:positivity-AQAQstar}, the
claim would follow from 
\begin{align*}
\|A_{Q}^{\ast}(\partial_{b}P_{2}-2bU_{2}\chi_{B_{0}})\|_{L^{2}} & \aleq\tfrac{b}{|\log b|}.
\end{align*}
We compute 
\begin{align*}
 & \partial_{b}P_{2}-2bU_{2}\chi_{B_{0}}\\
 & =\chi_{B_{0}}(-2i\eta\td U_{2}+3ib^{2}U_{3,0}+(b^{2}-\eta^{2}-2ib\eta)\partial_{b}\td U_{2}+ib^{3}\partial_{b}U_{3,0})+(\partial_{b}\chi_{B_{0}})\widehat{P}_{2}.
\end{align*}
Taking $A_{Q}^{\ast}$, we have 
\begin{align*}
 & A_{Q}^{\ast}(\partial_{b}P_{2}-2bU_{2}\chi_{B_{0}})\\
 & =\chi_{B_{0}}(-2i\eta g_{2}+3ib^{2}g_{3,0}+(b^{2}-\eta^{2}-2ib\eta)\partial_{b}g_{2}+ib^{3}\partial_{b}g_{3,0})\\
 & \quad+(\partial_{y}\chi_{B_{0}})(-2i\eta U_{2}+3ib^{2}U_{3,0}+(b^{2}-\eta^{2}-2ib\eta)\partial_{b}U_{2}+ib^{3}\partial_{b}U_{3,0}))\\
 & \quad+A_{Q}^{\ast}((\partial_{b}\chi_{B_{0}})\widehat{P}_{2}).
\end{align*}
Using $|\eta|\leq\frac{b}{|\log b|}$ and the sharp bounds \eqref{eq:g2-sharp},
\eqref{eq:db-g2-sharp}, \eqref{eq:g30-sharp} and \eqref{eq:db-g30-sharp},
we have 
\[
\|\chi_{B_{0}}(-2i\eta g_{2}+3ib^{2}g_{3,0}+(b^{2}-\eta^{2}-2ib\eta)\partial_{b}g_{2}+ib^{3}\partial_{b}g_{3,0})\|_{L^{2}}\aleq\tfrac{b}{|\log b|}.
\]
Next, using the logarithmic gain at $y\sim B_{0}$ in Lemma~\ref{lem:profiles},
we also have 
\begin{align*}
\nrm{(\partial_{y}\chi_{B_{0}})(-2i\eta U_{2}+3ib^{2}U_{3,0}+(b^{2}-\eta^{2}-2ib\eta)\partial_{b}U_{2}+ib^{3}\partial_{b}U_{3,0}))}_{L^{2}} & \aleq\tfrac{b}{|\log b|},\\
\nrm{A_{Q}^{\ast}((\partial_{b}\chi_{B_{0}})\widehat{P}_{2})}_{L^{2}} & \aleq\tfrac{b}{|\log b|}.
\end{align*}

Next, we claim that 
\begin{equation}
-\frac{\lmb_{s}}{\lmb}\Lambda_{-2}P_{2}=\chi_{B_{0}}b^{3}\Lambda_{-2}U_{2}-\Big(\frac{\lmb_{s}}{\lmb}+b\Big)\Lambda_{-2}P_{2}+O_{\dot{\calH}_{2}^{1}}(\tfrac{b^{3}}{|\log b|}).\label{eq:P2-equation-claim2}
\end{equation}
By \eqref{eq:positivity-AQAQstar}, it suffices to show 
\[
\|A_{Q}^{\ast}((y\partial_{y}\chi_{B_{0}})\widehat{P}_{2}+\chi_{B_{0}}\Lambda_{-2}(\widehat{P}_{2}-b^{2}U_{2}))\|_{L^{2}}\aleq\tfrac{b^{2}}{|\log b|}.
\]
For this, further using $A_{Q}^{\ast}\Lambda_{-2}=\Lambda_{-3}A_{Q}^{\ast}-\tfrac{yQ^{2}}{2}$,
it suffices to show 
\begin{align*}
\|\chf_{(0,2B_{0}]}|(-2ib\eta-\eta^{2})g_{2}+ib^{3}g_{3,0}|_{1}\|_{L^{2}} & \aleq\tfrac{b^{2}}{|\log b|},\\
\|\chf_{(0,2B_{0}]}\tfrac{1}{\langle y\rangle^{3}}|(-2ib\eta-\eta^{2})U_{2}+ib^{3}U_{3,0}|\|_{L^{2}} & \aleq\tfrac{b^{2}}{|\log b|},\\
\|\chf_{[B_{0},2B_{0}]}\tfrac{1}{y}(b^{2}|U_{2}|_{1}+b^{3}|U_{3,0}|_{1})\|_{L^{2}} & \aleq\tfrac{b^{2}}{|\log b|}.
\end{align*}
These are now immediate consequences of $\abs{\eta}\leq\tfrac{b}{\abs{\log b}}$,
the sharp pointwise bounds in Lemma~\ref{lem:profiles}, as well
as \eqref{eq:g2-sharp} and \eqref{eq:g30-sharp}.

Next, we claim that 
\begin{equation}
\td{\gmm}_{s}iP_{2}=(\td{\gmm}_{s}+\eta)iP_{2}+O_{\dot{\calH}_{2}^{1}}(\tfrac{b^{3}}{|\log b|}).\label{eq:P2-equation-claim3}
\end{equation}
This immediately follows from $|\eta|\leq\tfrac{b}{|\log b|}$ and
\eqref{eq:v2-estimate}.

Next, we claim that 
\begin{equation}
\|-({\textstyle \int_{0}^{y}}\Re(\overline{P}P_{1})dy')iP_{2}\|_{\dot{\calH}_{2}^{1}}\aleq\tfrac{b^{3}}{|\log b|}.\label{eq:P2-equation-claim4}
\end{equation}
By \eqref{eq:positivity-AQAQstar}, it suffices to show 
\[
\|\Re(\overline{P}P_{1})P_{2}\|_{L^{2}}+\|({\textstyle \int_{0}^{y}}\Re(\overline{P}P_{1})dy')A_{Q}^{\ast}P_{2}\|_{L^{2}}\aleq\tfrac{b^{3}}{|\log b|}.
\]
Since $P_{2}$ is supported in $(0,2B_{0}]$, it suffices to estimate
on that region. Note that $\Re(\overline{P}P_{1})=O(|\eta|+b^{2})$,
at least in $y\aleq1$. By \eqref{eq:Re-P-P1} and the rough bound
$\abs{P_{2}}\aleq b^{2}$, the first one $\|\Re(\overline{P}P_{1})P_{2}\|_{L^{2}}\aleq\frac{b^{3}}{|\log b|}$
follows. The second one follows from 
\[
\chf_{(0,2B_{0}]}|{\textstyle \int_{0}^{y}}\Re(\overline{P}P_{1})dy'|\aleq\tfrac{b}{|\log b|},
\]
which is proved using $\abs{\eta}\leq\tfrac{b}{\abs{\log b}}$ and
\eqref{eq:Re-P-P1}, as well as 
\begin{align*}
A_{Q}^{\ast}P_{2} & \aleq\chf_{(0,2B_{0}]}(b^{2}|g_{2}|+b^{3}|g_{3,0}|)+\chf_{[B_{0},2B_{0}]}\tfrac{1}{y}|P_{2}|\\
 & \aleq\chf_{(0,2B_{0}]}(b^{2}\tfrac{1}{\abs{\log b}\brk{y}}+b^{3}y)+\chf_{[B_{0},2B_{0}]}\tfrac{b^{2}}{\abs{\log b}y},
\end{align*}
where we used Lemma~\ref{lem:profiles}, \eqref{eq:g2-sharp} and
\eqref{eq:g30-sharp}.

Next, we claim that 
\begin{align}
-i\overline{P}(P_{1})^{2} & =\chi_{B_{0}}\{(ib^{2}+2b\eta-i\eta^{2})(\tfrac{y^{2}}{4}Q^{3})-b^{3}(yQ^{2}T_{2,0}+\tfrac{y^{4}}{16}Q^{3})\}\label{eq:P2-equation-claim5}\\
 & \quad+O_{\dot{\calH}_{2}^{1}}(\tfrac{b^{3}}{|\log b|}).\nonumber 
\end{align}
To see this, it suffices to use the rough estimates \eqref{eq:RoughPointwise}
and $|\eta|\leq\frac{b}{|\log b|}$, by which we have 
\begin{align*}
\overline{P}P_{1} & =-\chi_{B_{1}}(ib+\eta)\tfrac{y}{2}Q^{2}+b^{2}(\chi_{B_{0}}QT_{2,0}+\chi_{B_{1}}^{2}\tfrac{y^{3}}{8}Q^{2})\\
 & \peq+O(\chf_{(0,2B_{1}]}\tfrac{b^{2}}{|\log b|}\tfrac{1}{\brk{y}}+\chf_{(0,2B_{0}]}b^{3}y).
\end{align*}
Thus 
\begin{align*}
-i\overline{P}(P_{1})^{2} & =\chi_{B_{1}}^{2}(ib^{2}+2b\eta-i\eta^{2})\tfrac{y^{2}}{4}Q^{3}-b^{3}(\chi_{B_{0}}yQ^{2}T_{2,0}+\chi_{B_{1}}^{3}\tfrac{y^{4}}{16}Q^{3})\\
 & \peq+O(\chf_{(0,2B_{1}]}\tfrac{b^{3}}{|\log b|}\tfrac{1}{\brk{y}^{2}}+\chf_{(0,2B_{0}]}b^{4})\\
 & =\chi_{B_{0}}\{(ib^{2}+2b\eta-i\eta^{2})\tfrac{y^{2}}{4}Q^{3}-b^{3}(yQ^{2}T_{2,0}+\tfrac{y^{4}}{16}Q^{3})\}\\
 & \peq+O(\chf_{[B_{0},2B_{1}]}b^{3}\tfrac{1}{y^{2}}+\chf_{(0,2B_{1}]}\tfrac{b^{3}}{|\log b|}\tfrac{1}{\brk{y}^{2}}+\chf_{(0,2B_{0}]}b^{4}).
\end{align*}
Taking the $\nrm{\abs{\cdot}_{-1}}_{L^{2}}$ norm, the claim follows.

Next, we claim that 
\begin{equation}
iA_{P}A_{P}^{\ast}P_{2}=\chi_{B_{0}}\{(ib^{2}+2b\eta-i\eta^{2})A_{Q}g_{2}-b^{3}A_{Q}g_{3,0}\}+O_{\dot{\calH}_{2}^{1}}(\tfrac{b^{3}}{|\log b|}),\label{eq:P2-equation-claim6}
\end{equation}
Recall from \eqref{eq:A_P-P_2} that 
\[
A_{P}^{\ast}P_{2}=\chi_{B_{0}}\{(ib^{2}+2b\eta-i\eta^{2})g_{2}-b^{3}g_{3,0}\}+\chf_{(0,2B_{0}]}\cdot O_{\||\cdot|_{-2}\|_{L^{2}}}(\tfrac{b^{3}}{|\log b|}).
\]
Thus 
\[
iA_{Q}A_{P}^{\ast}P_{2}=\chi_{B_{0}}\{(ib^{2}+2b\eta-i\eta^{2})A_{Q}g_{2}-b^{3}A_{Q}g_{3,0}\}+O_{\dot{\calH}_{2}^{1}}(\tfrac{b^{3}}{|\log b|}).
\]
On the other hand, using $(A_{P}-A_{Q})f=-\tfrac{1}{y}(A_{\tht}[P]-A_{\tht}[Q])f$,
\eqref{eq:AthtP-AthtQ}, \eqref{eq:A_P-P_2}, \eqref{eq:g2-sharp}
and \eqref{eq:g30-sharp}, we have 
\[
\nrm{(A_{P}-A_{Q})A_{P}^{\ast}P_{2}}_{\dot{\calH}_{2}^{1}}\aleq\nrm{\chf_{(0,2B_{0}]}(\tfrac{b}{\abs{\log b}}\tfrac{\brk{\log y}}{\brk{y}}+b^{2}y)\abs{A_{P}^{\ast}P_{2}}_{-1}}_{L^{2}}\aleq\tfrac{b^{3}}{\abs{\log b}}.
\]
Thus the claim is shown.

Summing up the above claims \eqref{eq:P2-equation-claim1}--\eqref{eq:P2-equation-claim6}
yield 
\begin{align*}
 & (\partial_{s}-\frac{\lmb_{s}}{\lmb}\Lambda_{-2}+\td{\gmm}_{s}i)P_{2}+iA_{P}A_{P}^{\ast}P_{2}-\Big(\int_{0}^{y}\Re(\overline{P}P_{1})dy'\Big)iP_{2}-i\overline{P}(P_{1})^{2}\\
 & =-\td{\Mod}\cdot\mathbf{v}_{2}+\chi_{B_{0}}\{(ib^{2}+2b\eta-i\eta^{2})(\tfrac{y^{2}}{4}Q^{3}+A_{Q}g_{2})\\
 & \quad+b^{3}(\Lambda U_{2}-yQ^{2}T_{2,0}-\tfrac{y^{4}}{16}Q^{3}-A_{Q}g_{3,0})\}+O_{\dot{\calH}_{2}^{1}}(\tfrac{b^{3}}{|\log b|}).
\end{align*}
In fact, the $b^{3}$-order term vanishes, by the definition of $U_{3,0}$.
To see this, we rearrange the $b^{3}$-order term as 
\begin{align*}
 & \Lambda U_{2}-yQ^{2}T_{2,0}-\tfrac{y^{4}}{16}Q^{3}-A_{Q}g_{3,0}\\
 & =\Lambda(A_{Q}T_{2,0})-\tfrac{1}{2}yQ^{2}T_{2,0}-(QT_{2,0}+\tfrac{y^{3}}{8}Q^{2})(\tfrac{y}{2}Q)-A_{Q}g_{3,0}.
\end{align*}
Using the scaling identity $\Lambda A_{Q}T_{2,0}-\tfrac{1}{2}yQ^{2}T_{2,0}=A_{Q}\Lambda_{1}T_{2,0}$
and \eqref{eq:AQg30}, the above display continues as 
\[
=A_{Q}\Lambda_{1}T_{2,0}-(QT_{2,0}+\tfrac{y^{3}}{8}Q^{2})(\tfrac{y}{2}Q)-A_{Q}g_{3,0}=0.
\]
Next, by the definition of $g_{2}$, the quadratic order term \emph{almost
vanishes}. Indeed, using the scaling identity $A_{Q}\Lambda=\Lambda_{-1}A_{Q}-\tfrac{yQ^{2}}{2}$
and $A_{Q}(yQ)=0$, we have 
\begin{align*}
 & \chi_{B_{0}}\{(ib^{2}+2b\eta-i\eta^{2})(\tfrac{y^{2}}{4}Q^{3}+A_{Q}g_{2})\}\\
 & =-\chi_{B_{0}}\{c_{b}(ib^{2}+2b\eta-i\eta^{2})(\partial_{y}\chi_{B_{0}})\tfrac{y}{2}Q\}=O_{\dot{\calH}_{2}^{1}}(\tfrac{b^{3}}{|\log b|}).
\end{align*}
Therefore, 
\begin{align*}
 & (\partial_{s}-\frac{\lmb_{s}}{\lmb}\Lambda_{-2}+\td{\gmm}_{s}i)P_{2}+iA_{P}A_{P}^{\ast}P_{2}-\Big(\int_{0}^{y}\Re(\overline{P}P_{1})dy'\Big)iP_{2}-i\overline{P}(P_{1})^{2}\\
 & =-\td{\Mod}\cdot\mathbf{v}_{2}+O_{\dot{\calH}_{2}^{1}}(\tfrac{b^{3}}{|\log b|})\eqqcolon-\td{\Mod}\cdot\mathbf{v}_{2}+\Psi_{2}.
\end{align*}
The proof of \eqref{eq:Psi2-SharpEnergy} is now completed. 
\end{proof}

\section{\label{sec:Trapped-solutions}Trapped solutions}

So far, we constructed the modified profiles $P$, $P_{1}$, $P_{2}$,
and derived the formal modulation equations \eqref{eq:FormalParameterLaw}.
Applying the modulation parameters satisfying \eqref{eq:FormalParameterLaw}
to the modified profiles give approximate finite-time blow-up solutions
to \eqref{eq:CSS-equiv-u}. In this section, we hope to construct
a full nonlinear solution $u$ to \eqref{eq:CSS-equiv-u}, whose evolution
closely follows that of the approximate solution.

To achieve this, we will decompose our solution $u$ of the form 
\[
u(t,r)=\frac{e^{i\gmm(t)}}{\lmb(t)}[P(\cdot;b(t),\eta(t))+\eps(t,\cdot)]\Big(\frac{r}{\lmb(t)}\Big),
\]
where $\eps(t,y)$ is the error part of $u$. We will fix the decomposition
by imposing certain orthogonality conditions. We then apply a robust
energy method with a bootstrap argument to show that $\eps$ is sufficiently
small (and goes to $0$ at the blow-up time), guaranteeing that the
modulation parameters $\lmb,\gmm,b,\eta$ evolve as in \eqref{eq:FormalParameterLaw}.

As mentioned earlier, we carry out the analysis on the hierarchy of
equations for $w$, $w_{1}$, $w_{2}$: \eqref{eq:w-eqn-sd}, \eqref{eq:w1-eqn-sd},
and \eqref{eq:w2-eqn-sd}. As our modified profiles $P$, $P_{1}$,
$P_{2}$ are motivated from this hierarchical structure, the decomposition
of $u$ will also be based on this structure. Indeed, we use the decompositions
\begin{equation}
\begin{aligned}w & =e^{-i\gmm}\lmb u(\lmb\cdot)=P(\cdot;b,\eta)+\eps,\\
w_{1} & =\bfD_{w}w=P_{1}(\cdot;b,\eta)+\eps_{1},\\
w_{2} & =A_{w}w_{1}=P_{2}(\cdot;b,\eta)+\eps_{2},
\end{aligned}
\label{eq:w012}
\end{equation}
and impose four orthogonality conditions to fix the decomposition.

In this hierarchy, $\eps_{1}$ or $\eps_{2}$ are the same as $L_{Q}\eps$
or $A_{Q}L_{Q}\eps$, respectively, at the leading order. In the previous
work \cite{KimKwon2020arXiv}, the authors used linear adapted derivatives
such as $L_{Q}\eps$, $A_{Q}L_{Q}\eps$, or $A_{Q}^{\ast}A_{Q}L_{Q}\eps$.
Such adapted derivatives were used in the earlier works \cite{RaphaelRodnianski2012Publ.Math.,MerleRaphaelRodnianski2013InventMath,MerleRaphaelRodnianski2015CambJMath,Collot2018MemAMS}.
In this paper, however, we proceed to \emph{nonlinear} adapted derivatives.
Compared to that the linear adapted derivatives are chosen to respect
the linear flows, our \emph{nonlinear} adapted derivatives are chosen
to respect the nonlinear flows. It turns out that going up to higher
order by nonlinear adapted derivatives is more efficient, in the sense
that error terms in the evolution equations are much simpler.

The roles of the equations at different levels are all distinct. The
evolution equations of $\lmb$ and $\gmm$ are derived at the level
of the $w$-equation. The $w_{1}$-equation detects the sharp evolution
equations of $b$ and $\eta$, from which we observe the logarithmic
corrections in the blow-up rate \eqref{eq:sharp-lambda-asymptotics}.
Finally, the energy method will be applied to $\eps_{2}$, where we
observe the repulsivity \eqref{eq:Def-Vtilde}, and the full degeneracy
of $P_{2}$ \eqref{eq:v2-estimate}.

\subsection{Decompositions of solutions}

In this subsection, we explain in detail how we decompose our solutions.
We use the decomposition 
\[
u(t,r)=\frac{e^{i\gmm(t)}}{\lmb(t)}[P(\cdot;b(t),\eta(t))+\eps(t,\cdot)]\Big(\frac{r}{\lmb(t)}\Big).
\]
For each time $t$, there are four degrees of freedom to choose the
parameters $\lmb,\gmm,b,\eta$. We determine them by imposing four
orthogonality conditions on $\eps$. What follows is a fixed-time
analysis and we omit the time variable $t$.

We note that in the hierarchy of the variables $w$, $w_{1}$, $w_{2}$,
the modulation parameters $\lmb,\gmm,b,\eta$ and the error parts
$\eps$, $\eps_{1}$, $\eps_{2}$ are determined according to the
decomposition \eqref{eq:w012}: 
\begin{equation}
\begin{aligned}w & \coloneqq e^{-i\gmm}\lmb u(\lmb\cdot), & w_{1} & \coloneqq\bfD_{w}w, & w_{2} & \coloneqq A_{w}w_{1},\\
\eps & \coloneqq w-P(\cdot;b,\eta), & \eps_{1} & \coloneqq w_{1}-P_{1}(\cdot;b,\eta), & \eps_{2} & \coloneqq w_{2}-P_{2}(\cdot;b,\eta).
\end{aligned}
\label{eq:def-e-e2}
\end{equation}

We will consider two different decompositions, corresponding to two
different orthogonality conditions. Perhaps a standard decomposition
would require $\eps$ to lie in $N_{g}(\calL_{Q}i)^{\perp}$. However,
due to the slow decay of the generalized kernel elements, we will
use truncated orthogonality conditions. This means that, for some
large $M>1$ to be chosen later, we impose 
\begin{equation}
(\eps,\calZ_{1})_{r}=(\eps,\calZ_{2})_{r}=(\eps,\calZ_{3})_{r}=(\eps,\calZ_{4})_{r}=0,\label{eq:OrthogonalityRough}
\end{equation}
where (recall $\chi_{M}$ from the notation section)
\begin{align*}
\calZ_{1} & \coloneqq y^{2}Q\chi_{M}-\frac{2(\rho,y^{2}Q\chi_{M})_{r}}{(yQ,yQ\chi_{M})_{r}}L_{Q}^{\ast}(yQ\chi_{M}),\\
\calZ_{2} & \coloneqq i\rho\chi_{M}-\frac{(y^{2}Q,\rho\chi_{M})_{r}}{2(yQ,yQ\chi_{M})_{r}}L_{Q}^{\ast}(iyQ\chi_{M}),\\
\calZ_{3} & \coloneqq L_{Q}^{\ast}(iyQ\chi_{M}),\\
\calZ_{4} & \coloneqq L_{Q}^{\ast}(yQ\chi_{M}).
\end{align*}
Another way of putting this is to say $\eps\in\calZ^{\perp}$, where
$\calZ^{\perp}$ is a codimension four linear subspace of $H_{0}^{3}$
defined by 
\begin{equation}
\calZ^{\perp}\coloneqq\{\eps\in H_{0}^{3}:(\eps,\calZ_{1})_{r}=(\eps,\calZ_{2})_{r}=(\eps,\calZ_{3})_{r}=(\eps,\calZ_{4})_{r}=0\}.\label{eq:def-Z-perp}
\end{equation}
We call this decomposition the \emph{rough decomposition}. We will
use it as a preliminary decomposition, for instance when we describe
the initial data set and its coordinates. The choices of \eqref{eq:def-Z-perp}
is motivated from the transversality condition; see \eqref{eq:Transversality-Z1234}
below.

However, we will use a different decomposition that detects sharper
modulation equations for $b$ and $\eta$. In view of the hierarchical
structure, these are well-detected from the $\eps_{1}$-equation instead
of the $\eps$-equation. One may observe the error for a more refined
modulation equation $\td{\Mod}$ in the $P_{1}$-equation \eqref{eq:P1-equation}.
Thus we replace the third and fourth orthogonality conditions in \eqref{eq:OrthogonalityRough}
by orthogonality conditions for $\eps_{1}$: 
\begin{equation}
(\eps,\calZ_{1})_{r}=(\eps,\calZ_{2})_{r}=(\eps_{1},\td{\calZ}_{3})_{r}=(\eps_{1},\td{\calZ}_{4})_{r}=0,\label{eq:OrthogonalityNonlinear}
\end{equation}
where 
\begin{align*}
\td{\calZ}_{3} & \coloneqq iyQ\chi_{M},\\
\td{\calZ}_{4} & \coloneqq yQ\chi_{M}.
\end{align*}
In view of $\eps_{1}\approx L_{Q}\eps$ up to the leading order, this
is a slight modification of the rough decomposition. We will call
this the \emph{nonlinear decomposition, }as $\eps$ does not belong
to a fixed codimension four linear subspace. More precisely, after
writing \eqref{eq:OrthogonalityNonlinear} in terms of $b,\eta,\eps$,
we see that $\eps$ belongs to some codimension four manifold \emph{depending}
on $b$ and $\eta$. The nonlinear decomposition \emph{does not} in
general mean that $\eps$ belongs to $\calZ^{\perp}$. 
\begin{lem}[Estimates of $\calZ_{k}$'s]
The following estimates hold. 
\begin{enumerate}
\item (Logarithmic divergence) 
\begin{equation}
(yQ,yQ\chi_{M})_{r}=16\pi\log M+O(1).\label{eq:yQyQ}
\end{equation}
\item (Pointwise estimates) 
\begin{align*}
|\calZ_{1}|_{1}+|\calZ_{2}|_{1} & \aleq M^{2}Q\chf_{(0,2M]},\\
|\calZ_{3}|_{1}+|\calZ_{4}|_{1} & \aleq Q\chf_{(0,2M]},\\
|\td{\calZ}_{3}|_{1}+|\td{\calZ}_{4}|_{1} & \aleq yQ\chf_{(0,2M]}.
\end{align*}
\item (Transversality) For $k\in\{1,2,3,4\}$, we have 
\begin{equation}
\begin{aligned}(\Lambda Q,\calZ_{k})_{r} & =(-(yQ,yQ\chi_{M})_{r}+O(1))\delta_{1k},\\
(-iQ,\calZ_{k})_{r} & =(-\tfrac{1}{4}(yQ,yQ\chi_{M})_{r}+O(1))\delta_{2k},\\
(i\tfrac{y^{2}}{4}Q,\calZ_{k})_{r} & =\tfrac{1}{2}(yQ,yQ\chi_{M})_{r}\delta_{3k},\\
(\rho,\calZ_{k})_{r} & =\tfrac{1}{2}(yQ,yQ\chi_{M})_{r}\delta_{4k}.
\end{aligned}
\label{eq:Transversality-Z1234}
\end{equation}
For $k\in\{3,4\}$, we have 
\begin{equation}
\begin{aligned}(i\tfrac{y}{2}Q,\td{\calZ}_{k})_{r} & =\tfrac{1}{2}(yQ,yQ\chi_{M})_{r}\delta_{3k},\\
(\tfrac{y}{2}Q,\td{\calZ}_{k})_{r} & =\tfrac{1}{2}(yQ,yQ\chi_{M})_{r}\delta_{4k}.
\end{aligned}
\label{eq:Transversality-Z-tilde-34}
\end{equation}
\end{enumerate}
\end{lem}

\begin{proof}
(1) This is immediate from the explicit formula \eqref{eq:Q-formula}
of $Q$.

(2) The pointwise estimates for $\calZ_{1}$ and $\calZ_{2}$ follow
from \eqref{eq:rho-estimate} and 
\begin{align*}
|L_{Q}^{\ast}(yQ\chi_{M})|+|L_{Q}^{\ast}(iyQ\chi_{M})| & \aleq Q\chf_{(0,2M]},\\
|(\rho,y^{2}Q\chi_{M})_{r}| & \aleq M^{2},\\
(yQ,yQ\chi_{M})_{r} & \sim\log M.
\end{align*}
The pointwise estimates for $\td{\calZ}_{3}$ and $\td{\calZ}_{4}$
are immediate.

(3) Let $k\in\{1,2\}$. Since $\calZ_{1}$ is real, $\calZ_{2}$ is
imaginary, and $L_{Q}\Lambda Q=L_{Q}iQ=0$, we have 
\begin{align*}
(\Lambda Q,\calZ_{k})_{r} & =(\Lambda Q,y^{2}Q\chi_{M})_{r}\delta_{1k},\\
(-iQ,\calZ_{k})_{r} & =-(Q,\rho\chi_{M})_{r}\delta_{2k}.
\end{align*}
We then compute 
\begin{align*}
(\Lambda Q,y^{2}Q\chi_{M})_{r} & =\tfrac{1}{2}([y^{2}\chi_{M},\Lambda]Q,Q)_{r}=-(y^{2}Q\chi_{M},Q)_{r}+O(1),\\
(Q,\rho\chi_{M})_{r} & =\tfrac{1}{2}(yQ,L_{Q}(\rho\chi_{M}))_{r}=\tfrac{1}{4}(yQ,yQ\chi_{M})_{r}+O(1).
\end{align*}
Next, using $L_{Q}\rho=\tfrac{1}{2}yQ$ and $L_{Q}iy^{2}Q=2iyQ$,
we see that the additional terms in the definition of $\calZ_{1}$
and $\calZ_{2}$ are chosen to satisfy 
\[
(iy^{2}Q,\calZ_{k})_{r}=(\rho,\calZ_{k})_{r}=0.
\]

Let $k\in\{3,4\}$. Since $L_{Q}\Lambda Q=L_{Q}iQ=0$, we have 
\[
(\Lambda Q,\calZ_{k})_{r}=(-iQ,\calZ_{k})_{r}=0.
\]
Since $\calZ_{3}$ is imaginary, $\calZ_{4}$ is real, $L_{Q}\rho=\tfrac{1}{2}yQ$,
and $L_{Q}iy^{2}Q=2iyQ$, we have 
\begin{align*}
(i\tfrac{y^{2}}{4}Q,\calZ_{k})_{r} & =\tfrac{1}{2}(yQ,yQ\chi_{M})_{r}\delta_{3k},\\
(\rho,\calZ_{k})_{r} & =\tfrac{1}{2}(yQ,yQ\chi_{M})_{r}\delta_{4k}.
\end{align*}
Thus \eqref{eq:Transversality-Z1234} is proved. Finally, \eqref{eq:Transversality-Z-tilde-34}
for $\td{\calZ}_{3}$ and $\td{\calZ}_{4}$ are immediate from the
fact that $\td{\calZ}_{3}$ is imaginary and $\td{\calZ}_{4}$ is
real.
\end{proof}
We will define an open set $\calO_{\dec}\subseteq H_{0}^{3}$ near
the set of modulated solitons (i.e., the set of all $\frac{e^{i\gmm}}{\lmb}Q(\frac{\cdot}{\lmb})$'s),
on which both the above decompositions can be made. The set of coordinates
$(\lmb,\gmm,b,\eta,\eps)$ will be denoted by $\calU_{\dec}$. For
$\delta_{\dec}>0$ to be chosen, we define $\calU_{\dec}\subseteq\bbR_{+}\times\bbR/2\pi\bbZ\times\bbR\times\bbR\times\calZ^{\perp}$
by the set of $(\lmb,\gmm,b,\eta,\eps)$ satisfying 
\[
0<b<\delta_{\dec},\quad|\eta|<\tfrac{2b}{|\log b|},\quad\|\eps\|_{H_{0}^{3}}<\delta_{\dec}.
\]
The set $\calO_{\dec}$ is defined by the set of images 
\[
\calO_{\dec}\coloneqq\{\frac{e^{i\gmm}}{\lmb}[P(\cdot;b,\eta)+\eps]\Big(\frac{r}{\lmb}\Big):(\lmb,\gmm,b,\eta,\eps)\in\calU_{\dec}\}.
\]

\begin{lem}[Decompositions]
\label{lem:decomp}For all sufficiently large $M$, there exist $\delta_{1}>\delta_{1}'>\delta_{\dec}>0$
such that the following holds.
\begin{enumerate}
\item (The set $\calO_{\dec}$ and rough decomposition) The set $\calO_{\dec}$
is open in $H_{0}^{3}$. Moreover, the map 
\[
[\Phi(\lmb,\gmm,b,\eta,\eps)](r)\coloneqq\frac{e^{i\gmm}}{\lmb}[P(\cdot;b,\eta)+\eps]\Big(\frac{r}{\lmb}\Big)
\]
is a homeomorphism from $\overline{\calU}_{\dec}$ to $\overline{\calO}_{\dec}$.
We denote by $\mathbf{G}^{(1)}$ the $(\lmb,\gmm,b,\eta)$-components
of $\Phi^{-1}$. In other words, for any $u\in\br{\calO}_{\dec}$,
$\mathbf{G}^{(1)}(u)$ denotes the modulation parameters for the rough
decomposition satisfying \eqref{eq:OrthogonalityRough}.
\item (Nonlinear decomposition) For any $u\in\overline{\calO}_{\dec}$,
there exists unique $(\mathbf{G}^{(2)},\eps)=(\lmb,\gmm,b,\eta,\eps)\in\bbR_{+}\times\bbR/2\pi\bbZ\times B_{\delta_{1}}(0)\times B_{\delta_{1}}(0)\times B_{\delta_{1}'}(0)$\footnote{Since we are using two different decompositions, we have two different
$(\lmb,\gmm,b,\eta,\eps)$ for the same $u\in\overline{\calO}_{\dec}$.
We will use the same notation $(\lmb,\gmm,b,\eta,\eps)$ when no confusion
arises.} satisfying \eqref{eq:OrthogonalityNonlinear}, namely, 
\[
(\eps,\calZ_{1})_{r}=(\eps,\calZ_{2})_{r}=(\eps_{1},\td{\calZ}_{3})_{r}=(\eps_{1},\td{\calZ}_{4})_{r}=0.
\]
\item ($C^{1}$-regularity) The map $u\mapsto(\lmb,\gmm,b,\eta)$ for each
decomposition is $C^{1}$, i.e., the maps $\mathbf{G}^{(1)}$ and
$\mathbf{G}^{(2)}$ are $C^{1}$.
\item (Difference estimate) For $u\in\overline{\calO}_{\dec}$, we have
\begin{equation}
\mathrm{dist}(\mathbf{G}^{(1)}(u),\mathbf{G}^{(2)}(u))\aleq|(\eps_{1},\td{\calZ}_{3})_{r}|+|(\eps_{1},\td{\calZ}_{4})_{r}|,\label{eq:DecompositionDifference}
\end{equation}
where $\eps_{1}$ is computed using the rough decomposition and the
formula \eqref{eq:def-e-e2}.
\item (Initial data set) Recall the initial data sets \eqref{eq:def-Utilde-init}--\eqref{eq:def-O-init}.
If $b^{\ast}>0$ is sufficiently small depending on $M$ (in particular
$b^{\ast}\ll\delta_{\mathrm{dec}}$), then we have $\calU_{\init}\subseteq\calU_{\dec}$
and $\calO_{\init}\subseteq\calO_{\dec}$. Moreover, the statements
of (1) also hold when we replace $\calU_{\dec}$ and $\calO_{\dec}$
by $\calU_{\init}$ and $\calO_{\init}$, respectively.
\end{enumerate}
\end{lem}

\begin{proof}
The proof is an extension of \cite[Lemma 4.2]{KimKwon2020arXiv}.
We include the full proof for the reader's convenience.

Let us introduce some notation to be used in this proof. For $\lmb\in\bbR_{+}$
and $\gmm\in\bbR/2\pi\bbZ$, let us denote 
\[
f_{\lmb,\gmm}(y)\coloneqq\frac{e^{i\gmm}}{\lmb}f\Big(\frac{y}{\lmb}\Big),\quad X_{\lmb,\gmm}\coloneqq\{f_{\lmb,\gmm}:f\in X\}.
\]
We equip $\bbR_{+}$ with the metric $\mathrm{dist}(\lmb_{1},\lmb_{2})=|\log(\lmb_{1}/\lmb_{2})|$,
and equip $\bbR/2\pi\bbZ$ with the induced metric from $\bbR$. We
will choose small parameters $\delta_{1},\delta_{1}',\delta_{2},\delta_{\dec}>0$
on the way, with the parameter dependence 
\[
0<b^{\ast}\ll\delta_{\dec}\ll\delta_{2}\ll\delta_{1}'\ll\delta_{1}\ll M^{-1}\ll1,
\]
which means that $\delta_{1}$ is chosen sufficiently small depending
on the large parameter $M$, $\delta_{1}'$ is chosen sufficiently
small depending on $\delta_{1}$ (and hence only on $M$), and so
on.

\textbf{Step 1:} \emph{Extension of the profiles $P$ and $P_{1}$}.

Notice that in Section~\ref{sec:Modified-profiles}, the profiles
$P$ and $P_{1}$ are considered only for $(b,\eta)$ with $|\eta|\ll b$
(specifically $|\eta|\leq\frac{b}{|\log b|}$ with $b>0$ small),
\emph{not} for all $|(b,\eta)|\ll1$. As we want to apply the implicit
function theorem at $Q=P(\cdot;0,0)$, we will consider artificial
extensions $\td P(y;b,\eta)$ and $\td P_{1}(y;b,\eta)$ of $P(y;b,\eta)$
and $P_{1}(y;b,\eta)$ defined for all $(b,\eta)$ in a neighborhood
of $(0,0)$, respectively.

First, we extend $P(y;b,\eta)$ and $P_{1}(y;b,\eta)$ for $|\eta|\leq\frac{2b}{|\log b|}$
and $|b|<\delta_{1}$. If $b=0$ (hence $\eta=0$), then we set $P(\cdot;0,0)=Q$
and $P_{1}(\cdot;0,0)=0$. If $b\neq0$, then define $P$ and $P_{1}$
via the formulae \eqref{eq:P-equation} and \eqref{eq:P1-equation}
with $B_{0}=|b|^{-\frac{1}{2}}$, $B_{1}=|b|^{-\frac{1}{2}}|\log|b||^{-1}$,
$S_{2,0}(y;b,\eta)\coloneqq S_{2,0}(y;|b|,\eta)$, and similarly for
$\td T_{2},T_{3,0}$. We remark that the estimates \eqref{eq:v-estimate}
and \eqref{eq:v1-estimate} are still valid for $|b|<\delta_{1}$.
In particular $\partial_{b}P=-i\frac{y^{2}}{4}Q$ and $\partial_{b}P_{1}=-i\tfrac{y}{2}Q$
when $(b,\eta)=(0,0)$.

Next, in order to define the extensions $\td P$ and $\td P_{1}$
for all $|(b,\eta)|\ll1$, we will introduce a suitable cutoff function
for $\eta$. Choose a smooth function $\psi:\bbR\to\bbR$ such that
$\psi(\td{\eta})=\td{\eta}$ for $|\td{\eta}|\leq2$ and $\sup|\psi|_{1}\aleq1$.
For $|b|<\delta_{1}$, we define $\psi_{b}(\td{\eta})=\frac{|b|}{|\log|b||}\psi(\frac{|\log|b||}{|b|}\td{\eta})$
if $b\neq0$ and $\psi_{0}(\td{\eta})=0$. Thus $\partial_{b}\psi_{b}(\td{\eta})=-\mathrm{sgn}(b)(\frac{1}{|\log|b||}+\frac{1}{|\log|b||^{2}})[\Lambda_{2}\psi](\frac{|\log|b||}{|b|}\td{\eta})$
if $b\neq0$ and $\partial_{b=0}\psi_{b}(\td{\eta})=0$. In particular,
$\|\partial_{b}\psi_{b}\|_{L^{\infty}}\aleq\frac{1}{|\log|b||}$.
Finally, we define 
\begin{align*}
\td P(\cdot;b,\eta) & \coloneqq P(\cdot;b,\psi_{b}(\eta))-(\eta-\psi_{b}(\eta))\rho\chi_{2M},\\
\td P_{1}(\cdot;b,\eta) & \coloneqq P_{1}(\cdot;b,\psi_{b}(\eta))-(\eta-\psi_{b}(\eta))(\tfrac{y}{2}Q\chi_{2M}),
\end{align*}
for $|\eta|,|b|<\delta_{1}$. By the definition, $\td P(\cdot;b,\eta)=P(\cdot;b,\eta)$
for $|\eta|\leq\frac{2b}{|\log b|}$.

\textbf{Step 2:} \emph{Setting for the implicit function theorem}.

The main part of the proof is to use the implicit function theorem.
Define the maps 
\[
\mathbf{F}^{(1)},\mathbf{F}^{(2)}:\bbR_{+}\times\bbR/2\pi\bbZ\times B_{\delta_{1}}(0)\times B_{\delta_{1}}(0)\times L^{2}\to\bbR^{4}
\]
with variables $\lmb,\gmm,b,\eta,u$ and components $F_{1}^{(j)},F_{2}^{(j)},F_{3}^{(j)},F_{4}^{(j)}$,
by 
\begin{align*}
F_{1}^{(1)} & =(\eps,\calZ_{1})_{r}, & F_{2}^{(1)} & =(\eps,\calZ_{2})_{r}, & F_{3}^{(1)} & =(\eps,\calZ_{3})_{r}, & F_{4}^{(1)} & =(\eps,\calZ_{4})_{r},\\
F_{1}^{(2)} & =(\eps,\calZ_{1})_{r}, & F_{2}^{(2)} & =(\eps,\calZ_{2})_{r}, & F_{3}^{(2)} & =(\eps_{1},\td{\calZ}_{3})_{r}, & F_{4}^{(2)} & =(\eps_{1},\td{\calZ}_{4})_{r},
\end{align*}
where 
\[
\begin{aligned}w & \coloneqq e^{-i\gmm}\lmb u(\lmb\cdot), & w_{1} & \coloneqq\bfD_{w}w,\\
\eps & \coloneqq w-\td P(\cdot;b,\eta), & \eps_{1} & \coloneqq w_{1}-\td P_{1}(\cdot;b,\eta).
\end{aligned}
\]
Here, $\mathbf{F}^{(1)}$ and $\mathbf{F}^{(2)}$ correspond to the
rough and nonlinear decomposition, respectively.

We first consider $\mathbf{F}^{(1)}$. In order to use the implicit
function theorem, we will check that $\mathbf{F}^{(1)}$ is $C^{1}$
and $\partial_{\lmb,\gmm,b,\eta}\mathbf{F}^{(1)}$ is invertible at
$(\lmb,\gmm,b,\eta,u)=(1,0,0,0,Q)$. For different $(\lmb,\gmm)$,
we will apply scale/phase invariances in Step 3. For $(\lmb,\gmm,b,\eta,u)$
near $(1,0,0,0,Q)$, we compute using \eqref{eq:Transversality-Z1234}
\begin{align*}
\partial_{\lmb}F_{k}^{(1)} & =(\Lambda Q,[\calZ_{k}]_{\lmb,\gmm})_{r}-(u-Q,[\Lambda\calZ_{k}]_{\lmb,\gmm})_{r}\\
 & =(-(yQ,yQ\chi_{M})_{r}+O(1))\delta_{1k}+M^{C}O(\mathrm{dist}((\lmb,\gmm),(1,0))+\|u-Q\|_{L^{2}}),\\
\partial_{\gmm}F_{k}^{(1)} & =(-iQ,[\calZ_{k}]_{\lmb,\gmm})_{r}+(u-Q,[i\calZ_{k}]_{\lmb,\gmm})_{r}\\
 & =(-\tfrac{1}{4}(yQ,yQ\chi_{M})_{r}+O(1))\delta_{2k}+M^{C}O(\mathrm{dist}((\lmb,\gmm),(1,0))+\|u-Q\|_{L^{2}}).
\end{align*}
Next, by the pointwise estimates \eqref{eq:v-estimate} and $\|\partial_{b}\psi_{b}\|_{L^{\infty}}\aleq\frac{1}{|\log|b||}$,
we have 
\begin{align*}
 & \chf_{(0,2M]}|\partial_{b}\td P(0;b,\eta)+i\tfrac{y^{2}}{4}Q|\\
 & =\chf_{(0,2M]}\Big|\Big(\partial_{b}P(\cdot;b,\td{\eta})|_{\td{\eta}=\psi_{b}(\eta)}+i\tfrac{y^{2}}{4}Q\Big)+\partial_{b}\psi_{b}(\eta)\cdot\partial_{\td{\eta}=\psi_{b}(\eta)}P(\cdot;b,\td{\eta})+\partial_{b}\psi_{b}(\eta)\rho\Big|\\
 & \aleq\chf_{(0,2M]}(|b|y^{2}+\tfrac{1}{|\log|b||})
\end{align*}
Combining this with \eqref{eq:Transversality-Z1234}, we have 
\[
\partial_{b}F_{k}^{(1)}=(-\partial_{b}\td P,\calZ_{k})_{r}=\tfrac{1}{2}(yQ,yQ\chi_{M})_{r}\delta_{3k}+M^{C}O(\tfrac{1}{|\log|b||}).
\]
Next, again by pointwise estimates \eqref{eq:v-estimate}, we have
\[
\chf_{(0,2M]}|\partial_{\eta}\td P(\cdot;b,\eta)+\rho|=\chf_{(0,2M]}\Big|\psi_{b}'(\eta)\Big(\partial_{\td{\eta}=\psi_{b}(\eta)}P(\cdot;b,\td{\eta})+\rho\Big)\Big|\aleq\chf_{(0,2M]}|b|y^{2}.
\]
Combining this with \eqref{eq:Transversality-Z-tilde-34}, we have
\[
\partial_{\eta}F_{k}^{(1)}=(-\partial_{\eta}\td P,\calZ_{k})_{r}=\tfrac{1}{2}(yQ,yQ\chi_{M})_{r}\delta_{4k}+M^{C}O(|b|).
\]
Finally, we have 
\[
\frac{\delta F_{k}^{(1)}}{\delta u}=(\calZ_{k})_{\lmb,\gmm}\in L^{2}.
\]
In summary, $\mathbf{F}^{(1)}$ is $C^{1}$ and $\partial_{\lmb,\gmm,b,\eta}\mathbf{F}^{(1)}$
is invertible at $(\lmb,\gmm,b,\eta,u)=(1,0,0,0,Q)$ since the nonzero
leading terms are on the diagonal.

We turn to $\mathbf{F}^{(2)}$. We check that $\mathbf{F}^{(2)}$
is $C^{1}$ and $\partial_{\lmb,\gmm,b,\eta}\mathbf{F}^{(2)}$ is
invertible at $(\lmb,\gmm,b,\eta,u)=(1,0,0,0,Q)$. As $\mathbf{F}_{1}^{(2)}=\mathbf{F}_{1}^{(1)}$
and $\mathbf{F}_{2}^{(2)}=\mathbf{F}_{2}^{(1)}$, it suffices to consider
$\mathbf{F}_{k}^{(2)}$ for $k\in\{3,4\}$. Let us temporarily denote
$f_{\text{\ensuremath{\underbar{\ensuremath{\lmb}}}},\gmm}\coloneqq e^{i\gmm}f(\frac{\cdot}{\lmb})$
(the $\dot{H}^{1}$-scaling). For $(\lmb,\gmm,b,\eta,u)$ near $(1,0,0,0,Q)$,
we compute using $\bfD_{Q}Q=0$ and the linearization of the Bogomol'nyi
operator \eqref{eq:LinearizationBogomolnyi} that 
\begin{align*}
\partial_{\lmb}F_{k}^{(2)} & =-(\bfD_{u}u,[\Lambda_{1}\td{\calZ}_{k}]_{\text{\ensuremath{\underbar{\ensuremath{\lmb}}}},\gmm})_{r}\\
 & =-(u-Q,L_{Q}^{\ast}[\Lambda_{1}\td{\calZ}_{k}]_{\text{\ensuremath{\underbar{\ensuremath{\lmb}}}},\gmm})_{r}-(N_{Q}(u-Q),[\Lambda_{1}\td{\calZ}_{k}]_{\text{\ensuremath{\underbar{\ensuremath{\lmb}}}},\gmm})_{r}\\
 & =M^{C}O(\|u-Q\|_{L^{2}}).
\end{align*}
Similarly, 
\[
\partial_{\gmm}F_{k}^{(2)}=(\bfD_{u}u,[i\td{\calZ}_{k}]_{\text{\ensuremath{\underbar{\ensuremath{\lmb}}}},\gmm})_{r}=M^{C}O(\|u-Q\|_{L^{2}}).
\]
For $\partial_{b}$ and $\partial_{\eta}$, by $\|\partial_{b}\psi_{b}\|_{L^{\infty}}\aleq\frac{1}{|\log|b||}$
we have 
\begin{align*}
 & \chf_{(0,2M]}|\partial_{b}\td P_{1}(0;b,\eta)+i\tfrac{y}{2}Q|\\
 & =\chf_{(0,2M]}\Big|\Big(\partial_{b}P_{1}(\cdot;b,\td{\eta})|_{\td{\eta}=\psi_{b}(\eta)}+i\tfrac{y}{2}Q\Big)+\partial_{b}\psi_{b}(\eta)\cdot\partial_{\td{\eta}=\psi_{b}(\eta)}P_{1}(\cdot;b,\td{\eta})+\partial_{b}\psi_{b}(\eta)\tfrac{y}{2}Q\Big|\\
 & \aleq\chf_{(0,2M]}(|b|y+\tfrac{1}{|\log|b||}\tfrac{1}{y}).
\end{align*}
Combining this with \eqref{eq:Transversality-Z-tilde-34}, we have
\[
\partial_{b}F_{k}^{(2)}=(-\partial_{b}\td P_{1},\td{\calZ}_{k})_{r}=\tfrac{1}{2}(yQ,yQ\chi_{M})_{r}\delta_{3k}+M^{C}O(\tfrac{1}{|\log|b||}).
\]
Similarly, we have 
\[
\chf_{(0,2M]}|\partial_{\eta}\td P_{1}(\cdot;b,\eta)+\tfrac{y}{2}Q|=\chf_{(0,2M]}\Big|\psi_{b}'(\eta)\Big(\partial_{\td{\eta}=\psi_{b}(\eta)}P_{1}(\cdot;b,\td{\eta})+\tfrac{y}{2}Q\Big)\Big|\aleq\chf_{(0,2M]}|b|y
\]
so 
\[
\partial_{\eta}F_{k}^{(2)}=(-\partial_{\eta}\td P_{1},\td{\calZ}_{k})_{r}=\tfrac{1}{2}(yQ,yQ\chi_{M})_{r}\delta_{4k}+M^{C}O(|b|).
\]
Finally, we have 
\[
\frac{\delta F_{k}^{(2)}}{\delta u}=L_{u}^{\ast}[\td{\calZ}_{k}]_{\underline{\lmb},\gmm}\in L^{2}.
\]
This shows that $\mathbf{F}^{(2)}$ is $C^{1}$ and $\partial_{\lmb,\gmm,b,\eta}\mathbf{F}^{(2)}$
is invertible at $(\lmb,\gmm,b,\eta,u)=(1,0,0,0,Q)$.

Therefore, by the implicit function theorem, provided that $M\gg1$,
there exist $\delta_{1},\delta_{2}>0$, and $C^{1}$-maps $\mathbf{G}_{1,0}^{(j)}:B_{\delta_{2}}(Q)\to B_{\delta_{1}}(1,0,0,0)$
such that for given $u\in B_{\delta_{2}}(Q)\subseteq L^{2}$, $\mathbf{G}_{1,0}^{(j)}(u)$
is a unique solution to $\mathbf{F}^{(j)}(\mathbf{G}_{1,0}^{(j)}(u),u)=0$
in $B_{\delta_{1}}(1,0,0,0)$. We fix $\delta_{1}$ here, but we can
freely shrink $\delta_{2}$ and in particular we assume $\delta_{2}\ll\delta_{1}$.
Note that we also have a Lipschitz estimate 
\[
\mathrm{dist}(\mathbf{G}_{1,0}^{(j)}(u),(1,0,0,0))\aleq\|u-Q\|_{L^{2}}.
\]
The proof of the implicit function theorem also guarantees the difference
estimate: 
\[
\mathrm{dist}(\mathbf{G}_{1,0}^{(1)}(u),\mathbf{G}_{1,0}^{(2)}(u))\aleq|\mathbf{F}^{(2)}(\mathbf{G}_{1,0}^{(1)}(u),u)-\mathbf{F}^{(2)}(\mathbf{G}_{1,0}^{(2)}(u),u)|=|\mathbf{F}^{(2)}(\mathbf{G}_{1,0}^{(1)}(u),u)|.
\]

\textbf{Step 3:} \emph{Definition and uniqueness of $\mathbf{G}^{(j)}$}.

We now apply scale/phase invariances to cover the $\delta_{2}$-neighborhood
of $\{Q_{\lmb,\gmm}:\lmb\in\bbR_{+},\gmm\in\bbR/2\pi\bbZ\}$ in $L^{2}$.
For $\lmb\in\bbR_{+}$ and $\gmm\in\bbR/2\pi\bbZ$, apply the scale/phase
invariances to $\mathbf{G}_{1,0}^{(j)}$ to define $\mathbf{G}_{\lmb,\gmm}^{(j)}:B_{\delta_{2}}(Q)_{\lmb,\gmm}\to B_{\delta_{1}}(\lmb,\gmm,0,0)$
in the obvious way. Thus uniqueness property of $\mathbf{G}_{\lmb,\gmm}^{(j)}$
holds for values in $B_{\delta_{1}}(\lmb,\gmm,0,0)$ and there holds
the difference estimate 
\begin{equation}
\mathrm{dist}(\mathbf{G}_{\lmb,\gmm}^{(1)}(u),\mathbf{G}_{\lmb,\gmm}^{(2)}(u))\aleq|\mathbf{F}^{(2)}(\mathbf{G}_{\lmb,\gmm}^{(1)}(u),u)|.\label{eq:difference-temp}
\end{equation}

We claim that 
\[
\mathbf{G}^{(j)}\coloneqq{\textstyle \bigcup_{\lmb,\gmm}}\mathbf{G}_{\lmb_{1},\gmm_{1}}^{(j)}:{\textstyle \bigcup_{\lmb,\gmm}}B_{\delta_{2}}(Q)_{\lmb,\gmm}\to\bbR_{+}\times\bbR/2\pi\bbZ\times B_{\delta_{1}}(0)\times B_{\delta_{1}}(0)
\]
is well-defined, i.e. the family $\{\mathbf{G}_{\lmb,\gmm}^{(j)}\}_{\lmb,\gmm}$
is compatible. Indeed, if $u\in B_{\delta_{2}}(Q)_{\lmb_{1},\gmm_{1}}\cap B_{\delta_{2}}(Q)_{\lmb_{2},\gmm_{2}}$,
then $\mathrm{dist}((\lmb_{1},\gmm_{1}),(\lmb_{2},\gmm_{2}))\aleq\delta_{2}$
thus $\mathrm{dist}(\mathbf{G}_{\lmb_{2},\gmm_{2}}^{(j)}(u),(\lmb_{1},\gmm_{1},0,0))\aleq\delta_{2}\ll\delta_{1}$.
Since $\mathbf{G}_{\lmb_{2},\gmm_{2}}^{(j)}(u)$ satisfies the equation
$\mathbf{F}^{(j)}(\mathbf{G}_{\lmb_{2},\gmm_{2}}^{(j)}(u),u)=0$,
we have $\mathbf{G}_{\lmb_{2},\gmm_{2}}^{(j)}(u)=\mathbf{G}_{\lmb_{1},\gmm_{1}}^{(j)}(u)$
by the uniqueness of $\mathbf{G}_{\lmb_{1},\gmm_{1}}^{(j)}(u)$ in
$B_{\delta_{1}}(\lmb_{1},\gmm_{1},0,0)$.

Having defined $\mathbf{G}^{(j)}$, we can define the map 
\[
\eps^{(j)}:{\textstyle \bigcup_{\lmb,\gmm}}B_{\delta_{2}}(Q)_{\lmb,\gmm}\to B_{\delta_{1}'}(0)
\]
by $\eps^{(j)}(u)=u_{\lmb^{-1},-\gmm}-\td P(\cdot;b,\eta)$, where
$(\lmb,\gmm,b,\eta)=\mathbf{G}^{(j)}(u)$. At this point, the map
$\eps^{(j)}$ is defined whenever $0<\delta_{1}'<\delta_{1}$ and
$\delta_{2}\ll\delta_{1}'$. The small parameter $\delta_{1}'\ll\delta_{1}$
will be fixed in the next paragraph.

Next, we claim the uniqueness property of $\mathbf{G}^{(j)}$: given
$u\in\bigcup_{\lmb,\gmm}B_{\delta_{2}}(Q)_{\lmb,\gmm}$, $\mathbf{G}^{(j)}(u)\in\bbR_{+}\times\bbR/2\pi\bbZ\times B_{\delta_{1}}(0)\times B_{\delta_{1}}(0)$
is the unique solution to $\mathbf{F}^{(j)}(\mathbf{G}^{(j)}(u),u)=0$
such that $\|\eps^{(j)}\|_{L^{2}}<\delta_{1}'$. To see this, let
$\mathbf{G}'=(\lmb',\gmm',b',\eta')$ be a solution to $\mathbf{F}^{(j)}(\mathbf{G}',u)=0$
such that $\eps'=u_{(\lmb')^{-1},-\gmm'}-P(\cdot;b',\eta')$ satisfies
$\|\eps'\|_{L^{2}}<\delta_{1}'$. If $\mathrm{dist}(\mathbf{G}',\mathbf{G}^{(j)}(u))<\delta_{1}$,
then $\mathbf{G}^{(j)}(u)=\mathbf{G}'$ by the uniqueness of $\mathbf{G}^{(j)}(u)$.
If $\mathrm{dist}(\mathbf{G}',\mathbf{G}^{(j)}(u))\geq\delta_{1}$,
then $\|\td P(\cdot;b',\eta')_{\lmb',\gmm'}-\td P(\cdot;b,\eta)_{\lmb,\gmm}\|_{L^{2}}\ageq\delta_{1}$
but $\|\eps'\|_{L^{2}},\|\eps^{(j)}\|_{L^{2}}<\delta_{1}'\ll\delta_{1}$,
contradicting $[\td P(\cdot;b',\eta')+\eps']_{\lmb',\gmm'}=u=[\td P(\cdot;b,\eta)+\eps]_{\lmb,\gmm}$.

\textbf{Step 4:} \emph{Coordinate system of the rough decomposition}.

From now on, we work with the $H_{0}^{3}$-topology and $j=1$. Note
that $\eps^{(1)}$ is continuous on the $H_{0}^{3}$-topology, i.e.
\[
\eps^{(1)}:{\textstyle \bigcup_{\lmb,\gmm}}B_{\delta_{2}}^{H_{0}^{3}}(Q)_{\lmb,\gmm}\to B_{\delta_{1}'}^{\calZ^{\perp}}(0)
\]
is continuous. By the definition of $\eps^{(1)}$, the map 
\begin{align*}
(\mathbf{G}^{(1)},\eps^{(1)}):{\textstyle \bigcup_{\lmb,\gmm}}B_{\delta_{2}}^{H_{0}^{3}}(Q)_{\lmb,\gmm} & \to\bbR_{+}\times\bbR/2\pi\bbZ\times B_{\delta_{1}}(0)\times B_{\delta_{1}}(0)\times B_{\delta_{1}'}^{\calZ^{\perp}}(0)\\
u & \mapsto(\mathbf{G}^{(1)}(u),\eps^{(1)}(u))
\end{align*}
has a continuous left inverse 
\begin{gather*}
\Phi:\bbR_{+}\times\bbR/2\pi\bbZ\times B_{\delta_{1}}(0)\times B_{\delta_{1}}(0)\times B_{\delta_{1}'}^{\calZ^{\perp}}(0)\to H_{0}^{3}\\
(\lmb,\gmm,b,\eta,\eps)\mapsto[P(\cdot;b,\eta)+\eps]_{\lmb,\gmm}.
\end{gather*}
Moreover, the uniqueness of $\mathbf{G}^{(1)}$ implies that $\mathrm{Im}(\mathbf{G}^{(1)},\eps^{(1)})=\Phi^{-1}(\bigcup_{\lmb,\gmm}B_{\delta_{2}}^{H_{0}^{3}}(Q)_{\lmb,\gmm})$
(and in particular it is open) and $\Phi|_{\mathrm{Im}(\mathbf{G}^{(1)},\eps^{(1)})}$
is a right inverse of $(\mathbf{G}^{(1)},\eps^{(1)})$. Therefore,
the restriction 
\[
\Phi|_{\mathrm{Im}(\mathbf{G}^{(1)},\eps^{(1)})}:\mathrm{Im}(\mathbf{G}^{(1)},\eps^{(1)})\to{\textstyle \bigcup_{\lmb,\gmm}}B_{\delta_{2}}^{H_{0}^{3}}(Q)_{\lmb,\gmm}
\]
is a homeomorphism with the inverse $(\mathbf{G}^{(1)},\eps^{(1)})$.

\textbf{Step 5:} \emph{Completion of the proof}.

We finish the proof of this lemma.

(1) We further restrict to the sets $\calU_{\dec}$ and $\calO_{\dec}$.
Since $\overline{\calU}_{\dec}$ lies in the domain of $\Phi$ and
$\overline{\calO}_{\dec}\subseteq\bigcup_{\lmb,\gmm}B_{\delta_{2}}^{H_{0}^{3}}(Q)_{\lmb,\gmm}$,
we have $\overline{\calU}_{\dec}\subseteq\Im(\mathbf{G}^{(1)},\eps^{(1)})$
due to the uniqueness of $\mathbf{G}^{(1)}$. Therefore, restricting
the homeomorphism $\Phi|_{\Im(\mathbf{G}^{(1)},\eps^{(1)})}$ on $\overline{\calU}_{\dec}$
implies that $\calO_{\dec}$ is open, $\Phi(\overline{\calU}_{\dec})=\overline{\calO}_{\dec}$,
and $\Phi|_{\overline{\calU}_{\dec}}:\overline{\calU}_{\dec}\to\overline{\calO}_{\dec}$
is a homeomorphism.

(2) This is merely a summary of the properties of $\mathbf{G}^{(2)}$
shown above.

(3) We showed above that $\mathbf{G}^{(j)}$ is $C^{1}$ with respect
to the $L^{2}$-topology. The $C^{1}$ property of $\mathbf{G}^{(j)}$
on the $H_{0}^{3}$-topology is immediate from the embedding $H_{0}^{3}\hookrightarrow L^{2}$.

(4) \eqref{eq:DecompositionDifference} follows from $\mathbf{G}^{(j)}=\bigcup_{\lmb,\gmm}\mathbf{G}_{\lmb,\gmm}^{(j)}$,
the difference estimate \eqref{eq:difference-temp} for $\mathbf{G}_{\lmb,\gmm}^{(1)}$
and $\mathbf{G}_{\lmb,\gmm}^{(2)}$, and the definition of $\mathbf{F}^{(2)}$.
Note that $\mathbf{F}_{k}^{(2)}=\mathbf{F}_{k}^{(1)}=0$ for $k\in\{1,2\}$.

(5) This follows from the parameter dependence $b^{\ast}\ll\delta_{\dec}=\delta_{\dec}(M)$.
\end{proof}

\subsection{Trapped solutions and reduction of Theorem~\ref{thm:main-thm}}

In this subsection, we reduce Theorem~\ref{thm:main-thm} to Propositions~\ref{prop:main-bootstrap},
\ref{prop:Sets-I-pm}, and \ref{prop:SharpDescription}. We also prove
Corollary~\ref{cor:InfiniteBlowup}. Among these, the main ingredient
is a bootstrap argument, Proposition~\ref{prop:main-bootstrap}.
We will call solutions satisfying the bootstrap conditions the \emph{trapped
solutions}. By bootstrapping (Proposition~\ref{prop:main-bootstrap})
with a connectivity argument (Proposition~\ref{prop:Sets-I-pm}),
we show the existence of trapped solutions. We then show that (Proposition
\ref{prop:SharpDescription}) those solutions are finite-time blow-up
solutions as described in Theorem~\ref{thm:main-thm}. Such an argument
is standard in the literature.

Roughly speaking, trapped solutions are required to satisfy $|\eta|\ll b$
and certain smallness conditions on $\eps$ on its maximal forward
lifespan, to guarantee the blow-up derived in Section~\ref{sec:Modified-profiles}.
To describe more precisely, we quantify $|\eta|\ll b$ and the smallness
conditions on $\eps$ in terms of the nonlinear decomposition (see
Lemma~\ref{lem:decomp}) and nonlinear adapted derivatives of $\eps$.
Namely, for a function $u\in\calO_{\dec}$, we decompose it as 
\[
u(r)=\frac{e^{i\gmm}}{\lmb}[P(\cdot;b,\eta)+\eps]\Big(\frac{r}{\lmb}\Big)
\]
with the orthogonality conditions \eqref{eq:OrthogonalityNonlinear}
according to Lemma~\ref{lem:decomp}. We recall the \emph{nonlinear
adapted derivatives}, which are given by 
\[
\begin{aligned}w & \coloneqq e^{-i\gmm}\lmb u(\lmb\cdot), & w_{1} & \coloneqq\bfD_{w}w, & w_{2} & \coloneqq A_{w}w_{1},\\
\eps & \coloneqq w-P(\cdot;b,\eta), & \eps_{1} & \coloneqq w_{1}-P_{1}(\cdot;b,\eta), & \eps_{2} & \coloneqq w_{2}-P_{2}(\cdot;b,\eta).
\end{aligned}
\tag{\ref{eq:def-e-e2}}
\]
We further define $\eps_{3}$ by taking the linear operator $A_{Q}^{\ast}$
to $\eps_{2}$: 
\[
\eps_{3}\coloneqq A_{Q}^{\ast}\eps_{2}.
\]
Here it suffices to use this \emph{linear} adapted derivative $\eps_{3}$
of $\eps_{2}$, as opposed to $\eps_{1}$ or $\eps_{2}$. With these
adapted derivatives, we can rigorously state our bootstrap hypothesis.
For a large universal constant $K>1$ to be chosen later, we set the
bootstrap assumptions 
\begin{equation}
\begin{gathered}0<b<b^{\ast},\quad|\eta|<\tfrac{b}{|\log b|},\\
\|\eps\|_{L^{2}}<(b^{\ast})^{\frac{1}{4}},\ \|\eps_{1}\|_{L^{2}}<Kb|\log b|^{2},\ \|\eps_{3}\|_{L^{2}}<K\tfrac{b^{2}}{|\log b|}.
\end{gathered}
\label{eq:BootstrapHypothesis}
\end{equation}
Let $u$ be a solution to \eqref{eq:CSS-self-dual-form} with the
initial data $u_{0}\in\calO_{\init}$ and maximal forward-in-time
lifespan $[0,T)$. This $u$ is called a \emph{trapped solution} if
it admits the nonlinear decomposition for each time $t\in[0,T)$ and
satisfies the bootstrap assumptions \eqref{eq:BootstrapHypothesis}.

We note that the assumptions \eqref{eq:BootstrapHypothesis} are initially
satisfied at $t=0$. In other words, any elements of $\calO_{\init}$
satisfy \eqref{eq:BootstrapHypothesis}. Indeed, if we are given $(\widehat{\lmb},\widehat{\gmm},\widehat{b},\widehat{\eta},\widehat{\eps})\in\calU_{\init}$
and denote $\widehat{w}=P(\cdot;\widehat{b},\widehat{\eta})+\widehat{\eps}$
and $\widehat{\eps}_{1}=\bfD_{\widehat{w}}\widehat{w}-P_{1}(\cdot;\widehat{b},\widehat{\eta})$,
then we have for $k\in\{3,4\}$ 
\begin{equation}
\begin{aligned}(\widehat{\eps}_{1},\td{\calZ}_{k})_{r} & =(\widehat{\eps}_{1},\td{\calZ}_{k})_{r}-(\widehat{\eps},\calZ_{k})_{r}=(\bfD_{\widehat{w}}\widehat{w}-P_{1}-L_{Q}\widehat{\eps},\calZ_{k})_{r}\\
 & \aleq M^{C}(\|\bfD_{P}P-P_{1}\|_{\dot{\calH}_{1}^{2}}+\|(L_{P}-L_{Q})\widehat{\eps}\|_{\dot{\calH}_{1}^{2}}+\|N_{P}(\widehat{\eps})\|_{\dot{\calH}_{1}^{2}})\aleq M^{C}(\widehat{b})^{2},
\end{aligned}
\label{eq:initial-decomposition-transition}
\end{equation}
where the last inequality can be proved by the proof of \eqref{eq:NonlinearCoercivityH3-e1e3}
below. Therefore, by the difference estimate \eqref{eq:DecompositionDifference},
the rough decomposition $(\widehat{\lmb},\widehat{\gmm},\widehat{b},\widehat{\eta},\widehat{\eps})\in\calU_{\init}$
and the nonlinear decomposition $(\lmb,\gmm,b,\eta,\eps)$ only differ
by $O(M^{C}(\widehat{b})^{2})$ for data in $\calO_{\init}$.

In the sequel, we will see that all the assumptions except the bound
$|\eta|<\frac{b}{|\log b|}$ can be bootstrapped. Note that $\eta$
is almost conserved by $\eta_{s}\approx0$, whereas $b$ tends to
zero by $b_{s}+b^{2}+\frac{2b^{2}}{|\log b|}\approx0$. Thus the $\eta$-bound
$|\eta|<\frac{b}{|\log b|}$ cannot be bootstrapped and the trapped
solutions are non-generic. This is the source of codimension one as
illustrated before. We will construct these non-generic solutions
using a soft connectivity argument.

We conclude this subsection by reducing the proof of Theorem~\ref{thm:main-thm}
into three propositions: main bootstrap (Proposition~\ref{prop:main-bootstrap}),
a proposition for the connectivity argument (Proposition~\ref{prop:Sets-I-pm}),
and a sharp description of the trapped solutions (Proposition~\ref{prop:SharpDescription}).
The heart of the proof is the main boostrap, Proposition~\ref{prop:main-bootstrap}.
\begin{proof}[Proof of Theorem~\ref{thm:main-thm} assuming Propositions~\ref{prop:main-bootstrap},
\ref{prop:Sets-I-pm}, and \ref{prop:SharpDescription}]
Let $(\widehat{\lmb}_{0},\widehat{\gmm}_{0},\widehat{b}_{0},\widehat{\eps}_{0})\in\td{\calU}_{\init}$
and consider $\widehat{\eta}_{0}$ which varies in the range $(-\frac{\widehat{b}_{0}}{2|\log\widehat{b}_{0}|},\frac{\widehat{b}_{0}}{2|\log\widehat{b}_{0}|})$.
Define $u_{0}\in\calO_{\init}$ via \eqref{eq:initial-u0} and let
$u$ be the forward-in-time maximal solution to \eqref{eq:CSS-self-dual-form}
with the initial data $u_{0}$ and lifespan $[0,T)$.

Our main goal is to show that $u$ is a trapped solution for a well-chosen
$\widehat{\eta}_{0}$. Notice that $u_{0}$ is formed by the rough
decomposition. Define the exit time of $\calO_{\dec}$: 
\[
T_{\dec}\coloneqq\sup\{\tau\in[0,T):u(\tau')\in\calO_{\dec}\text{ for }\tau'\in[0,\tau]\}\in(0,T].
\]
Thus $u(t)$ for $t\in[0,T_{\dec})$ admits the nonlinear decomposition
$(\lmb(t),\gmm(t),b(t),\eta(t),\eps(t))$ according to Lemma~\ref{lem:decomp}.
Moreover, if $T_{\dec}<T$, then $u(T_{\dec})\in\overline{\calO}_{\dec}\setminus\calO_{\dec}$
and it also admits the nonlinear decomposition at time $t=T_{\dec}$.
Next, thanks to \eqref{eq:initial-decomposition-transition}, the
nonlinear decomposition $(\lmb_{0},\gmm_{0},b_{0},\eta_{0},\eps_{0})$
at $t=0$ satisfies the bootstrap assumption \eqref{eq:BootstrapHypothesis}.
Thus we can also define the exit time of the bootstrap hypotheses:
\[
T_{\mathrm{exit}}\coloneqq\sup\{\tau\in[0,T_{\dec}):\eqref{eq:BootstrapHypothesis}\text{ holds for all }\tau'\in[0,\tau]\}\in(0,T_{\dec}].
\]
Thus our goal is to show that $T_{\mathrm{exit}}=T_{\dec}=T$ for
some $\widehat{\eta}_{0}$. Then $u$ is a trapped solution with this
$\widehat{\eta}_{0}$.

In fact, it suffices to show that $T_{\mathrm{exit}}=T_{\dec}$ for
some $\widehat{\eta}_{0}$. Indeed, if $T_{\mathrm{exit}}=T_{\dec}$
but $T_{\dec}<T$, then $u(T_{\dec})\in\overline{\calO}_{\dec}\setminus\calO_{\dec}$
but $(\lmb,\gmm,b,\eta,\eps)$ at $t=T_{\mathrm{exit}}$ lies in the
closure of the bootstrap hypotheses. Since $u(T_{\mathrm{exit}})=u(T_{\dec})$,
we must have $b=\eta=0$ and $\eps=0$ at $t=T_{\mathrm{exit}}$.
In other words, $u$ is a rescaled $Q$, which is a static solution.
This contradicts the assumption $u_{0}\in\calO_{\init}$.

To show that $T_{\mathrm{exit}}=T_{\dec}$ for some $\widehat{\eta}_{0}$,
assume for the sake of contradiction that $T_{\mathrm{exit}}<T_{\dec}$
for all $\widehat{\eta}_{0}$. The following proposition is shown
in Section~\ref{subsec:ClosingBootstrap}, and is the heart of the
proof of Theorem~\ref{thm:main-thm}: 
\begin{prop}[Main bootstrap]
\label{prop:main-bootstrap}Let $u$ have the nonlinear decomposition
$(\lmb,\gmm,b,\eta,\eps)$. If the boostrap hypotheses \eqref{eq:BootstrapHypothesis}
hold for $t\in[0,\tau_{\ast}]$ for some $\tau_{\ast}>0$, then the
following hold for $t\in[0,\tau_{\ast}]$: 
\[
b\in(0,b_{0}],\quad\|\eps\|_{L^{2}}<\tfrac{1}{2}(b^{\ast})^{\frac{1}{4}},\quad\|\eps_{1}\|_{L^{2}}<\tfrac{K}{2}b|\log b|^{2},\quad\|\eps_{3}\|_{L^{2}}<\tfrac{K}{2}\tfrac{b^{2}}{|\log b|}.
\]
\end{prop}

The fact that $T_{\mathrm{exit}}<T_{\dec}$ together with Proposition~\ref{prop:main-bootstrap}
imply that $|\eta|=\frac{b}{2|\log b|}$ at $t=T_{\mathrm{exit}}$.
To derive a contradiction, we use a basic connectivity argument. Let
$\calI_{\pm}$ be the set of initial $\widehat{\eta}_{0}$ such that
$\eta=\pm\frac{b}{2|\log b|}$ at $t=T_{\mathrm{exit}}$. Note that
$\calI_{\pm}$ partitions $(-\frac{b_{0}}{2|\log b_{0}|},\frac{b_{0}}{2|\log b_{0}|})$.
The following proposition is shown in Section~\ref{subsec:ClosingBootstrap}. 
\begin{prop}[The sets $\calI_{\pm}$]
\label{prop:Sets-I-pm}The sets $\calI_{\pm}$ are nonempty and open. 
\end{prop}

We have a contradiction from the connectivity of $(-\frac{b_{0}}{2|\log b_{0}|},\frac{b_{0}}{2|\log b_{0}|})$.
Thus our claim, $T_{\mathrm{exit}}=T_{\dec}$ for some $\widehat{\eta}_{0}$,
is proved. Therefore, there exists a trapped solution $u$ with this
$\widehat{\eta}_{0}$.

The remaining part of the proof is the sharp description of this trapped
solution. The following is proved in Section~\ref{subsec:ClosingBootstrap}. 
\begin{prop}[Sharp description]
\label{prop:SharpDescription}Let $u$ be a trapped solution. Then,
it blows up in finite time as described in Theorem~\ref{thm:main-thm}. 
\end{prop}

This ends the proof of Theorem~\ref{thm:main-thm} assuming Propositions
\ref{prop:main-bootstrap}, \ref{prop:Sets-I-pm}, and \ref{prop:SharpDescription}. 
\end{proof}
Using Theorem~\ref{thm:main-thm} and the pseudoconformal transform,
we prove Corollary~\ref{cor:InfiniteBlowup}.
\begin{proof}[Proof of Corollary~\ref{cor:InfiniteBlowup}]
Let $v$ be a finite-time blow-up solution with smooth compactly
supported initial data $v_{0}$, constructed in Theorem~\ref{thm:main-thm}
(see also Comments on Theorem~\ref{thm:main-thm}). Applying scaling,
phase rotation, and time translation symmetries, we may assume that
$v$ is defined on $[-T,0)$ with $v(-T)=v_{0}$ and blows up at time
$0$ with the decomposition 
\[
v(t)-\frac{|\log|t||^{2}}{|t|}Q\Big(\frac{|\log|t||^{2}}{|t|}r\Big)-v^{\ast}\to0\text{ in }L^{2}
\]
as $t\to0^{-}$. For convenience, we rewrite this as 
\[
v(t)=\frac{|\log|t||^{2}}{|t|}Q\Big(\frac{|\log|t||^{2}}{|t|}r\Big)+e^{it\Delta}v^{\ast}+\err(t)
\]
with $\|\err(t)\|_{L^{2}}\to0$ as $t\to0^{-}$. We now apply the
pseudoconformal transform $\calC$ \eqref{eq:discrete-pseudo-transf}
to obtain the solution $u$ on $[1/T,\infty)$ defined by 
\[
u(t)\coloneqq[\calC v](t).
\]
Note that the initial data $u(1/T)$ is smooth and compactly supported.
Since $\calC$ preserves the $L^{2}$-norm, the contribution of $\err(t)$
is negligible: $\|[\calC(\err)](t)\|_{L^{2}}\to0$ as $t\to\infty$.
Moreover, since $\calC$ preserves linear Schrödinger waves, 
\[
[\calC(e^{it\Delta}v^{\ast})](t)=e^{it\Delta}u^{\ast}
\]
for some $u^{\ast}\in L^{2}$ with $\|u^{\ast}\|_{L^{2}}=\|v^{\ast}\|_{L^{2}}$.
Finally, we have 
\[
\bigg[\calC\bigg(\frac{|\log|t||^{2}}{|t|}Q\Big(\frac{|\log|t||^{2}}{|t|}r\Big)\bigg)\bigg](t)=e^{i\frac{r^{2}}{4t}}|\log(t)|^{2}Q\Big(|\log(t)|^{2}r\Big).
\]
We can remove $e^{i\frac{r^{2}}{4t}}$ by applying the dominated convergence
theorem (after rescaling): 
\[
\bigg\|(e^{i\frac{r^{2}}{4t}}-1)\bigg\{|\log(t)|^{2}Q\Big(|\log(t)|^{2}r\Big)\bigg\}\bigg\|_{L^{2}}\to0
\]
as $t\to\infty$. Therefore, 
\[
u(t)-|\log(t)|^{2}Q\Big(|\log(t)|^{2}r\Big)-e^{it\Delta}u^{\ast}\to0\text{ in }L^{2}
\]
as $t\to\infty$. 
\end{proof}
In the remaining sections, we show Propositions~\ref{prop:main-bootstrap},
\ref{prop:Sets-I-pm}, and \ref{prop:SharpDescription}. The main
bootstrap Proposition \ref{prop:main-bootstrap}, which is the heart
of the proof, is proved through Sections~\ref{subsec:CoercivityNonlinearAdaptedDer}--\ref{subsec:ClosingBootstrap}.
Propositions~\ref{prop:Sets-I-pm} and~\ref{prop:SharpDescription}
are proved in Section~\ref{subsec:ClosingBootstrap}.

In the rest of this paper, \textbf{we assume the bootstrap hypotheses
\eqref{eq:BootstrapHypothesis}.} Moreover, \textbf{we assume the
parameter dependence} 
\begin{equation}
0\ll b^{\ast}\ll M^{-1}\ll K^{-1}\ll1,\label{eq:parameter-dependence}
\end{equation}
where $K$ is the constant in the bootstrap, $M$ is a large parameter
introduced in the decomposition Lemma \ref{lem:decomp}, and $b^{\ast}$
is a small parameter introduced in the definition of initial data
sets \eqref{eq:def-Utilde-init}--\eqref{eq:def-O-init} that restricts
the admissible range of $b$: $0<b<b^{\ast}$. In the sequel, we will
freely shrink $b^{\ast}>0$ and enlarge $M\gg1$ (at the cost of further
shrinking $b^{\ast})$. Finally, we adopt the following abuse of notation:
When there is a string of $\aleq$'s, we only express the dependencies
of the implicit constants in relation to the left-most expression.
For instance, if we have an estimate $\nrm{\eps}_{\dot{\calH}_{0}^{1}}\aleq_{M}Kb\abs{\log b}^{2}$,
then $\nrm{\eps}_{\dot{\calH}_{0}^{1}}\aleq b\abs{\log b}^{2+}$ by
parameter dependence (using $C(M)K\leq|\log b^{\ast}|^{0+}\leq|\log b|^{0+}$).
We simply write this chain of estimates as $\nrm{\eps}_{\dot{\calH}_{0}^{1}}\aleq_{M}Kb\abs{\log b}^{2}\aleq b\abs{\log b}^{2+}$.

\subsection{\label{subsec:CoercivityNonlinearAdaptedDer}Coercivity for nonlinear
adapted derivatives}

Recall that we decomposed our solution $u$ according to the nonlinear
decomposition. That is, 
\[
\begin{aligned}w & \coloneqq e^{-i\gmm}\lmb u(\lmb\cdot), & w_{1} & \coloneqq\bfD_{w}w, & w_{2} & \coloneqq A_{w}w_{1},\\
\eps & \coloneqq w-P(\cdot;b,\eta), & \eps_{1} & \coloneqq w_{1}-P_{1}(\cdot;b,\eta), & \eps_{2} & \coloneqq w_{2}-P_{2}(\cdot;b,\eta),
\end{aligned}
\tag{\ref{eq:def-e-e2}}
\]
and the orthogonality conditions 
\[
(\eps,\calZ_{1})_{r}=(\eps,\calZ_{2})_{r}=(\eps_{1},\td{\calZ}_{3})_{r}=(\eps_{1},\td{\calZ}_{4})_{r}=0\tag{\ref{eq:OrthogonalityNonlinear}}
\]
are satisfied. We defined $\eps_{3}$ by $\eps_{3}\coloneqq A_{Q}^{\ast}\eps_{2}$.

The goal of this section is to transfer the linear coercivity estimates
(Proposition \ref{prop:LinearCoercivity}) to the nonlinear adapted
derivatives $\eps_{1},\eps_{2},\eps_{3}$, under the bootstrap assumptions
\eqref{eq:BootstrapHypothesis}. By the linearization of the Bogomol'nyi
operator (see \eqref{eq:LinearizationBogomolnyi}), $\bfD_{P}P\approx P_{1}$
(see \eqref{eq:Comparison-P-P1-H1} and \eqref{eq:Comparison-P-P1-H3}),
and $A_{P}P_{1}\approx P_{2}$ (see \eqref{eq:Comparison-P1-P2-H3}),
we see that $\eps_{1}\approx L_{Q}\eps$ and $\eps_{2}\approx A_{Q}\eps_{1}$.

As mentioned earlier, we will take advantage of using nonlinear adapted
derivatives in various places. Compared to using the linear ones,
one can observe that error terms are simplified in the evolution equations
of $\eps_{1},\eps_{2},\eps_{3}$ in Sections~\ref{subsec:Modulation-estimates}
and~\ref{subsec:Energy-estimate}. The following estimates are the
trade-offs. We need additional arguments to establish the coercivity
relations of the nonlinear adapted derivatives.
\begin{lem}[Nonlinear coercivity estimates]
\label{lem:NonlinearCoercivity}The following estimates hold. 
\begin{enumerate}
\item ($\dot{H}^{1}$-level) 
\begin{equation}
\|\eps\|_{\dot{\calH}_{0}^{1}}\aleq_{M}Kb|\log b|^{2}\aleq b|\log b|^{2+}.\label{eq:NonlinearCoercivityH1}
\end{equation}
\item ($\dot{H}^{3}$-level) 
\begin{gather}
\|\eps_{2}\|_{\dot{\calH}_{2}^{1}}\sim\|\eps_{3}\|_{L^{2}},\label{eq:NonlinearCoercivityH3-e2e3}\\
\|\eps_{1}\|_{\dot{\calH}_{1}^{2}}\aleq_{M}K\tfrac{b^{2}}{|\log b|}\aleq\tfrac{b^{2}}{|\log b|^{1-}}.\label{eq:NonlinearCoercivityH3-e1e3}\\
\|\eps\|_{\dot{\calH}_{0}^{3}}\aleq_{M}Kb^{2}\aleq b^{2}|\log b|^{0+}.\label{eq:NonlinearCoercivityH3-e0e3}
\end{gather}
\item (Interpolation estimates at $\dot{H}^{2}$-level) 
\begin{equation}
\||\eps_{1}|_{-1}\|_{L^{2}}+\|\eps_{2}\|_{L^{2}}\aleq_{M}Kb^{\frac{3}{2}}|\log b|^{\frac{1}{2}}\aleq b^{\frac{3}{2}}|\log b|^{\frac{1}{2}+}.\label{eq:H2-interpolation}
\end{equation}
\end{enumerate}
\end{lem}

\begin{proof}
(1) From the relation 
\[
P_{1}+\eps_{1}=w_{1}=\bfD_{w}w=\bfD_{P}P+L_{Q}\eps+(L_{P}-L_{Q})\eps+N_{P}(\eps),
\]
the coercivity estimate \eqref{eq:LQ-coer-H1-sec2} implies that 
\[
\|\eps\|_{\dot{\calH}_{0}^{1}}\aleq_{M}\|L_{Q}\eps\|_{L^{2}}\aleq\|\eps_{1}\|_{L^{2}}+\|\bfD_{P}P-P_{1}\|_{L^{2}}+\|(L_{P}-L_{Q})\eps\|_{L^{2}}+\|N_{P}(\eps)\|_{L^{2}}.
\]
The second term is estimated by $b$ due to \eqref{eq:Comparison-P-P1-H1}.
We claim that the last two terms are estimated by 
\begin{equation}
\|(L_{P}-L_{Q})\eps\|_{L^{2}}+\|N_{P}(\eps)\|_{L^{2}}\aleq(o_{b^{\ast}\to0}(1)+\|\eps\|_{\dot{\calH}_{0}^{1}})\|\eps\|_{\dot{\calH}_{0}^{1}}.\label{eq:NonlinearCoercivityTemp1}
\end{equation}
Assuming this, these terms are absorbed into the LHS and we have 
\[
\|\eps\|_{\dot{\calH}_{0}^{1}}\aleq_{M}\|\eps_{1}\|_{L^{2}}+b.
\]
The bootstrap hypothesis \eqref{eq:BootstrapHypothesis} on $\eps_{1}$
and the parameter dependence \eqref{eq:parameter-dependence} yield
\eqref{eq:NonlinearCoercivityH1}.

Henceforth, we show the claim \eqref{eq:NonlinearCoercivityTemp1}.
Notice that $(L_{P}-L_{Q})\eps$ and $N_{P}(\eps)$ are linear combinations
of $\tfrac{1}{y}A_{\theta}[\psi_{1},\psi_{2}]\psi_{3}$, which we
estimate by 
\[
\|\tfrac{1}{y}A_{\theta}[\psi_{1},\psi_{2}]\psi_{3}\|_{L^{2}}\aleq\min_{\{j_{1},j_{2},j_{3}\}=\{1,2,3\}}\|\tfrac{1}{\langle y\rangle}\psi_{j_{1}}\|_{L^{2}}\|(\tfrac{\langle y\rangle}{y})^{\frac{1}{2}}\psi_{j_{2}}\|_{L^{2}}\|(\tfrac{\langle y\rangle}{y})^{\frac{1}{2}}\psi_{j_{3}}\|_{L^{2}}.
\]
For $(L_{P}-L_{Q})\eps$, we can assume $\psi_{j_{1}}=\eps$, $\psi_{j_{2}}=P-Q$,
and $\psi_{j_{3}}\in\{P,Q\}$ so 
\[
\|(L_{P}-L_{Q})\eps\|_{L^{2}}\aleq b^{\frac{1}{2}-}\|\eps\|_{\dot{\calH}_{0}^{1}}.
\]
For $N_{P}(\eps)$, we can assume $\psi_{j_{1}}=\psi_{j_{2}}=\eps$
and $\psi_{j_{3}}\in\{P,\eps\}$ so 
\[
\|N_{P}(\eps)\|_{L^{2}}\aleq(\|\eps\|_{\dot{\calH}_{0}^{1}}+\|\eps\|_{L^{2}})\|\eps\|_{\dot{\calH}_{0}^{1}}\aleq(\|\eps\|_{\dot{\calH}_{0}^{1}}+o_{b^{\ast}\to0}(1))\|\eps\|_{\dot{\calH}_{0}^{1}}.
\]
This shows the claim \eqref{eq:NonlinearCoercivityTemp1}.

(2) The equivalence \eqref{eq:NonlinearCoercivityH3-e2e3} follows
from the coercivity of $A_{Q}A_{Q}^{\ast}$ \eqref{eq:positivity-AQAQstar}.

We turn to \eqref{eq:NonlinearCoercivityH3-e1e3}. We simultaneously
consider the relations 
\begin{align*}
P_{2}+\eps_{2} & =w_{2}=A_{w}w_{1}=A_{P}P_{1}+A_{Q}\eps_{1}+(A_{w}-A_{Q})\eps_{1}+(A_{w}-A_{P})P_{1},\\
P_{1}+\eps_{1} & =w_{1}=\bfD_{w}w=\bfD_{P}P+L_{Q}\eps+(L_{P}-L_{Q})\eps+N_{P}(\eps).
\end{align*}
By the coercivity estimates \eqref{eq:AQ-coer-H2-sec2} and \eqref{eq:LQ-coer-H3-sec2},
we have 
\begin{align*}
\|\eps_{1}\|_{\dot{\calH}_{1}^{2}} & \aleq_{M}\|\eps_{2}\|_{\dot{\calH}_{2}^{1}}+\|A_{P}P_{1}-P_{2}\|_{\dot{\calH}_{2}^{1}}+\|(A_{w}-A_{Q})\eps_{1}\|_{\dot{\calH}_{2}^{1}}+\|(A_{w}-A_{P})P_{1}\|_{\dot{\calH}_{2}^{1}},\\
\|\eps\|_{\dot{\calH}_{0}^{3}} & \aleq_{M}\|\eps_{1}\|_{\dot{\calH}_{1}^{2}}+\|\bfD_{P}P-P_{1}\|_{\dot{\calH}_{1}^{2}}+\|(L_{P}-L_{Q})\eps\|_{\dot{\calH}_{1}^{2}}+\|N_{P}(\eps)\|_{\dot{\calH}_{1}^{2}}.
\end{align*}
Here, we have $\|\eps_{2}\|_{\dot{\calH}_{2}^{1}}\sim\|\eps_{3}\|_{L^{2}}$
by \eqref{eq:NonlinearCoercivityH3-e2e3}, and the $\eps$-independent
terms $\bfD_{P}P-P_{1}$ and $A_{P}P_{1}-P_{2}$ are estimated in
\eqref{eq:Comparison-P-P1-H3} and \eqref{eq:Comparison-P1-P2-H3}.
Therefore, 
\begin{align*}
\|\eps_{1}\|_{\dot{\calH}_{1}^{2}}+\tfrac{1}{|\log b|}\|\eps\|_{\dot{\calH}_{0}^{3}} & \aleq_{M}\|\eps_{3}\|_{L^{2}}+\tfrac{b^{2}}{|\log b|}+\|(A_{w}-A_{Q})\eps_{1}\|_{\dot{\calH}_{2}^{1}}+\|(A_{w}-A_{P})P_{1}\|_{\dot{\calH}_{2}^{1}}\\
 & \quad+\tfrac{1}{|\log b|}\|(L_{P}-L_{Q})\eps\|_{\dot{\calH}_{1}^{2}}+\tfrac{1}{|\log b|}\|N_{P}(\eps)\|_{\dot{\calH}_{1}^{2}}.
\end{align*}
We claim the estimates 
\begin{align}
 & \|(A_{w}-A_{Q})\eps_{1}\|_{\dot{\calH}_{2}^{1}}+\|(A_{w}-A_{P})P_{1}\|_{\dot{\calH}_{2}^{1}}\label{eq:NonlinearCoercivityTemp2}\\
 & \quad\aleq b^{3-}+(o_{b^{\ast}\to0}(1)+\|\eps\|_{\dot{\calH}_{0}^{3}}+\|\eps\|_{\dot{\calH}_{0}^{3}}^{2})(\|\eps_{1}\|_{\dot{\calH}_{1}^{2}}+\tfrac{1}{|\log b|}\|\eps\|_{\dot{\calH}_{0}^{3}}),\nonumber \\
 & \|(L_{P}-L_{Q})\eps\|_{\dot{\calH}_{1}^{2}}+\|N_{P}(\eps)\|_{\dot{\calH}_{1}^{2}}\label{eq:NonlinearCoercivityTemp3}\\
 & \quad\aleq b^{3-}+(o_{b^{\ast}\to0}(1)+\|\eps\|_{\dot{\calH}_{0}^{3}})\|\eps\|_{\dot{\calH}_{0}^{3}}.\nonumber 
\end{align}
Assuming these claims, we have 
\[
\|\eps_{1}\|_{\dot{\calH}_{1}^{2}}+\tfrac{1}{|\log b|}\|\eps\|_{\dot{\calH}_{0}^{3}}\aleq_{M}\|\eps_{3}\|_{L^{2}}+\tfrac{b^{2}}{|\log b|},
\]
which implies \eqref{eq:NonlinearCoercivityH3-e1e3} and \eqref{eq:NonlinearCoercivityH3-e0e3}
after substituting the bootstrap hypothesis for $\eps_{3}$.

Henceforth, we show the claims \eqref{eq:NonlinearCoercivityTemp2}
and \eqref{eq:NonlinearCoercivityTemp3}. First, we show \eqref{eq:NonlinearCoercivityTemp2}.
For the first term on the LHS of \eqref{eq:NonlinearCoercivityTemp2},
we use the definition of $\dot{\calH}_{2}^{1}$ to have 
\begin{align*}
\|(A_{w}-A_{Q})\eps_{1}\|_{\dot{\calH}_{2}^{1}} & \aleq\|(|w|^{2}-Q^{2})\eps_{1}\|_{L^{2}}+\|A_{\theta}[w]-A_{\theta}[Q]\|_{L^{\infty}}\|\partial_{y}(\tfrac{1}{y}\eps_{1})\|_{L^{2}}\\
 & \quad+\|\langle\log_{-}y\rangle(A_{\theta}[w]-A_{\theta}[Q])\|_{L^{\infty}}\|\eps_{1}\|_{\dot{\calH}_{1}^{2}}.
\end{align*}
Since $\partial_{y}(\tfrac{1}{y}\eps_{1})=\tfrac{1}{y}\partial_{+}\eps_{1}$,
where $\partial_{+}=\partial_{y}-\tfrac{1}{y}$ when acting on $1$-equivariant
functions, \eqref{eq:ComparisonH2H2} says that $\|\partial_{y}(\tfrac{1}{y}\eps_{1})\|_{L^{2}}\aleq\|\partial_{+}\eps_{1}\|_{\dot{H}_{2}^{1}}\aleq\|\eps_{1}\|_{\dot{H}_{1}^{2}}\aleq\|\eps_{1}\|_{\dot{\calH}_{1}^{2}}$.
Thus we have 
\[
\|(A_{w}-A_{Q})\eps_{1}\|_{\dot{\calH}_{2}^{1}}\aleq\|(|w|^{2}-Q^{2})\eps_{1}\|_{L^{2}}+\|\langle\log_{-}y\rangle(A_{\theta}[w]-A_{\theta}[Q])\|_{L^{\infty}}\|\eps_{1}\|_{\dot{\calH}_{1}^{2}}.
\]
Using the estimates 
\begin{align*}
\|(|w|^{2}-Q^{2})\eps_{1}\|_{L^{2}} & \aleq\||w|^{2}-|P|^{2}\|_{L^{\infty}}\|\eps_{1}\|_{L^{2}}+\|(|P|^{2}-Q^{2})\eps_{1}\|_{L^{2}}\\
 & \aleq(\|\eps\|_{L^{\infty}}^{2}+\|P\eps\|_{L^{\infty}})\|\eps_{1}\|_{L^{2}}+b^{1-}\|\eps_{1}\|_{\dot{\calH}_{1}^{2}},\\
\|\langle\log_{-}y\rangle(A_{\theta}[w]-A_{\theta}[Q])\|_{L^{\infty}} & \aleq\||w|^{2}-Q^{2}\|_{L^{1}}+\||w|^{2}-Q^{2}\|_{L^{\infty}}\\
 & \aleq b^{1-}+\|\eps\|_{L^{2}}+\|\eps\|_{L^{\infty}}+\|\eps\|_{L^{\infty}}^{2},
\end{align*}
weighted $L^{\infty}$-estimates (see Lemma~\ref{lem:Weighted-Linfty})
\begin{align*}
\|\eps\|_{L^{\infty}}^{2} & \aleq\|\eps\|_{\dot{\calH}_{0}^{1}}^{2}+\|\eps\|_{\dot{\calH}_{0}^{3}}^{2},\\
\|P\eps\|_{L^{\infty}} & \aleq\|\eps\|_{\dot{\calH}_{0}^{3}}+\|\eps\|_{\dot{\calH}_{0}^{1}}^{0+}\|\eps\|_{\dot{\calH}_{0}^{3}}^{1-},
\end{align*}
and substituting \eqref{eq:NonlinearCoercivityH1} and $\|\eps\|_{L^{2}}=o_{b^{\ast}\to0}(1)$,
we obtain 
\[
\|(A_{w}-A_{Q})\eps_{1}\|_{\dot{\calH}_{2}^{1}}\aleq b^{3-}+(o_{b^{\ast}\to0}(1)+\|\eps\|_{\dot{\calH}_{0}^{3}}+\|\eps\|_{\dot{\calH}_{0}^{3}}^{2})(\|\eps_{1}\|_{\dot{\calH}_{1}^{2}}+\tfrac{1}{|\log b|}\|\eps\|_{\dot{\calH}_{0}^{3}}).
\]
as desired in \eqref{eq:NonlinearCoercivityTemp2}. Next, the second
term on the LHS of \eqref{eq:NonlinearCoercivityTemp2} is estimated
by 
\begin{align*}
\|(A_{w}-A_{P})P_{1}\|_{\dot{\calH}_{2}^{1}} & \aleq\||A_{w}-A_{P}|_{-1}\|_{L^{\infty}}\||P_{1}|_{1}\|_{L^{2}}\\
 & \aleq b^{1-}\||w|^{2}-|P|^{2}\|_{L^{\infty}}\aleq b^{1-}(\|\eps\|_{L^{\infty}}^{2}+\|P\eps\|_{L^{\infty}}).
\end{align*}
Recalling how we dealt with $\|\eps\|_{L^{\infty}}^{2}+\|P\eps\|_{L^{\infty}}$
above, this bound suffices. This completes the proof of \eqref{eq:NonlinearCoercivityTemp2}.

Next, we show \eqref{eq:NonlinearCoercivityTemp3}. Recall that $(L_{P}-L_{Q})\eps$
and $N_{P}(\eps)$ are linear combinations of $\tfrac{1}{y}A_{\theta}[\psi_{1},\psi_{2}]\psi_{3}$.
In view of \eqref{eq:ComparisonH2H2} (see also its proof), we have
\begin{align*}
 & \|\tfrac{1}{y}A_{\theta}[\psi_{1},\psi_{2}]\psi_{3}\|_{\dot{\calH}_{1}^{2}}\\
 & \aleq\|\Delta_{1}(\tfrac{1}{y}A_{\theta}[\psi_{1},\psi_{2}]\psi_{3})\|_{L^{2}}+\|\chf_{y\sim1}\tfrac{1}{y}A_{\theta}[\psi_{1},\psi_{2}]\psi_{3}\|_{L^{2}},\\
 & \aleq\|\chf_{y\sim1}A_{\theta}[\psi_{1},\psi_{2}]\psi_{3}\|_{L^{2}}+\|\tfrac{1}{y}({\textstyle \int_{0}^{y}}\Re(\overline{\psi_{1}}\psi_{2})y'dy')(\partial_{y}-\tfrac{1}{y})\partial_{y}\psi_{3}\|_{L^{2}}\\
 & \quad+\|(\partial_{y}\psi_{3})\Re(\overline{\psi_{1}}\psi_{2})\|_{L^{2}}+\|\psi_{3}\partial_{y}\Re(\overline{\psi_{1}}\psi_{2})\|_{L^{2}}.
\end{align*}
We will only consider choices of $\psi_{1},\psi_{2},\psi_{3}$ that
can contribute to $(L_{P}-L_{Q})\eps$ or $N_{P}(\eps)$. That is,
the set of $\psi_{1},\psi_{2},\psi_{3}$ contains at least two $\eps$'s
or one $\eps$ and one $P-Q$. The first two terms can be estimated
using weighted $L^{\infty}$-estimates (Lemma \ref{lem:Weighted-Linfty}):
\begin{align*}
 & \|\chf_{y\sim1}A_{\theta}[\psi_{1},\psi_{2}]\psi_{3}\|_{L^{2}}+\|\tfrac{1}{y}({\textstyle \int_{0}^{y}}\Re(\overline{\psi_{1}}\psi_{2})y'dy')(\partial_{y}-\tfrac{1}{y})\partial_{y}\psi_{3}\|_{L^{2}}\\
 & \aleq\begin{cases}
\|\psi_{1}\psi_{2}\|_{L^{1}}\|\tfrac{1}{y}(\rd_{y}-\tfrac{1}{y})\rd_{y}\eps\|_{L^{2}}\aleq(b^{1-}+\|\eps\|_{L^{2}}+\|\eps\|_{L^{2}}^{2})\|\eps\|_{\dot{\calH}_{0}^{3}} & \text{if }\psi_{3}=\eps,\\
\|\langle y\rangle^{-3+}\psi_{1}\psi_{2}\|_{L^{\infty}}\aleq(b+\|\eps\|_{L^{\infty}})\|\eps\|_{\dot{\calH}_{0}^{3}} & \text{if }\psi_{3}\in\{P,Q\},\\
b\|\langle y\rangle^{-1}\psi_{1}\psi_{2}\|_{L^{\infty}}\aleq b\|\eps\|_{\dot{\calH}_{0}^{3}} & \text{if }\psi_{3}=P-Q,
\end{cases}\\
 & \aleq(o_{b^{\ast}\to0}(1)+\|\eps\|_{\dot{\calH}_{0}^{3}})\|\eps\|_{\dot{\calH}_{0}^{3}}.
\end{align*}
We note that in the case $\psi_{3}=\eps$, we used $(\rd_{y}-\tfrac{1}{y})\rd_{y}=\rd_{+}\rd_{+}$
and \eqref{eq:d_plus_plus-H3}.

The last two terms can be estimated by 
\begin{align*}
 & \|(\partial_{y}\psi_{3})\Re(\overline{\psi_{1}}\psi_{2})\|_{L^{2}}+\|\psi_{3}\partial_{y}\Re(\overline{\psi_{1}}\psi_{2})\|_{L^{2}}\\
 & \aleq\begin{cases}
\|\partial_{y}\eps\|_{L^{2}}\|\eps\|_{L^{\infty}}^{2}\aleq b^{1-}(b^{2-}+\|\eps\|_{\dot{\calH}_{0}^{3}}^{2}) & \text{if }\psi_{1}=\psi_{2}=\psi_{3}=\eps,\\
\|\langle y\rangle^{-2+}\partial_{y}(\psi_{j_{1}}\psi_{j_{2}})\|_{L^{2}}+\|\langle y\rangle^{-3+}\psi_{j_{1}}\psi_{j_{2}}\|_{L^{2}} & \text{if }\psi_{j_{3}}\in\{P,Q\}\text{ for some }j_{3}.
\end{cases}
\end{align*}
In the latter case, we can further estimate by 
\begin{align*}
 & \|\langle y\rangle^{-2+}\partial_{y}(\psi_{j_{1}}\psi_{j_{2}})\|_{L^{2}}+\|\langle y\rangle^{-3+}\psi_{j_{1}}\psi_{j_{2}}\|_{L^{2}}\\
 & \aleq b^{1-}\|\eps\|_{\dot{\calH}_{0}^{3}}+\|\eps\|_{\dot{\calH}_{0}^{1}}\|\langle y\rangle^{-2+}\eps\|_{L^{\infty}}\\
 & \aleq b^{1-}\|\eps\|_{\dot{\calH}_{0}^{3}}+\|\eps\|_{\dot{\calH}_{0}^{1}}(\|\eps\|_{\dot{\calH}_{0}^{3}}+\|\eps\|_{\dot{\calH}_{0}^{1}}^{0+}\|\eps\|_{\dot{\calH}_{0}^{3}}^{1-}),
\end{align*}
so 
\[
\|(\partial_{y}\psi_{3})\Re(\overline{\psi_{1}}\psi_{2})\|_{L^{2}}+\|\psi_{3}\partial_{y}\Re(\overline{\psi_{1}}\psi_{2})\|_{L^{2}}\aleq b^{3-}+(o_{b^{\ast}\to0}(1)+\|\eps\|_{\dot{\calH}_{0}^{3}})\|\eps\|_{\dot{\calH}_{0}^{3}}.
\]
This completes the proof of \eqref{eq:NonlinearCoercivityTemp3}.

(3) To prove \eqref{eq:H2-interpolation}, we interpolate \eqref{eq:NonlinearCoercivityH1}
and \eqref{eq:NonlinearCoercivityH3-e1e3}. First, the interpolation
estimate \eqref{eq:interpolation-2} says 
\[
\||\eps_{1}|_{-1}\|_{L^{2}}\aleq\|\eps_{1}\|_{L^{2}}^{\frac{1}{2}}\|\eps_{1}\|_{\dot{\calH}_{1}^{2}}^{\frac{1}{2}}\aleq_{M}Kb^{\frac{3}{2}}|\log b|^{\frac{1}{2}}\aleq b^{\frac{3}{2}}|\log b|^{\frac{1}{2}+}.
\]
Next, we use $\|A_{P}P_{1}-P_{2}\|_{L^{2}}\aleq\frac{b^{3/2}}{|\log b|}$
(which can be proved by \eqref{eq:APP1-P2-temp}) to get 
\begin{align*}
\|\eps_{2}\|_{L^{2}} & \aleq\|A_{P}P_{1}-P_{2}\|_{L^{2}}+\|A_{w}\eps_{1}\|_{L^{2}}+\|(A_{w}-A_{P})P_{1}\|_{L^{2}}.\\
 & \aleq\tfrac{b^{3/2}}{|\log b|}+\||\eps_{1}|_{-1}\|_{L^{2}}+\||w|^{2}-|P|^{2}\|_{L^{\infty}}\|yP_{1}\|_{L^{2}}.
\end{align*}
Since $\||\eps_{1}|_{-1}\|_{L^{2}}\aleq_{M}Kb^{\frac{3}{2}}|\log b|^{\frac{1}{2}}$
and $\||w|^{2}-|P|^{2}\|_{L^{\infty}}\aleq b^{2-}$, we have 
\[
\|\eps_{2}\|_{L^{2}}\aleq_{M}Kb^{\frac{3}{2}}|\log b|^{\frac{1}{2}}\aleq b^{\frac{3}{2}}|\log b|^{\frac{1}{2}+}.
\]
This completes the proof of \eqref{eq:H2-interpolation}. 
\end{proof}

\subsection{\label{subsec:Modulation-estimates}Modulation estimates}

In this subsection, we prove that the modulation parameters roughly
evolve according to the formal parameter ODEs \eqref{eq:FormalParameterLaw}.
The evolution laws of $\lmb$ and $\gmm$ will be obtained from differentiating
the first two orthogonality conditions $(\eps,\calZ_{1})_{r}=(\eps,\calZ_{2})_{r}=0$.
The evolution laws of $b$ and $\eta$ will be obtained from the $\eps_{1}$-equation,
thanks to the conditions $(\eps_{1},\td{\calZ}_{3})_{r}=(\eps_{1},\td{\calZ}_{4})_{r}=0$
from the nonlinear decomposition.

We start by deriving the equation for $\eps$. Recall \eqref{eq:w-eqn-sd}
and \eqref{eq:P-equation}: 
\begin{align*}
(\partial_{s}-\frac{\lmb_{s}}{\lmb}\Lambda+\gmm_{s}i)w+iL_{w}^{\ast}w_{1} & =0,\\
(\partial_{s}-\frac{\lmb_{s}}{\lmb}\Lambda+\gmm_{s}i)P+iL_{P}^{\ast}P_{1} & =-\Mod\cdot\mathbf{v}+i\Psi.
\end{align*}
Subtracting the second from the first, we get the equation for $\eps$:
\begin{equation}
(\partial_{s}-\frac{\lmb_{s}}{\lmb}\Lambda+\gmm_{s}i)\eps+(iL_{w}^{\ast}w_{1}-iL_{P}^{\ast}P_{1})=\Mod\cdot\mathbf{v}-i\Psi.\label{eq:e-equation}
\end{equation}
From the identity 
\[
iL_{w}^{\ast}w_{1}-iL_{P}^{\ast}P_{1}=iL_{Q}^{\ast}\eps_{1}+(iL_{P}^{\ast}-iL_{Q}^{\ast})\eps_{1}+(iL_{w}^{\ast}-iL_{P}^{\ast})w_{1},
\]
the first term $iL_{Q}^{\ast}\eps_{1}$ can be considered as the leading
term of $iL_{w}^{\ast}w_{1}-iL_{P}^{\ast}P_{1}$.

Next, we derive the equation for $\eps_{1}$. Recall that 
\begin{equation}
\td{\gmm}_{s}=\gmm_{s}+\int_{0}^{\infty}\Re(\overline{w}w_{1})dy'.\label{eq:def-gamma-tilde}
\end{equation}
Recall also \eqref{eq:w1-eqn-sd} and \eqref{eq:P1-equation}: 
\begin{align*}
(\partial_{s}-\frac{\lmb_{s}}{\lmb}\Lambda_{-1}+\td{\gmm}_{s}i)w_{1}+iA_{w}^{\ast}w_{2}-\Big(\int_{0}^{y}\Re(\overline{w}w_{1})dy'\Big)iw_{1} & =0,\\
(\partial_{s}-\frac{\lmb_{s}}{\lmb}\Lambda_{-1}+\td{\gmm}_{s}i)P_{1}+iA_{P}^{\ast}P_{2}-\Big(\int_{0}^{y}\Re(\overline{P}P_{1})dy'\Big)iP_{1} & =-\td{\Mod}\cdot\mathbf{v}_{1}+i\Psi_{1}.
\end{align*}
Subtracting the second from the first, we get the equation for $\eps_{1}$:
\begin{align}
 & (\partial_{s}-\frac{\lmb_{s}}{\lmb}\Lambda_{-1}+\td{\gmm}_{s}i)\eps_{1}+iA_{Q}^{\ast}\eps_{2}\nonumber \\
 & =-(iA_{w}^{\ast}-iA_{P}^{\ast})w_{2}-(iA_{P}^{\ast}-iA_{Q}^{\ast})\eps_{2}+({\textstyle \int_{0}^{y}}\Re(\overline{w}w_{1})dy')i\eps_{1}\label{eq:e1-equation}\\
 & \quad+({\textstyle \int_{0}^{y}}(\Re(\overline{w}w_{1}-\overline{P}P_{1}))dy')iP_{1}+\td{\Mod}\cdot\mathbf{v}_{1}-i\Psi_{1}.\nonumber 
\end{align}

\begin{lem}[Modulation estimates]
\label{lem:ModulationEstimates}We have 
\begin{align}
\Big|\frac{\lmb_{s}}{\lmb}+b\Big|+|\gmm_{s}-\eta|+|\td{\gmm}_{s}+\eta| & \aleq b^{2-},\label{eq:ModEstimateScalePhase}\\
|b_{s}+b^{2}+\eta^{2}+c_{b}(b^{2}-\eta^{2})|+|\eta_{s}+2c_{b}b\eta| & \aleq\tfrac{1}{\sqrt{\log M}}\|\eps_{3}\|_{L^{2}}+b^{3-}.\label{eq:ModEstimate-b-eta}
\end{align}
\end{lem}

\begin{proof}
In this proof, we freely use the bootstrap hypotheses \eqref{eq:BootstrapHypothesis},
as well as Lemmas~\ref{lem:NonlinearCoercivity}, \ref{lem:interpolation},
and \ref{lem:Weighted-Linfty} to estimate $\eps$, $\eps_{1}$, and
$\eps_{2}$. We also abuse the notation and identify the operator
$A_{w}^{\ast}-A_{w'}^{\ast}$, which is simply the multiplication
by a function (namely, the difference of the zeroth order terms),
with that function.

We note that the estimate of $|\gmm_{s}-\eta|$ will follow from the
estimate of $|\td{\gmm}_{s}+\eta|$ and the claim 
\begin{equation}
\td{\gmm}_{s}-\gmm_{s}={\textstyle \int_{0}^{\infty}}\Re(\overline{w}w_{1})dy=-2\eta+O(b^{2-}).\label{eq:Gamma-GammaTilde}
\end{equation}
The claim can be obtained from the computations 
\[
{\textstyle \int_{0}^{\infty}}\Re(\overline{P}P_{1})dy=\eta{\textstyle \int_{0}^{\infty}}(-\tfrac{y}{2}Q^{2})dy+O(b^{2-})=-2\eta+O(b^{2-})
\]
and 
\begin{align*}
|{\textstyle \int_{0}^{\infty}}\Re(\overline{\eps}P_{1})dy| & \aleq\|\eps\|_{\dot{\calH}_{0}^{1}}\|\tfrac{\langle y\rangle}{y}\langle\log y\rangle P_{1}\|_{L^{2}}\aleq b^{2-},\\
|{\textstyle \int_{0}^{\infty}}\Re(\overline{P}\eps_{1})dy| & \aleq\|\tfrac{1}{y\langle y\rangle^{1-}}\eps_{1}\|_{L^{2}}\|\langle y\rangle^{1-}P\|_{L^{2}}\aleq b^{2-},\\
|{\textstyle \int_{0}^{\infty}}\Re(\overline{\eps}\eps_{1})dy| & \aleq\|\eps\|_{L^{\infty-}}\|\tfrac{1}{\brk{y}}\|_{{L^{2+}}}\|\tfrac{\langle y\rangle}{y}\eps_{1}\|_{L^{2}}\aleq b^{2-},
\end{align*}
where in the last inequality we used \eqref{eq:interpolation-1}.

In order to derive the modulation estimates for $\lmb$ and $\td{\gmm}$,
we differentiate the orthogonality conditions $(\eps,\calZ_{k})_{r}=0$
for $k\in\{1,2\}$. It is convenient to rearrange the equation \eqref{eq:e-equation}
as 
\begin{align*}
 & \td{\Mod}\cdot(\mathbf{v}+(\Lambda\eps,-i\eps,0,0)^{t})\\
 & =\partial_{s}\eps+b\Lambda\eps+\eta i\eps+(iL_{w}^{\ast}w_{1}-iL_{P}^{\ast}P_{1})+i\Psi\\
 & \quad-(\td{\gmm}_{s}-\gmm_{s}+2\eta)iw-c_{b}(b^{2}-\eta^{2})(\partial_{b}P)-2c_{b}b\eta(\partial_{\eta}P).
\end{align*}
Taking the inner product with $\calZ_{k}$ with $k\in\{1,2\}$, we
get 
\begin{align}
 & \sum_{j=1}^{4}\{(v_{j},\calZ_{k})_{r}+O(M^{C}\|\eps\|_{\dot{\calH}_{0}^{3}})\}\td{\mathrm{Mod}}_{j}\nonumber \\
 & =(iL_{w}^{\ast}w_{1}-iL_{P}^{\ast}P_{1},\calZ_{k})_{r}-b(\eps,\Lambda\calZ_{k})_{r}-\eta(\eps,i\calZ_{k})_{r}-(\Psi,i\calZ_{k})_{r}\label{eq:Modulation-e-eqn}\\
 & \quad+(\td{\gmm}_{s}-\gmm_{s}+2\eta)(w,i\calZ_{k})_{r}-c_{b}(b^{2}-\eta^{2})(\partial_{b}P,\calZ_{k})_{r}-2c_{b}b\eta(\partial_{\eta}P,\calZ_{k})_{r}.\nonumber 
\end{align}
We first look at the matrix structure of the LHS of \eqref{eq:Modulation-e-eqn}.
By the transversality computation \eqref{eq:Transversality-Z1234}
and the fact that $\calZ_{k}$ is supported in the region $y\leq2M$,
we have 
\begin{multline}
\{(v_{j},\calZ_{k})_{r}+O(M^{C}\|\eps\|_{\dot{\calH}_{0}^{3}})\}_{1\leq k\leq2,\,1\leq j\leq4}\\
=\begin{pmatrix}-(yQ,yQ\chi_{M})_{r}+O(1) & 0 & 0 & 0\\
0 & -\tfrac{1}{4}(yQ,yQ\chi_{M})_{r}+O(1) & 0 & 0
\end{pmatrix}+O(M^{C}b).\label{eq:Modulation-e-matrix}
\end{multline}
Note that this matrix has logarithmic divergence due to $(yQ,yQ\chi_{M})_{r}\sim\log M$
by \eqref{eq:yQyQ}.

We turn to estimate the RHS of \eqref{eq:Modulation-e-eqn}. We claim
that 
\begin{equation}
|\text{RHS of }\eqref{eq:Modulation-e-eqn}|\aleq b^{2-}.\label{eq:Modulation-e-claim}
\end{equation}
For the first term on the RHS of \eqref{eq:Modulation-e-eqn}, we
have 
\begin{align*}
 & |(iL_{w}^{\ast}w_{1}-iL_{P}^{\ast}P_{1},\calZ_{k})_{r}|\\
 & \aleq|(\eps_{1},L_{P}i\calZ_{k})_{r}|+|(w_{1},(L_{w}-L_{P})i\calZ_{k})_{r}|\\
 & \aleq\|\langle y\rangle^{-2+}\eps_{1}\|_{L^{2}}\|\langle y\rangle^{2-}L_{P}i\calZ_{k}\|_{L^{2}}+\|w_{1}\|_{L^{2}}\|(L_{w}-L_{P})i\calZ_{k}\|_{L^{2}}\\
 & \aleq b^{2-}\|\langle y\rangle^{2-}L_{P}i\calZ_{k}\|_{L^{2}}+b^{1-}\|(L_{w}-L_{P})i\calZ_{k}\|_{L^{2}},
\end{align*}
so it suffices to show 
\begin{align*}
\|\langle y\rangle^{2-}L_{P}i\calZ_{k}\|_{L^{2}} & \aleq M^{C},\\
\|(L_{w}-L_{P})i\calZ_{k}\|_{L^{2}} & \aleq b^{1-}.
\end{align*}
The estimate for $L_{P}i\calZ_{k}$ follows from 
\[
|L_{P}i\calZ_{k}|\aleq M^{C}\brk{y}^{-3+}.
\]
The estimate for $(L_{w}-L_{P})i\calZ_{k}$ follows from 
\[
|(L_{w}-L_{P})i\calZ_{k}|\aleq\tfrac{1}{y}|(A_{\theta}[w]-A_{\theta}[P])\calZ_{k}|+\tfrac{1}{y}|A_{\theta}[\eps,i\calZ_{k}]w|+\tfrac{1}{y}|A_{\theta}[P,i\calZ_{k}]\eps|
\]
and 
\begin{align*}
\|\tfrac{1}{y}(A_{\theta}[w]-A_{\theta}[P])\calZ_{k}\|_{L^{2}}+\|\tfrac{1}{y}A_{\theta}[\eps,i\calZ_{k}]w\|_{L^{2}} & \aleq M^{C}\|\eps\|_{\dot{\calH}_{0}^{3}}\aleq b^{2-},\\
\|\tfrac{1}{y}A_{\theta}[P,i\calZ_{k}]\eps\|_{L^{2}} & \aleq M^{C}\|\eps\|_{\dot{\calH}_{0}^{1}}\aleq b^{1-}.
\end{align*}
The remaining terms on the RHS of \eqref{eq:Modulation-e-eqn} can
be estimated using \eqref{eq:Psi-RoughBound} and \eqref{eq:Gamma-GammaTilde};
we have 
\begin{align*}
|b(\eps,\Lambda\calZ_{k})_{r}|+|\eta(\eps,i\calZ_{k})_{r}| & \aleq bM^{C}\|\eps\|_{\dot{\calH}_{0}^{3}}\aleq b^{3-},\\
|(\Psi,i\calZ_{k})_{r}| & \aleq M^{C}b^{2}|\log b|\aleq b^{2-},\\
|(\td{\gmm}_{s}-\gmm_{s}+2\eta)(w,i\calZ_{k})_{r}| & \aleq M^{C}|\td{\gmm}_{s}-\gmm_{s}+2\eta|\aleq b^{2-},
\end{align*}
and 
\begin{align*}
 & |c_{b}(b^{2}-\eta^{2})(\partial_{b}P,\calZ_{k})_{r}|+|2c_{b}b\eta(\partial_{\eta}P,\calZ_{k})_{r}|\\
 & \aleq b^{2}\big(|(\partial_{b}P,\calZ_{k})_{r}|+|(\partial_{\eta}P,\calZ_{k})_{r}|\big)\aleq b^{2}M^{C}\aleq b^{2-}.
\end{align*}
Therefore, the claim \eqref{eq:Modulation-e-claim} is proved.

Next, in order to derive the modulation estimates for $b$ and $\eta$,
we differentiate the orthogonality conditions $(\eps_{1},\td{\calZ}_{k})_{r}=0$
for $k\in\{3,4\}$. We rearrange the equation \eqref{eq:e1-equation}
as 
\begin{equation}
\begin{aligned} & \td{\Mod}\cdot(\mathbf{v}_{1}+(\Lambda_{-1}\eps_{1},-i\eps_{1},0,0)^{t})\\
 & =\partial_{s}\eps_{1}+iA_{Q}^{\ast}\eps_{2}+b\Lambda_{-1}\eps_{1}-\eta i\eps_{1}+(iA_{w}^{\ast}-iA_{P}^{\ast})w_{2}+(iA_{P}^{\ast}-iA_{Q}^{\ast})\eps_{2}\\
 & \quad-({\textstyle \int_{0}^{y}}\Re(\overline{w}w_{1})dy')i\eps_{1}-({\textstyle \int_{0}^{y}}\Re(\overline{w}w_{1}-\overline{P}P_{1})dy')iP_{1}+i\Psi_{1}.
\end{aligned}
\label{eq:Modulation-e1-prelim}
\end{equation}
Taking the inner product with $\td{\calZ}_{k}$ with $k\in\{3,4\}$,
we get 
\begin{equation}
\begin{aligned} & \sum_{j=1}^{4}\{((\mathbf{v}_{1})_{j},\td{\calZ}_{k})_{r}+O(M^{C}\|\eps_{1}\|_{\dot{\calH}_{1}^{2}})\}\td{\mathrm{Mod}}_{j}=(iA_{Q}^{\ast}\eps_{2},\td{\calZ}_{k})_{r}\\
 & \quad+(b\Lambda_{-1}\eps_{1}-\eta i\eps_{1},\td{\calZ}_{k})_{r}+((iA_{w}^{\ast}-iA_{P}^{\ast})w_{2}+(iA_{P}^{\ast}-iA_{Q}^{\ast})\eps_{2},\td{\calZ}_{k})_{r}\\
 & \quad-(({\textstyle \int_{0}^{y}}\Re(\overline{w}w_{1})dy')i\eps_{1}+({\textstyle \int_{0}^{y}}\Re(\overline{w}w_{1}-\overline{P}P_{1}))dy')iP_{1},\td{\calZ}_{k})_{r}+(i\Psi_{1},\td{\calZ}_{k})_{r}.
\end{aligned}
\label{eq:Modulation-e1-eqn}
\end{equation}
We first look at the matrix structure of the LHS of \eqref{eq:Modulation-e1-eqn}.
By the structure of $\mathbf{v}_{1}$ \eqref{eq:v1-estimate} (in
particular the \emph{degeneracy} $\Lambda_{-1}P_{1}=O(b)=iP_{1}$),
the transversality computation \eqref{eq:Transversality-Z-tilde-34},
and the fact that $\td{\calZ}_{k}$ is supported in $(0,2M]$, we
obtain 
\begin{multline}
\{((\mathbf{v}_{1})_{j},\td{\calZ}_{k})_{r}+O(M^{C}\|\eps_{1}\|_{\dot{\calH}_{1}^{2}})\}_{3\leq k\leq4,\,1\leq j\leq4}\\
=\begin{pmatrix}0 & 0 & (\tfrac{1}{2}yQ,yQ\chi_{M})_{r} & 0\\
0 & 0 & 0 & (\tfrac{1}{2}yQ,yQ\chi_{M})_{r}
\end{pmatrix}+O(M^{C}b).\label{eq:Modulation-e1-matrix}
\end{multline}
As before, $(\tfrac{1}{2}yQ,yQ\chi_{M})_{r}\sim\log M$.

We turn to estimate the RHS of \eqref{eq:Modulation-e1-eqn}. We claim
that 
\begin{align}
|\text{RHS of }\eqref{eq:Modulation-e1-eqn}| & \aleq\sqrt{\log M}\|\eps_{3}\|_{L^{2}}+b^{3-},\label{eq:Modulation-e1-claim}
\end{align}
For the first term, since $A_{Q}^{\ast}\eps_{2}=\eps_{3}$, we estimate
as\footnote{The way of estimating this contribution is quite different from the
$m\geq1$ case. When $m\geq1$, the inner product matrix \eqref{eq:Modulation-e1-matrix}
has no logarithmic divergence in $M$. Instead, the smallness factor
in $M$ of \eqref{eq:ModEstimate-b-eta} comes from $A_{Q}i\td{\calZ}_{k}\approx0$
for $k\in\{3,4\}$ and $\|\frac{1}{y}\eps_{2}\|_{L^{2}}\sim\|\eps_{3}\|_{L^{2}}$:
\[
(\eps_{3},i\td{\calZ}_{k})_{r}=(\eps_{2},A_{Q}i\td{\calZ}_{k})_{r}\aleq\|\tfrac{1}{y}\eps_{2}\|_{L^{2}}\|yA_{Q}i\td{\calZ}_{k}\|_{L^{2}}\aleq M^{-1}\|\eps_{3}\|_{L^{2}}.
\]
When $m=0$, the smallness factor in $M$ of \eqref{eq:ModEstimate-b-eta}
simply comes from $(\log M)^{\frac{1}{2}}/(\log M)$, where $(\log M)^{\frac{1}{2}}$
and $\log M$ come from $\|\td{\calZ}_{k}\|_{L^{2}}\sim(\log M)^{\frac{1}{2}}$
and the inner product matrix \eqref{eq:Modulation-e1-matrix}, respectively.} 
\[
|(iA_{Q}^{\ast}\eps_{2},\td{\calZ}_{k})_{r}|\aleq\|\eps_{3}\|_{L^{2}}\|\td{\calZ}_{k}\|_{L^{2}}\aleq\sqrt{\log M}\|\eps_{3}\|_{L^{2}}.
\]
For the remaining terms, we claim the following weighted $L^{2}$-estimates
(this is also for a later use in the Morawetz correction; see the
proof of Lemma~\ref{lem:MorawetzCorrection}): 
\begin{gather}
\||\eps_{1}|_{1}\|_{X}\aleq b^{2-},\label{eq:virial-correction-temp1}\\
\|(A_{w}^{\ast}-A_{P}^{\ast})w_{2}\|_{X}+\|(A_{P}^{\ast}-A_{Q}^{\ast})\eps_{2}\|_{X}\aleq b^{3},\label{eq:virial-correction-temp2}\\
\|({\textstyle \int_{0}^{y}}\Re(\overline{w}w_{1})dy')\eps_{1}\|_{X}+\|({\textstyle \int_{0}^{y}}\Re(\overline{w}w_{1}-\overline{P}P_{1})dy')P_{1}\|_{X}\aleq b^{3-},\label{eq:virial-correction-temp3}\\
\|\Psi_{1}\|_{X}\aleq b^{3-}.\label{eq:virial-correction-temp4}
\end{gather}
Here, we recall from \eqref{eq:Def-Xnorm} that the $X$-norm is given
by $\|f\|_{X}=\|\langle y\rangle^{-2}\langle\log_{+}y\rangle f\|_{L^{2}}$.
We note that \eqref{eq:Modulation-e1-claim} follows from combining
\eqref{eq:virial-correction-temp1}-\eqref{eq:virial-correction-temp3}
with $\|\langle y\rangle^{2}\langle\log_{+}y\rangle^{-1}|\td{\calZ}_{k}|_{1}\|_{L^{2}}\aleq M^{C}$.
Henceforth, we focus on proving \eqref{eq:virial-correction-temp1}-\eqref{eq:virial-correction-temp3}.

The estimate \eqref{eq:virial-correction-temp1} follows from 
\[
\||\eps_{1}|_{1}\|_{X}=\|\langle y\rangle^{-2}\langle\log_{+}y\rangle|\eps_{1}|_{1}\|_{L^{2}}\aleq\|\eps_{1}\|_{\dot{\calH}_{1}^{2}}^{1-}\|\eps_{1}\|_{L^{2}}^{0+}\aleq b^{2-}.
\]
For \eqref{eq:virial-correction-temp2}, since $\|\eps_{2}\|_{\dot{\calH}_{2}^{1}}+\|P_{2}\|_{\dot{\calH}_{2}^{1}}\aleq b^{2}$,
it suffices to show 
\[
\|y\langle y\rangle^{-2}\langle\log_{+}y\rangle^{2}(|A_{w}^{\ast}-A_{P}^{\ast}|+|A_{P}^{\ast}-A_{Q}^{\ast}|)\|_{L^{\infty}}\aleq b.
\]
The estimate for $A_{w}^{\ast}-A_{P}^{\ast}$ follows from the observation
that $A_{w}^{\ast}-A_{P}^{\ast}$ is a linear combination of $\tfrac{1}{y}A_{\theta}[\psi_{1},\psi_{2}]$,
where $\psi_{1}\in\{P,\eps\}$ and $\psi_{2}=\eps$ and the estimate
\[
\|\langle y\rangle^{-2}\langle\log_{+}y\rangle^{2}A_{\theta}[\psi_{1},\psi_{2}]\|_{L^{\infty}}\aleq\|\psi_{1}\|_{L^{2}}\|\langle y\rangle^{-2+}\eps\|_{L^{2}}\aleq\nrm{\eps}_{\dot{\calH}_{0}^{3}}^{\frac{2}{3}-}\nrm{\eps}_{L^{2}}^{\frac{1}{3}+}\aleq b^{\frac{4}{3}-}.
\]
The estimate for $A_{P}^{\ast}-A_{Q}^{\ast}=-\frac{1}{y}(A_{\theta}[P]-A_{\theta}[Q])$
follows from \eqref{eq:AthtP-AthtQ}. The estimate \eqref{eq:virial-correction-temp3}
follows from 
\begin{align*}
\|({\textstyle \int_{0}^{y}}\Re(\overline{w}w_{1})dy')\eps_{1}\|_{X} & \aleq\|{\textstyle \int_{0}^{y}}\Re(\overline{w}w_{1})dy'\|_{L^{\infty}}\|\eps_{1}\|_{X}\\
 & \aleq\|w\|_{L^{2}}\|y^{-1}w_{1}\|_{L^{2}}\|\eps_{1}\|_{\dot{\calH}_{1}^{2}}^{1-}\|\eps_{1}\|_{L^{2}}^{0+}\aleq b^{3-}
\end{align*}
and 
\begin{align*}
 & \|({\textstyle \int_{0}^{y}}\Re(\overline{w}w_{1}-\overline{P}P_{1})dy')P_{1}\|_{X}\\
 & \aleq\|\langle\log_{+}y\rangle P_{1}\|_{L^{2}}\|y^{-1}\langle y\rangle^{-2}(|P\eps_{1}|+|\eps w_{1}|)\|_{L^{1}}\\
 & \aleq b^{1-}(\|P\|_{L^{2}}\|y^{-1}\langle y\rangle^{-2}\eps_{1}\|_{L^{2}}+\|\langle y\rangle^{-3}\eps\|_{L^{2}}\|y^{-1}\langle y\rangle w_{1}\|_{L^{2}})\aleq b^{3-}.
\end{align*}
Finally, the claim \eqref{eq:virial-correction-temp4} is proved in
\eqref{eq:Psi1-LocalEnergy}. Thus the claims \eqref{eq:virial-correction-temp1}-\eqref{eq:virial-correction-temp4}
and hence \eqref{eq:Modulation-e1-claim} are proved.

To complete the proof, we use the structures of the matrices \eqref{eq:Modulation-e-matrix}
and \eqref{eq:Modulation-e1-matrix}, and the logarithmic divergence
\eqref{eq:yQyQ} to find 
\begin{align*}
\Big|\frac{\lmb_{s}}{\lmb}+b\Big|+|\td{\gmm}_{s}+\eta| & \aleq\frac{1}{\log M}\eqref{eq:Modulation-e-claim}+M^{C}b\eqref{eq:Modulation-e1-claim},\\
|b_{s}+b^{2}+\eta^{2}+c_{b}(b^{2}-\eta^{2})|+|\eta_{s}+2c_{b}b\eta| & \aleq\frac{1}{\log M}\eqref{eq:Modulation-e1-claim}+M^{C}b\eqref{eq:Modulation-e-claim}.
\end{align*}
Substituting the claims finishes the proof.
\end{proof}
The estimates \eqref{eq:ModEstimateScalePhase} and \eqref{eq:ModEstimate-b-eta}
suffice to close our bootstrap procedure and derive finite-time blow-up.
However, these do not suffice to derive the sharp blow-up rates. Substituting
$\|\eps_{3}\|_{L^{2}}\leq K\frac{b^{2}}{|\log b|}$, the estimate
\eqref{eq:ModEstimate-b-eta} would only yield 
\[
|b_{s}+b^{2}+\tfrac{2b^{2}}{|\log b|}|\aleq\tfrac{K}{\sqrt{\log M}}\tfrac{b^{2}}{|\log b|},
\]
which would not be enough to determine the precise coefficient of
$\frac{b^{2}}{|\log b|}$. The sharp blow-up rate \emph{depends} on
the coefficient of $\frac{b^{2}}{|\log b|}$.

To overcome this issue, we note that the estimates are saturated by
the contribution of $(iA_{Q}^{\ast}\eps_{2},\td{\calZ}_{k})_{r}$.
To make this term smaller, we test \eqref{eq:Modulation-e1-prelim}
against better approximations of the kernel elements $yQ$, $iyQ$
of $A_{Q}$ instead of $\td{\calZ}_{k}$ ($k=3,4$). With this correction,
we improve the bound $\tfrac{1}{\sqrt{\log M}}\|\eps_{3}\|_{L^{2}}$
of \eqref{eq:ModEstimate-b-eta} by a logarithmic factor $\sqrt{|\log b|}$.
From this, we can see that the sharp coefficient of $\frac{b^{2}}{|\log b|}$
is $2$. The same argument was previously used in \cite{MerleRaphaelRodnianski2013InventMath}.

For a small universal constant $\delta>0$ (e.g., $\delta=\frac{1}{100C}$
for the $C$'s used in $M^{C}$ bounds), we introduce 
\[
B_{\delta}\coloneqq b^{-\delta},\quad\td{\calZ}_{3,\delta}=yQ\chi_{B_{\delta}},\quad\td{\calZ}_{4,\delta}=iyQ\chi_{B_{\delta}}.
\]
The refined modulation estimates will be derived from differentiating
$(\eps_{1},\td{\calZ}_{k,\delta})_{r}$. We remark that we do not
use $(\eps_{1},\td{\calZ}_{k,\delta})_{r}=0$ as orthogonality conditions
from the beginning. If $(\eps_{1},\td{\calZ}_{k,\delta})_{r}=0$ were
used, then the implicit constants of the coercivity relations would
depend on $b$ and create serious complications.
\begin{lem}[Refined modulation estimates for $b$ and $\eta$]
\label{lem:RefinedModulationEstimates}Define 
\[
\td b\coloneqq b-\frac{(\eps_{1},\td{\calZ}_{3,\delta})_{r}}{(\frac{1}{2}yQ,yQ\chi_{B_{\delta}})_{r}}\quad\text{and}\quad\td{\eta}\coloneqq\eta-\frac{(\eps_{1},\td{\calZ}_{4,\delta})_{r}}{(\tfrac{1}{2}yQ,yQ\chi_{B_{\delta}})_{r}}.
\]
Then, 
\begin{align}
|\td b-b|+|\td{\eta}-\eta| & \aleq b^{2-C\delta},\label{eq:btilde-b}\\
|\td b_{s}+b^{2}+\eta^{2}+c_{b}(b^{2}-\eta^{2})|+|\td{\eta}_{s}+2c_{b}b\eta| & \aleq\tfrac{1}{\sqrt{|\log b|}}\|\eps_{3}\|_{L^{2}}+b^{3-C\delta}.\label{eq:RefinedModulation}
\end{align}
In particular, 
\begin{equation}
|\td b_{s}+\td b^{2}+\tfrac{2\td b^{2}}{|\log b|}|+|\td{\eta}_{s}|\aleq\tfrac{b^{2}}{|\log b|^{\frac{3}{2}-}}.\label{eq:btilde-modulation}
\end{equation}
\end{lem}

\begin{proof}
In the following, we will compute $\partial_{s}(\eps_{1},\td{\calZ}_{k,\delta})_{r}$.
We take the inner product of \eqref{eq:Modulation-e1-prelim} and
$\td{\calZ}_{k,\delta}$ to obtain a variant of \eqref{eq:Modulation-e1-eqn}:
\begin{equation}
\begin{aligned} & \sum_{j=1}^{4}\{((\mathbf{v}_{1})_{j},\td{\calZ}_{k,\delta})_{r}+O(b^{-C\delta}\|\eps_{1}\|_{\dot{\calH}_{1}^{2}})\}\td{\mathrm{Mod}}_{j}=(\partial_{s}\eps_{1},\td{\calZ}_{k,\delta})_{r}+(iA_{Q}^{\ast}\eps_{2},\td{\calZ}_{k,\delta})_{r}\\
 & \quad+(b\Lambda_{-1}\eps_{1}-\eta i\eps_{1},\td{\calZ}_{k,\delta})_{r}+((iA_{w}^{\ast}-iA_{P}^{\ast})w_{2}+(iA_{P}^{\ast}-iA_{Q}^{\ast})\eps_{2},\td{\calZ}_{k,\delta})_{r}\\
 & \quad-(({\textstyle \int_{0}^{y}}\Re(\overline{w}w_{1})dy')i\eps_{1}+({\textstyle \int_{0}^{y}}\Re(\overline{w}w_{1}-\overline{P}P_{1}))dy')iP_{1},\td{\calZ}_{k,\delta})_{r}+(i\Psi_{1},\td{\calZ}_{k,\delta})_{r}.
\end{aligned}
\label{eq:RefinedModulation-eqn}
\end{equation}
We remark that there is an additional term $(\partial_{s}\eps_{1},\td{\calZ}_{k,\delta})_{r}$
on the RHS of \eqref{eq:RefinedModulation-eqn}. The matrix on the
LHS of \eqref{eq:RefinedModulation-eqn} satisfies (c.f.~\eqref{eq:Modulation-e1-matrix})
\begin{equation}
\begin{aligned} & \{((\mathbf{v}_{1})_{j},\td{\calZ}_{k,\delta})_{r}+O(b^{-C\delta}\|\eps_{1}\|_{\dot{\calH}_{1}^{2}})\}_{3\leq k\leq4,\,1\leq j\leq4}\\
 & =\begin{pmatrix}0 & 0 & (\tfrac{1}{2}yQ,yQ\chi_{B_{\delta}})_{r} & 0\\
0 & 0 & 0 & (\tfrac{1}{2}yQ,yQ\chi_{B_{\delta}})_{r}
\end{pmatrix}+O(b^{1-C\delta}).
\end{aligned}
\label{eq:RefinedModulation-matrix}
\end{equation}
For the terms on the RHS of \eqref{eq:RefinedModulation-eqn}, estimates
are very similar to those in Lemma~\ref{lem:ModulationEstimates}
with replacing $M$ by $B_{\delta}$. We use $\|\td{\calZ}_{k,\delta}\|_{L^{2}}\aleq\sqrt{\log B_{\delta}}$
and $\|\langle y\rangle^{2}\langle\log_{+}y\rangle^{-1}|\td{\calZ}_{k,\delta}|_{1}\|_{L^{2}}\aleq b^{-C\delta}$,
and follow the proof of \eqref{eq:Modulation-e1-claim} to obtain
\begin{equation}
\text{RHS of }\eqref{eq:RefinedModulation-eqn}=(\partial_{s}\eps_{1},\td{\calZ}_{k,\delta})_{r}+O(\sqrt{\log B_{\delta}}\|\eps_{3}\|_{L^{2}}+b^{3-C\delta}).\label{eq:RefinedModulation-temp}
\end{equation}
Summing up \eqref{eq:RefinedModulation-matrix} and \eqref{eq:RefinedModulation-temp},
and then applying the previous modulation estimates (Lemma~\ref{lem:ModulationEstimates})
to treat the term $O(b^{1-C\dlt}\td{\Mod})$, we arrive at 
\begin{equation}
\td{\mathrm{Mod}}_{k}=\frac{(\partial_{s}\eps_{1},\td{\calZ}_{k,\delta})_{r}}{(\frac{1}{2}yQ,yQ\chi_{B_{\delta}})_{r}}+O\Big(\frac{1}{\sqrt{\log B_{\delta}}}\|\eps_{3}\|_{L^{2}}+b^{3-C\delta}\Big).\label{eq:RefinedModulation-temp2}
\end{equation}
for $k\in\{3,4\}$.

We now differentiate $(\partial_{s}\eps_{1},\td{\calZ}_{k,\delta})_{r}$
by parts: 
\begin{align*}
 & \bigg|\frac{(\partial_{s}\eps_{1},\td{\calZ}_{k,\delta})_{r}}{(\frac{1}{2}yQ,yQ\chi_{B_{\delta}})_{r}}-\partial_{s}\bigg(\frac{(\eps_{1},\td{\calZ}_{k,\delta})_{r}}{(\frac{1}{2}yQ,yQ\chi_{B_{\delta}})_{r}}\bigg)\bigg|\\
 & \aleq\frac{|(\eps_{1},\partial_{s}\td{\calZ}_{k,\delta})_{r}|}{(\frac{1}{2}yQ,yQ\chi_{B_{\delta}})_{r}}+\frac{|(\eps_{1},\td{\calZ}_{k,\delta})_{r}(\frac{1}{2}yQ,yQ\partial_{s}\chi_{B_{\delta}})_{r}|}{(\frac{1}{2}yQ,yQ\chi_{B_{\delta}})_{r}^{2}}.
\end{align*}
Using $|b_{s}\partial_{b}\chi_{B_{\delta}}|\aleq b\cdot|b\partial_{b}\chi_{B_{\delta}}|\aleq b\chf_{[B_{\delta},2B_{\delta}]}$
and the $\dot{\calH}_{1}^{2}$-bound of $\eps_{1}$ from Lemma~\ref{lem:NonlinearCoercivity},
we obtain 
\begin{equation}
\bigg|\frac{(\partial_{s}\eps_{1},\td{\calZ}_{k,\delta})_{r}}{(\frac{1}{2}yQ,yQ\chi_{B_{\delta}})_{r}}-\partial_{s}\bigg(\frac{(\eps_{1},\td{\calZ}_{k,\delta})_{r}}{(\frac{1}{2}yQ,yQ\chi_{B_{\delta}})_{r}}\bigg)\bigg|\aleq b^{3-C\delta}.\label{eq:RefinedModulation-temp3}
\end{equation}

The definitions of $\td b$ and $\td{\eta}$ are motivated in view
of \eqref{eq:RefinedModulation-temp2} and \eqref{eq:RefinedModulation-temp3},
and the estimates \eqref{eq:btilde-b}, \eqref{eq:RefinedModulation},
and \eqref{eq:btilde-modulation} are immediate. 
\end{proof}

\subsection{\label{subsec:Energy-estimate}Energy estimate in $\dot{\protect\calH}_{0}^{3}$}

In this subsection, we propagate the control of $\eps$ forward-in-time.
The main idea is the energy method in higher derivatives with repulsivity.
More precisely, we proceed to higher order derivatives by adapted
derivatives, say $\eps_{k}$. We then apply the energy method with
correction terms. The correction terms are designed to exploit the
repulsivity observed in the variable $\eps_{2}$. Such an idea appeared
in \cite{RodnianskiSterbenz2010Ann.Math.,RaphaelRodnianski2012Publ.Math.,MerleRaphaelRodnianski2013InventMath}
in the context of wave maps and Schrödinger maps.

We will apply the energy method to $\eps_{2}$ with the energy functional
$\|A_{Q}^{\ast}\eps_{2}\|_{L^{2}}^{2}=\|\eps_{3}\|_{L^{2}}^{2}$.
Indeed, we need to work at least in the $\dot{H}^{3}$-level due to
scaling reasons. More precisely, as we are in the situation $\lmb\sim b|\log b|^{2}$
(which is dictated by the formal parameter law \eqref{eq:FormalParameterLaw}),
we can expect at best $\|\eps_{k}\|_{L^{2}}\aleq\lmb^{k}\sim b^{k}|\log b|^{2k}$.
In order to guarantee the modulation equation $b_{s}+b^{2}+\tfrac{2b^{2}}{|\log b|}\approx0$,
we need $k>2$ in view of Lemma~\ref{lem:ModulationEstimates}. On
the other hand, when $k=3$, a toy model 
\[
(\partial_{s}-\tfrac{\lmb_{s}}{\lmb}\Lambda_{-3})\eps_{3}\approx-iA_{Q}^{\ast}\Psi_{2}
\]
implies 
\[
(\partial_{s}-3\tfrac{\lmb_{s}}{\lmb})\|\eps_{3}\|_{L^{2}}\aleq\|A_{Q}^{\ast}\Psi_{2}\|_{L^{2}}\aleq\tfrac{b^{3}}{|\log b|}
\]
by \eqref{eq:Psi2-SharpEnergy}. Integrating this loses $b$, which
yields 
\[
\|\eps_{3}\|_{L^{2}}\aleq\tfrac{b^{2}}{|\log b|}.
\]
In view of Lemma~\ref{lem:ModulationEstimates}, this bound suffices
to guarantee the modulation equation $b_{s}+b^{2}+\frac{2b^{2}}{|\log b|}\approx0$.
Moreover, this motivates the bootstrap hypothesis for $\|\eps_{3}\|_{L^{2}}$.

In the energy estimate, there appear two non-perturbative contributions
in $\tfrac{1}{2}(\partial_{s}-6\tfrac{\lmb_{s}}{\lmb})\|\eps_{3}\|_{L^{2}}^{2}$.
One is from the commutator of the scaling operator $\Lmb_{-2}$ and
$A_{Q}^{\ast}$ acting on $\eps_{2}$. In the energy estimate, we
will see that this contribution has the \emph{good} (negative) sign,
thanks to the \emph{repulsivity} \eqref{eq:Def-Vtilde} of the operator
$A_{Q}A_{Q}^{\ast}$, i.e. $-\partial_{\lmb}(A_{Q_{\lmb}}A_{Q_{\lmb}}^{\ast})=\frac{y\partial_{y}\td V}{y^{2}}\leq0$
where $Q_{\lmb}=\lmb^{-1}Q(\lmb^{-1}\cdot)$. Another non-perturbative
contribution comes from the cubic nonlinearity. This will be treated
by both a Morawetz correction and the above repulsivity.

We start by deriving the equation for $\eps_{2}$. Recall \eqref{eq:w2-eqn-sd}
and \eqref{eq:P2-equation}: 
\begin{align*}
(\rd_{s}-\frac{\lmb_{s}}{\lmb}\Lmb_{-2}+\td{\gmm}_{s}i)w_{2}+iA_{w}A_{w}^{\ast}w_{2}-\left(\int_{0}^{y}\Re(\br ww_{1})dy'\right)iw_{2} & -i\br ww_{1}^{2}=0,\\
(\partial_{s}-\frac{\lmb_{s}}{\lmb}\Lambda_{-2}+\td{\gmm}_{s}i)P_{2}+iA_{P}A_{P}^{\ast}P_{2}-\Big(\int_{0}^{y}\Re(\overline{P}P_{1})dy'\Big)iP_{2} & -i\overline{P}(P_{1})^{2}\\
 & =-\td{\Mod}\cdot\mathbf{v}_{2}+i\Psi_{2}.
\end{align*}
Subtracting the second from the first and using the identity 
\begin{align*}
 & iA_{w}A_{w}^{\ast}w_{2}-iA_{P}A_{P}^{\ast}P_{2}\\
 & =iA_{Q}A_{Q}^{\ast}\eps_{2}+(iA_{w}A_{w}^{\ast}-iA_{P}A_{P}^{\ast})w_{2}+(iA_{P}A_{P}^{\ast}-iA_{Q}A_{Q}^{\ast})\eps_{2},
\end{align*}
we obtain the equation for $\eps_{2}$: 
\begin{align}
 & (\partial_{s}-\frac{\lmb_{s}}{\lmb}\Lambda_{-2}+\td{\gmm}_{s}i)\eps_{2}+iA_{Q}A_{Q}^{\ast}\eps_{2}-(i\overline{w}w_{1}^{2}-i\overline{P}P_{1}^{2})\nonumber \\
 & =-(iA_{w}A_{w}^{\ast}-iA_{P}A_{P}^{\ast})w_{2}-(iA_{P}A_{P}^{\ast}-iA_{Q}A_{Q}^{\ast})\eps_{2}\label{eq:e2-equation}\\
 & \quad+({\textstyle \int_{0}^{y}}\Re(\overline{w}w_{1}-\overline{P}P_{1}))dy')iw_{2}+({\textstyle \int_{0}^{y}}\Re(\overline{P}P_{1})dy')i\eps_{2}+\td{\Mod}\cdot\mathbf{v}_{2}-i\Psi_{2}.\nonumber 
\end{align}
Here we wrote the cubic difference term $-(i\overline{w}w_{1}^{2}-i\overline{P}P_{1}^{2})$
on the LHS, because it is a non-perturbative term. This term will
be handled using a Morawetz correction. All the terms on the RHS are
perturbative. 
\begin{lem}[Energy identity of $\eps_{3}$]
We have 
\begin{equation}
\begin{aligned}\tfrac{1}{2}(\partial_{s}-6\tfrac{\lmb_{s}}{\lmb})\|\eps_{3}\|_{L^{2}}^{2} & =b(\eps_{3},\tfrac{y}{2}Q^{2}\eps_{2}+A_{Q}^{\ast}(yQ^{2}\eps_{1}))_{r}\\
 & \quad+b\|\eps_{3}\|_{L^{2}}\cdot O(\tfrac{1}{\sqrt{\log M}}\|\eps_{3}\|_{L^{2}}+\tfrac{b^{2}}{|\log b|}).
\end{aligned}
\label{eq:EnergyIdentity}
\end{equation}
\end{lem}

\begin{rem}
We remind the reader that the relations between $\eps$, $\eps_{1}$,
and $\eps_{2}$ are highly nonlinear. If one were to proceed to higher
order derivatives in a linear fashion, e.g. $\eps_{2}=A_{Q}L_{Q}\eps$,
then one would encounter a lot of non-perturbative errors $O(b\eps^{2})$
in the energy identity. Such errors would contain nonlocal expressions
from $A_{\theta}$ or $A_{t}$, thus it would be very difficult to
find correction terms. However, as we proceed with nonlinear adapted
derivatives, we are able to take advantage from the degeneracies $P_{1}=O(b)$
and $P_{2}=O(b^{2})$ to simplify the non-perturbative terms significantly.
In this sense, we believe that using nonlinear adapted derivatives
is more efficient and describes the blow-up regime more precisely
than using the linear ones. 
\end{rem}

\begin{rem}
\label{rem:energy-estimate-weak-repul}When $m\geq1$, the situation
is simpler than here. In that case, the authors in \cite{KimKwon2020arXiv}
were able to close the argument using linear adapted derivatives.
This is mainly due to the \emph{stronger repulsivity} of $A_{Q}A_{Q}^{\ast}$
and \emph{better decay} of $Q$. The stronger repulsivity enables
(a localized version of) the monotonicity from the virial functional
$(\eps_{2},i\Lambda\eps_{2})_{r}$, see \cite[(2.8) and (5.36)]{KimKwon2020arXiv}.
Moreover, thanks to the better decay of $Q$, many nonlocal contributions
of size $O(b\eps^{2})$ can in fact be estimated by some local norms
of $\eps$. See \cite[(5.33) and Lemma 5.1]{KimKwon2020arXiv}.

In contrast, the case $m=0$ has serious problems from the slower
decay of $Q$ and weaker repulsivity of $A_{Q}A_{Q}^{\ast}$. In fact,
$A_{Q}A_{Q}^{\ast}\approx-\Delta_{0}$ near the spatial infinity,
as the potential $\frac{\tilde{V}}{y^{2}}$ decays faster than $\frac{1}{y^{2}}$.
Thus the argument using a localized virial functional as in \cite{KimKwon2020arXiv}
meets a serious difficulty from the fact that $-\Delta_{0}$ (on 2D)
has zero resonance. Thus we do not rely on the virial functional in
this paper, but rather construct a precise correction term to handle
non-perturbative terms. To find such corrections, it is also crucial
to proceed with nonlinear adapted derivatives, to simplify the structure
of non-perturbative terms significantly. 
\end{rem}

\begin{proof}
As before, in this proof, we freely use the bootstrap hypotheses \eqref{eq:BootstrapHypothesis},
as well as Lemmas~\ref{lem:NonlinearCoercivity}, \ref{lem:interpolation},
and \ref{lem:Weighted-Linfty} to estimate $\eps$, $\eps_{1}$, and
$\eps_{2}$. We also abuse the notation and identify the operator
$A_{w}A_{w}^{\ast}-A_{w'}A_{w'}^{\ast}$, which is simply the multiplication
by a function (namely, the difference of the zeroth order terms),
with that function.

The equation for $\eps_{3}=A_{Q}^{\ast}\eps_{2}$ is given as 
\begin{align*}
 & (\partial_{s}-\tfrac{\lmb_{s}}{\lmb}\Lambda_{-3}+\td{\gmm}_{s}i)\eps_{3}+iA_{Q}^{\ast}A_{Q}\eps_{3}\\
 & =\tfrac{\lmb_{s}}{\lmb}(\partial_{\lmb}A_{Q_{\lmb}}^{\ast})\eps_{2}+A_{Q}^{\ast}(i\overline{w}w_{1}^{2}-i\overline{P}P_{1}^{2})+A_{Q}^{\ast}(\text{RHS of }\eqref{eq:e2-equation}).
\end{align*}
As opposed to $\eps_{1}$ or $\eps_{2}$, we take a linear adapted
derivative to get $\eps_{3}$. Taking the inner product with $\eps_{3}$,
we have the energy identity 
\begin{align*}
\tfrac{1}{2}(\partial_{s}-6\tfrac{\lmb_{s}}{\lmb})\|\eps_{3}\|_{L^{2}}^{2} & =\tfrac{\lmb_{s}}{\lmb}(\eps_{3},(\partial_{\lmb}A_{Q_{\lmb}}^{\ast})\eps_{2})_{r}+(\eps_{3},A_{Q}^{\ast}(i\overline{w}w_{1}^{2}-i\overline{P}P_{1}^{2}))_{r}\\
 & \quad+\|\eps_{3}\|_{L^{2}}\cdot O(\|\text{RHS of }\eqref{eq:e2-equation}\|_{\dot{\calH}_{2}^{1}}).
\end{align*}

The first and second terms have non-perturbative contributions. For
the first term, using $\partial_{\lmb}A_{Q_{\lmb}}^{\ast}=-\tfrac{y}{2}Q^{2}$
and the modulation estimate, 
\[
\tfrac{\lmb_{s}}{\lmb}(\eps_{3},(\partial_{\lmb}A_{Q_{\lmb}}^{\ast})\eps_{2})_{r}=b(\eps_{3},\tfrac{y}{2}Q^{2}\eps_{2})_{r}+O(b^{2-})\|\eps_{3}\|_{L^{2}}^{2}.
\]
For the second term, we first write 
\[
i\overline{w}w_{1}^{2}-i\overline{P}P_{1}^{2}=byQ^{2}\eps_{1}+(2i\overline{P}P_{1}-byQ^{2})\eps_{1}+i\overline{P}\eps_{1}^{2}+i\overline{\eps}w_{1}^{2}.
\]
We keep $byQ^{2}\eps_{1}$ and estimate the rest: (we use the weighted
$L^{\infty}$-bounds from Lemma~\ref{lem:Weighted-Linfty} for $L^{\infty}$
terms) 
\begin{align*}
 & \|(2i\overline{P}P_{1}-byQ^{2})\eps_{1}\|_{\dot{\calH}_{2}^{1}}\aleq\|(|\eta|yQ^{2}+b^{2}y^{3}Q^{2})|\eps_{1}|_{-1}\|_{L^{2}}\\
 & \phantom{\|(2i\overline{P}P_{1}-byQ^{2})\eps_{1}\|_{\dot{\calH}_{2}^{1}}}\aleq\tfrac{b}{|\log b|}\|\eps_{1}\|_{\dot{\calH}_{1}^{2}}\aleq_{M}K\tfrac{b^{3}}{|\log b|^{2}}\aleq\tfrac{b^{3}}{|\log b|},\\
 & \|\overline{P}\eps_{1}^{2}\|_{\dot{\calH}_{2}^{1}}\aleq\|\langle y\rangle^{-2+}\eps_{1}|\eps_{1}|_{-1}\|_{L^{2}}\aleq\|\langle y\rangle^{-1+}\eps_{1}\|_{L^{\infty}}\|\langle y\rangle^{-1-}|\eps_{1}|_{-1}\|_{L^{2}}\aleq b^{4-},\\
 & \|\overline{\eps}(w_{1}^{2}-\eps_{1}^{2})\|_{\dot{\calH}_{2}^{1}}\aleq\||P_{1}|_{1}|\overline{\eps}(w_{1}+\eps_{1})|_{-1}\|_{L^{2}}\aleq b^{1-}\|\langle y\rangle^{-1-}|\overline{\eps}(P_{1}+2\eps_{1})|_{-1}\|_{L^{2}}\\
 & \phantom{\|\overline{\eps}(w_{1}^{2}-\eps_{1}^{2})\|_{\dot{\calH}_{2}^{1}}}\aleq b^{1-}(\nrm{\brk{y}^{-2-}\rd_{y}\eps}_{L^{2}}\nrm{\brk{y}P_{1}}_{L^{\infty}}+\nrm{\brk{y}^{-2-}\eps}_{L^{\infty}}\nrm{\brk{y}\abs{P_{1}}_{-1}}_{L^{2}})\\
 & \phantom{\|\overline{\eps}(w_{1}^{2}-\eps_{1}^{2})\|_{\dot{\calH}_{2}^{1}}\aleq}+b^{1-}(\nrm{\rd_{y}\eps}_{L^{2}}\nrm{\brk{y}^{-1-}\eps_{1}}_{L^{\infty}}+\nrm{\eps}_{L^{\infty}}\nrm{\brk{y}^{-1-}\abs{\eps_{1}}_{-1}}_{L^{2}})\aleq b^{4-},\\
 & \|\overline{\eps}\eps_{1}^{2}\|_{\dot{\calH}_{2}^{1}}\aleq\|\partial_{y}\eps\|_{L^{2}}\|\eps_{1}\|_{L^{\infty}}^{2}\\
 & \phantom{\|\overline{\eps}\eps_{1}^{2}\|_{\dot{\calH}_{2}^{1}}\aleq}+\|\eps\|_{L^{\infty}}\|\partial_{y}\eps_{1}\|_{L^{2}}\|\eps_{1}\|_{L^{\infty}}+\|\eps\|_{\dot{\calH}_{0}^{1}}\|\langle\log_{-}y\rangle|\eps_{1}|^{2}\|_{L^{\infty}}\aleq b^{4-}.
\end{align*}
Therefore, we have 
\[
(\eps_{3},A_{Q}^{\ast}(i\overline{w}w_{1}^{2}-i\overline{P}P_{1}^{2}))_{r}=b(\eps_{3},A_{Q}^{\ast}(yQ^{2}\eps_{1}))_{r}+b\|\eps_{3}\|_{L^{2}}\cdot O(\tfrac{b^{2}}{|\log b|}).
\]

The remaining terms are all treated as errors; we claim 
\begin{equation}
\|\text{RHS of }\eqref{eq:e2-equation}\|_{\dot{\calH}_{2}^{1}}\aleq\tfrac{b}{\sqrt{\log M}}\|\eps_{3}\|_{L^{2}}+\tfrac{b^{3}}{|\log b|}.\label{eq:energy-estimate-claim}
\end{equation}
In fact, we will see that $\tfrac{b}{\sqrt{\log M}}\|\eps_{3}\|_{L^{2}}$
is saturated by the modulation term and $\frac{b^{3}}{|\log b|}$
is saturated by $\Psi_{2}$.

First, we show 
\[
\|(A_{w}A_{w}^{\ast}-A_{P}A_{P}^{\ast})w_{2}\|_{\dot{\calH}_{2}^{1}}\aleq b^{4-}.
\]
We first note that 
\begin{align*}
 & \|(A_{w}A_{w}^{\ast}-A_{P}A_{P}^{\ast})w_{2}\|_{\dot{\calH}_{2}^{1}}\\
 & \aleq\|A_{w}A_{w}^{\ast}-A_{P}A_{P}^{\ast}\|_{L^{\infty}}\|w_{2}\|_{\dot{\calH}_{2}^{1}}+\|\partial_{y}(A_{w}A_{w}^{\ast}-A_{P}A_{P}^{\ast})\|_{L^{2+}}\|w_{2}\|_{L^{\infty-}}\\
 & \aleq b^{2}\|A_{w}A_{w}^{\ast}-A_{P}A_{P}^{\ast}\|_{L^{\infty}}+b^{2-}\|\partial_{y}(A_{w}A_{w}^{\ast}-A_{P}A_{P}^{\ast})\|_{L^{2+}}.
\end{align*}
Recall that 
\[
A_{w}A_{w}^{\ast}=-\partial_{yy}-\tfrac{1}{y}\partial_{y}+\tfrac{1}{y^{2}}((2+A_{\theta}[w])^{2}+\tfrac{1}{2}y^{2}|w|^{2}).
\]
Thus 
\[
\|A_{w}A_{w}^{\ast}-A_{P}A_{P}^{\ast}\|_{L^{\infty}}\aleq\||w|^{2}-|P|^{2}\|_{L^{\infty}}\aleq\|\langle y\rangle^{-2+}\eps\|_{L^{\infty}}+\|\eps\|_{L^{\infty}}^{2}\aleq b^{2-}.
\]
On the other hand, using $|w|\aleq1$ and $|A_{\theta}[w]|+|A_{\theta}[P]|\aleq y^{2}\langle y\rangle^{-2}$,
we have the pointwise bound 
\begin{align*}
 & |\partial_{y}(A_{w}A_{w}^{\ast}-A_{P}A_{P}^{\ast})|\\
 & \aleq\tfrac{1}{y\langle y\rangle^{2}}|A_{\theta}[w]-A_{\theta}[P]|+\tfrac{y}{\langle y\rangle^{2}}||w|^{2}-|P|^{2}|+|\overline{w}\partial_{y}w-\overline{P}\partial_{y}P|.
\end{align*}
We estimate the $L^{2+}$ norms by 
\begin{align*}
 & \|\tfrac{1}{y\langle y\rangle^{2}}|A_{\theta}[w]-A_{\theta}[P]|\|_{L^{2+}}+\|\tfrac{y}{\langle y\rangle^{2}}(|w|^{2}-|P|^{2})\|_{L^{2+}}\\
 & \aleq\|\tfrac{1}{\langle y\rangle}(|w|^{2}-|P|^{2})\|_{L^{2+}}\aleq\|\tfrac{1}{\langle y\rangle^{3-}}\eps\|_{L^{2+}}+\|\tfrac{1}{\langle y\rangle}\eps\|_{L^{2+}}\|\eps\|_{L^{\infty}}\aleq b^{2-}
\end{align*}
and 
\[
\|\overline{w}\partial_{y}w-\overline{P}\partial_{y}P\|_{L^{2+}}\aleq\|\tfrac{1}{\langle y\rangle^{2-}}|\eps|_{-1}\|_{L^{2+}}+\|\partial_{y}\eps\|_{L^{2+}}\|\eps\|_{L^{\infty}}\aleq b^{2-}.
\]

Next, we show 
\[
\|(A_{P}A_{P}^{\ast}-A_{Q}A_{Q}^{\ast})\eps_{2}\|_{\dot{\calH}_{2}^{1}}\aleq\tfrac{b}{|\log b|}\|\eps_{3}\|_{L^{2}}.
\]
This follows from 
\[
\|(A_{P}A_{P}^{\ast}-A_{Q}A_{Q}^{\ast})\eps_{2}\|_{\dot{\calH}_{2}^{1}}\aleq\|\langle\log_{+}y\rangle|A_{P}A_{P}^{\ast}-A_{Q}A_{Q}^{\ast}|_{1}\|_{L^{\infty}}\|\eps_{2}\|_{\dot{\calH}_{2}^{1}}\aleq\tfrac{b}{|\log b|}\|\eps_{3}\|_{L^{2}}.
\]
Note that $\frac{b}{|\log b|}$ comes from $||P|^{2}-Q^{2}|\aleq\chf_{(0,2B_{1}]}(|\eta|Q+b^{2}y^{2}Q)$.

Next, we show 
\[
\|({\textstyle \int_{0}^{y}}\Re(\overline{w}w_{1}-\overline{P}P_{1})dy')iw_{2}\|_{\dot{\calH}_{2}^{1}}\aleq b^{4-}.
\]
If $\partial_{y}$ does not hit the integral term, we estimate this
by (using the estimates shown in the proof of \eqref{eq:Gamma-GammaTilde})
\[
\|\tfrac{1}{y}\Re(\overline{w}w_{1}-\overline{P}P_{1})\|_{L^{1}}\|w_{2}\|_{\dot{\calH}_{2}^{1}}\aleq b^{2-}\|w_{2}\|_{\dot{\calH}_{2}^{1}}\aleq b^{4-}.
\]
If $\partial_{y}$ hits the integral term, we would like to put $w_{2}\in L^{\infty}$,
but here we have a technical problem that $\eps_{2}\notin L^{\infty}$.
Instead, we put $w_{2}$ in $L^{\infty-}$ using \eqref{eq:interpolation-1}
and \eqref{eq:H2-interpolation}: 
\[
\|w_{2}\|_{L^{\infty-}}\aleq\|w_{2}\|_{L^{2}}^{0+}\|\partial_{y}w_{2}\|_{L^{2}}^{1-}\aleq b^{2-}.
\]
Thus we estimate this contribution as 
\begin{align*}
 & \|\Re(\overline{w}w_{1}-\overline{P}P_{1})iw_{2}\|_{L^{2}}\\
 & \aleq(b\|\langle y\rangle^{-1}\eps\|_{L^{2+}}+\|\langle y\rangle^{-2+}\eps_{1}\|_{L^{2+}}+\|\eps\|_{L^{\infty}}\|\eps_{1}\|_{L^{2+}})\|w_{2}\|_{L^{\infty-}}\aleq b^{4-}.
\end{align*}

Next, it is easy to see that 
\[
\|\Re(\overline{P}P_{1})i\eps_{2}\|_{L^{2}}\aleq\tfrac{b}{|\log b|}\|\eps_{2}\|_{\dot{\calH}_{2}^{1}}\sim\tfrac{b}{|\log b|}\|\eps_{3}\|_{L^{2}}.
\]

Next, by the modulation estimates (Lemma~\ref{lem:ModulationEstimates})
and cancellation estimates \eqref{eq:v2-estimate}, we have 
\[
\|\Mod\cdot\mathbf{v}_{2}\|_{\dot{\calH}_{2}^{1}}\aleq b(\tfrac{1}{\sqrt{\log M}}\|\eps_{3}\|_{L^{2}}+b^{3-}).
\]

Lastly, we use the sharp energy estimate \eqref{eq:Psi2-SharpEnergy}:
\[
\|\Psi_{2}\|_{\dot{\calH}_{2}^{1}}\aleq\tfrac{b^{3}}{|\log b|}.
\]
This completes the proof. 
\end{proof}
We now aim to handle the non-perturbative contribution $b(\eps_{3},\tfrac{y}{2}Q^{2}\eps_{2}+A_{Q}^{\ast}(yQ^{2}\eps_{1}))_{r}$.
To motivate this, we write 
\begin{align*}
 & b(\eps_{3},\tfrac{y}{2}Q^{2}\eps_{2}+A_{Q}^{\ast}(yQ^{2}\eps_{1}))_{r}\\
 & =3b(\eps_{3},\tfrac{y}{2}Q^{2}\eps_{2})_{r}+b\{(A_{Q}\eps_{3},yQ^{2}\eps_{1})_{r}-(yQ^{2}\eps_{2},A_{Q}^{\ast}\eps_{2})_{r}\}.
\end{align*}
The first term is \emph{non-positive}, thanks to the repulsivity:
\begin{equation}
\begin{aligned}(\eps_{3},yQ^{2}\eps_{2})_{r} & =-2(A_{Q}^{\ast}\eps_{2},(\partial_{\lmb}A_{Q_{\lmb}}^{\ast})\eps_{2})_{r}\\
 & =-(\eps_{2},\partial_{\lmb}(A_{Q_{\lmb}}A_{Q_{\lmb}}^{\ast})\eps_{2})_{r}=(\eps_{2},\tfrac{y\partial_{y}\td V}{y^{2}}\eps_{2})_{r}\leq0.
\end{aligned}
\label{eq:MorawetzRepulsivity}
\end{equation}
The second term can be deleted by a \emph{Morawetz correction}: 
\[
(A_{Q}\eps_{3},yQ^{2}\eps_{1})_{r}-(yQ^{2}\eps_{2},A_{Q}^{\ast}\eps_{2})_{r}\approx\partial_{s}(i\eps_{2},yQ^{2}\eps_{1})_{r}
\]
from $i\partial_{s}\eps_{2}\approx A_{Q}\eps_{3}$ and $i\partial_{s}\eps_{1}\approx A_{Q}^{\ast}\eps_{2}$.
Note that this Morawetz correction term shares a similar spirit of
that of \cite{MerleRaphaelRodnianski2013InventMath} in the Schrödinger
maps case. More precisely, we have the following.
\begin{lem}[Morawetz correction]
\label{lem:MorawetzCorrection}We have 
\begin{align}
|b(i\eps_{2},yQ^{2}\eps_{1})_{r}| & \aleq b^{5-},\label{eq:Morawetz-bdd}\\
(\partial_{s}-6\tfrac{\lmb_{s}}{\lmb})\{b(i\eps_{2},yQ^{2}\eps_{1})_{r}\} & =b(\eps_{3},A_{Q}^{\ast}(yQ^{2}\eps_{1})-yQ^{2}\eps_{2})_{r}\label{eq:Morawetz-deriv}\\
 & \quad+O(\tfrac{b}{\sqrt{\log M}}\|\eps_{3}\|_{L^{2}}^{2}+b^{6-}).\nonumber 
\end{align}
\end{lem}

\begin{proof}
The first bound \eqref{eq:Morawetz-bdd} is immediate from \eqref{eq:virial-correction-temp1}
and the bootstrap hypothesis: 
\begin{align*}
b(i\eps_{2},yQ^{2}\eps_{1})_{r} & \aleq b\|\eps_{2}\|_{\dot{\calH}_{2}^{1}}\|\langle y\rangle^{-2}\langle\log_{+}y\rangle\eps_{1}\|_{L^{2}}\aleq b^{3-}\|\eps_{3}\|_{L^{2}}\aleq b^{5-}.
\end{align*}

We turn to the derivative estimate \eqref{eq:Morawetz-deriv}. We
compute 
\begin{equation}
\begin{aligned} & (\partial_{s}-6\tfrac{\lmb_{s}}{\lmb})\{b(i\eps_{2},yQ^{2}\eps_{1})_{r}\}\\
 & =b(A_{Q}\eps_{3},yQ^{2}\eps_{1})_{r}-b(yQ^{2}\eps_{2},A_{Q}^{\ast}\eps_{2})_{r}+(b_{s}-6\tfrac{\lmb_{s}}{\lmb}b)(i\eps_{2},yQ^{2}\eps_{1})_{r}\\
 & \quad+b(i\partial_{s}\eps_{2}-A_{Q}\eps_{3},yQ^{2}\eps_{1})_{r}-b(yQ^{2}\eps_{2},i\partial_{s}\eps_{1}-A_{Q}^{\ast}\eps_{2})_{r}
\end{aligned}
\label{eq:Morawetz-temp}
\end{equation}
As illustrated in the above, the first two terms of \eqref{eq:Morawetz-temp}
are the desired corrections: 
\[
b(A_{Q}\eps_{3},yQ^{2}\eps_{1})_{r}-b(yQ^{2}\eps_{2},A_{Q}^{\ast}\eps_{2})_{r}=b(\eps_{3},A_{Q}^{\ast}(yQ^{2}\eps_{1})-yQ^{2}\eps_{2})_{r}.
\]

The remaining terms of \eqref{eq:Morawetz-temp} are all treated as
errors. The third term is easily estimated by 
\[
|(b_{s}-6\tfrac{\lmb_{s}}{\lmb}b)(i\eps_{2},yQ^{2}\eps_{1})_{r}|\aleq b^{2}|(i\eps_{2},yQ^{2}\eps_{1})_{r}|\aleq b^{6-}.
\]

For the fourth term, by the estimate 
\begin{align*}
 & |b(i\partial_{s}\eps_{2}-A_{Q}\eps_{3},yQ^{2}\eps_{1})_{r}|\\
 & \aleq b\|y^{-1}\langle\log_{+}y\rangle^{-1}(i\partial_{s}\eps_{2}-A_{Q}\eps_{3})\|_{L^{2}}\|\langle y\rangle^{-2}\langle\log_{+}y\rangle\eps_{1}\|_{L^{2}}\\
 & \aleq b^{3-}\|y^{-1}\langle\log_{+}y\rangle^{-1}(i\partial_{s}\eps_{2}-A_{Q}\eps_{3})\|_{L^{2}}
\end{align*}
it suffices to prove 
\[
\|y^{-1}\langle\log_{+}y\rangle^{-1}(\partial_{s}\eps_{2}+iA_{Q}\eps_{3})\|_{L^{2}}\aleq b^{3-}.
\]
To show this, rewrite the equation \eqref{eq:e2-equation} as 
\begin{align*}
\partial_{s}\eps_{2}+iA_{Q}\eps_{3} & =\tfrac{\lmb_{s}}{\lmb}\Lambda_{-2}\eps_{2}-\td{\gmm}_{s}i\eps_{2}+(i\overline{w}w_{1}^{2}-i\overline{P}P_{1}^{2})+(\text{RHS of }\eqref{eq:e2-equation}).
\end{align*}
It only suffices to estimate the first three terms on the RHS above,
because we know from the proof of energy estimate \eqref{eq:energy-estimate-claim}
that $\|\text{RHS of }\eqref{eq:e2-equation}\|_{\dot{\calH}_{2}^{1}}\aleq\tfrac{b^{3}}{|\log b|}$.
We now estimate 
\[
\|y^{-1}\langle\log_{+}y\rangle^{-1}(\tfrac{\lmb_{s}}{\lmb}\Lambda_{-2}\eps_{2}-\td{\gmm}_{s}i\eps_{2})\|_{L^{2}}\aleq b\|\eps_{2}\|_{\dot{\calH}_{2}^{1}}\aleq b^{3}.
\]
Next, (from the proof of \eqref{eq:energy-estimate-claim}) 
\begin{align*}
 & \|y^{-1}\langle\log_{+}y\rangle^{-1}(i\overline{w}w_{1}^{2}-i\overline{P}P_{1}^{2})\|_{L^{2}}\\
 & \aleq\|y^{-1}\langle\log_{+}y\rangle^{-1}byQ^{2}\eps_{1}\|_{L^{2}}+\tfrac{b^{3}}{|\log b|}\aleq b\|\eps_{1}\|_{\dot{\calH}_{1}^{2}}+\tfrac{b^{3}}{|\log b|}\aleq b^{3}.
\end{align*}

Finally, the last term of \eqref{eq:Morawetz-temp} can be estimated
by 
\begin{align*}
|b(yQ^{2}\eps_{2},i\partial_{s}\eps_{1}-A_{Q}^{\ast}\eps_{2})_{r}| & \aleq b\|\eps_{2}\|_{\dot{\calH}_{2}^{1}}\|\langle y\rangle^{-2}\langle\log_{+}y\rangle(\partial_{s}\eps_{1}+iA_{Q}^{\ast}\eps_{2})\|_{L^{2}}\\
 & \aleq b\|\eps_{3}\|_{L^{2}}\|\partial_{s}\eps_{1}+iA_{Q}^{\ast}\eps_{2}\|_{X},
\end{align*}
so it suffices to establish the bound 
\[
\|\partial_{s}\eps_{1}+iA_{Q}^{\ast}\eps_{2}\|_{X}\aleq\tfrac{1}{\sqrt{\log M}}\|\eps_{3}\|_{L^{2}}+b^{3-}.
\]
To show this, we rewrite the equation \eqref{eq:e1-equation} of $\eps_{1}$
as 
\begin{align*}
\partial_{s}\eps_{1}+iA_{Q}^{\ast}\eps_{2} & =\tfrac{\lmb_{s}}{\lmb}\Lambda_{-1}\eps_{1}-\td{\gmm}_{s}i\eps_{1}-(iA_{w}^{\ast}-iA_{P}^{\ast})w_{2}-(iA_{P}^{\ast}-iA_{Q}^{\ast})\eps_{2}\\
 & \quad+({\textstyle \int_{0}^{y}}\Re(\overline{w}w_{1})dy')i\eps_{1}+({\textstyle \int_{0}^{y}}(\Re(\overline{w}w_{1})-\Re(\overline{P}P_{1}))dy')iP_{1}\\
 & \quad+\td{\Mod}\cdot\mathbf{v}_{1}-i\Psi_{1}.
\end{align*}
Recall that all terms except $\td{\Mod}\cdot\mathbf{v}_{1}$ on the
RHS are already estimated in the proof of the modulation estimates;
see \eqref{eq:virial-correction-temp1}-\eqref{eq:virial-correction-temp4}.
Thus these terms contribute to the error $O(b^{3-})$. The term $\td{\Mod}\cdot\mathbf{v}_{1}$
can be estimated by the modulation estimates (Lemma~\ref{lem:ModulationEstimates})
and estimates for $\mathbf{v}_{1}$ \eqref{eq:v1-estimate}: 
\[
\|\td{\Mod}\cdot\mathbf{v}_{1}\|_{X}\aleq\tfrac{1}{\sqrt{\log M}}\|\eps_{3}\|_{L^{2}}+b^{3-}.
\]
This completes the proof. 
\end{proof}
Define the \emph{modified third energy} by 
\[
\calF_{3}\coloneqq\tfrac{1}{2}\|\eps_{3}\|_{L^{2}}^{2}-b(i\eps_{2},yQ^{2}\eps_{1})_{r}.
\]

\begin{prop}[The modified energy inequality]
\label{prop:ModifiedEnergyInequality}We have 
\begin{align}
|\calF_{3}-\tfrac{1}{2}\|\eps_{3}\|_{L^{2}}^{2}| & \leq b^{5-}\label{eq:ModifiedEnergyCoer}\\
(\partial_{s}-6\tfrac{\lmb_{s}}{\lmb})\calF_{3} & \leq b(\tfrac{1}{100}\|\eps_{3}\|_{L^{2}}^{2}+C\tfrac{b^{4}}{|\log b|^{2}}),\label{eq:ModifiedEnergyDeriv}
\end{align}
where $C$ is some universal constant. 
\end{prop}

\begin{proof}
The coercivity \eqref{eq:ModifiedEnergyCoer} follow from \eqref{eq:Morawetz-bdd}.
For the monotonicity \eqref{eq:ModifiedEnergyDeriv}, we combine \eqref{eq:EnergyIdentity}
and \eqref{eq:MorawetzRepulsivity} to have 
\begin{align*}
 & |(\partial_{s}-6\tfrac{\lmb_{s}}{\lmb})\calF_{3}-\tfrac{3}{2}b(\eps_{3},yQ^{2}\eps_{2})_{r}|\\
 & \leq Cb(\tfrac{1}{\sqrt{\log M}}\|\eps_{3}\|_{L^{2}}^{2}+\tfrac{b^{2}}{|\log b|}\|\eps_{3}\|_{L^{2}})+Cb(\tfrac{1}{\sqrt{\log M}}\|\eps_{3}\|_{L^{2}}^{2}+b^{5-})\\
 & \leq b\big((\tfrac{C}{\sqrt{\log M}}+\tfrac{1}{200})\|\eps_{3}\|_{L^{2}}^{2}+\tfrac{Cb^{4}}{|\log b|^{2}}\big).
\end{align*}
By the repulsivity \eqref{eq:MorawetzRepulsivity} and $M\gg1$, we
have 
\[
(\partial_{s}-6\tfrac{\lmb_{s}}{\lmb})\calF_{3}\leq b(\tfrac{1}{100}\|\eps_{3}\|_{L^{2}}^{2}+C\tfrac{b^{4}}{|\log b|^{2}}).
\]
This completes the proof. 
\end{proof}

\subsection{\label{subsec:ClosingBootstrap}Proofs of Propositions~\ref{prop:main-bootstrap},
\ref{prop:Sets-I-pm}, and \ref{prop:SharpDescription}}

In this last subsection, we finish the proofs of Propositions~\ref{prop:main-bootstrap},
\ref{prop:Sets-I-pm}, and \ref{prop:SharpDescription}. The arguments
here are very similar to the Schrödinger map case \cite{MerleRaphaelRodnianski2013InventMath}.
We include the proofs for the sake of completeness. We note that there
are some simplifications in our case, thanks to the conservation of
mass and energy. 
\begin{lem}[Consequences of modulation estimates]
We have 
\begin{align}
\int_{0}^{t}\frac{b}{\lmb^{2}}\cdot\frac{b^{4}}{\lmb^{6}|\log b|^{2}}d\tau & \leq\frac{b^{4}(t)}{\lmb^{6}(t)|\log b(t)|^{2}},\label{eq:Conseq-Modulation-claim1}\\
\frac{b(t)|\log b(t)|^{2}}{\lmb(t)} & =\bigg(1+O\Big(\frac{1}{|\log b_{0}|^{\frac{1}{2}-}}\Big)\bigg)\frac{b_{0}|\log b_{0}|^{2}}{\lmb_{0}},\label{eq:Conseq-Modulation-claim2}\\
\frac{\lmb(t)}{\lmb_{0}} & \leq\Big(\frac{b(t)}{b_{0}}\Big)^{\frac{3}{4}}.\label{eq:Conseq-Modulation-claim3}
\end{align}
\end{lem}

\begin{proof}
The estimate \eqref{eq:Conseq-Modulation-claim1} follows from $\frac{b}{\lmb^{2}}=-\frac{\lmb_{t}}{\lmb}+O(\frac{b^{2-}}{\lmb^{2}})$
and integration by parts: 
\begin{align*}
 & \int_{0}^{t}\frac{b}{\lmb^{2}}\cdot\frac{b^{4}}{\lmb^{6}|\log b|^{2}}d\tau\\
 & =\frac{1}{6}\Big[\frac{b^{4}}{\lmb^{6}|\log b|^{2}}\Big]_{0}^{t}-\frac{4}{6}\int_{0}^{t}\frac{b_{t}b^{3}}{\lmb^{6}|\log b|^{2}}d\tau+O\Big(\frac{1}{|\log b^{\ast}|}\int_{0}^{t}\frac{b}{\lmb^{2}}\cdot\frac{b^{4}}{\lmb^{6}|\log b|^{2}}d\tau\Big)\\
 & \leq\frac{1}{6}\frac{b^{4}(t)}{\lmb^{6}(t)|\log b|^{2}}+\frac{4}{6}\int_{0}^{t}\frac{b}{\lmb^{2}}\cdot\frac{b^{4}}{\lmb^{6}|\log b|^{2}}d\tau+O\Big(\frac{1}{|\log b^{\ast}|}\int_{0}^{t}\frac{b}{\lmb^{2}}\cdot\frac{b^{4}}{\lmb^{6}|\log b|^{2}}d\tau\Big).
\end{align*}

To show the estimate \eqref{eq:Conseq-Modulation-claim2}, we need
the refined modulation estimates (Lemma~\ref{lem:RefinedModulationEstimates}).
We compute using \eqref{eq:btilde-b}, \eqref{eq:btilde-modulation}
and $|b_{s}+b^{2}|+|\td b_{s}+\td b^{2}|\aleq\frac{b^{2}}{|\log b|}$
to get 
\[
\partial_{s}\log\Big(\frac{\lmb}{\td b|\log\td b|^{2}}\Big)=\Big(\frac{\lmb_{s}}{\lmb}+\td b\Big)-\Big(\frac{\td b_{s}+\td b^{2}+\frac{2\td b^{2}}{|\log\td b|}}{\td b}\Big)+O\Big(\frac{b}{|\log b|^{2}}\Big)=O\Big(\frac{b}{|\log b|^{\frac{3}{2}-}}\Big).
\]
Integrating this, we have 
\[
\bigg|\Big(\frac{\td b(t)|\log\td b(t)|^{2}}{\lmb(t)}\Big)^{-1}\frac{\td b_{0}|\log\td b_{0}|^{2}}{\lmb_{0}}-1\bigg|\aleq\int_{0}^{t}\frac{b}{\lmb^{2}}\cdot\frac{1}{|\log b|^{\frac{3}{2}-}}d\tau.
\]
The error term (the RHS) can be estimated using $\frac{b}{\lmb^{2}}=-\frac{b_{t}}{b}+O(\frac{1}{|\log b|}\frac{b}{\lmb^{2}})$:
\[
\bigg(1+O\Big(\frac{1}{|\log b^{\ast}|}\Big)\bigg)\int_{0}^{t}\frac{b}{\lmb^{2}}\cdot\frac{1}{|\log b|^{\frac{3}{2}-}}d\tau=-\int_{0}^{t}\frac{b_{t}}{b|\log b|^{\frac{3}{2}-}}d\tau\aleq\frac{1}{|\log b_{0}|^{\frac{1}{2}-}}.
\]
Finally replacing $\td b$ by $b$ using \eqref{eq:btilde-b} completes
the proof of \eqref{eq:Conseq-Modulation-claim2}.

The estimate \eqref{eq:Conseq-Modulation-claim3} follows from 
\[
\partial_{s}\log\Big(\frac{\lmb^{\frac{4}{3}}}{b}\Big)=\frac{\lmb_{s}}{3\lmb}+\Big(\frac{\lmb_{s}}{\lmb}+b\Big)-\Big(\frac{b_{s}+b^{2}}{b}\Big)=-\frac{b}{3}+O\Big(\frac{b}{|\log b|}\Big)\leq0.
\]
This completes the proof. 
\end{proof}
We are now ready to prove the main bootstrap Proposition~\ref{prop:main-bootstrap}\@. 
\begin{proof}[Proof of the main boostrap Proposition~\ref{prop:main-bootstrap}]
Note that $b(t)\leq b_{0}$ is immediate from $b_{s}\approx-b^{2}<0$.

We first close the $\|\eps_{3}\|_{L^{2}}$-bound. By the modified
energy inequality (Proposition~\ref{prop:ModifiedEnergyInequality}),
we have 
\[
\frac{1}{2}\frac{\|\eps_{3}(t)\|_{L^{2}}^{2}}{\lmb^{6}(t)}\leq\frac{1}{2}\frac{\|\eps_{3}(0)\|_{L^{2}}^{2}}{\lmb_{0}^{6}}+\frac{b_{0}^{9/2}}{\lmb_{0}^{6}}+\frac{b^{9/2}(t)}{\lmb^{6}(t)}+\Big(\frac{K^{2}}{100}+C\Big)\int_{0}^{t}\frac{b}{\lmb^{2}}\cdot\frac{b^{4}}{\lmb^{6}|\log b|^{2}}d\tau.
\]
Applying the claims \eqref{eq:Conseq-Modulation-claim1} and \eqref{eq:Conseq-Modulation-claim3}
yields 
\[
\|\eps_{3}(t)\|_{L^{2}}^{2}\leq\Big(\frac{b(t)}{b_{0}}\Big)^{\frac{9}{2}}\|\eps_{3}(0)\|_{L^{2}}^{2}+\Big(\frac{K^{2}}{50}+C\Big)\frac{b^{4}(t)}{|\log b(t)|^{2}}.
\]
Applying the initial bound \eqref{eq:def-Utilde-init} and $K\gg1$,
this closes the $\|\eps_{3}\|_{L^{2}}$-bound.

We now close the $\|\eps_{1}\|_{L^{2}}$-bound. Thanks to the energy
conservation, we have 
\[
\frac{\|w_{1}(t)\|_{L^{2}}}{\lmb(t)}=\frac{\|w_{1}(0)\|_{L^{2}}}{\lmb_{0}}.
\]
Thus we have 
\begin{align*}
\|\eps_{1}(t)\|_{L^{2}} & \leq\|w_{1}(t)\|_{L^{2}}+\|P_{1}(t)\|_{L^{2}}\\
 & \leq\frac{\lmb(t)}{\lmb_{0}}(\|\eps_{1}(0)\|_{L^{2}}+Cb_{0}|\log b_{0}|^{\frac{1}{2}})+Cb(t)|\log b(t)|^{\frac{1}{2}}\\
 & \leq\frac{\lmb(t)}{\lmb_{0}}(\|\eps_{1}(0)\|_{L^{2}}+b_{0}|\log b_{0}|^{2})+b(t)|\log b(t)|^{2}.
\end{align*}
Applying the initial bound \eqref{eq:def-Utilde-init}, \eqref{eq:Conseq-Modulation-claim2},
and $K\gg1$, this closes the $\|\eps_{1}\|_{L^{2}}$-bound.

We now close the $\|\eps\|_{L^{2}}$-bound. Thanks to the mass conservation,
\[
\|w(t)\|_{L^{2}}=\|w(0)\|_{L^{2}}.
\]
We manipulate 
\begin{align*}
\|w\|_{L^{2}}^{2} & =\|P\|_{L^{2}}^{2}+2(P,\eps)_{r}+\|\eps\|_{L^{2}}^{2}\\
 & =\|Q\|_{L^{2}}^{2}+\|\eps\|_{L^{2}}^{2}+O(\|P\eps\|_{L^{1}}+|\|P\|_{L^{2}}^{2}-\|Q\|_{L^{2}}^{2}|)\\
 & =\|Q\|_{L^{2}}^{2}+\|\eps\|_{L^{2}}^{2}+O(b^{1-}).
\end{align*}
Therefore, 
\[
\|\eps(t)\|_{L^{2}}\leq\|\eps(0)\|_{L^{2}}+O(b_{0}^{\frac{1}{2}-}).
\]
Applying the initial bound \eqref{eq:def-Utilde-init} and $b^{\ast}\ll1$,
this closes the $\|\eps\|_{L^{2}}$-bound. 
\end{proof}
We turn to the proof of Proposition~\ref{prop:Sets-I-pm}. Let us
recall the situation in the proof of Theorem~\ref{thm:main-thm}.
For a fixed $(\widehat{\lmb}_{0},\widehat{\gmm}_{0},\widehat{b}_{0},\widehat{\eps}_{0})\in\td{\calU}_{\init}$,
we were considering the one-parameter family of solutions $u^{(\widehat{\eta}_{0})}$
starting from the initial data formed by the rough decomposition,
i.e. $u_{0}^{(\widehat{\eta}_{0})}=\frac{e^{i\widehat{\gmm}_{0}}}{\widehat{\lmb}_{0}}[P(\cdot;\widehat{b}_{0},\widehat{\eta}_{0})+\widehat{\eps}_{0}](\frac{\cdot}{\widehat{\lmb}_{0}})$,
$\widehat{\eta}_{0}\in(-\frac{\widehat{b}_{0}}{2|\log\widehat{b}_{0}|},\frac{\widehat{b}_{0}}{2|\log\widehat{b}_{0}|})$.
Here we added a superscript $(\widehat{\eta}_{0})$ for clarification.
We then changed the decomposition into the nonlinear decomposition
$(\lmb_{0},\gmm_{0},b_{0},\eta_{0},\eps_{0})$, and denote by $(\lmb(t),\gmm(t),b(t),\eta(t),\eps(t))$
the nonlinear decomposition of $u^{(\widehat{\eta}_{0})}(t)$ at time
$t$. We also recall by \eqref{eq:initial-decomposition-transition}
that the difference of $(\widehat{\lmb}_{0},\widehat{\gmm}_{0},\widehat{b}_{0},\widehat{\eta}_{0},\widehat{\eps}_{0})$
and $(\lmb_{0},\gmm_{0},b_{0},\eta_{0},\eps_{0})$ is bounded by $\widehat{b}_{0}^{2-}$.
Finally, we assumed (for a contradiction argument) that for any $\widehat{\eta}_{0}$
the solution $u^{(\widehat{\eta}_{0})}$ exits the trapped regime
by violating the $\eta$-bound: $|\eta(T_{\mathrm{exit}}^{(\widehat{\eta}_{0})})|=\frac{b(T_{\mathrm{exit}}^{(\widehat{\eta}_{0})})}{|\log b(T_{\mathrm{exit}}^{(\widehat{\eta}_{0})})|}$. 
\begin{proof}[Proof of Proposition~\ref{prop:Sets-I-pm}]
We need to show that $\calI_{\pm}$ are nonempty open sets.

To show that $\calI_{\pm}$ is nonempty, we show $\pm\frac{1}{5}\frac{\widehat{b}_{0}}{|\log\widehat{b}_{0}|}\in\calI_{\pm}$.
We compute the variation of the ratio $\frac{\eta|\log b|}{b}$ using
the modulation estimates \eqref{eq:ModEstimate-b-eta}: 
\begin{align*}
\partial_{s}\Big(\frac{\eta|\log b|}{b}\Big) & =\frac{\eta|\log b|}{b}\Big(-\frac{b_{s}}{b}\big(1+\frac{1}{|\log b|}\big)\Big)+\frac{\eta_{s}|\log b|}{b}\\
 & =\frac{\eta|\log b|}{b}\cdot b\Big(1+O\big(\frac{1}{|\log b|}\big)\Big)+O\big(\frac{Kb}{\sqrt{\log M}}\big).
\end{align*}
Thus if $|\frac{\eta|\log b|}{b}|\geq\frac{1}{10}$ holds at some
time, $|\frac{\eta|\log b|}{b}|$ starts to increase, thanks to $\frac{K}{\sqrt{\log M}}\ll1$.
In particular, if $\widehat{\eta}_{0}=\pm\frac{1}{5}\frac{\widehat{b}_{0}}{|\log\widehat{b}_{0}|}$,
by \eqref{eq:initial-decomposition-transition} $\pm\eta_{0}\geq\frac{1}{10}\frac{b_{0}}{|\log b_{0}|}$
so $\eta$ must have same sign with $\eta_{0}$ at $T_{\mathrm{exit}}^{(\eta_{0})}$,
saying that $\pm\frac{1}{5}\frac{\widehat{b}_{0}}{|\log\widehat{b}_{0}|}\in\calI_{\pm}$.

We turn to show that $\calI_{\pm}$ is open. Since $\widehat{\eta}_{0}\in\calI_{\pm}$,
there exists $t^{(\widehat{\eta}_{0})}\in[0,T_{\mathrm{exit}}^{(\widehat{\eta}_{0})})$
such that $\pm\eta^{(\widehat{\eta}_{0})}(t^{(\widehat{\eta}_{0})})>\frac{1}{2}\frac{b_{0}}{|\log b_{0}|}(t^{(\widehat{\eta}_{0})})$.
By the continuous dependence, (obtained by combining the local well-posedness
and Lemma~\ref{lem:decomp}) for all $\eta_{0}'$ near $\widehat{\eta}_{0}$
we have $t^{(\widehat{\eta}_{0})}\in[0,T_{\mathrm{exit}}^{(\eta_{0}')})$
and $\pm\eta^{(\eta_{0}')}(t^{(\widehat{\eta}_{0})})>\frac{1}{2}\frac{b_{0}}{|\log b_{0}|}(t^{(\widehat{\eta}_{0})})$.
Such $\eta_{0}'$ belongs to $\calI_{\pm}$ due to the argument in
the previous paragraph. This completes the proof.
\end{proof}
In view of Propositions~\ref{prop:main-bootstrap} and \ref{prop:Sets-I-pm},
we have constructed a trapped solution $u$. The remaining task is
to show that $u$ blows up in finite time as described in Theorem~\ref{thm:main-thm}. 
\begin{proof}[Proof of Proposition~\ref{prop:SharpDescription}]
The proof is very similar to \cite[Section 6]{MerleRaphaelRodnianski2013InventMath}.

(1) By the claim \eqref{eq:Conseq-Modulation-claim3}, we have 
\[
\partial_{t}\lmb^{\frac{2}{3}}=-\frac{b}{3\lmb^{\frac{4}{3}}}+\frac{1}{3\lmb^{\frac{4}{3}}}\Big(\frac{\lmb_{s}}{\lmb}+b\Big)=-\frac{b}{3\lmb^{\frac{4}{3}}}\Big(1+O\big((b^{\ast})^{1-}\big)\Big)\leq-\frac{b_{0}}{4\lmb_{0}^{\frac{4}{3}}}.
\]
This implies the finite-time blow-up, $T<+\infty$. By the standard
blow-up criterion, i.e., a $H^{1}$-solution blows up at a finite
time $T<+\infty$ only if $\lim_{t\uparrow T}\|u(t)\|_{\dot{H}^{1}}=\infty$,
we have $\lmb(T)\coloneqq\lim_{t\uparrow T}\lmb(t)=0$. Moreover,
due to \eqref{eq:Conseq-Modulation-claim2} and $|\eta|<\frac{b}{|\log b|}$,
we have $b(T)\coloneqq\lim_{t\uparrow T}b(t)=0$ and $\eta(T)\coloneqq\lim_{t\uparrow}\eta(t)=0$.

(2) We start by rewriting the claim \eqref{eq:Conseq-Modulation-claim2}
as 
\begin{equation}
\frac{b|\log b|^{2}}{\lmb}=\ell\Big(1+O\big(\frac{1}{|\log b|^{\frac{1}{2}-}}\big)\Big),\qquad\ell\coloneqq\lim_{t\uparrow T}\frac{b(t)|\log b(t)|^{2}}{\lmb(t)}\in(0,\infty),\label{eq:Conseq-Modulation-claim4}
\end{equation}
where the existence of $\ell\in(0,\infty)$ follows from \eqref{eq:Conseq-Modulation-claim2}
(on $[t,T)$ instead of on $[0,t]$) and $b(T)=0$.

We now claim the asymptotics of the parameters $\lmb$ and $b$: 
\begin{align}
\lmb(t) & =\ell\cdot\frac{T-t}{|\log(T-t)|^{2}}(1+o_{t\to T}(1)),\label{eq:sharp-lambda-asymptotics}\\
b(t) & =\ell^{2}\cdot\frac{T-t}{|\log(T-t)|^{4}}(1+o_{t\to T}(1)).\label{eq:sharp-b-asymptotics}
\end{align}
To see this, we first derive the asymptotics of $\lmb$ and $b$ in
the $s$-variable. We integrate the refined modulation estimate \eqref{eq:btilde-modulation}
in the $s$-variable from $[s,\infty)$ to obtain 
\[
\td b(s)=\frac{1}{s}-\frac{2}{s\log s}+O\Big(\frac{1}{s|\log s|^{\frac{3}{2}-}}\Big).
\]
By \eqref{eq:btilde-b}, the same asymptotics apply to $b(s)$. Thus
\eqref{eq:Conseq-Modulation-claim4} yields 
\[
\ell\lmb(s)=\frac{(\log s)^{2}}{s}(1+o_{s\to\infty}(1))
\]
and hence 
\begin{equation}
b=\frac{\ell\lmb}{|\log(\ell\lmb)|^{2}}(1+o_{s\to\infty}(1)).\label{eq:Conseq-Modulation-claim5}
\end{equation}
In the original time variable $t$, the sharp $\lmb$-asymptotics
\eqref{eq:sharp-lambda-asymptotics} follow from integrating 
\[
\lmb_{t}=-\frac{b}{\lmb}(1+o_{t\to T}(1))=-\frac{\ell}{|\log(\ell\lmb)|^{2}}(1+o_{t\to T}(1))
\]
backwards in time from $T$ to $t$ with $\lmb(T)=0$. The sharp $b$-asymptotics
\eqref{eq:sharp-b-asymptotics} follow from substituting the sharp
$\lmb$-asymptotics into \eqref{eq:Conseq-Modulation-claim5}.

Next, we claim that $\gmm(t)$ converges to some $\gmm^{\ast}$ as
$t\to T$. Indeed, from the refined modulation estimate \eqref{eq:btilde-modulation}
and $\eta\to0$, we have 
\[
|\td{\eta}(s)|\aleq\frac{1}{s(\log s)^{\frac{3}{2}-}}.
\]
By \eqref{eq:btilde-b}, the same bound holds for $\eta(s)$. Thus
the modulation estimate \eqref{eq:ModEstimateScalePhase} says that
$\gmm_{s}$ is integrable in $[s,\infty)$: 
\[
|\gmm_{s}|\aleq\frac{1}{s(\log s)^{\frac{3}{2}-}}.
\]
Hence $\gmm(t)$ converges to some $\gmm^{\ast}$ as $t\to T$.

(3) It now remains to show that $u$ decomposes as in Theorem~\ref{thm:main-thm}.

We first claim the outer $L^{2}$-convergence: $\chf_{[R,\infty)}u(t)$
converges in $L^{2}$ for any $R>0$. To show this, choose any $R>0$
and we show that $\chf_{[R,\infty)}u(t)$ converges in $L^{2}$. In
view of $i\partial_{t}(\chf_{[R,\infty)}u)=\chf_{[R,\infty)}L_{u}^{\ast}\bfD_{u}u$,
it suffices to show that $t\mapsto\|\chf_{[R,\infty)}L_{u}^{\ast}\bfD_{u}u(t)\|_{L^{2}}$
is integrable. By scaling, we observe that 
\[
\|\chf_{[R,\infty)}L_{u}^{\ast}\bfD_{u}u\|_{L^{2}}=\lmb^{-2}(t)\|\chf_{[\lmb^{-1}(t)R,\infty)}L_{w}^{\ast}w_{1}\|_{L^{2}}.
\]
Since 
\[
\chf_{[\lmb^{-1}R,\infty)}|L_{w}^{\ast}w_{1}|\aleq\chf_{[\lmb^{-1}R,\infty)}(|w_{1}|_{-1}+|w|{\textstyle \int_{y}^{\infty}}|ww_{1}|dy'),
\]
we have 
\[
\|\chf_{[\lmb^{-1}R,\infty)}L_{w}^{\ast}w_{1}\|_{L^{2}}\aleq\|\chf_{[\lmb^{-1}R,\infty)}|w_{1}|_{-1}\|_{L^{2}}(1+\|w\|_{L^{2}}^{2}).
\]
Because $P_{1}$ is supported in $(0,2B_{1}]$ and $2B_{1}<\lmb^{-1}R$
for $t$ sufficiently close to $T$, we have by \eqref{eq:H2-interpolation}
\[
\|\chf_{[\lmb^{-1}R,\infty)}L_{w}^{\ast}w_{1}\|_{L^{2}}\aleq\||\eps_{1}|_{-1}\|_{L^{2}}\aleq b^{\frac{3}{2}}|\log b|^{\frac{1}{2}-}.
\]
Using the sharp asymptotics \eqref{eq:sharp-lambda-asymptotics} and
\eqref{eq:sharp-b-asymptotics}, $\lmb^{-2}b^{\frac{3}{2}}|\log b|^{\frac{1}{2}-}$
is integrable, and hence the claim is proved.

The above claim says that there exists a function $u^{\ast}$ such
that $\chf_{[R,\infty)}u^{\ast}\in L^{2}$ and $\chf_{[R,\infty)}u(t)\to\chf_{[R,\infty)}u^{\ast}$
in $L^{2}$ for any $R>0$. We show that this $u^{\ast}$ satisfies
the statement of Theorem~\ref{thm:main-thm}. Let 
\[
\eps^{\sharp}(t,r)\coloneqq\frac{e^{i\gmm(t)}}{\lmb(t)}\eps\Big(t,\frac{r}{\lmb(t)}\Big)
\]
Since $(\gmm,b,\eta)\to(\gmm^{\ast},0,0)$ and $\frac{\ell(T-t)}{\lmb(t)|\log(T-t)|^{2}}\to1$,
we have 
\[
\frac{e^{i\gmm(t)}}{\lmb(t)}P\Big(\frac{r}{\lmb(t)};b(t),\eta(t)\Big)-e^{i\gmm^{\ast}}\frac{|\log(T-t)|^{2}}{\ell(T-t)}Q\Big(\frac{|\log(T-t)|^{2}}{\ell(T-t)}r\Big)\to0\text{ in }L^{2}.
\]
Thus it suffices to show that $u^{\ast}\in H_{0}^{1}$ and $\eps^{\sharp}(t)\to u^{\ast}$
in $L^{2}$ as $t\to T$. On one hand, $\chf_{[R,\infty)}\eps^{\sharp}(t)\to\chf_{[R,\infty)}u^{\ast}$
in $L^{2}$ for any $R>0$, as the outer convergence is insensitive
to the concentrating bubble. On the other hand, due to the boundedness
of $\frac{b|\log b|^{2}}{\lmb}$ (see \eqref{eq:Conseq-Modulation-claim4})
and $\|\eps^{\sharp}\|_{\dot{H}_{0}^{1}}=\lmb^{-1}\|\eps\|_{\dot{H}_{0}^{1}}$,
we see that $\eps^{\sharp}(t)$ is uniformly bounded in $H_{0}^{1}$.
Therefore, $u^{\ast}\in H_{0}^{1}$ and $\eps^{\sharp}(t)\rightharpoonup u^{\ast}$
weakly in $H_{0}^{1}$. By the Rellich-Kondrachov compactness theorem,
$\eps^{\sharp}(t)\to u^{\ast}$ in $L_{\mathrm{loc}}^{2}$. Combining
this with outer $L^{2}$-convergence shows that $\eps^{\sharp}(t)\to u^{\ast}$
in $L^{2}$. This finishes the proof. 
\end{proof}

\appendix

\section{\label{sec:Adapted-function-spaces}Adapted function spaces}

In this section, we prove the facts regarding to the adapted function
spaces introduced in Section~\ref{subsec:Adapted-function-spaces}.
Our main focuses are on (sub-)coercivity estimates of Proposition
\ref{prop:LinearCoercivity}. On the way, we compare the adapted function
spaces with the usual equivariant Sobolev spaces and prove various
$L^{\infty}$-estimates and interpolation estimates.

Our main tools are weighted Hardy's inequalities: 
\begin{lem}[{{Weighted Hardy's inequality for $\partial_{r}$; see \cite[Lemma A.1]{KimKwon2020arXiv}}}]
\label{lem:WeightedHardy-dr}Let $0<r_{1}<r_{2}<\infty$; let $\varphi:[r_{1},r_{2}]\to\bbR_{+}$
be a $C^{1}$ weight function such that $\partial_{r}\varphi$ is
nonvanishing and $\varphi\aleq|r\partial_{r}\varphi|$. Then, for
smooth $f:[r_{1},r_{2}]\to\bbC$, we have 
\[
\int_{r_{1}}^{r_{2}}\Big|\frac{f}{r}\Big|^{2}|r\partial_{r}\varphi|rdr\aleq\int_{r_{1}}^{r_{2}}|\partial_{r}f|^{2}\varphi\,rdr+\begin{cases}
\varphi(r_{2})|f(r_{2})|^{2} & \text{if }\partial_{r}\varphi>0,\\
\varphi(r_{1})|f(r_{1})|^{2} & \text{if }\partial_{r}\varphi<0.
\end{cases}
\]
\end{lem}

By carefully choosing $\varphi$, we also have logarithmic Hardy's
inequality: 
\begin{lem}[{{Logarithmic Hardy's inequality; see \cite[Corollary A.3]{KimKwon2020arXiv}}}]
\label{lem:LogarithmicHardy}For $k\in\mathbb{R}$, we have 
\begin{equation}
\int_{r_{1}}^{r_{2}}\Big|\frac{f}{r^{k+1}\langle\log r\rangle}\Big|^{2}rdr\aleq\int_{r_{1}}^{r_{2}}\Big|\frac{(\partial_{r}-\frac{k}{r})f}{r^{k}}\Big|^{2}rdr+\begin{cases}
|f(1)|^{2} & \text{if }1\in[r_{1},r_{2}],\\
|(r_{2})^{-k}f(r_{2})|^{2} & \text{if }r_{2}\leq1,\\
|(r_{1})^{-k}f(r_{1})|^{2} & \text{if }r_{1}\geq1.
\end{cases}\label{eq:LogHardy}
\end{equation}
\end{lem}

We now introduce the adapted function spaces $\dot{\calH}_{0}^{1}$,
$\dot{\calH}_{2}^{1}$, $\dot{\calH}_{1}^{2}$, and $\dot{\calH}_{0}^{3}$.
These are all different from $\dot{H}_{0}^{1}$, $\dot{H}_{1}^{2}$,
$\dot{H}_{1}^{2}$, and $\dot{H}_{0}^{3}$, but are essentially same
for functions with high frequency. As a result, their inhomogeneous
versions are the same: $\dot{\calH}_{m}^{k}\cap L^{2}=H_{m}^{k}$.

The adapted function spaces are motivated to have boundedness and
subcoercivity estimates for the linear adapted derivatives, e.g. $L_{Q}\eps$,
$A_{Q}L_{Q}\eps$, and $A_{Q}^{\ast}A_{Q}L_{Q}\eps$ with various
levels of regularity. The first one $\dot{\calH}_{0}^{1}$ is designed
to control $\eps$, provided that $\eps_{1}\approx L_{Q}\eps\in L^{2}$.
On the other hand, the spaces $\dot{\calH}_{2}^{1}$, $\dot{\calH}_{1}^{2}$,
and $\dot{\calH}_{0}^{3}$ are designed to control $\eps_{2}$, $\eps_{1}$,
and $\eps$, provided that $\eps_{3}=A_{Q}^{\ast}\eps_{2}\in L^{2}$.

\subsubsection*{The space $\dot{\protect\calH}_{0}^{1}$}

\ 

For $0$-equivariant Schwartz functions $f$, define 
\[
\|f\|_{\dot{\calH}_{0}^{1}}\coloneqq\|\partial_{r}f\|_{L^{2}}+\|r^{-1}\langle\log_{-}r\rangle^{-1}f\|_{L^{2}}.
\]
Define the space $\dot{\calH}_{0}^{1}$ by taking the completion of
$\calS_{0}$ under this norm. This is the adapted function space at
$\dot{H}^{1}$-level. We note that $\dot{\calH}_{0}^{1}$ is \emph{stronger}
than $\dot{H}_{0}^{1}$, due to its control at infinity. Nevertheless,
$L^{2}\cap\dot{\calH}_{0}^{1}=H_{0}^{1}$. 
\begin{lem}[Boundedness and subcoercivity of $L_{Q}$]
For $v\in\dot{\calH}_{0}^{1}$, we have 
\[
\|L_{Q}v\|_{L^{2}}+\|\chf_{r\sim1}v\|_{L^{2}}\sim\|v\|_{\dot{\calH}_{0}^{1}}
\]
Moreover, the kernel of $L_{Q}:\dot{\calH}_{0}^{1}\to L^{2}$ is $\mathrm{span}_{\bbR}\{\Lambda Q,iQ\}$. 
\end{lem}

\begin{proof}
By density, we may assume $v\in\calS_{0}$. Recall that $L_{Q}=\bfD_{Q}+QB_{Q}$.
First, $QB_{Q}$ is perturbative in the sense that 
\[
\|QB_{Q}v\|_{L^{2}}\aleq\|\tfrac{1}{r^{2}}{\textstyle \int_{0}^{r}}\langle r'\rangle^{-3}|v|r'dr'\|_{L^{2}}\aleq\|\langle r\rangle^{-3}v\|_{L^{2}}\aleq\|\chf_{[r_{0}^{-1},r_{0}]}v\|_{L^{2}}+r_{0}^{-1+}\|v\|_{\dot{\calH}_{0}^{1}},
\]
for any $r_{0}\geq1$. Therefore, it suffices to show 
\[
\|\bfD_{Q}v\|_{L^{2}}+\|\chf_{r\sim1}v\|_{L^{2}}\sim\|v\|_{\dot{\calH}_{0}^{1}}.
\]
We note that the boundedness $(\aleq)$ is obvious. Henceforth, we
focus on the subcoercivity $(\ageq)$ of $\bfD_{Q}$. We use the operator
identity $\bfD_{Q}=Q\partial_{r}Q^{-1}$ and try to apply weighted
Hardy's inequality (Lemma~\ref{lem:WeightedHardy-dr}) for $Q^{-1}v$.
In the region $r\geq10$, we have $-r\partial_{r}(Q^{2})\sim Q^{2}$,
so applying Lemma~\ref{lem:WeightedHardy-dr} for $f=Q^{-1}v$ with
$\varphi=Q^{2}$ yields 
\[
\|\chf_{[r_{0},\infty)}\tfrac{1}{r}v\|_{L^{2}}^{2}\aleq\|\chf_{[r_{0},\infty)}\bfD_{Q}v\|_{L^{2}}^{2}+|v(r_{0})|^{2},
\]
provided that $r_{0}\geq10$. Averaging over $r_{0}\in[10,20]$, we
get 
\[
\|\chf_{[20,\infty)}\tfrac{1}{r}v\|_{L^{2}}^{2}\aleq\|\chf_{[10,\infty)}\bfD_{Q}v\|_{L^{2}}^{2}+\|\chf_{[10,20]}v\|_{L^{2}}^{2}.
\]
In the region $r\leq\frac{1}{10}$, we have $Q\sim1$. We choose $\varphi:(0,\frac{1}{10}]\to\bbR_{+}$
such that $r\partial_{r}\varphi=Q^{2}\langle\log_{-}r\rangle^{-2}$
and $\lim_{r\to0^{+}}\varphi(r)=0$. This $\varphi$ is very similar
to that used in the proof of logarithmic Hardy's inequality (Lemma~\ref{lem:LogarithmicHardy}).
Note that $\varphi(r)\sim\langle\log_{-}r\rangle^{-1}$ so $\varphi\aleq r\partial_{r}\varphi$
\emph{does not hold} (and hence Lemma~\ref{lem:WeightedHardy-dr}
cannot be applied) but the proof of the logarithmic Hardy inequality
applies. After averaging the boundary term, we have 
\[
\|\chf_{(0,\frac{1}{20}]}\tfrac{1}{r\langle\log_{-}r\rangle}v\|_{L^{2}}^{2}\aleq\|\chf_{(0,\frac{1}{10}]}\bfD_{Q}v\|_{L^{2}}^{2}+\|\chf_{[\frac{1}{20},\frac{1}{10}]}v\|_{L^{2}}^{2}.
\]
Therefore, we have proved that 
\[
\|\bfD_{Q}v\|_{L^{2}}^{2}+\|\chf_{r\sim1}v\|_{L^{2}}^{2}\ageq\|(\chf_{r\ll1}+\chf_{r\gg1})\tfrac{1}{r\langle\log_{-}r\rangle}v\|_{L^{2}}^{2}.
\]
Adding both sides by $\|\chf_{r\sim1}v\|_{L^{2}}^{2}$, we get 
\[
\|\bfD_{Q}v\|_{L^{2}}^{2}+\|\chf_{r\sim1}v\|_{L^{2}}^{2}\ageq\|\tfrac{1}{r\langle\log_{-}r\rangle}v\|_{L^{2}}^{2}.
\]
Combining this with $\|\bfD_{Q}v\|_{L^{2}}=\|\partial_{r}v\|_{L^{2}}+O(\|\tfrac{1}{r\langle\log_{-}r\rangle}v\|_{L^{2}})$
yields the conclusion.

For the kernel characterization, we refer to \cite[Lemma A.5]{KimKwon2020arXiv}.
The argument there still works for $m=0$ with a slight modification. 
\end{proof}
\begin{lem}[Coercivity of $L_{Q}$ at $\dot{H}^{1}$-level]
\label{lem:coercivityAppendix}Let $\psi_{1},\psi_{2}$ be elements
of the dual space $(\dot{\calH}_{0}^{1})^{\ast}$. If the $2\times2$
matrix $(a_{ij})$ defined by $a_{i1}=(\psi_{i},\Lambda Q)_{r}$ and
$a_{i2}=(\psi_{i},iQ)_{r}$ has nonzero determinant, then we have
a coercivity estimate 
\[
\|v\|_{\dot{\calH}_{0}^{1}}\aleq_{\psi_{1},\psi_{2}}\|L_{Q}v\|_{L^{2}}\aleq\|v\|_{\dot{\calH}_{0}^{1}},\qquad\forall v\in\dot{\calH}_{m}^{1}\cap\{\psi_{1},\psi_{2}\}^{\perp}.
\]
\end{lem}

\begin{proof}
We omit the proof and refer to \cite[Lemma A.6]{KimKwon2020arXiv}.
\end{proof}

\subsubsection*{The space $\dot{\protect\calH}_{2}^{1}$}

\ 

Define the space $\dot{\calH}_{2}^{1}$ by taking the completion of
$\calS_{2}$ under the norm for $2$-equivariant functions 
\[
\|v\|_{\dot{\calH}_{2}^{1}}\coloneqq\|\partial_{r}v\|_{L^{2}}+\|r^{-1}\langle\log_{+}r\rangle^{-1}v\|_{L^{2}}.
\]
Note that $\dot{\calH}_{2}^{1}$ is \emph{weaker} than $\dot{H}_{2}^{1}$
at infinity. Nevertheless, we have $\dot{\calH}_{2}^{1}\cap L^{2}=H_{2}^{1}$. 
\begin{lem}[Coercivity of $A_{Q}^{\ast}$]
For $v\in\dot{\calH}_{2}^{1}$, we have 
\begin{equation}
\|A_{Q}^{\ast}v\|_{L^{2}}\sim\|v\|_{\dot{\calH}_{2}^{1}}.\label{eq:Coercivity-AQstar-appendix}
\end{equation}
\end{lem}

\begin{proof}
By density, we may assume $v\in\calS_{2}$. From 
\begin{gather*}
A_{Q}A_{Q}^{\ast}=-\partial_{rr}-\tfrac{1}{r}\partial_{r}+\tfrac{\td V}{r^{2}},\\
\td V=(2+A_{\theta}[Q])^{2}+r^{2}Q^{2}\sim\langle r\rangle^{-2},
\end{gather*}
we have 
\[
\|A_{Q}^{\ast}v\|_{L^{2}}^{2}\sim\|\partial_{r}v\|_{L^{2}}^{2}+\|r^{-1}\langle r\rangle^{-1}v\|_{L^{2}}^{2}.
\]
Applying the logarithmic Hardy's inequality \eqref{eq:LogHardy},
we have 
\[
\|\chf_{r\geq1}r^{-1}\langle\log_{+}r\rangle^{-1}v\|_{L^{2}}\aleq\|\partial_{r}v\|_{L^{2}}+\|\chf_{r\sim1}v\|_{L^{2}}.
\]
Absorbing $\|\chf_{r\sim1}v\|_{L^{2}}$ into $\|r^{-1}\langle r\rangle^{-1}v\|_{L^{2}}$,
the conclusion follows. 
\end{proof}

\subsubsection*{The space $\dot{\protect\calH}_{1}^{2}$}

\ 

Define the space $\dot{\calH}_{1}^{2}$ by taking the completion of
$\calS_{1}$ under the norm for $1$-equivariant functions 
\[
\|v\|_{\dot{\calH}_{1}^{2}}\coloneqq\|\partial_{rr}v\|_{L^{2}}+\|r^{-1}\langle\log r\rangle^{-1}|v|_{-1}\|_{L^{2}}.
\]
It turns out that $\dot{\calH}_{1}^{2}$ is \emph{stronger} than $\dot{H}_{1}^{2}$
and $\dot{\calH}_{1}^{2}\cap L^{2}=H_{1}^{2}$. 
\begin{lem}[Comparison of $\dot{\calH}_{1}^{2}$ and $\dot{H}_{1}^{2}$]
For $v\in\calS_{1}$, we have 
\begin{equation}
\|v\|_{\dot{\calH}_{1}^{2}}\sim\|v\|_{\dot{H}_{1}^{2}}+\|\chf_{r\sim1}v\|_{L^{2}}.\label{eq:ComparisonH2H2}
\end{equation}
Moreover, one cannot remove $\|\chf_{r\sim1}v\|_{L^{2}}$ in the estimate
\eqref{eq:ComparisonH2H2}. 
\end{lem}

\begin{proof}
For the $(\ageq)$-direction, due to $\|v\|_{\dot{H}_{1}^{2}}\sim\|\partial_{+}v\|_{\dot{H}_{2}^{1}}\sim\||\partial_{+}v|_{-1}\|_{L^{2}}$
by \eqref{eq:GenHardyAppendix-1}, it suffices to establish 
\begin{equation}
\|\partial_{+}v\|_{\dot{H}_{2}^{1}}\sim\||\partial_{+}v|_{-1}\|_{L^{2}}\aleq\|v\|_{\dot{\calH}_{1}^{2}}.\label{eq:d_plus-H2}
\end{equation}
To show \eqref{eq:d_plus-H2}, we recognize that $\partial_{rr}=(\partial_{r}+\frac{1}{r})(\partial_{r}-\frac{1}{r})$
and $\partial_{r}-\frac{1}{r}$ is the radial part of $\partial_{+}$
acting on $1$-equivariant functions. We then apply Hardy's inequality
(Lemma~\ref{lem:WeightedHardy-dr}) to the operator $\partial_{r}+\frac{1}{r}=\frac{1}{r}\partial_{r}r$
with $f=r(\partial_{r}+\frac{1}{r})v$, $\varphi=\frac{1}{r^{2}}$,
$r_{1}\to0$, and $r_{2}\to\infty$. Note that the boundary term at
$r_{1}$ goes to zero as $r_{1}\to0$ because $(\partial_{r}-\frac{1}{r})v$
degenerates at the origin of order $r^{2}$ for $v\in\calS_{1}$.
As a result, we obtain 
\[
\|\tfrac{1}{r}(\partial_{r}-\tfrac{1}{r})v\|_{L^{2}}\aleq\|(\partial_{r}+\tfrac{1}{r})(\partial_{r}-\tfrac{1}{r})v\|_{L^{2}}=\|\partial_{rr}v\|_{L^{2}}.
\]
Since $\partial_{r}=(\partial_{r}+\tfrac{1}{r})-\tfrac{1}{r}$, it
is also possible to upgrade the above as 
\[
\||\partial_{+}v|_{-1}\|_{L^{2}}=\||(\partial_{r}-\tfrac{1}{r})v|_{-1}\|_{L^{2}}\aleq\|(\partial_{r}+\tfrac{1}{r})(\partial_{r}-\tfrac{1}{r})v\|_{L^{2}}=\|\partial_{rr}v\|_{L^{2}}.
\]
This shows \eqref{eq:d_plus-H2} and hence the $(\ageq)$-direction
of \eqref{eq:ComparisonH2H2}.

For the $(\aleq)$-direction, we note that 
\[
\|\partial_{rr}v\|_{L^{2}}=\|(\partial_{r}+\tfrac{1}{r})\partial_{+}v\|_{L^{2}}\aleq\||\partial_{+}v|_{-1}\|_{L^{2}}\aleq\|\partial_{+}v\|_{\dot{H}_{2}^{1}}\aleq\|v\|_{\dot{H}_{1}^{2}}.
\]
Next, by the logarithmic Hardy's inequality \eqref{eq:LogHardy},
we have 
\[
\|r^{-2}\langle\log r\rangle^{-1}v\|_{L^{2}}\aleq\|\partial_{r}(\tfrac{1}{r}v)\|_{L^{2}}+\|\chf_{r\sim1}v\|_{L^{2}}\aleq\|\tfrac{1}{r}(\partial_{r}-\tfrac{1}{r})v\|_{L^{2}}+\|\chf_{r\sim1}v\|_{L^{2}}.
\]
Using $\partial_{r}v=(\partial_{r}-\tfrac{1}{r})v+\tfrac{1}{r}v$,
we further deduce that 
\[
\|r^{-1}\langle\log r\rangle^{-1}|v|_{-1}\|_{L^{2}}\aleq\||(\partial_{r}-\tfrac{1}{r})v|_{-1}\|_{L^{2}}+\|\chf_{r\sim1}v\|_{L^{2}}\aleq\|v\|_{\dot{H}_{1}^{2}}+\|\chf_{r\sim1}v\|_{L^{2}}.
\]
This completes the proof of \eqref{eq:ComparisonH2H2}.

To see why $\|\chf_{r\sim1}v\|_{L^{2}}$ in \eqref{eq:ComparisonH2H2}
cannot be removed, consider $v(x)=(x_{1}+ix_{2})\sum_{n=1}^{N}\chi_{2^{n}}(x)$
with $N\in\bbN$ sufficiently large. Then $\|v\|_{\dot{\calH}_{1}^{2}}\ageq N$
but $\|v\|_{\dot{H}_{1}^{2}}\aleq N^{\frac{1}{2}}$. 
\end{proof}
We turn to the subcoercivity estimate. We want to control $v$, provided
that $A_{Q}v\in\dot{\calH}_{2}^{1}$. 
\begin{lem}[Boundedness and subcoercivity of $A_{Q}$]
For $v\in\dot{\calH}_{1}^{2}$, we have 
\begin{equation}
\|A_{Q}v\|_{\dot{\calH}_{2}^{1}}+\|\chf_{r\sim1}v\|_{L^{2}}\sim\|v\|_{\dot{\calH}_{1}^{2}}.\label{eq:subcoercivity-AQ}
\end{equation}
Moreover, the kernel of $A_{Q}:\dot{\calH}_{1}^{2}\to\dot{\calH}_{2}^{1}$
is $\mathrm{span}_{\bbC}\{rQ\}$. 
\end{lem}

\begin{rem}
\label{rem:AQ-subcoercivity-remark}The log weight in the definition
of $\dot{\calH}_{1}^{2}$ cannot be improved (or, removed). Indeed,
if one considers $v(x)=(x_{1}+ix_{2})\chi_{R}(x)$ for large $R$,
then $\|A_{Q}v\|_{\dot{\calH}_{2}^{1}}$ is uniformly bounded in $R$,
but both $\|r^{-2}v\|_{L^{2}}$ and $\|r^{-1}\partial_{r}v\|_{L^{2}}$
diverge as $R\to\infty$.
\end{rem}

\begin{proof}
By density, we may assume $v\in\calS_{1}$. We note that 
\begin{align*}
A_{Q}^{\ast}A_{Q} & =-\partial_{rr}-\tfrac{1}{r}\partial_{r}+\tfrac{1}{r^{2}}-Q^{2}=-\Delta_{1}-Q^{2},\\
\|Q^{2}v\| & \aleq\|\chf_{[r_{0}^{-1},r_{0}]}v\|_{L^{2}}+r_{0}^{-2+}\|v\|_{\dot{\calH}_{1}^{2}},\\
\|v\|_{\dot{\calH}_{1}^{2}} & \sim_{r_{0}}\|\Delta_{1}v\|_{L^{2}}+\|\chf_{[r_{0}^{-1},r_{0}]}v\|_{L^{2}},
\end{align*}
for $r_{0}\geq10$. Taking $r_{0}$ sufficiently large, we obtain
\[
\|v\|_{\dot{\calH}_{1}^{2}}\sim_{r_{0}}\|A_{Q}^{\ast}A_{Q}v\|_{L^{2}}+\|\chf_{[r_{0}^{-1},r_{0}]}v\|_{L^{2}}.
\]
Applying the coercivity \eqref{eq:Coercivity-AQstar-appendix} shows
the subcoercivity estimate. For the kernel characterization, notice
that $A_{Q}$ is a first-order differential operator such that $A_{Q}(rQ)=0$.
A standard ODE theory concludes the proof. 
\end{proof}
\begin{lem}[Coercivity of $A_{Q}$ at $\dot{H}^{2}$-level]
Let $\psi_{1},\psi_{2}$ be elements of $(\dot{\calH}_{1}^{2})^{\ast}$,
which is the dual space of $\dot{\calH}_{1}^{2}$. If the $2\times2$
matrix $(a_{ij})$ defined by $a_{i1}=(\psi_{i},rQ)_{r}$ and $a_{i2}=(\psi_{i},irQ)_{r}$
has nonzero determinant, then we have a coercivity estimate 
\[
\|v\|_{\dot{\calH}_{1}^{2}}\aleq_{\psi_{1},\psi_{2}}\|A_{Q}v\|_{\dot{\calH}_{2}^{1}}\aleq\|v\|_{\dot{\calH}_{1}^{2}},\qquad\forall v\in\dot{\calH}_{1}^{2}\cap\{\psi_{1},\psi_{2}\}^{\perp}.
\]
\end{lem}

\begin{proof}
We omit the proof as it can be proved in a similar manner to Lemma
\ref{lem:coercivityAppendix}. 
\end{proof}

\subsubsection*{The space $\dot{\protect\calH}_{0}^{3}$}

\ 

Define the space $\dot{\calH}_{0}^{3}$ by taking the completion of
$\calS_{0}$ under the norm for $0$-equivariant functions 
\[
\|v\|_{\dot{\calH}_{0}^{3}}\coloneqq\|\partial_{rrr}v\|_{L^{2}}+\|r^{-1}\langle\log r\rangle^{-1}|\partial_{r}v|_{-1}\|_{L^{2}}+\|r^{-1}\langle r\rangle^{-2}\langle\log r\rangle^{-1}v\|_{L^{2}}.
\]
It turns out that $\dot{\calH}_{0}^{3}$ is \emph{stronger} than $\dot{H}_{0}^{3}$
but $\dot{\calH}_{0}^{3}\cap L^{2}=H_{0}^{3}$.
\begin{lem}[Comparison of $\dot{\calH}_{0}^{3}$ and $\dot{H}_{0}^{3}$]
For $v\in\calS_{0}$, we have 
\begin{equation}
\|v\|_{\dot{\calH}_{0}^{3}}\sim\|v\|_{\dot{H}_{0}^{3}}+\|\chf_{r\sim1}v\|_{L^{2}}.\label{eq:ComparisonH3H3}
\end{equation}
Moreover, $\|\chf_{r\sim1}v\|_{L^{2}}$ cannot be removed.
\end{lem}

\begin{proof}
For the $(\ageq)$-direction, it suffices to establish 
\begin{equation}
\|\partial_{+}\partial_{+}v\|_{\dot{H}_{2}^{1}}\sim\||\partial_{+}\partial_{+}v|_{-1}\|_{L^{2}}\aleq\|v\|_{\dot{\calH}_{0}^{3}},\label{eq:d_plus_plus-H3}
\end{equation}
due to $\|v\|_{\dot{H}_{0}^{3}}\sim\|\partial_{+}\partial_{+}v\|_{\dot{H}_{2}^{1}}\sim\||\partial_{+}\partial_{+}v|_{-1}\|_{L^{2}}$.
To show \eqref{eq:d_plus_plus-H3}, we recognize that $\partial_{rrr}=(\partial_{r}+\frac{1}{r})(\partial_{r}-\frac{1}{r})\partial_{r}$
and $(\partial_{r}-\frac{1}{r})\partial_{r}$ is the radial part of
$\partial_{+}\partial_{+}$ acting on $0$-equivariant functions.
Therefore, we use Hardy's inequality for $\partial_{r}+\frac{1}{r}=\frac{1}{r}\partial_{r}r$
in the proof of \eqref{eq:ComparisonH2H2} to have 
\[
\||\partial_{+}\partial_{+}v|_{-1}\|_{L^{2}}\aleq\|\partial_{rrr}v\|_{L^{2}}.
\]
This shows \eqref{eq:d_plus_plus-H3} and hence the $(\ageq)$-direction
of \eqref{eq:ComparisonH3H3}.

For the $(\aleq)$-direction, we use the definition of the $\dot{\calH}_{1}^{2}$-norm
to have 
\[
\|\partial_{rrr}v\|_{L^{2}}+\|r^{-1}\langle\log r\rangle^{-1}|\partial_{r}v|_{-1}\|_{L^{2}}\aleq\|\partial_{r}v\|_{\dot{\calH}_{1}^{2}}\aleq\|\partial_{+}v\|_{\dot{\calH}_{1}^{2}},
\]
use weighted Hardy's inequality (Lemma \ref{lem:WeightedHardy-dr})
for $r\ageq1$ and weighted logarithmic Hardy's inequality \eqref{eq:LogHardy}
for $r\aleq1$ to have 
\[
\|r^{-1}\langle r\rangle^{-2}\langle\log r\rangle^{-1}v\|_{L^{2}}\aleq\|\langle r\rangle^{-2}\langle\log_{+}r\rangle^{-1}\partial_{r}v\|_{L^{2}}+\|\chf_{r\sim1}v\|_{L^{2}},
\]
and use \eqref{eq:ComparisonH2H2} to have 
\[
\|\partial_{+}v\|_{\dot{\calH}_{1}^{2}}\aleq\|\partial_{+}v\|_{\dot{H}_{1}^{2}}+\|\chf_{r\sim1}\partial_{+}v\|_{L^{2}}\aleq\|v\|_{\dot{H}_{0}^{3}}+\|\chf_{r\sim1}\partial_{+}v\|_{L^{2}}.
\]
Combining the above three displays yields 
\[
\|v\|_{\dot{\calH}_{0}^{3}}\aleq\|v\|_{\dot{H}_{0}^{3}}+\|\chf_{r\sim1}\partial_{+}v\|_{L^{2}}+\|\chf_{r\sim1}v\|_{L^{2}}.
\]
In order to remove $\|\chf_{r\sim1}\partial_{+}v\|_{L^{2}}=\|\chf_{r\sim1}\partial_{r}v\|_{L^{2}}$,
we use an interpolation bound 
\[
\|\chf_{r\sim1}\partial_{r}v\|_{L^{2}}\aleq\|\chf_{r\sim1}v\|_{L^{2}}+\|\chf_{r\sim1}\partial_{rrr}v\|_{L^{2}}\aleq\|\chf_{r\sim1}v\|_{L^{2}}+\|v\|_{\dot{H}_{0}^{3}}.
\]
This shows the $(\aleq)$-direction of \eqref{eq:ComparisonH3H3}.

To see why $\|\chf_{r\sim1}v\|_{L^{2}}$ in \eqref{eq:ComparisonH3H3}
cannot be removed, consider $v(x)=|x|^{2}\sum_{n=1}^{N}\chi_{2^{n}}(x)$
with $N\in\bbN$ sufficiently large. Then $\|v\|_{\dot{\calH}_{0}^{3}}\ageq N$
but $\|v\|_{\dot{H}_{0}^{3}}\aleq N^{\frac{1}{2}}$.
\end{proof}
We turn to the subcoercivity estimates of $L_{Q}$.
\begin{lem}[Boundedness and subcoercivity of $L_{Q}$ at $\dot{H}^{3}$-level]
For $v\in\dot{\calH}_{0}^{3}$, we have 
\begin{equation}
\|L_{Q}v\|_{\dot{\calH}_{1}^{2}}+\|\chf_{r\sim1}v\|_{L^{2}}\sim\|v\|_{\dot{\calH}_{0}^{3}}.\label{eq:subcoercivity-H3}
\end{equation}
Moreover, the kernel of $L_{Q}:\dot{\calH}_{0}^{3}\to\dot{\calH}_{1}^{2}$
is $\mathrm{span}_{\bbR}\{\Lambda Q,iQ\}$. 
\end{lem}

\begin{rem}
The log weight in the definition of $\dot{\calH}_{0}^{3}$ cannot
be improved (or, removed), by arguing similarly as in Remark~\ref{rem:AQ-subcoercivity-remark}
with the function $v(x)=|x|^{2}\chi_{R}(x)$. 
\end{rem}

\begin{proof}
By density, we may assume $v\in\calS_{0}$. Recall $L_{Q}=\bfD_{Q}+QB_{Q}$.
We first claim that the contribution of $QB_{Q}$ is perturbative:
\begin{equation}
\|QB_{Q}v\|_{\dot{\calH}_{1}^{2}}\aleq\|\langle r\rangle^{-5}v\|_{L^{2}}+\|\langle r\rangle^{-4}\partial_{r}v\|_{L^{2}}.\label{eq:QBQ-contribution-H3}
\end{equation}
To see this, we estimate using \eqref{eq:ComparisonH2H2} 
\[
\|QB_{Q}v\|_{\dot{\calH}_{1}^{2}}\aleq\|QB_{Q}v\|_{\dot{H}_{1}^{2}}+\|\chf_{r\sim1}QB_{Q}v\|_{L^{2}}\aleq\|\Delta_{1}(QB_{Q}v)\|_{L^{2}}+\|\chf_{r\sim1}QB_{Q}v\|_{L^{2}}.
\]
The RHS can be bounded by 
\begin{align*}
\|\Delta_{1}(QB_{Q}v)\|_{L^{2}} & =\|\partial_{r}(\partial_{r}+\tfrac{1}{r})(QB_{Q}v)\|_{L^{2}}\\
 & =\|\partial_{r}\{(\partial_{r}Q)B_{Q}v+Q\Re(Qv)\}\|_{L^{2}}\\
 & =\|((\partial_{r}-\tfrac{1}{r})\partial_{r}Q)B_{Q}v+2(\partial_{r}Q)\Re(Qv)+Q\partial_{r}\Re(Qv)\|_{L^{2}}\\
 & \aleq\|\langle r\rangle^{-5}v\|_{L^{2}}+\|\langle r\rangle^{-4}\partial_{r}v\|_{L^{2}}
\end{align*}
and 
\[
\|\chf_{r\sim1}QB_{Q}v\|_{L^{2}}\aleq\|\chf_{r\aleq1}v\|_{L^{2}}.
\]

Next, we show the $(\aleq)$-direction of \eqref{eq:subcoercivity-H3}.
By \eqref{eq:QBQ-contribution-H3}, it suffices to show 
\[
\|\bfD_{Q}v\|_{\dot{\calH}_{1}^{2}}\aleq\|v\|_{\dot{\calH}_{0}^{3}}.
\]
In view of \eqref{eq:ComparisonH2H2}, we have 
\[
\|\bfD_{Q}v\|_{\dot{\calH}_{1}^{2}}\aleq\|\bfD_{Q}v\|_{\dot{H}_{1}^{2}}+\|\chf_{r\sim1}\bfD_{Q}v\|_{L^{2}}\aleq\||\partial_{+}\bfD_{Q}v|_{-1}\|_{L^{2}}+\|\chf_{r\sim1}|v|_{-1}\|_{L^{2}}.
\]
In the region $r\ll1$, we have $\bfD_{Q}\approx\partial_{r}$, so
\[
\|\chf_{(0,1]}|\partial_{+}\bfD_{Q}v|_{-1}\|_{L^{2}}\aleq\|\chf_{(0,1]}|\partial_{+}\partial_{+}v|_{-1}\|_{L^{2}}+\|\chf_{(0,1]}|(\partial_{r}-\tfrac{1}{r})(\tfrac{A_{\theta}[Q]}{r}v)|_{-1}\|_{L^{2}}.
\]
In the region $r\gg1$, we have $\bfD_{Q}\approx\partial_{r}+\frac{2}{r}$,
so 
\begin{align*}
 & \|\chf_{[1,\infty)}|\partial_{+}\bfD_{Q}v|_{-1}\|_{L^{2}}\\
 & \aleq\|\chf_{[1,\infty)}|(\partial_{r}-\tfrac{1}{r})(\partial_{r}+\tfrac{2}{r})v|_{-1}\|_{L^{2}}+\|\chf_{[1,\infty)}|(\partial_{r}-\tfrac{1}{r})(\tfrac{2+A_{\theta}[Q]}{r}v)|_{-1}\|_{L^{2}}.
\end{align*}
One crucial observation is that $|(\partial_{r}-\tfrac{1}{r})(\partial_{r}+\tfrac{2}{r})v|_{-1}$
can be controlled by $|\partial_{+}\partial_{+}v|_{-1}$ in view of
\begin{equation}
(\partial_{r}+\tfrac{2}{r})(\partial_{r}-\tfrac{1}{r})(\partial_{r}+\tfrac{2}{r})v=(\partial_{r}+\tfrac{4}{r})(\partial_{r}-\tfrac{1}{r})\partial_{r}v=(\partial_{r}+\tfrac{4}{r})\partial_{+}\partial_{+}v\label{eq:ComparisonH3H3-identity}
\end{equation}
and Hardy's inequality: 
\begin{align*}
 & \|\chf_{[1,\infty)}|(\partial_{r}-\tfrac{1}{r})(\partial_{r}+\tfrac{2}{r})v|_{-1}\|_{L^{2}}\\
 & \aleq\|\chf_{[\frac{1}{2},\infty)}(\partial_{r}+\tfrac{2}{r})(\partial_{r}-\tfrac{1}{r})(\partial_{r}+\tfrac{2}{r})v\|_{L^{2}}+\|\chf_{[\frac{1}{2},1]}|v|_{-2}\|_{L^{2}}\\
 & \aleq\|\chf_{[\frac{1}{2},\infty)}|\partial_{+}\partial_{+}v|_{-1}\|_{L^{2}}+\|\chf_{[\frac{1}{2},1]}|v|_{-2}\|_{L^{2}}.
\end{align*}
Combining the above estimates, we arrive at 
\begin{align*}
\|\bfD_{Q}v\|_{\dot{\calH}_{1}^{2}} & \aleq\||\partial_{+}\partial_{+}v|_{-1}\|_{L^{2}}+\|\chf_{(0,1]}|(\partial_{r}-\tfrac{1}{r})(\tfrac{A_{\theta}[Q]}{r}v)|_{-1}\|_{L^{2}}\\
 & \quad+\|\chf_{[1,\infty)}|(\partial_{r}-\tfrac{1}{r})(\tfrac{2+A_{\theta}[Q]}{r}v)|_{-1}\|_{L^{2}}+\|\chf_{r\sim1}|v|_{-2}\|_{L^{2}}.
\end{align*}
Applying the estimates 
\begin{align*}
 & \|\chf_{(0,1]}|(\partial_{r}-\tfrac{1}{r})(\tfrac{A_{\theta}[Q]}{r}v)|_{-1}\|_{L^{2}}+\|\chf_{[1,\infty)}|(\partial_{r}-\tfrac{1}{r})(\tfrac{2+A_{\theta}[Q]}{r}v)|_{-1}\|_{L^{2}}\\
 & \qquad\aleq\|\langle r\rangle^{-3}\partial_{rr}v\|_{L^{2}}+\|\langle r\rangle^{-4}\partial_{r}v\|_{L^{2}}+\|\langle r\rangle^{-5}v\|_{L^{2}}
\end{align*}
and \eqref{eq:d_plus_plus-H3}, we get 
\begin{align*}
\|\bfD_{Q}v\|_{\dot{\calH}_{1}^{2}} & \aleq\||\partial_{+}\partial_{+}v|_{-1}\|_{L^{2}}+\|\langle r\rangle^{-3}\partial_{rr}v\|_{L^{2}}+\|\langle r\rangle^{-4}\partial_{r}v\|_{L^{2}}+\|\langle r\rangle^{-5}v\|_{L^{2}}\aleq\|v\|_{\dot{\calH}_{0}^{3}}.
\end{align*}
Combining this with \eqref{eq:QBQ-contribution-H3}, the $(\aleq)$-direction
of \eqref{eq:subcoercivity-H3} is proved.

Next, we show the $(\ageq)$-direction of \eqref{eq:subcoercivity-H3}.
By \eqref{eq:ComparisonH3H3}, we have 
\[
\|v\|_{\dot{\calH}_{0}^{3}}\aleq\|v\|_{\dot{H}_{0}^{3}}+\|\chf_{r\sim1}v\|_{L^{2}}\aleq\||\partial_{+}\partial_{+}v|_{-1}\|_{L^{2}}+\|\chf_{r\sim1}v\|_{L^{2}}.
\]
Thus we aim to control $\||\partial_{+}\partial_{+}v|_{-1}\|_{L^{2}}$
in terms of $\|\bfD_{Q}v\|_{\dot{\calH}_{1}^{2}}$. Again, we separately
consider the regions $r\leq1$ and $r\geq1$. In the region $r\leq1$,
\[
\|\chf_{(0,1]}|\partial_{+}\partial_{+}v|_{-1}\|_{L^{2}}\aleq\|\chf_{(0,1]}|\partial_{+}\bfD_{Q}v|_{-1}\|_{L^{2}}+\|\chf_{(0,1]}|(\partial_{r}-\tfrac{1}{r})(\tfrac{A_{\theta}[Q]}{r}v)|_{-1}\|_{L^{2}}.
\]
In the region $r\geq1$, we use \eqref{eq:ComparisonH3H3-identity}
with Hardy's inequality that 
\begin{align*}
 & \|\chf_{[1,\infty)}|\partial_{+}\partial_{+}v|_{-1}\|_{L^{2}}\\
 & \aleq\|\chf_{[\frac{1}{2},\infty)}|(\partial_{r}-\tfrac{1}{r})(\partial_{r}+\tfrac{2}{r})v|_{-1}\|_{L^{2}}+\|\chf_{[\frac{1}{2},1]}|v|_{-2}\|_{L^{2}}\\
 & \aleq\|\chf_{[\frac{1}{2},\infty)}|\partial_{+}\bfD_{Q}v|_{-1}\|_{L^{2}}+\|\chf_{[\frac{1}{2},1]}|v|_{-2}\|_{L^{2}}+\|\chf_{[\frac{1}{2},\infty)}|(\partial_{r}-\tfrac{1}{r})(\tfrac{A_{\theta}[Q]}{r}v)|_{-1}\|_{L^{2}}.
\end{align*}
Therefore, 
\begin{align*}
\|v\|_{\dot{\calH}_{0}^{3}} & \aleq\||\partial_{+}\partial_{+}v|_{-1}\|_{L^{2}}+\|\chf_{r\sim1}v\|_{L^{2}}\\
 & \aleq\||\partial_{+}\bfD_{Q}v|_{-1}\|_{L^{2}}+\|\langle r\rangle^{-3}\partial_{rr}v\|_{L^{2}}+\|\langle r\rangle^{-4}\partial_{r}v\|_{L^{2}}+\|\langle r\rangle^{-5}v\|_{L^{2}}\\
 & \aleq\|\bfD_{Q}v\|_{\dot{\calH}_{1}^{2}}+\|\langle r\rangle^{-3}\partial_{rr}v\|_{L^{2}}+\|\langle r\rangle^{-4}\partial_{r}v\|_{L^{2}}+\|\langle r\rangle^{-5}v\|_{L^{2}},
\end{align*}
where in the last inequality we used \eqref{eq:d_plus-H2}. Combining
this with \eqref{eq:QBQ-contribution-H3}, we have proved that 
\[
\|L_{Q}v\|_{\dot{\calH}_{1}^{2}}+\|\langle r\rangle^{-3}\partial_{rr}v\|_{L^{2}}+\|\langle r\rangle^{-4}\partial_{r}v\|_{L^{2}}+\|\langle r\rangle^{-5}v\|_{L^{2}}\ageq\|v\|_{\dot{\calH}_{0}^{3}}.
\]
It now remains to replace the perturbative terms by $\|\chf_{r\sim1}v\|_{L^{2}}$.
For this, we use 
\[
\|\langle r\rangle^{-3}\partial_{rr}v\|_{L^{2}}+\|\langle r\rangle^{-4}\partial_{r}v\|_{L^{2}}+\|\langle r\rangle^{-5}v\|_{L^{2}}\aleq\|\chf_{[r_{0}^{-1},r_{0}]}|v|_{-2}\|_{L^{2}}+r_{0}^{-\frac{1}{2}}\|v\|_{\dot{\calH}_{0}^{3}}
\]
and choose $r_{0}$ large enough to obtain 
\[
\|L_{Q}v\|_{\dot{\calH}_{1}^{2}}+\|\chf_{r\sim1}|v|_{-2}\|_{L^{2}}\ageq\|v\|_{\dot{\calH}_{0}^{3}}.
\]
Finally applying an interpolation bound 
\[
\|\chf_{r\sim1}|v|_{-2}\|_{L^{2}}\aleq\|\chf_{r\sim1}v\|_{L^{2}}+\|\chf_{r\sim1}v\|_{L^{2}}^{\frac{1}{3}}\|\chf_{r\sim1}\partial_{rrr}v\|_{L^{2}}^{\frac{2}{3}}
\]
completes the proof of the $(\ageq)$-direction of \eqref{eq:subcoercivity-H3}.

The kernel characterization can be proved by a slight modification
of the argument in \cite[Lemma A.13]{KimKwon2020arXiv}. 
\end{proof}
\begin{lem}[Coercivity of $L_{Q}$ at $\dot{H}^{3}$-level]
Let $\psi_{1},\psi_{2}$ be elements of the dual space $(\dot{\calH}_{0}^{3})^{\ast}$.
If the $2\times2$ matrix $(a_{ij})$ defined by $a_{i1}=(\psi_{i},\Lambda Q)_{r}$
and $a_{i2}=(\psi_{i},iQ)_{r}$ has nonzero determinant, then we have
a coercivity estimate 
\[
\|v\|_{\dot{\calH}_{0}^{3}}\aleq_{\psi_{1},\psi_{2}}\|L_{Q}v\|_{\dot{\calH}_{1}^{2}}\aleq\|v\|_{\dot{\calH}_{0}^{3}},\qquad\forall v\in\dot{\calH}_{0}^{3}\cap\{\psi_{1},\psi_{2}\}^{\perp}.
\]
\end{lem}

\begin{proof}
We omit the proof and refer to \cite[Lemma A.15]{KimKwon2020arXiv}. 
\end{proof}

\subsubsection*{Interpolation and $L^{\infty}$ estimates}
\begin{lem}[Interpolation estimates]
\label{lem:interpolation} Let $v_{2}$ be a radial function and
$v_{1}\in H_{1}^{2}$. We have 
\begin{align}
\|v_{2}\|_{L^{\infty-}} & \aleq\|v_{2}\|_{L^{2}}^{0+}\|\partial_{r}v_{2}\|_{L^{2}}^{1-},\label{eq:interpolation-1}\\
\||v_{1}|_{-1}\|_{L^{2}} & \aleq\|v_{1}\|_{L^{2}}^{\frac{1}{2}}\|v_{1}\|_{\dot{\calH}_{1}^{2}}^{\frac{1}{2}}.\label{eq:interpolation-2}
\end{align}
\end{lem}

\begin{proof}
For the estimate \eqref{eq:interpolation-1}, we will in fact show
\[
\|v_{2}\|_{L^{2p}}\aleq_{p}\|v_{2}\|_{L^{2}}^{\frac{1}{p}}\|\partial_{r}v_{2}\|_{L^{2}}^{1-\frac{1}{p}},\qquad\forall p\in[1,\infty).
\]
As the case $p=1$ is immediate, it suffices to show for $p\in[2,\infty)$
by interpolation. Applying the FTC to the expression $\partial_{r}|v_{2}|^{p}(r)\aleq_{p}|v_{2}|^{p-1}|\partial_{r}v_{2}|$
and using Minkowski's inequality, we get 
\begin{align*}
\||v_{2}|^{p}\|_{L^{2}} & \aleq_{p}\|{\textstyle \int_{r}^{\infty}}|v_{2}|^{p-1}|\partial_{r}v_{2}|dr'\|_{L^{2}(rdr)}\\
 & \aleq_{p}{\textstyle \int_{0}^{\infty}}\|\chf_{r\leq r'}\|_{L^{2}(rdr)}|v_{2}|^{p-1}|\partial_{r}v_{2}|dr'\\
 & \aleq_{p}{\textstyle \int_{0}^{\infty}}|v_{2}|^{p-1}|\partial_{r}v_{2}|\,r'dr'\\
 & \aleq_{p}\||v_{2}|^{p-1}\|_{L^{2}}\|\partial_{r}v_{2}\|_{L^{2}}.
\end{align*}
Therefore, 
\[
\|v_{2}\|_{L^{2p}}^{p}\aleq_{p}\|v_{2}\|_{L^{2(p-1)}}^{p-1}\|\partial_{r}v_{2}\|_{L^{2}}\aleq_{p}\|v_{2}\|_{L^{2}}^{\frac{1}{p-1}}\|v_{2}\|_{L^{2p}}^{p(1-\frac{1}{p-1})}\|\partial_{r}v_{2}\|_{L^{2}}.
\]
Rearranging this completes the proof of \eqref{eq:interpolation-1}.

The estimate \eqref{eq:interpolation-2} follows from 
\[
\||v_{1}|_{-1}\|_{L^{2}}\aleq\|v_{1}\|_{\dot{H}_{1}^{1}}\aleq\|v_{1}\|_{L^{2}}^{\frac{1}{2}}\|v_{1}\|_{\dot{H}_{1}^{2}}^{\frac{1}{2}}\aleq\|v_{1}\|_{L^{2}}^{\frac{1}{2}}\|v_{1}\|_{\dot{\calH}_{1}^{2}}^{\frac{1}{2}},
\]
where in the last inequality we used \eqref{eq:ComparisonH2H2}. 
\end{proof}
\begin{lem}[Weighted $L^{\infty}$-estimates]
\label{lem:Weighted-Linfty}Let $v_{m}$ be $m$-equivariant functions,
$m\in\{0,1,2\}$. Near the origin, we have 
\begin{align*}
\|\chf_{(0,1]}v\|_{L^{\infty}} & \aleq\|v\|_{\dot{\calH}_{0}^{3}},\\
\|\chf_{(0,1]}v_{1}\|_{L^{\infty}} & \aleq\|v_{1}\|_{\dot{\calH}_{1}^{2}},\\
\|\chf_{(0,1]}v_{2}\|_{L^{\infty}} & \aleq\|v_{2}\|_{\dot{\calH}_{2}^{1}}.
\end{align*}
Near infinity, we have 
\begin{align*}
\|\chf_{[1,\infty)}v\|_{L^{\infty}} & \aleq\|v\|_{\dot{\calH}_{0}^{1}},\\
\|\chf_{[1,\infty)}\langle\log_{+}r\rangle^{-1}|v|_{-2}\|_{L^{\infty}} & \aleq\|v\|_{\dot{\calH}_{0}^{3}},\\
\|\chf_{[1,\infty)}v_{1}\|_{L^{\infty}} & \aleq\|v_{1}\|_{L^{2}}^{\frac{1}{2}}\|v_{1}\|_{\dot{\calH}_{1}^{2}}^{\frac{1}{2}},\\
\|\chf_{[1,\infty)}\langle\log_{+}r\rangle^{-1}|v_{1}|_{-1}\|_{L^{\infty}} & \aleq\|v_{1}\|_{\dot{\calH}_{1}^{2}},\\
\|\chf_{[1,\infty)}\langle\log_{+}r\rangle^{-1}v_{2}\|_{L^{\infty}} & \aleq\|v_{2}\|_{\dot{\calH}_{2}^{1}}.
\end{align*}
\end{lem}

\begin{proof}
Near the origin, $L^{\infty}$-estimates for $v$ and $v_{1}$ follow
from $\dot{\calH}_{0}^{3}\hookrightarrow H_{\mathrm{loc}}^{3}$ and
$\dot{\calH}_{1}^{2}\hookrightarrow H_{\mathrm{loc}}^{2}$, and the
Sobolev embeddings. For $v_{2}$, we use the FTC argument: 
\[
\|\chf_{(0,1]}|v_{2}|^{2}\|_{L^{\infty}}\aleq\int_{0}^{1}|\tfrac{1}{r}v_{2}||\partial_{r}v_{2}|r'dr'\aleq\|\chf_{(0,1]}\tfrac{1}{r}v_{2}\|_{L^{2}}\|\chf_{(0,1]}\partial_{r}v_{2}\|_{L^{2}}\aleq\|v_{2}\|_{\dot{\calH}_{2}^{1}}^{2}.
\]

Near infinity, all the estimates except $\|\chf_{[1,\infty)}v_{1}\|_{L^{\infty}}$
follow from the FTC arguments and the definitions of our adapted function
spaces. We omit their proofs. For $\|\chf_{[1,\infty)}v_{1}\|_{L^{\infty}}$,
we use \eqref{eq:HardySobolevAppendix-1} and \eqref{eq:interpolation-2}
instead: 
\[
\|\chf_{[1,\infty)}v_{1}\|_{L^{\infty}}\aleq\|v_{1}\|_{\dot{H}_{1}^{1}}\aleq\|v_{1}\|_{L^{2}}^{\frac{1}{2}}\|v_{1}\|_{\dot{\calH}_{1}^{2}}^{\frac{1}{2}}.
\]
This completes the proof.
\end{proof}
\bibliographystyle{abbrv}
\bibliography{References}

\end{document}